\documentclass[11pt]{article}

\usepackage{amsmath, amsthm}
\usepackage{amsmath, amsfonts}
\usepackage{amsmath, amssymb}
\usepackage{amsmath}
\usepackage{graphics}
\usepackage{graphicx}
\usepackage{color}
\usepackage{hyperref}
\usepackage[all]{xy}

\newtheorem{theorem}{Theorem}[subsection]
\newtheorem{definition}[theorem]{Definition}
\newtheorem{definition-lemma}[theorem]{Definition/Lemma}
\newtheorem{definition-explanation}[theorem]{Definition/Explanation}
\newtheorem{explanation-definition}[theorem]{Explanation/Definition}
\newtheorem{definition-fact}[theorem]{Definition/Fact}
\newtheorem{definition-notation}[theorem]{Definition/Notation}
\newtheorem{definition-conjecture}[theorem]{Definition/Conjecture}
\newtheorem{lemma}[theorem]{Lemma}
\newtheorem{lemma-definition}[theorem]{Lemma/Definition}
\newtheorem{proposition}[theorem]{Proposition}

\newtheorem{remark}[theorem]{\it Remark}
\newtheorem{remark-notation}[theorem]{\it Remark/Notation}

\newtheorem{application-lemma}[theorem]{Application/Lemma}

\newtheorem{convention}[theorem]{\it Convention}

\newtheorem{example}[theorem]{Example}
\newtheorem{example-definition}[theorem]{Example/Definition}

\newtheorem{definition-prototype}[theorem]{Definition-Prototype}

\newtheorem{terminology}[theorem]{\it Terminology}
\newtheorem{lesson}[theorem]{Lesson}

\numberwithin{equation}{subsection}

\newtheorem{stheorem}{Theorem}[section]

\newtheorem{sdefinition-lemma}[stheorem]{Definition/Lemma}
\newtheorem{sdefinition-explanation}[stheorem]{Definition/Explanation}
\newtheorem{sexplanation-definition}[stheorem]{Explanation/Definition}
\newtheorem{sdefinition-fact}[stheorem]{Definition/Fact}
\newtheorem{sdefinition-notation}[stheorem]{Definition/Notation}
\newtheorem{sdefinition-conjecture}[stheorem]{Definition/Conjecture}

\newtheorem{slemma-definition}[stheorem]{Lemma/Definition}

\newtheorem{sremark}[stheorem]{\it Remark}
\newtheorem{sremark-notation}[stheorem]{\it Remark/Notation}

\newtheorem{sapplication-lemma}[stheorem]{Application/Lemma}

\newtheorem{sansatz}[stheorem]{Ansatz}

\newtheorem{sexample}[stheorem]{Example}
\newtheorem{sexample-definition}[stheorem]{Example/Definition}

\newtheorem{sdefinition-prototype}[stheorem]{Definition-Prototype}

\newtheorem{sstheorem}{Theorem}[subsubsection]
\newtheorem{ssdefinition}[sstheorem]{Definition}
\newtheorem{ssdefinition-lemma}[sstheorem]{Definition/Lemma}
\newtheorem{ssdefinition-explanation}[sstheorem]{Definition/Explanation}
\newtheorem{ssexplanation-definition}[sstheorem]{Explanation/Definition}
\newtheorem{ssdefinition-fact}[sstheorem]{Definition/Fact}
\newtheorem{ssdefinition-notation}[sstheorem]{Definition/Notation}
\newtheorem{ssdefinition-conjecture}[sstheorem]{Definition/Conjecture}
\newtheorem{sslemma}[sstheorem]{Lemma}
\newtheorem{sslemma-definition}[sstheorem]{Lemma/Definition}
\newtheorem{ssproposition}[sstheorem]{Proposition}

\newtheorem{ssremark}[sstheorem]{\it Remark}
\newtheorem{ssremark-notation}[sstheorem]{\it Remark/Notation}

\newtheorem{ssapplication-lemma}[sstheorem]{Application/Lemma}

\newtheorem{ssexample}[sstheorem]{Example}
\newtheorem{ssexample-definition}[sstheorem]{Example/Definition}

\newtheorem{ssnotation}[sstheorem]{Notation}

\newtheorem{ssdefinition-prototype}[sstheorem]{Definition-Prototype}

\newcommand{\Br}{\mbox{\it Br}\,}

\newcommand{\Coh}{\mbox{\it Coh}\,}

\newcommand{\Coker}{\mbox{\it Coker}\,}

\newcommand{\Der}{\mbox{\it Der}\,}

\newcommand{\End}{\mbox{\it End}\,}
 \newcommand{\smallEnd}{\mbox{\small\it End}\,}
 \newcommand{\scriptsizeEnd}{\mbox{\scriptsize\it End}\,}
\newcommand{\Endsheaf}{\mbox{\it ${\cal E}\!$nd}\,}

\newcommand{\Ext}{\mbox{\rm Ext}\,}

\newcommand{\GL}{\mbox{\it GL}}

\newcommand{\Homsheaf}{\mbox{\it ${\cal H}$om}\,}

\newcommand{\Id}{\mbox{\it Id}\,}

\newcommand{\Image}{\mbox{\it Im}\,}

\newcommand{\Inn}{\mbox{\it Inn}\,}
\newcommand{\Innsheaf}{\mbox{\it ${\cal I}\!$nn}\,}

\newcommand{\Ker}{\mbox{\it Ker}\,}

\newcommand{\Map}{\mbox{\it Map}\,}

\newcommand{\ModCategory}{\mbox{\it ${\cal M}$\!od}\,}

\newcommand{\Quot}{\mbox{\it Quot}\,}

\newcommand{\Real}{\mbox{\it Re}\,}

\newcommand{\slLie}{\mbox{\it sl}\,}

\newcommand{\Space}{\mbox{\it Space}\,}

\newcommand{\Spec}{\mbox{\it Spec}\,}
 \newcommand{\boldSpec}{\mbox{\it\bf Spec}\,}
 
 \newcommand{\smallSpec}{\mbox{\small\it Spec}\,}

\newcommand{\Supp}{\mbox{\it Supp}\,}
 
\newcommand{\Sym}{\mbox{\it Sym}}
 \newcommand{\scriptsizeSym}{\mbox{\scriptsize\it Sym}}

 \newcommand{\scriptsizeZss}{\mbox{\scriptsize\it $Z$-$ss$}\,}
 \newcommand{\tinyZss}{\mbox{\tiny\it $Z$-$ss$}\,}

\newcommand{\scriptsizedegree}{\mbox{\scriptsize\it deg}\,}

\newcommand{\determinant}{\mbox{\it det}\,}
\newcommand{\dimm}{\mbox{\it dim}\,}

\newcommand{\scriptsizenoncommutative}{\mbox{\scriptsize\rm
                                             noncommutative}}

\newcommand{\pr}{\mbox{\it pr}}

\newcommand{\redscriptsize}{\mbox{\scriptsize\rm red}\,}
\newcommand{\redtiny}{\mbox{\tiny\rm red}\,}

\newcommand{\smoothscriptsize}{\mbox{\scriptsize\it smooth}\,}

\newcommand{\longrightaarrow}{\longrightarrow\hspace{-3ex}\longrightarrow}

\begin{document}

\enlargethispage{24cm}

\begin{titlepage}

$ $

\vspace{-1.5cm} % Re: -1.5cm for PC; -2.5cm for UT-Math-system

\noindent\hspace{-1cm}
\parbox{6cm}{\small April 2014}\
   \hspace{6cm}\
   \parbox[t]{6cm}{yymm.nnnn [math.DG] \\
                D(11.1): D-brane, matrix dg,  \\
                $\mbox{\hspace{3.8em}} $ 
			 	 map from matrix-brane
				}

%\vspace{2em}
\vspace{2cm}

%title
\centerline{\large\bf
  D-branes and Azumaya/matrix noncommutative differential geometry,}
\vspace{1ex}
\centerline{\large\bf
 I: D-branes as fundamental objects in string theory  and differentiable maps}
\vspace{1ex}
\centerline{\large\bf
 from Azumaya/matrix manifolds with a fundamental module to real manifolds}
% \vspace{1ex}
% \centerline{\large\bf
%   -- foundations,  and examples in symplectic geometry and string theory}
% \vspace{1ex}
% \centerline{\large\bf
%  ?????? }
% end-title

\bigskip

\centerline{({\it
  In memory of Professor William P.\ Thurston})}

%\bigskip
\vspace{2.4em}
%\vspace{3em}

%authors-'n-addresses
\centerline{\large
  Chien-Hao Liu    % \,, \hspace{1ex}
                                % Cumrun Vafa\,,   
            \hspace{1ex} and \hspace{1ex}
  Shing-Tung Yau
}

%\vspace{2em}
\vspace{3em}
%\vspace{4em}

%abstract%
\begin{quotation}
\centerline{\bf Abstract}

\vspace{0.3cm}

\baselineskip 12pt  %13pt for [12pt] style
{\small
 We consider D-branes in string theory and address the issue of 
   how to describe them mathematically as a fundamental object (as opposed to a solitonic object)
   of string theory in the realm in differential and symplectic geometry. 
 The notion of continuous maps, $k$-times differentiable maps, and smooth maps
    from an Azumaya/matrix manifold with a fundamental module  to a (commutative) real manifold $Y$
  is developed.
 Such maps are meant to describe D-branes or matrix branes in string theory
   when these branes are light and soft with only small enough or even zero brane-tension.
  When $Y$ is a symplectic manifold
   (resp.\  a Calabi-Yau manifold; a $7$-manifold with $G_2$-holonomy;
                   a manifold with an almost complex structure $J$), 	
   the corresponding notion of
      Lagrangian maps
    (resp.\ special Lagrangian maps; associative maps, coassociative maps; $J$-holomorphic maps)	
	 are introduced.
 Indicative examples linking to symplectic geometry and string theory are given.
 This provides us with a language and part of the foundation required
  to study themes, new or old, in symplectic geometry and string theory, including
  (1) $J$-holomorphic D-curves (with or without boundary),
  (2) quantization and dynamics of D-branes in string theory,
  (3) a definition of Fukaya category guided by Lagrangian maps from Azumaya manifolds
	         with  a fundamental module with a connection,   
  (4) a theory of fundamental matrix strings or D-strings, and
  (5) the nature of Ramond-Ramond fields in a space-time.
  The current note D(11.1) is the symplectic/differential-geometric counterpart
    of the more algebraic-geometry-oriented first two notes
	  D(1) ([L-Y1]) (arXiv:0709.1515 [math.AG])  and
	  D(2) ([L-L-S-Y], with Si Li and Ruifang Song) (arXiv:0809.2121 [math.AG])
	 in this project.
 } % endsmall
\end{quotation}

%\smallskip
\bigskip
%\vspace{2em}

\baselineskip 12pt
{\footnotesize
\noindent
{\bf Key words:} \parbox[t]{14cm}{D-brane; 
      $C^k$-scheme; Azumaya manifold, matrix manifold;  
      $C^k$-map; special Lagrangian map;\\
	  Chan-Paton sheaf, fundamental module; connection; push-forward.
 }} %end-footnotesize

%\smallskip
 \bigskip

\noindent {\small MSC number 2010: 58A40, 14A22, 81T30; 51K10, 16S50, 46L87.
} % end-small

% \smallskip
\bigskip

\baselineskip 10pt
% Re: 11pt for [11pt] style; 12pt for [12pt] style
{\scriptsize
\noindent{\bf Acknowledgements.}
We thank
 Cumrun Vafa for stringy consultations in the early stage of the work;
 Eduardo J.\ Dubuc, Dominic Joyce, Anders Kock, Ieke Moerdijk, Gonzalo E.\ Reyes
     for their works on synthetic differential geometry and $C^{\infty}$-algebraic geometry  and
 Michel Dubois-Violette, Richard Kerner, John Madore, Thierry Masson, Emmanuel S\'{e}ri\'{e}
     for their works on noncommutative differential geometry on endomorphism algebras
   that together provide the two beginning building blocks of this note (cf.\ Sec.~2 and Sec.~4).	
C.-H.L.\ thanks in addition
 Yng-Ing Lee, Katrin Wehrheim
   for discussions on related symplectic issues during the long brewing years of the current work;
 Siu-Cheong Lau
   for topic courses in SYZ mirror symmetry and symplectic geometry, spring 2013 and spring 2014,
   and the in-or-after-class discussions that tremendously influenced his understanding;
 Girma Hailu, Pei-Ming Ho, Shiraz Minwalla
   for discussions/literature guide related to open D-strings;
 Si Li, Robert Myers, Eric Sharpe
   for communications;
 Yaiza Canzani, Girma Hailu, Hiro Lee Tanaka and
         Mboyo Esole, Stefan Patrikis, Mathew Reece
   for other topic courses,  fall 2013 and spring 2014;		
 Ling-Miao Chou for moral support and comments that improve the illustrations.
% S.-T.Y.\ thanks in addition
% Department of Mathematics at National Taiwan University, fall and winter 2013,
% ????????????, spring 2014,
%   for hospitality,
%   while part of the manuscript is in preparation.
The project is supported by NSF grants DMS-9803347 and DMS-0074329.
} %endscriptsize

\end{titlepage}

\newpage

\begin{titlepage}

$ $

\vspace{2em}
%\vspace{12em} % \vspace{4em}

\centerline{\small\it
 Chien-Hao Liu dedicates this note to the memory of}
\centerline{\small\it
 his adviser Prof.\ William P.\ Thurston (1946--2012)}
\centerline{\small\it during his Princeton and Berkeley years.}

\vspace{3.5em}

\baselineskip 11pt

{\footnotesize
\noindent
 (From C.H.L.)
 In retrospect, I went to Princeton University at late 1980s completely in the wrong mind set.
 After the depression due to my father's passing away in my senior year at college and
     then the break of study due to the two-year military service,
   I arrived there practically with an empty soul and an empty mind.
 In my senior year in college I suddenly felt so unbearable for pure mathematics:
 There are so many sufferings in this world;
  how could one keep one's cool and address problems like,
    ``Are there integer solutions to the equation $x^n+y^n=z^n$ for $n \ge 3$?"
    while keeping a humane heart.	
 Something between mathematics and physics should give one
    more meaning to devote one's life to:
   at least it's more connecting to the real world.
 Yet I had no idea what specific field I wanted to be in.  	
 I went to Professor {\it Neil Turok}'s course on cosmology
    and got extremely fascinated by his view of the universe
	  and the way he did research:
	  a combination of mathematics, theoretical physics, observed data,
	    and computer simulations
  --- what could be more perfect than such a combination?
 Despite his being extremely friendly to me,
    I felt there was too much to catch up if I were to follow him.
 I also went to other topic courses in mathematics and physics,
   though clearly most faculty at that time seemed more like
     communicating to themselves rather than to their audience.
 After all, this is Princeton; you are expected to be a genius 
   and  a genius doesn't require anyone to teach him/her anything.
 Yet, I am not a genius.   
  
 I was in such an embarrassing situation
   when stepping into Prof.\ {\it Thurston}'s topic course
   on pseudo-Anosov flow and train tracks on surfaces
       and later on geometry and topology of $3$-manifolds and orbifolds.
 His lively lectures make doing mathematics almost like entertaining.	
 His air, not as a Fields medalist but rather as a big boy who is curious about everything,
   stood out very uniquely even at Princeton.
 My research result in the first year with him began with a conversation in his office.
 After my explaining to him Prof.\ John A.\ Wheeler's interpretation of
   Einstein's vision of quantum gravity via geometrodynamics,
 Prof.\ Thurston looked struck with something and then started to explain to me on the blackboard
   Gromov's $\varepsilon$-approximation topology between metric spaces
      (i.e.\ topological spaces with a distance function)
   -- originally developed in part for the study hyperbolic groups and their limits.
 For example, on the space of isometry classes $[(S, ds^2)]]$
    of orientable Riemannian surfaces --- without fixing the genus --- equipped with such topology,
   a neighborhood of an $[(S_0, ds^2_0)]$ will contain $[(S,ds^2)]$'s of any higher genus,
   resembling various quantum fluctuations of the topology(!) of $S_0$.	
 With a few more discussions with him,  I was finally able to prove that
 
   \begin{itemize}
     \item[$\cdot$] \parbox[t]{14.6cm}{\it
	  {\bf [conformal deformation]} \hspace{1em}
	   Let $M$ be a closed smooth $n$-manifold   and
    	   $ds_0$, $ds_1$ be two Riemannian metrics on $M$.
       Let $\varepsilon_0>0$ be any positive number.
	   Then there exists a smooth function $f$ on $M$ such that $(M,ds^2_0)$
	     and $M,e^{2f} ds^2_1)$ are $\varepsilon_0$-close to each other.	
	  }
	
	\smallskip
	
	\item[$\cdot$] \parbox[t]{14.6cm}{\it
	   {\bf [conformal class is dense]} \hspace{1em}
         Let $M$ be a closed smooth $n$-manifold. Then each conformal class is dense
         in the space of isometry classes of of Riemannian metrics on $M$,
		 endowed with the Gromov's $\varepsilon$-approximation topology.		
	 }
   \end{itemize}
  
  \noindent
  The second statement is a re-statement of the first.
  I would expect a possibility to have joint advisors at Princeton for my graduate study:
    Prof.\ Thurston from mathematics side and Prof.~Turok from the physics side
  as Prof.~Turok happened to be exploring applications of hyperbolic geometry to cosmology
    at that time.
  Unfortunately, as if life has its own way/plan/destiny, I didn't have that luck. 	
  
  A year after Prof.\ Thurston took me as one of his students, he moved to M.S.R.I.\ at Berkeley.
  Life completely changed and he became extremely busy with his duty as the director
   and we became detached.
  Yet, without my slightest anticipation of it when preparing for the move,
    it is a whole miraculous world out there waiting for me at Berkeley from other sources.
  But that's for another story.
 
  Most people we came across in life
       will be eventually just passers-by to us,
               leaving at best only some vague and shallow memories.
  Only a very limited few left some imprint on us throughout our life,
    whether we think about them or keep in touch with them or whether they still exist.  	
  Going to Princeton is my choice and also my honor to be chosen.
  Like it or not,
  the picturesque campus and surroundings and
    the surreal feeling of walking the path Einstein may have walked between GC and IAS
    are unique enough to deserve it.
 Transforming to Berkeley is not my choice; yet it is there I found a world belonging to me.
 Somehow, Prof.~Thurston, a unique mathematician, happened to bring good to his student
  in such a unique, unconventional, and quite unexpected way!
     
 Prof.\ Thurston is a very visionary mathematician, who sees pictures in his mind
   before stating them or proving them;
 this note D(11.1) is thus made as pictorial and diagrammatic as can be for a dedication to him.
 } % end-footnotesize

\end{titlepage}

%paper

\newpage
$ $

\vspace{-3em}
% \vspace{-4em}  % Re: -4cm for PC; -6cm for UT-Math-system

%short heading
\centerline{\sc
 D-Brane and NCDG I: Differentiable Maps from Azumaya/Matrix Manifolds
 } %

\vspace{2em}

% \baselineskip 14pt  %Re: 14pt for [11pt] style
                                      %Re: 15pt for [12pt] style.

\begin{flushleft}
{\Large\bf 0. Introduction and outline}
\end{flushleft}
A {\it D-brane}, in full name: {\it Dirichlet brane} , in string theory is by definition
(i.e.\ by the very word `Dirichlet') a boundary condition for the end-points of open strings.
{From} the viewpoint of the field theory on the open-string world-sheet aspect,
 it is a boundary state in the $2$-dimensional conformal field theory with boundary.
{From} the viewpoint of open string target space(-time) $Y$,
 it is a cycle or a union of submanifolds $Z$ in $Y$ with a gauge bundle(on $Z$)
  that carries the Chan-Paton index for the end-points of open strings.
For the second viewpoint, Polchinski recognized in 1995 in [Pol2]
  that a D-brane is indeed a source of the Ramond-Ramond fields on $Y$
    created by the oscillations of closed superstrings in $Y$.
In particular,
  in the region of the Wilson's theory-space for string theory
   where the tension $g_s$ of the fundamental string is small and the tension of D-branes is large,
 D-branes can be identified with the solitonic/black branes studied earlier\footnote{See
                                                [D-K-L] for a review and more references.}
 in supergravity and (target) space-time aspect of superstrings.
This recognition is so fundamental that it gave rise to the second revolution of string theory.

However, in the region of Wilson's theory-space of string theory where the D-brane tension is small,
 D-branes stand in an equal footing with strings as fundamental objects.
In this region, they are soft and can move around and vibrate, just like a fundamental string can,
  in the space-time $Y$.
Thus, a D-brane in this case is better described as a map from the D-brane world-volume $X$ to $Y$.
Such non-solitonic aspect was already taken in the original works,
  [P-C] (1989) of {\it Joseph Polchinski} and {\it Yunhai Cai}  and
  [D-L-P]  (1989) of {\it Jin Dai},  {\it Robert Leigh}, and {\it Joseph Polchinski},
 that introduced the notion of D-branes to string theory.

Something novel and mysterious at the first sight happens
 when a collection of D-branes in space-time coincide:
 (cf.\ [Pol3], [Pol4], [Wi6])
\begin{itemize}
 \item[$\cdot$] \parbox[t]{37.4em}{\it
  {\bf [enhancement of scalar field on D-brane world-volume]}\hspace{1em}
   When a collection of D-branes in space-time coincide,
    the open-string-induced massless spectrum on the world-volume of the D-brane
	 is enhanced.
  In particular,
    the gauge field is enhanced to one with a larger gauge group  and
    the scalar field that describes the deformations of the brane in space-time
     is enhanced to one that is matrix-valued.}
\end{itemize}
This leads to the following question:
 \begin{itemize}
  \item[{\bf Q.}] \parbox[t]{37.4em}{\it {\bf [D-brane]}\hspace{1em}
   What is a D-brane as a fundamental object (as opposed to a solitonic object) in string theory?}
 \end{itemize}
In other words, what is the intrinsic definition of D-branes
 so that by itself it can produce the properties of D-branes
 that are consistent with, governed by, or originally produced by open strings as well?
This is the guiding question of the whole project.

It is clear from the behavior under coincidence
that one cannot expect to have a good answer to Question [D-brane]
 without bringing some noncommutative geometry into the intrinsic definition of D-branes.
It turns out that
 Polchinski's {\it description of deformations of stacked D-branes} in [Pol3] and [Pol4] together with
 Grothendieck's {\it local equivalence of rings and spaces/geometries}
 implies immediately:
 \begin{itemize}
  \item[$\cdot$]
  \parbox[t]{37.4em}{\it
  {\bf Ansatz [D-brane: matrix noncommutativity on brane]}\hspace{1em}
   The world-volume of a D-brane carries a noncommutative structure locally associated to
     a function ring of the form $M_{r\times r}(R)$,
    i.e., the $r\times r$ matrix-ring over a ring $R$ for some $r\in {\Bbb Z}_{\ge 1}$.}
  \end{itemize}
This is an observation that had been already made in the work
 [Ho-W] (1996) of {\it Pei-Ming Ho} and {\it Yong-Shi Wu} from their study of D-branes
  and repicked up in [L-Y1] (D(1), 2007).
This brings us to a technical world in mathematics: noncommutative geometry.
In particular,
 \begin{itemize}
   \item[(1)] \parbox[t]{37.4em}{\it {\bf [local]}\hspace{1em}
    What is the noncommutative space $U^{nc}_{\alpha}$
	  associated to matrix rings $R_{\alpha}$?}
   
   \item[(2)] \parbox[t]{37.4em}{\it {\bf [local to global from gluing]}\hspace{1em}
    How to glue the local noncommutative spaces $U^{nc}_{\alpha}$ in (1)
     to a global noncommutative space $X^{nc}$?}
	
   \item[(3)]	\parbox[t]{37.4em}{\it
   {\bf [map from noncommutative space to space-time]}\hspace{1em}
    What is the notion of maps  $\varphi:X^{nc}\rightarrow Y$
  	that fits the study of D-branes in the space-time $Y$?}
 \end{itemize}
This project aims to answer the above questions   and then
 apply the result to themes in mathematics motivated by string theory and D-branes
   and themes in string theory in its own right.
See [L-Y1: Introduction] (D(1))
 for more details of the string-theory background on D-branes that motivated us the whole project.
 
\bigskip
 
Influenced by the setup of modern algebraic geometry through Grothendieck's language of schemes,
 so far in this project we focused more on the algebro-geometric side,
   cf.\ [L-Y1] (D(1)), [L-L-S-Y] (D(2)), [L-Y2] (D(3)), [L-Y3] (D(4)), [L-Y4] (D(5)),
         [L-Y8] (D(9.1)), [L-Y9] (D(10.1)), and [L-Y10] (D(10.2)),
 and less on the differential-and-symplectic-geometry side,
   cf.\ [L-Y5] (D(6)), [L-Y6] (D(7)), and [L-Y7] (D(8.1)).
{To} remedy this and
 to fill in the missing foundation --- as [L-Y1] (D(1)) and [L-L-S-Y] (D(2)) to the whole project  ---
     as a preparation for further works,
in the current note
 we take the attitude of re-starting everything from the ground floor again and
 reconsider D-branes in string theory and address the issue of
 how to describe them mathematically as a fundamental object
 (as opposed to a solitonic object) in string theory in the realm in differential topology and geometry.
 
After a terse review
  of how the Azumaya (i.e.\ matrix-)type noncommutative structure occurs
     on the world-volume of stacked D-branes from [L-Y1] (D(1))  and
  of the relevant background on $C^k$-algebraic geometry and matrix differential geometry from literature,
 we develop the notion of continuous maps, $k$-times differentiable maps, and smooth maps
   from an Azumaya/matrix manifold with a fundamental module  to a (commutative) real manifold $Y$.
Such maps are meant to describe D-branes or matrix branes in string theory
  when these branes are light and soft with only small enough or even zero brane-tension.
When $Y$ is a symplectic manifold
  (resp.\  a Calabi-Yau manifold; a $7$-manifold with $G_2$-holonomy;
                   a manifold with an almost complex structure $J$), 	
  the corresponding notion of
      Lagrangian maps
    (resp.\ special Lagrangian maps; associative maps, coassociative maps; $J$-holomorphic maps)
	 are also introduced.
Indicative examples linking to symplectic geometry and string theory are given.
This provides us with a language and part of the foundation required
 to study themes, new or old, in symplectic geometry and string theory, including
	\begin{itemize}
	 \item[(1)]
      $J$-holomorphic D-curves (with or without boundary),
	  as opposed to $J$-holomorphic curve in symplectic (open or closed) Gromov-Witten theory,
	
	 \item[(2)]
   	  quantization and dynamics of D-branes (in particular, A-branes) in string theory,
	
     \item[(3)]
 	  a definition of Fukaya category guided by Lagrangian maps from Azumaya manifolds
	    with  a fundamental module with a connection,
	   as opposed to the various existing settings in the study of homological mirror symmetry,    
				
	 \item[(4)]
   	  a theory of fundamental matrix strings or D-strings,
	  as opposed to the by-now-standard theory of ordinary fundamental strings, and
	
	 \item[(5)]
      the nature of Ramond-Ramond fields in a space-time, which on the one hand is generated
	  by excitations of closed fundamental superstrings and on the other hand is sourced by D-branes.
    \end{itemize}
The current note D(11.1) is the symplectic/differential-geometric counterpart
    of the algebraic-geometry-oriented first two notes
	  [L-Y1] (D(1)) (arXiv:0709.1515 [math.AG])  and
	  [L-L-S-Y] (D(2), with Si Li and Ruifang Song) (arXiv:0809.2121 [math.AG])
	 in this project.

\bigskip
 
\bigskip

\noindent
{\bf Convention.}
 References for standard notations, terminology, operations and facts are\\
  (1) Azumaya/matrix  algebra: [Ar], [Az], [A-N-T], [Wed]; \hspace{3.1em}
  (2) analysis: [Ap], [Br], [Mal]; \hspace{1em}\\			
  (3) sheaves and bundles: [Dim], [Hu-L], [Kob], [Kas-S], [Ste];  \hspace{1em}
  (4) algebraic geometry: [E-H], [F-G-I-K-N-V], [G-H], [Hart]; \hspace{2em}
  (5) differential topology; [B-T], [G-G], [Gu-P], [Hir], [Wa]; \hspace{1em}\\
  (6) differential geometry: [Hic], [Joy1], [K-N]; \hspace{2.6em}
  (7) symplectic geometry: [McD-S1]; \hspace{1em}\\
  (8) calibrated geometry:  [Harv], [Ha-L], [McL]; \hspace{1em}
  (9) synthetic geometry and $C^{\infty}$-algebraic geometry:
              [Du1],  [Joy3], [Koc], [M-R], [NG-SdS]; \hspace{1em}
  (10) noncommutative differential geometry: [Co1], [Co2], [GB-V-F], [Mad]; \hspace{1em}
  (11) string theory: [B-B-Sc], [G-S-W], [Pol4], [Zw]; \hspace{1em}
  (12) D-branes: [As], [A-B-C-D-G-K-M-S-S-W], [Bac], [Joh], [H-K-K-P-T-V-V-Z], [Pol3], [Sh3], [Sz].
 \begin{itemize}
  \item[$\cdot$]
   For clarity, the {\it real line} as a real $1$-dimensional manifold is denoted by ${\Bbb R}^1$,
    while the {\it field of real numbers} is denoted by ${\Bbb R}$.
   Similarly, the {\it complex line} as a complex $1$-dimensional manifold is denoted by ${\Bbb C}^1$,
    while the {\it field of complex numbers} is denoted by ${\Bbb C}$.
	
  \item[$\cdot$]	
  The inclusion `${\Bbb R}\hookrightarrow{\Bbb C}$' is referred to the {\it field extension
   of ${\Bbb R}$ to ${\Bbb C}$} by adding $\sqrt{-1}$, unless otherwise noted.

 \item[$\cdot$]	
  The {\it real $n$-dimensional vector spaces} ${\Bbb R}^{\oplus n}$
      vs.\ the {\it real $n$-manifold} $\,{\Bbb R}^n$; \\
  similarly, the {\it complex $r$-dimensional vector space ${\Bbb C}^{\oplus r}$}
     vs.\ the {\it complex $r$-fold} $\,{\Bbb C}^r$.
   
 \item[$\cdot$]
  All manifolds are paracompact, Hausdorff, and admitting a (locally finite) partition of unity.
  We adopt the {\it index convention for tensors} from differential geometry.
   In particular, the tuple coordinate functions on an $n$-manifold is denoted by, for example,
   $(y^1,\,\cdots\,y^n)$.
  However, no up-low index summation convention is used.
   
  %--------------------------------------------------------------------------------------------------
  % \item[$\cdot$]
  % ??????????
  % All varieties, schemes and their products are over ${\Bbb C}$;
  % a `{\it curve}' means a $1$-dimensional proper scheme over ${\Bbb C}$.
  % % a `{\it stack}' means an {\it Artin stack}.
  %======================================================
  
  \item[$\cdot$]
   `{\it differentiable}' = $k$-times differentiable (i.e.\ $C^k$)
         for some $k\in{\Bbb Z}_{\ge 1}\cup{\infty}$;
   `{\it smooth}' $=C^{\infty}$;
   $C^0$ = {\it continuous} by standard convention.

  \item[$\cdot$]
   $\Spec R $ ($:=\{\mbox{prime ideals of $R$}\}$)
         of a commutative Noetherian ring $R$  in algebraic geometry\\
   vs.\ $\Spec R$ of a $C^k$-ring $R$
  ($:=\Spec^{\Bbb R}R :=\{\mbox{$C^k$-ring homomorphisms $R\rightarrow {\Bbb R}$}\}$).

  \item[$\cdot$]
  {\it morphism} between schemes in algebraic geometry
    vs.\ {\it $C^k$-map} between $C^k$-manifolds or $C^k$-schemes
         	in differential topology and geometry or $C^k$-algebraic geometry.
   
  \item[$\cdot$]
   The `{\it support}' $\Supp({\cal F})$
    of a quasi-coherent sheaf ${\cal F}$ on a scheme $Y$ in algebraic geometry
     	or on a $C^k$-scheme in $C^k$-algebraic geometry
    means the {\it scheme-theoretical support} of ${\cal F}$
   unless otherwise noted;
    ${\cal I}_Z$ denotes the {\it ideal sheaf} of
    a (resp.\ $C^k$-)subscheme of $Z$ of a (resp.\ $C^k$-)scheme $Y$;
    $l({\cal F})$ denotes the {\it length} of a coherent sheaf ${\cal F}$ of dimension $0$.

  \item[$\cdot$]
   {\it coordinate-function index}, e.g.\ $(y^1,\,\cdots\,,\, y^n)$ for a real manifold
      vs.\  the {\it exponent of a power},
	  e.g.\  $a_0y^r+a_1y^{r-1}+\,\cdots\,+a_{r-1}y+a_r\in {\Bbb R}[y]$.
	     	
  \item[$\cdot$]	
 {\it cotangent sheaf} $\, \Omega^1_Y$ of a manifold $Y$
     vs.\ {\it holomorphic $n$-form} $\Omega$ on a Calabi-Yau $n$-fold.
 
  \item[$\cdot$]
    {\it global section functor} $\varGamma (\,\cdot\,)$ on sheaves
	   vs.\  {\it graph} $\varGamma_f$ of a function $f$.

  \item[$\cdot$]
    {\it algebraic operation} $\mu$ of a ring on  a module
	   vs.\ {\it Maslov index} $\mu$ of a loop in a symplectic group
	   vs.\ {\it dummy indices} in  e.g.\   space(-time) coordinates $y^{\mu},\, y^{\nu}$.
   
  \item[$\cdot$]
   `{\it d-manifold}$\,$' in the sense of `{\it d}erived manifold'
    vs.\  `{\it D-manifold}$\,$' in the sense of `{\it D}(irichlet)-brane that is supported on a manifold' 
	vs.\  `{\it D-manifold}$\,$' in the sense of works [B-V-S1] and [B-V-S2]
	   of Michael Bershadsky, Cumrun Vafa and Vladimir Sadov.
   
  \item[$\cdot$]
   The current Note D(11.1) continues the study in
	  \begin{itemize}
	   \item[$\cdot$] [L-Y1]  (arXiv:0709.1515 [math.AG], D(1)) and		
	   \item[$\cdot$] [L-L-S-Y] (arXiv:0809.2121 [math.AG], D(2), with Si Li and Ruifang Song),
	  \end{itemize}  	
	 with the direction toward D-branes in the realm of differential/symplectic topology and geometry.
   It was hinted at in
	   \begin{itemize}
	    \item[$\cdot$] [L-Y5] (arXiv:1003.1178 [math.SG], D(6)),
	   \end{itemize}
	pushed along in
	   \begin{itemize}
	    \item[$\cdot$] [L-Y6] (arXiv:1012.0525 [math.SG] , D(7))  and
		\item[$\cdot$] [L-Y7] (arXiv:1109.1878 [math.DG], D(8.1)),
	   \end{itemize}
    and recently re-motivated partially but strongly by the construction in progress
	   of a symplectic theory of both closed D-string world-sheet instantons
	      and open D-string world-sheet instantons
	   that can be paired to the algebraic theory of D-string world-sheet instantons
	     that is partially studied in
	  \begin{itemize}
	   \item[$\cdot$] [L-Y9] (arXiv:1302.2054 [math.AG] D(10.1)) and
	   \item[$\cdot$] [L-Y10] (arXiv:1310.5195 [math.AG] D(10.2)).
	  \end{itemize}
   A partial review of D-branes and Azumaya noncommutative geometry is given in
	  \begin{itemize}
	   \item[$\cdot$] [L-Y5] (arXiv:1003.1178 [math.SG], D(6)) and
       \item[$\cdot$] [Liu1] (arXiv:1112.4317 [math.AG]); see also [Liu2] and [Liu3].
      \end{itemize}
   Notations and conventions follow these earlier works when applicable.
 \end{itemize}

\bigskip

%\bigskip
\newpage
   
\begin{flushleft}
{\bf Outline}
\end{flushleft}
\nopagebreak
{\small
\baselineskip 12pt  %13pt
\begin{itemize}
 \item[0.]
  Introduction.
   
 \item[1.]
 D-branes as fundamental objects in string theory and their matrix-type noncommutativity.
  \vspace{-.6ex}
  \begin{itemize}
   \item[$\cdot$]
    What is a D-brane?
	
   \item[$\cdot$]
    Azumaya/matrix noncommutative structures on a D-brane world-volume.
	
   \item[$\cdot$]
    What is the mathematical notion of `maps' from Azumaya/matrix spaces
	 to a space-time that can describe D-branes in string theiry correctly?
	
   \item[$\cdot$]
    New directions in geometry motivated by D-branes as maps from Azumaya spaces.
  \end{itemize}

 \item[2.]
 Algebraic aspect of differentiable maps, and $C^k$-algebraic geometry.
   \vspace{-.6ex}
   \begin{itemize}
    \item[2.1]
     Algebraic aspect of differentiable maps between differentiable manifolds.
    
	\item[2.2]
	 Sheaves in differential topology and geometry with input from algebraic geometry\\
     -- with a view toward the Chan-Paton sheaf on a D-brane.					
   \end{itemize}
   
 \item[3.]
  The case of D0-branes on ${\Bbb R}^n$:
  Differentiable maps from an Azumaya point with a fundamental module
  to ${\Bbb R}^n$ as a differentiable manifold.
  \vspace{-.6ex}
  \begin{itemize}
	\item[3.1]
	 Warm-up:
      Morphisms from an Azumaya point with a fundamental module
	  to the real affine space ${\Bbb A}^n_{\Bbb R}$ in algebraic geometry.
	
	\item[3.2]
	 Smooth maps from an Azumaya point with a fundamental module to ${\Bbb R}^n$ as a smooth manifold.	 
	\begin{itemize}
     \item[3.2.1]
	  Smooth maps from an Azumaya point with a fundamental module to ${\Bbb R}^1$.
		
	 \item[3.2.2]
	  Smooth maps from an Azumaya point with a fundamental module to ${\Bbb R}^n$.
    \end{itemize}
	 	
	\item[3.3]	
	 Continuous maps from an Azumaya point with a fundamental module
	 to ${\Bbb R}^n$ as a topological manifold.
		
	\item[3.4]
     $C^k$-maps from an Azumaya point with a fundamental module
	 to ${\Bbb R}^n$ as a $C^k$-manifold.  		 	
	
    \item[3.5]
	Lessons from the case study.
  \end{itemize}
  
 \item[4.]
  Differential calculus and geometry of Azumaya manifolds with a fundamental module.
  \vspace{-.6ex}
  \begin{itemize}   	
   \item[4.1]
   	Differential calculus on noncommutative rings with center a $C^k$-ring.
   
   \item[4.2]
    Differential calculus on Azumaya differentiable manifolds with a fundamental module.

   \item[4.3]
    Metric structures on an Azumaya differentiable manifold with a fundamental module.
  \end{itemize}
  
 \item[5.]
  Differentiable maps from an Azumaya/matrix manifold with a fundamental module to a real\\
  manifold.
  \vspace{-.6ex}
  \begin{itemize}	
   \item[5.1]
	A generalization of Sec.\ 2.1 and Sec.\ 3 to Azumaya/matrix manifolds with a fundamental module
	-- local study.
		
   \item[5.2]
    The role of the bundle $E$: The Chan-Paton bundle on a D-brane (continuing Sec.~2.2).
   
   \item[5.3]
    Differentiable maps from an Azumaya manifold with a fundamental module to a real manifold.
	\begin{itemize}
	  \item[5.3.1]
	   Aspect I [fundamental]: Maps from gluing systems of ring-homomorphisms.
	
	  \item[5.3.2]
      Aspect II: The graph of a differentiable map.
	
      \item[5.3.3]
      Aspect III: From maps to the stack of D0-branes.
	
	  \item[5.3.4]
	  Aspect IV: From associated $\GL_r({\Bbb C})$-equivariant maps.
	  %
	  %---------------------------------------------
	  % \item[5.3.5]
	  % Aspect V: From synthetic differential geometry and $C^k$-algebraic geometry.
	  %=========================
	\end{itemize}
  \end{itemize}	
   
 \item[6.]
  Push-pulls and differentiable maps adapted to additional structures on the target-manifold
  \vspace{-.6ex}
  \begin{itemize}
   \item[6.1]
   The induced map on derivations, differentials, and tensors.
 
   \item[6.2]
   Remarks on stringy regularizations of the push-forward of a sheaf under a differentiable map.

   \item[6.3]
   Differentiable maps adapted to additional structures on the target manifold.
  \end{itemize}
  
 \item[7.]
  Examples of differentiable maps from Azumaya manifolds with a fundamental module.
  \vspace{-.6ex}
  \begin{itemize}
   \item[7.1]
    Examples generated from branched coverings of manifolds.
	
   \item[7.2]	
    {From} an immersed special Lagragian brane with a flat bundle to a fuzzy special Lagrangian brane
    with a flat bundle in the Calabi-Yau $1$-fold ${\Bbb C}^1$.
  \end{itemize}
      
 %---------------------------------------------------------------------------
 %
 % \item[?.]
 %  ??????????????????????.
 %  %
 %  \vspace{-.6ex}
 %  \begin{itemize}
 %   \item[?.?]
 %    ??????????????.
 %
 %   \item[?.?]
 %    ?????????????
 %    %
 %    \begin{itemize}
 %     \item[?.?.?]
 %      ???????????????.
 %    \end{itemize}
 %
 %  \end{itemize}
 %===========================================
 
\end{itemize}
} %endsmall

\newpage

\section{D-branes as fundamental objects in string theory and their matrix-type noncommutativity }

In this first section of the current note, which points to new directions of the project,
 we rerun what motivated us in the year 2007.
As this note is the symplectic/differential-geometric counterpart
    of the more algebraic-geometry-oriented first two notes
	   [L-Y1] (D(1))  and
	   [L-L-S-Y] (D(2), with Si Li and Ruifang Song)
	 from the project,
 this section serves to make the current D(11.1) and its sequels more conceptually self-contained as well.

\bigskip

\begin{flushleft}
{\bf What is a D-brane?}
\end{flushleft}
A {\it D-brane} (in full name,{\it Dirichlet brane}) is meant to be a boundary condition for open strings
  in whatever form it may take, depending on where we are in the related Wilson's theory-space.
A realization of D-branes that is most related to the current work is
  an embedding $f:X\rightarrow Y$ of a manifold $X$ into the open-string target space-time $Y$
  with the end-points of open strings being required to lie in $f(X)$.
This sets up a $2$-dimensional Dirichlet boundary-value problem
  for the field theory on the world-sheet of open strings.
Oscillations of open strings with end-points in $f(X)$ then create a mass-tower of various fields on $X$,
 whose dynamics is governed by open string theory.
This is parallel to the mechanism that oscillations of closed strings create fields in space-time $Y$,
 whose dynamics is governed by closed string theory.
Cf.~{\sc Figure}~1-1.
%
%  \marginpar{\raggedright\tiny $\bullet$
%        {\sc Figure}:  \\  db-os.pdf}

\begin{figure}[htbp]
 \bigskip
  \centering
  \includegraphics[width=0.80\textwidth]{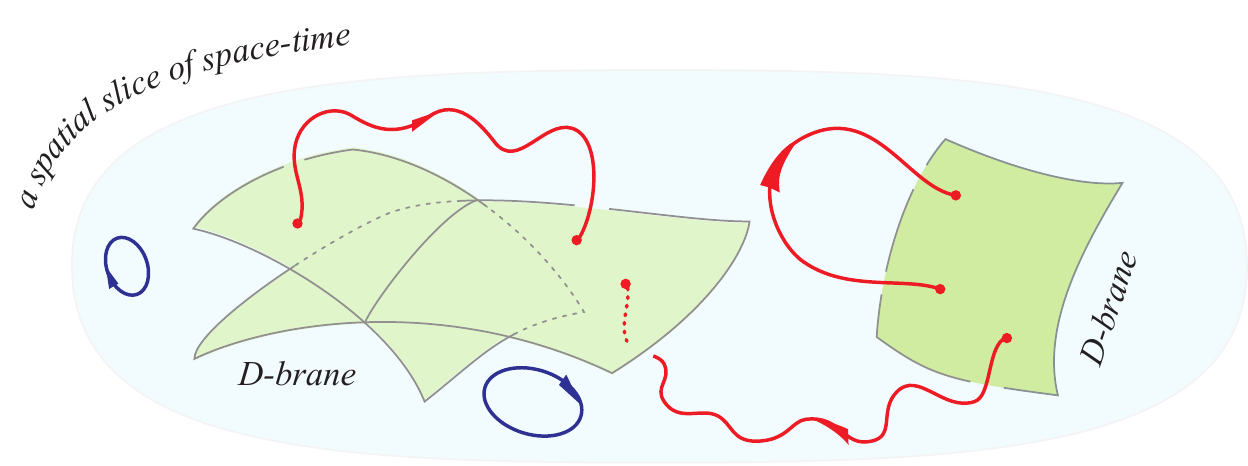}
 
  \bigskip
  \bigskip
 \centerline{\parbox{13cm}{\small\baselineskip 12pt
  {\sc Figure}~1-1.
  D-branes as boundary conditions for open strings in space-time.
  This gives rise to interactions of D-brane world-volumes with open strings.
  Properties of D-branes,
     including the quantum field theory on their world-volume and deformations of such,
   are governed by open strings via this interaction.
  Both oriented open (resp.\ closed) strings and a D-brane configuration are shown.
  }}
\end{figure}

\noindent
Let $\xi:=(\xi^a)_a$ be local coordinates on $X$ and
  $\Phi:=(\Phi^a;\Phi^{\mu})_{a,\mu}$ be local coordinates on $Y$
  such that the embedding $f:X\hookrightarrow Y$ is locally expressed as
  $$
   \Phi\; =\; \Phi(\xi)\; =\; (\Phi^a(\xi); \Phi^{\mu}(\xi))_{a,\mu}\;
   =\; (\xi^a,\Phi^{\mu}(\xi))_{a,\mu}\,;
  $$
 i.e., $\Phi^a$'s (resp.\ $\Phi^{\mu}$'s) are local coordinates along
       (resp.\ transverse to) $f(X)$ in $Y$.
This choice of local coordinates removes redundant degrees of freedom of the map $f$, and
$\Phi^{\mu}=\Phi^{\mu}(\xi)$ can be regarded as (scalar) fields on $X$
 that collectively describes the postions/shapes/fluctuations of $X$ in $Y$ locally.
Here, both $\xi^a$'s, $\Phi^a$'s, and $\Phi^{\mu}$'s are ${\Bbb R}$-valued.
The open-string-induced gauge field on $X$ is locally given
 by the connection $1$-form $A=\sum_a A_a(\xi)d\xi^a$ of a $U(1)$-bundle on $X$.

When $r$-many such D-branes $X$ are coincident/stacked,
  from the associated massless spectrum of (oriented) open strings with both end-points on $f(X)$
 one can draw the conclusion that
 \begin{itemize}
  \item[(1)]
   The gauge field $A=\sum_a A_a(\xi)d\xi^a$ on $X$ is enhanced to $u(r)$-valued.

  \item[(2)]
   Each scalar field $\Phi^{\mu}(\xi)$ on $X$ is also enhanced to matrix-valued.
 \end{itemize}
Property (1) says that there is now a $U(r)$-bundle on $X$.
But
 \begin{itemize}
  \item[$\cdot$]
   {\bf Q.}\ {\it What is the meaning of Property (2)?}
 \end{itemize}
{For} this, Polchinski remarks that:
 (Note:
  Polchinski's $X^{\mu}$ and $n$ $\;=\;$ our $\Phi^{\mu}$ and $r$.)
\begin{itemize}
 \item[$\cdot$] \parbox[t]{37.4em}{\it
 {\rm [quote from [Pol4: vol.~I, Sec.~8.7, p.~272]]}\hspace{1em}
 ``{\it
  For the collective coordinate $X^{\mu}$, however, the meaning is mysterious:
   the collective coordinates for the embedding of $n$ D-branes in space-time
     are now enlarged to $n\times n$ matrices.
  This `noncommutative geometry' has proven to play a key role in the dynamics of D-branes,
   and there are conjectures that it is an important hint about the nature of space-time.}"}
\end{itemize}
(See also comments in
    [Joh: Sec.~4.10 (p.\ 125)] and [A-B-C-D-G-K-M-S-S-W: Sec.~3.5.2.9].)
{From} the mathematical/geometric perspective,
 \begin{itemize}
  \item[$\cdot$]
  Property (2) of D-branes, the above question, and Polchinski's remark
 \end{itemize}
 can be incorporated into the following single guiding question:
 \begin{itemize}
  \item[{\bf Q.}]
  {\bf [D-brane]}$\;$ {\it What is a D-brane intrinsically?}
 \end{itemize}
In other words, what is the {\it intrinsic} nature/definition of D-branes so that {\it by itself}
 it can produce the properties of D-branes (e.g.\ Property (1) and Property (2) above)
 that are consistent with, governed by, or originally produced by open strings as well?

\bigskip

\begin{flushleft}
{\bf Azumaya/matrix noncommutative structures on D-brane world-volume}
\end{flushleft}
{To} understand Property (2), one has two perspectives:
 \begin{itemize}
  \item[(A1)]
   [{\it coordinate tuple as point}]\hspace{1em}
   A tuple $(\xi^a)_a$ (resp.\ $(\Phi^a; \Phi^{\mu})_{a,\mu}$)
    represents a point on the world-volume $X$ of the D-brane (resp.\ on the target space-time $Y$).

  \item[(A2)]
   [{\it local coordinates as generating set of local functions}]\hspace{1em}
   Each local coordinate $\xi^a$ of $X$ (resp.\ $\Phi^a$, $\Phi^{\mu}$ of $Y$)
     is a local function on $X$ (resp.\ on $Y$)  and
     the local coordinates $\xi^a$'s (resp.\ $\Phi^a$'s and $\Phi^{\mu}$'s) together
   form a generating set of local functions on the world-volume $X$ of the D-brane
   (resp.\ on the target space-time $Y$).
 \end{itemize}
While Aspect (A1) leads one to the anticipation of a noncommutative space
  from a noncommutatization of the target space-time $Y$ when probed by coincident D-branes,
Aspect (A2) of Grothendieck leads one to a different -- seemingly dual but not quite -- conclusion:
  a noncommutative space from a noncommutatization of the world-volume $X$ of coincident D-branes,
 as follows.

Denote by ${\Bbb R}\langle \xi^a\rangle_{a}$
  (resp.\ ${\Bbb R}\langle \Phi^a; \Phi^{\mu}\rangle_{a, \mu}$)
 the local function ring on the associated local coordinate chart on $X$ (resp.\ on $Y$).
Then the embedding $f:X\rightarrow Y$,
 locally expressed as
  $\Phi=\Phi(\xi)
     =(\Phi^a(\xi); \Phi^{\mu}(\xi))_{a,\mu}=(\xi^a; \Phi^{\mu}(\xi))$,
 is locally contravariantly equivalent to a ring-homomorphism\footnote{For
                                           string-theorists:
                                            I.e.\ pull-back of functions from the target-space $Y$
                                            to the domain-space $X$ via $f$.}
 $$
  f^{\sharp}\;:\;
   {\Bbb R}\langle \Phi^a; \Phi^{\mu}\rangle_{a, \mu}\;
   \longrightarrow\; {\Bbb R}\langle \xi^a\rangle_{a}\,,
   \hspace{1em}\mbox{generated by}\hspace{1em}
  \Phi^a\;\longmapsto\; \xi^a\,,\;
  \Phi^{\mu}\;\longmapsto\;\Phi^{\mu}(\xi)\,.
 $$
When $r$-many such D-branes are coincident,
 $\Phi^{\mu}(\xi)$'s become $M_{r\times r}({\Bbb C})$-valued.
Thus, $f^{\sharp}$ is promoted to a new local ring-homomorphism:
 $$
  \hat{f}^{\sharp}\;:\;
   {\Bbb R}\langle \Phi^a; \Phi^{\mu}\rangle_{a, \mu}\;
   \longrightarrow\; M_{r\times r}({\Bbb C}\langle \xi^a\rangle_{a})\,,
  \hspace{1em}\mbox{generated by}\hspace{1em}
  \Phi^a\;\longmapsto\; \xi^a\cdot{\mathbf 1}\,,\;
  \Phi^{\mu}\;\longmapsto\;\Phi^{\mu}(\xi)\,.
 $$
Under Grothendieck's contravariant local equivalence of function rings
 and spaces, $\hat{f}^{\sharp}$ is equivalent to saying that we have
 now a map $\hat{f}: X_{\scriptsizenoncommutative}\rightarrow Y$,
  where $X_{\scriptsizenoncommutative}$ is the new domain-space,
   associated now to the enhanced function-ring
   $M_{r\times r}({\Bbb C}\langle \xi^a\rangle_{a})$.
Thus, the D-brane-related noncommutativity in Polchinski's treatise
 [Pol4], as recalled above, implies the following ansatz
 when it is re-read from the viewpoint of Grothendieck:

\bigskip

\begin{sansatz} {\bf [D-brane: Azumaya/matrix-type noncommutativity].}
 The world-volume of a D-brane carries a noncommutative structure locally associated to
  a function ring of the form $M_{r\times r}(R)$,   where
  $r\in {\Bbb Z}_{\ge 1}$ and $M_{r\times r}(R)$ is the $r\times r$ matrix ring over $R$.
\end{sansatz}

\bigskip

We call a geometry associated to a local function-rings of matrix-type
 {\it Azumaya-type noncommutative geometry};
cf.~[Ar], [Az], and Remark~1.2 (3) below.
                 % Remark [Azumaya/matrix noncommutative structure on D-brane world-volume]

Note that
 when the closed-string-created $B$-field on the open-string target space(-time) $Y$ is turned on,
 $R$ in the Ansatz can become noncommutative itself;
cf.\ for example, [S-W] and see [L-Y4: Sec.~5.1] (D(5)) for such a case.

\bigskip

\begin{sremark} $[$Azumaya/matrix noncommutative structure on D-brane world-volume$]$. \\{\rm
 (1)
 Ansatz~1.1 was originally observed by {\it Pei-Ming Ho} and {\it Yong-Shi Wu}
   from their study of D-branes [Ho-W: Sec.~5 `{\it D-branes as quantum spaces}'] (1996).
 {From} the physical aspect, it can be perceived also from a comparison with quantum mechanics.
 In quantum mechanics,
  when a particle moving in a space-time with spatial coordinates collectively denoted by $y$,
    $y$ becomes operator-valued.
 There we don't take the attitude that
   just because $y$ becomes operator-valued, the nature of the space-time is changed.
 Rather, we say that {\it the particle is quantized but the space-time remains classical.}
 In other words, it is the nature of the particle that is changed, {\it not} the space-time.
 Replacing the word `{\it quantized}' by `{\it matrix/Azumaya noncommutatized}',
  one concludes that this matrix/Azumaya-noncommutativity happens on D-branes,
  {\it not (immediately on) the space-time}.
  
  \medskip
  
  (2)
 {From} the mathematical/Grothendieck aspect, the function ring $R$ is more fundamental than
   the topological space $\Space(R)$ associated to it, if definable.
 A morphism
     $$
      \varphi\; :\;  \Space(R)\; \longrightarrow\; \Space(S)
     $$
     is specified {\it contravariantly} by a ring-homomorphism
     $$
      \varphi^{\sharp}\; :\; S\; \longrightarrow\; R\,.
     $$
   If the function ring $R$ of the domain space $\Space(R)$
     is commutative, then $\varphi^{\sharp}$ factors through
     a  ring-homomorphism $\bar{\varphi}^{\sharp}:S/[S,S]\rightarrow R$,
     $$
      \xymatrix{
        R  &&  S\ar[ll]_-{\varphi^{\sharp}} \ar@{->>}[d]^-{\pi_{S/[S,S]}}
                \ar@{}[dl]|{\mbox{\raisebox{2.4ex}{$\circ$}}}   \\
           &&  S/[S,S]  \ar[llu]^-{\bar{\varphi}^{\sharp}}   &.
      }
     $$
   Here,
    $[S,S]$, the {\it commutator} of $S$, is the bi-ideal of $S$
     generated by elements of the form $s_1s_2-s_2s_1$
     for some $s_1,\,s_2\in S$;  and
    $S/[S,S]$ is the commutatization of $S$.
   It follows that
     $$
      \xymatrix{
       \Space(R)\ar[rr]^-{\varphi} \ar[rrd]_-{\bar{\varphi}}
        &&  \Space(S)\ar@{}[dl]|{\mbox{\raisebox{2.4ex}{$\circ$}}}\\
        &&  \Space(S/[S,S])\rule{0ex}{2.4ex} \ar@{^{(}->}[u]_-{\iota}   &.
      }
     $$
   In other words,
    \begin{itemize}
     \item[$\cdot$] \parbox[t]{37.4em}{\it
	 If the function ring on the D-brane world-volume is only commutative,
      then it won't be able to detect the noncommutativity, if any,
      of the open-string target-space!}
    \end{itemize}
	
 \medskip
 
 (3) For a reference to string-theorists,
 the name `{\it Azumaya algebra}' is due to the introduction of the notion of  `{\it maximally central algebra}
   by Goro Azumaya in [Az] (1951), which later came to be called ``Azumaya algebra" , cf.\ [Ar],
 This is the same abstraction of the matrix algebras as `$r$-dimensional vector space over $k$' vs.\ $k^{\oplus r }$.
 The study of such algebras from the viewpoint of algebras and representation theory
   is a classical mathematical subject.
 However, the investigation of it as a geometric object started only much later,
   cf.\ related reference in
                        [L-Y1] (D(1)), [L-Y2] (D(3)), and [L-Y3] (D(4)).
 The full richness of Azumaya geometry remains to be explored.
 Cf.~{\sc Figure}~1-2.
 %
 % \marginpar{\raggedright\tiny $\bullet$
 %   {\sc Figure}:\\  surrogate.pdf  }

\begin{figure}[htbp]
 \bigskip
 \centering
  \includegraphics[width=0.80\textwidth]{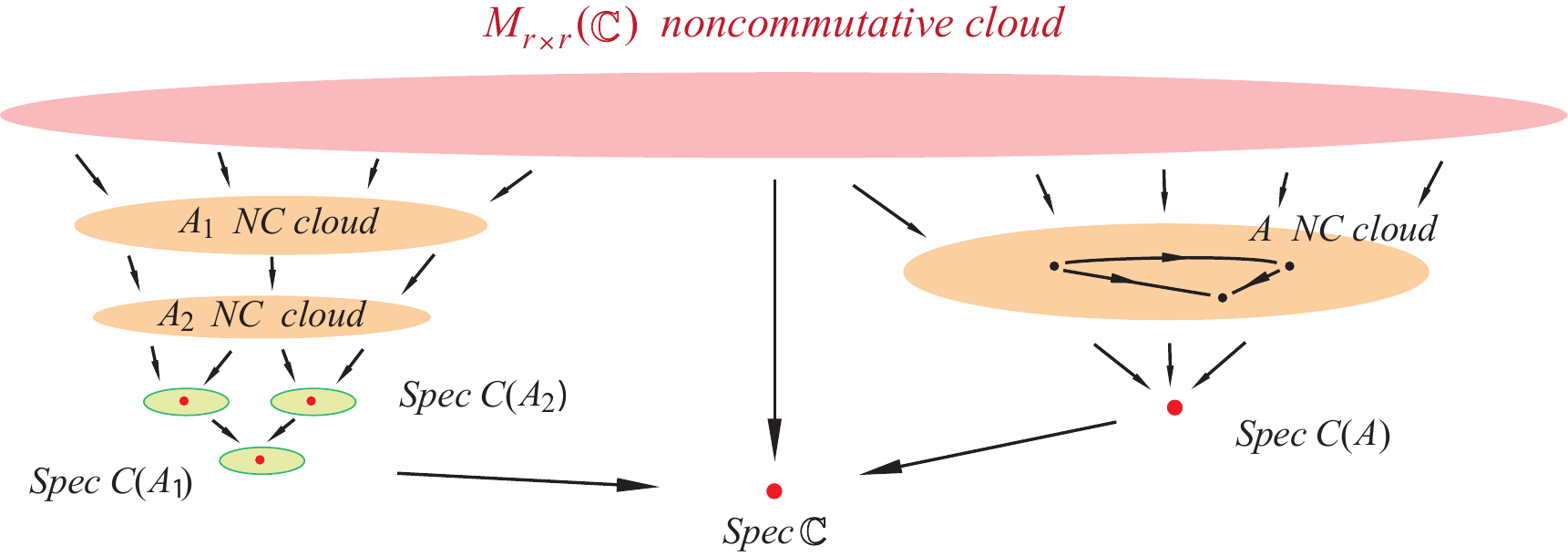}

 \bigskip
 \bigskip
 \centerline{\parbox{13cm}{\small\baselineskip 13pt
  {\sc Figure}~1-2.
  An Azumaya scheme contains a very rich amount of geometry,
   revealed via its surrogates;
   cf.\ [L-L-S-Y: {\sc Figure}~1-3] (D(2)).
  Indicated here is the geometry of an Azumaya point
   $p^{A\!z} := (\smallSpec{\Bbb C}, M_{r\times r}({\Bbb C}))$.
  Here, $A_i$ are ${\Bbb C}$-subalgebras of $M_{r\times r}({\Bbb C})$
    and $C(A_i)$ is the center of $A_i$ with
   $$
     \begin{array}{cccccccl}
      M_{r\times r}({\Bbb C}) & \supset  & A_1  & \supset  &  A_2
                    &\supset   & \cdots \\
       \cup  && \cup && \cup \\
     {\Bbb C}\cdot {\mathbf 1} & \subset  & C(A_1)
                    & \subset  & C(A_2)   & \subset  & \cdots & .
     \end{array}
   $$
  According to Ansatz~1.1,
   a D$0$-brane can be modelled prototypically
   by an Azumaya point with a fundamental module of type $r$,
    $(\smallSpec{\Bbb C},\smallEnd({\Bbb C}^{\oplus r}),{\Bbb C}^{\oplus r})$.
  When the target space $Y$ is commutative,
   the surrogates involved are commutative ${\Bbb C}$-sub-algebras
    of the matrix algebra $M_{r\times r}({\Bbb C})=\End({\Bbb C}^{\oplus r})$.
  This part already contains an equal amount of
   information/richness/complexity
   as the moduli space of $0$-dimensional coherent sheaves
   of length $r$.
  When the target space is noncommutative,
   more surrogates to the Azumaya point will be involved.
  Allowing $r$ to go to $\infty$ enables Azumaya points to probe
   ``infinitesimally nearby points" to points on a scheme
   to arbitrary level/order/depth.
  In (commutative) algebraic geometry,
   a resolution of a scheme $Y$ comes from a blow-up.
  In other words,
   a resolution of a singularity $p$ of $Y$ is achieved
   by adding an appropriate family of
   infinitesimally nearby points to $p$.
  Since D-branes with an Azumaya-type structure
   are able to ``see" these infinitesimally nearby points
    via morphisms therefrom to $Y$,
   they can be used to resolve singularities of $Y$.
  Thus, from the viewpoint of Ansatz~1.1,
   the Azumaya-type structure on D-branes is
   why D-branes have the power to ``see" a singularity
   of a scheme not just as a point,
   but rather as a partial or complete resolution of it.
  Such effect should be regarded as a generalization of
   the standard technique in algebraic geometry
   of probing a singularity of a scheme by arcs
   of the form $\Spec({\Bbb C}[\varepsilon]/(\varepsilon^r))$,
   which leads to the notion of jet-schemes
   in the study of singularity and birational geometry.}}
\end{figure}
 %
 %-------------------------------------------------------------------------------------------------------------
 % \medskip
 %
 %(4)
 % Purely physically, one may take the noncommutative scalar field addressed  in the quote
 %   from Polchinski directly as Lie algebra valued.
 % Works have been done to understand D-branes in such a setting,
 %   e.g.\ works [Dor1] and [Dor2]  of {\it Harald Dorn}.
 % While we are not qualified to comment on this from the physical aspect,
 %  we do want to point out some issues from the mathematical aspect.
 %  %
 % \begin{itemize}
 %  \item[]
 %   \begin{itemize}
 %    \item[(4.1)]
 %      ????????????????.
 %   \end{itemize}	
 % \end{itemize}
 %=============------------------------------================------------------------------
}\end{sremark}

\bigskip

\begin{sremark} $[$D-brane as fundamental object in string theory$\,]$.   {\rm
 When D-branes are taken as fundamental objects as strings,
   it is no longer in existence as solitonic objects from string theory.
 Furthermore,
  we no longer want to think of their properties as derived from open strings.
 Rather, D-branes should have their own intrinsic nature in discard of open strings.
 Only that when D-branes co-exist with open strings in space-time,
   their nature has to be compatible/consistent with the originally-open-string-induced properties thereon.
 It is in this sense that we think of a D-brane world-volume
   as an Azumaya-type noncommutative space, following the Ansatz,
   on which other additional compatible structures
   -- in particular, a Chan-Paton sheaf -- are defined.
}\end{sremark}

\bigskip

\begin{flushleft}
{\bf What is the mathematical notion of `maps' from Azumaya/matrix spaces
	 to\\  a space-time that can describe D-branes in string theory correctly?}
\end{flushleft}
Having {\it derived} that D-brane world-volume is an Azumaya/matrix-type noncomutative space,
 our next question is
 \begin{itemize}
  \item[{\bf Q.}]  \parbox[t]{37.4em}{\it {\bf [map from Azumaya space].}
   What is the mathematical notion of `maps' from Azumaya/matrix spaces to a space-time
   that can describe D-branes in string theory correctly?}
 \end{itemize}
{To} be precise, we now confine ourselves
  to the realm of (commutative or noncommutative) algebraic geometry for the rest of the theme.
  
Our candidate for a D-brane (or a Wick-rotated D-brane world-volume) in the realm of algebraic geometry
  is an {\it Azumaya/matrix scheme} obtained from gluing central localizations of matrix rings.
As such, it is an equivalence class of gluing systems of matrix rings which can be compactly
 realized as a ringed-space
 $$
  (X^{A\!z},{\cal O})\;
   :=\;   (X,{\cal O}_X^{A\!z}:=\Endsheaf_{{\cal O}_X}({\cal E}),{\cal E})\,.
 $$
 Here,
   $X:=(X,{\cal O}_X)$ is the ordinary (commutative, Noetherian) scheme over ${\Bbb C}$
    (e.g.\ [E-H] and [Hart]) and
   ${\cal E}$ is a locally free ${\cal O}_X$-module, say of rank $r$.
 $X^{\!A\!z}$ has the same underlying topology as $X$ but with the noncommutative structure sheaf
  ${\cal O}_X^{A\!z}$, the sheaf of endomorphism algebras of ${\cal E}$,
		that contains ${\cal O}_X$  as its sheaf of central elements.
 
However, it turns out that to reflect the behavior of D-branes correctly,
 the notion of a {\it morphism}
  $$
     \varphi\; :\; (X^{\!A\!z},{\cal E})\; \longrightarrow\;  (Y,{\cal O}_Y)
  $$
    to a (commutative or noncommutative) ringed-space $(Y,{\cal O}_Y)$
  is not defined ordinarily as a morphism between ringed-spaces as in, e.g.\ [E-H] and [Hart].
Rather, it is only {\it defined contravariantly as an equivalence class of gluing systems of ring-homomorphisms},
 in notation
 $$
   \varphi^{\sharp}\; :\; {\cal O}_Y\; \longrightarrow\; {\cal O}_X^{A\!z}\,,
 $$
 without specifying in company any map $X\rightarrow Y$ between the underlying topological spaces.
Indeed, for general $\varphi$ there is no map $X\rightarrow Y$ that can be assigned naturally/functorially.
Despite being nonconventional,
  the built-in/fundamental (left) ${\cal O }_X^{A\!z}$-module ${\cal E}$ is realized
  as an ${\cal O}_Y$-module through $\varphi^{\sharp}$.
This defines the push-forward $\varphi_{\ast}{\cal E}$ on $Y$.

%-------------------------------------------------------------------------------------------------------------------------
%
% A holomorphic D-brane in superstring theory can be described
%  as a morphism from a scheme with a matrix-type noncommutative structure sheaf
%  (i.e. an Azumaya scheme $(X,{\cal O}_X^{A\!z})$)
%    together with a fundamental ${\cal O}_X^{A\!z}$-module ${\cal E}$
%  to $(Y, {\cal O}_Y)$, where ${\cal O}_Y$ is the structure sheaf of $Y$
%   in either commutative or noncommutative setting; in notation/symbol,
%  $$
%    \varphi\; :\;  (X,{\cal O}_X^{A\!z}, {\cal E})\;
%      \longrightarrow\;   (Y,{\cal O}_Y)\,,
%  $$
%  with a built-in isomorphism
%    ${\cal O}_X^{A\!z}\simeq \Endsheaf_{{\cal O}_X}({\cal E})$.
% In true contents, this means
%  a contravariant gluing system of ring-homomorphisms
%  $$
%     {\cal O}_X^{A\!z}\;
% 	   \longleftarrow\; {\cal O}_Y\; :\; \varphi^{\sharp}\,,
% $$
% which in general {\it does not} induce any morphisms directly
%    from $X$ to $Y$.
% It is through $\varphi^{\sharp}$ that
%  the ${\cal O}_X^{A\!z}$-module ${\cal E}$
%  can be pushed forward to an ${\cal O}_Y$-module,
%  in notation $\varphi_{\ast}({\cal E})$, on $Y$.
%
%----------------------------------------------------------------------------------------------------------------
  
\bigskip
 
When the target space $Y$ is a commutative scheme
   and ${\cal E}$ is locally free ${\cal O}_X$-module,
 then
  associated to a morphism
   $\varphi :(X^{\!A\!z},{\cal E})\rightarrow (Y,{\cal O}_Y)$
   is the following diagram
   $$
    \xymatrix{
    {\cal O}_X^{A\!z}= \Endsheaf_{{\cal O}_X}({\cal E})\\
       {\cal A}_{\varphi}\;:=\, \rule{0ex}{3ex}
         		{\cal O}_X\langle\Image\varphi^{\sharp}\rangle \ar@{^{(}->}[u]
                    				&&& {\cal O}_Y\ar[lll]_-{\varphi^{\sharp}}\\
    {\cal O}_X\rule{0ex}{3ex}  \ar@{^{(}->}[u]								&&&&,
    }
  $$
 where ${\cal O}_X\langle\Image\varphi^{\sharp}\rangle$
    is the sheaf of ${\cal O}_X$-subalgebras of ${\cal O}_X^{A\!z}$
	generated by $\Image\varphi^{\sharp}$.
Note that ${\cal A}_{\varphi}$ is a sheaf of commutative ${\cal O}_X$-algebras
    and that ${\cal E}$ is naturally realized as an ${\cal A}_{\varphi}$-module as well,
	in notation $_{{\cal A}_{\varphi}}{\cal E}$.
Let $X_{\varphi}:= \boldSpec{\cal A}_{\varphi}$ be the associated scheme. 	
Then, by tautology,  $_{{\cal A}_{\varphi}}{\cal E}$
    is an ${\cal O}_{X_{\varphi}}$-module  and
  one has the corresponding diagram of spaces:
   $$
    \xymatrix{
     &   X^{\!A\!z}\ar[rrrd]^-{\varphi}\ar@{->>}[d]   \\
     &  X_{\varphi}\ar[rrr]_-{f_{\varphi}}  \ar@{->>}[d]^-{\pi_{\varphi}} &&& Y \\
     & X      &&&&.
      }
   $$
   
Observe that  $X_{\varphi}$ can be canonically embedded in $X\times Y$
   through $(\pi_{\varphi},f_{\varphi})$.
 The built-in ${\cal O}_{X_{\varphi}}$-module
   $_{{\cal A}_{\varphi}}{\cal E}$ can then be pushed forward to
   an ${\cal O}_{X\times Y}$-module
    $\tilde{\cal E}_{\varphi}
	    := (\pi_{\varphi},f_{\varphi})_{\ast}
		                                   (_{{\cal A}_{\varphi}}{\cal E})$
    that is supported on $(\pi_{\varphi},f_{\varphi})(X_{\varphi})$.	
 $\tilde{\cal E}_{\varphi}$ is called the {\it graph} of  the morphism $\varphi$.
 It is a coherent sheaf on $X\times Y$
  that is flat over $X$, of relative dimension $0$.
Conversely, given such a coherent sheaf $\tilde{\cal E}$ on $X\times Y$,
  a morphism
    $\varphi_{\tilde{\cal E}}:(X,{\cal O}_X^{A\!z}, {\cal E})\rightarrow Y$
 can be constructed from $\tilde{\cal E}$ by taking
 \begin{itemize}
  \item[$\cdot$]
    ${\cal E}=\pr_{1\,\ast}\tilde{\cal E}$,
	
  \item[$\cdot$] 	
   ${\cal O}_X^{A\!z}=\Endsheaf_{{\cal O}_X}({\cal E})$,  and

  \item[$\cdot$]
   $\varphi_{\tilde{\cal E}}^{\;\sharp}:
     {\cal O}_Y\rightarrow {\cal O}_X^{A\!z}\,$
   is defined by the composition
   $$
     {\cal O}_Y
        \xrightarrow{\hspace{1em}pr_2^{\sharp}\hspace{1em}}
      {\cal O}_{X\times Y}
        \xrightarrow{\hspace{1.2em}\iota^{\sharp}\hspace{1.2em}}
		{\cal O}_{Supp(\tilde{\cal E})}\;
		\hookrightarrow\; {\cal O}_X^{A\!z}\,.
   $$
 \end{itemize}
Here,
  $X\xleftarrow{\pr_1} X\times Y \xrightarrow{\pr_2} $
      are the projection maps,
  $\iota:\Supp(\tilde{\cal E})\rightarrow X\times Y$	
    is the embedding of the subscheme,   and
note that $\Supp(\tilde{\cal E})$ is affine over $X$.
Treating $\tilde{\cal E}$ as an object in the bounded derived category
 $D^b(\Coh(X\times Y))$ of coherent sheaves on $X\times Y$,
 $\tilde{\cal E}$ defines a Fourier-Mukai transform
 $\Phi_{\tilde{\cal F}} : D^b(\Coh(X))\rightarrow D^b(\Coh(Y))$,
 in short name, a {\it Fourier-Mukai transform from $X$ to $Y$}.
In this way, the data that specifies a morphism
 $\varphi:(X,{\cal O}_X^{A\!z}, {\cal E})\rightarrow Y$
 is matched to a data that specifies a special kind of Fourier-Mukai transform.
 
\bigskip

While our notion of a morphism $\varphi:(X^{\!A\!z},{\cal E})\rightarrow (Y,{\cal O}_Y)$
 may look odd at the first sight,
it is tested on a few important situations of D-branes in string theory.  Thus, to conclude,

\bigskip

\begin{sdefinition-prototype} {\bf [D-brane as morphism from Azumaya/matrix space].} {\rm
 In the realm of algebraic geometry,  a {\it D-brane on $Y$}  is a {\it morphism}
  $\varphi:(X^{\!A\!z},{\cal E})\rightarrow Y$, defined contravariantly
  by an equivalence class $\varphi^{\sharp}:{\cal O}_Y\rightarrow {\cal O}_X^{A\!z}$
   of gluing systems of ring-homomorphisms.
}\end{sdefinition-prototype}

\bigskip

See [L-Y1] (D(1)), [L-L-S-Y] (D(2)) for foundations and
 [L-Y2] (D(3)), [L-Y3] (D(4)),  [L-Y4] (D(5)), [L-Y5] (D(6)), [L-Y6] (D(7))
  for tests/applications in various situations.

\bigskip

\begin{flushleft}
{\bf New directions in geometry motivated by D-branes as maps from Azumay spaces}
\end{flushleft}
This notion of D-branes in line with Definition-Prototype~1.4
                        % Definition-Prototype [D-brane as morphism from Azumaya/matrix space]
  is not just a formal mathematical game without real consequences.
{\sc Figure}~1-3 illustrate
 how it motivates/drives new directions in mathematics motivated by string theory.
%
%  \marginpar{\raggedright\tiny $\bullet$
%        {\sc Figure}:  \\  mD-Az.pdf}
%
\begin{figure}[htbp]
 \bigskip
  \centering
  \includegraphics[width=0.60\textwidth]{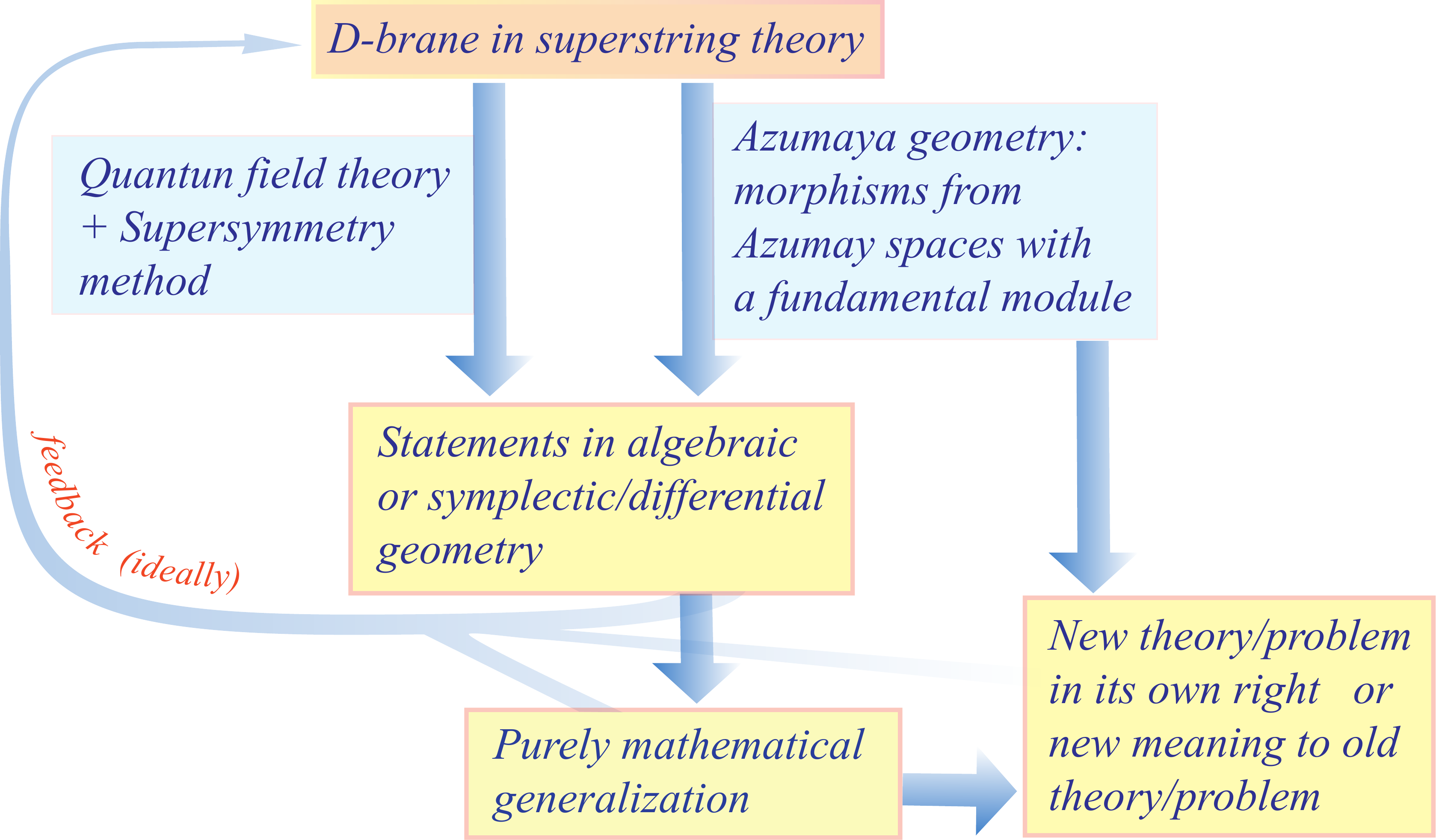}
  
   \bigskip
  % \bigskip
 \centerline{\parbox{13cm}{\small\baselineskip 12pt
  {\sc Figure}~1-3. From D-branes in string theory to new directions in geometry.
  }}
\end{figure}
Two exmaples in the realm of algebraic geometry are given below.

\bigskip

\begin{sexample} {\bf [D-brane resolution of singularity].} {\rm
 {From} the study of D-brane probe resolution of singularities of space-time, beginning with [Do-M]
 is extracted a conjecture:
 \begin{itemize}
  \item[$\cdot$] \parbox[t]{37.4em}{\it {\bf Conjecture [D0-brane resolution of singularity]}
    \hspace{1em}  {\rm ([L-Y8: Conjecture 1.5 and Conjecture 1.6] (D(9.1) with Baosen Wu).)}
   Let $Y$ be a projective variety over ${\Bbb C}$ and
         $f:Y^{\prime}\rightarrow Y$ be a birational morphism from a projective variety $Y^{\prime}$ to $Y$.	
   Then, $f$ factors through an embedding
	  $\tilde{f}:Y^{\prime}
	       \hookrightarrow{\frak M}_r^{0^{Az^f}_{\;p}}\!\!(Y) $, for some $r$.
    I.e.\ 		     		
	  $$
	    \xymatrix{
         &&  {\frak M}^{0^{Az^f}_{\;p}}_r\!\!(Y) \ar[d]^-{\pi_Y} \\
        Y^{\prime}\hspace{1ex}\ar @{^{(}->}[rru]^-{\tilde{f}} \ar[rr]^-f
         && \hspace{1ex}Y\hspace{1ex}  &\hspace{-4em}.
        }
      $$
     In particular, any resolution $\rho:\widetilde{Y}\rightarrow Y$ of $Y$ factors
    	 through some ${\frak M}_r^{0^{Az^f}_{\;p}}\!\!(Y)$.		
 Here, ${\frak M}_r^{0^{Az^f}_{\;p}}\!\!(Y)$
   is the stack of punctual D0-branes of rank $r$ on $Y$ in the sense of Definition-Prototype~1.4.
                                       % Definition-Prototype [D-brane as morphism from Azumaya/matrix space]
  }\end{itemize}
  The conjecture is true for $Y$ a reduced curve over ${\Bbb C}$
   ([L-Y8: Proposition~2.1] (D(9.1)).
}\end{sexample}

\bigskip

\begin{sexample} {\bf [D-string world-sheet instanton].} {\rm
 A mathematical theory of D-string world-sheet instantons for genus $g\ge 2$ was developed
      in [L-Y8] (D(10.1)) and [L-Y9] (D(10.2))
   in line with Definition-Prototype~1.4.
 This is a counter theory for D-strings as algebraic Gromov-Witten theory for fundamental strings
   and one has a similar compactness result:
   \begin{itemize}
    \item[$\cdot$] \parbox[t]{37.4em}{\it
  {\bf [${\frak M}_{Az^{\!f}\!(g;r,\chi)}^{\scriptsizeZss}(Y;\beta,c)$ compact]}
     \hspace{1em} {\rm ([L-Y10: Theorem~4.0] (D(10.2)).)}
    The stack ${\frak M}_{Az^{\!f}\!(g;r,\chi)}^{\tinyZss}(Y; \beta, c)$
      of $Z$-semistable morphisms from Azumaya nodal curves with a fundamental module
      to the Calabi-Yau $3$-fold $Y$ of type $(g;r,\chi;\beta,c)$
     is bounded and complete (i.e.\ compact).}
   \end{itemize}
 See ibidem for details.
}\end{sexample}	
	
\bigskip

Our path leading to Definition-Prototype~1.4
                         % Definition-prototype [D-brane as morphism from Azumaya/matrix space]
  naturally confines us a little bit to the realm of algebraic geometry.	
Yet,  on the string-theory side
  the occurrence of Azumaya/matrix-type noncommutative structure on D-brane world-volume
   is indifferent to whether one thinks of them mathematically in the realm of algebraic geometry or
   of differential or symplectic geometry;
on the mathematical side,
 the notion of encoding a topology space contravariantly by its function ring
  dated much earlier to, for example,
  the work [G-N] (1943) of Israel Gelfand and Mark  Na\u{i}mark
  and beyond algebraic geometry (or perhaps even to Karl Weierstrass in the19th century).
Several major topics motivated by string theory in differential or symplectic geometry
 require us to go beyond algebraic geometry as well.
This leads us to the current note.

\bigskip

\section{Algebraic aspect of differentiable maps, and $C^k$-algebraic\\ geometry}

As already being hinted at in, e.g.,
   the works [Vafa1] , [Vafa2] of Cumrun Vafa in studying D0-brane moduli space
      and the need to go from the classical moduli space to a quamtum moduli space
  and later being pointed out explicitly in the work [G-S] of T\'{o}mas G\'{o}mez and Eric Sharpe
 (cf.\ [L-Y5: Sec.~2.4, Item (4) ] (D(6))
 (and we'll see more from Sec.~3, Sec.~5.2, Sec.~7.2),
 even if we begin with a collection of simple D-branes (i.e.\ without any multiplicity),
  the support of Chan-Paton sheaf from deformations of branes should in general has a structure sheaf
  that contains nilpotent elements,
  reflecting its partial remembrance of the infinitesimal information of the deformation process.
This is a nature of D-branes themselves in string theory
  and has nothing to do with whether they are studied mathematically
   in the context of algebraic geometry or differential/symplectic geometry.
However, on the algebraic-geometry side,
 Grothendieck's theory of schemes and coherent sheaves take care of this information
    automatically inside the theory
  and is itself a standard general language in (commutative)  algebraic geometry.
In contrast, while there were efforts to merge Grothendieck's language to differential geometry
 under the names `{\it $C^{\infty}$-schemes}' and `{\it synthetic differential geometry}'
 (e.g.\ [Du1],  [Joy3], [Koc], [M-R], [NG-SdS]) by enlarging function rings in the realm of differential geometry
    to contain also nilpotent elements,
 the latter mathematical language are confined more to a subcommunity of differential geometers.
Thus,
 to fix some terminology/notation we need and to illuminate why we are doing this or that later,
 we recall in this section notions and objects in synthetic differential topology/geometry
  or $C^k$-algebraic geometry that are most relevant to this Note D(11.1)	
   in the spirit of a recent work [Joy3] of Dominic Joyce,
    who studied $C^{\infty}$-algebraic geometry for another reason
      (but could be relevant to us in the future
         when we come to the moduli problems in the current context).
The setting in this section follows [Joy3] essentially.
Some notational changes are made to fit in our works in the series.
Some additional terminologies are introduced for our later purpose.
With the current section as a preliminary,
we generalize in Sec.~5.1 the notion of differentiable maps between differentiable manifolds
 to the case that involves Azumaya algebras,
 which gives us the foundation toward the notion of a differentiable map
  from an Azumaya manifold with a fundamental module to a differentiable manifold in Sec.~5.3.
As explained in [L-Y1] (D(1)) and further illuminated in [Liu1],
   such a map is a proto-model to describe a D-brane in string theory
    in the regime where their tension is weak and hence are soft and flexible.

\bigskip
	
\subsection{Algebraic aspect of differentiable maps between differentiable manifolds}	

First, a guiding question from synthetic differential topology/geometry and
 $C^{\infty}$-algebraic geometry that is most relevant to us:
   \begin{itemize}
     \item[{\Large $\cdot$}]
	   Given a differentiable map $f:X\rightarrow Y$ between $C^k$-manifolds,
	   this induces an ${\Bbb R}$-algebra-homomorphism
	    $f^{\sharp}:C^k(Y)\rightarrow C^k(X)$.
		\begin{itemize}
           \item[{\bf Q.}]
		   {\it What is special about $f^{\sharp}$ as an algebra-homomorphism?\\[.6ex]
		           What feature of $f^{\sharp}$, $C^k(Y)$, and $C^k(X)$
			    	  makes $f$ recoverable from $f^{\sharp}$?}
		\end{itemize}   	
   \end{itemize}
One feature immediately catches the eye:
    \begin{itemize}
	 \item[{\Large $\cdot$}]
	  Let $g:{\Bbb R}^m\rightarrow {\Bbb R}$
	     be a $C^k$-function on ${\Bbb R}^m$  and
		 $h_1,\, \cdots\,,\, h_m\in C^k(Y)$,
	    then $g(h_1,\,\cdots\,,\, h_m)\in C^k(Y)$.
	  Furthermore,
	    $$
		  \begin{array}{ccl}
		   f^{\sharp}(g(h_1,\,\cdots\,,\, h_m))
		     & =   &  g(h_1,\,\cdots\,,\, h_m)\circ f \hspace{1em}\mbox{by definition,} \\[1.2ex]
           & =  & g(h_1\circ f,\,\cdots\,,\,h_m\circ f)\;\;\;
		                =\;\;\;  g(f^{\sharp}(h_1),\,\cdots\,,\, f^{\sharp}(h_m))\,.
		   \end{array}
         $$		   	
	\end{itemize}
  In other words, $f^{\sharp}$ satisfies many many more relations
     than just those that are required to be an ${\Bbb R}$-algebra-homomorphism.
These generally uncountably-many relations are essentially all non-algebraic. 	
This motivates a path to do or extend differential topology and geometry
with input from algebraic geometry.

\bigskip

\begin{flushleft}
{\bf Basic definitions in $C^k$-commutative algebra}
\end{flushleft}
The following definitions in the $C^{\infty}$ case are in the literature.
One can form the general $C^k$ version as well.
	 	
\bigskip

\begin{definition} {\bf [$C^k$-ring].} {\rm
 (Cf.\ [Joy: Definition 2.1, Definition 2.6].)
  Let $k\in {\Bbb Z}_{\ge 0}\cup\{\infty\}$.
  A {\it $C^k$-ring} is a set $R$ together with operations
   $$
    \Xi_f\;:\;
	   R^{\times n}\;
	     :=\,  \underbrace{R\times\,\cdots\,\times R}_{\mbox{\scriptsize $n$ copies}} \;
	   \longrightarrow\;  R                  	
   $$
   for all $n\ge 0$ and $C^k$-maps $f:{\Bbb R}^n\rightarrow {\Bbb R}$,
    where by convention  $R^{\times 0}:=\{\emptyset\}$.	
 These operations must satisfy the following relations:
   \begin{itemize}
    \item[(1)]
	  For $m$, $n\ge 0$  and
	    $f_i:{\Bbb R}^n\rightarrow {\Bbb R}$, $i=1,\,\ldots\,,\,m$,
		  and $g:{\Bbb R}^m\rightarrow {\Bbb R}$ $C^k$-functions,
		define a $C^k$-function $h:{\Bbb R}^n\rightarrow {\Bbb R}$ by
        $$
          h(x_1,\,\cdots\,,\, x_n)\;
		   =\;  g(f_1(x_1,\,\cdots\,,\, x_n)\,,\, \cdots\,,\, f_m(x_1,\,\cdots\,,\, x_n))\,,
        $$	
	 	for all $(x_1,\,\cdots\,,\, x_n)\in {\Bbb R}^n$.
	  Then
	    $$
		  \Xi_h(r_1,\, \cdots\,,\, r_n)\;
		    =\;  \Xi_g(\Xi_{f_1}(r_1,\, \cdots\,,\, r_n)\,,\, \cdots\,,\,
			                        \Xi_{f_m}(r_1,\, \cdots\,,\, r_n))
		$$
        for all $(r_1,\, \cdots\,,\, r_n)\in R^{\times n}$.
		
    \item[(2)]
	 For all $1\le j\le n$,
	 define $\pi_j:{\Bbb R}^n\rightarrow{\Bbb R}$
	   by $\pi_j:(x_1,\,\cdots\,,\,x_n)\mapsto x_j$.
	 Then
      $$
        \Xi_{\pi_j}(c_1,\,\cdots\,,\,c_n)\;=\; c_j
	  $$	
      for all $(c_1,\,\cdots\,,c_n)\in R^{\times n}$. 	 			
   \end{itemize}
 For brevity of notation, we may denote a $C^k$-ring
   $(R,(\Xi_f)_{f\in\cup_{n\ge 0}C^k({\Bbb R}^n)})$
   simply as $R$.
   
 The set $R$ is then equipped with a naturally induced commutative
  ${\Bbb R}$-algebra structure: \\
  (Below, $r_1$, $r_2$, $r\in R$ and $a\in {\Bbb R}$.)
  \begin{itemize}
    \item[(1)] ({\it addition}$\,$) \hspace{1em}
	  $r_1+r_2:= \Phi_f(r_1,r_2)$,
	   where $f:{\Bbb R}^2\rightarrow {\Bbb R}$ with $f(x_1,x_2)=x_1+x_2$.
	
    \item[(2)] ({\it multiplication}$\,$) \hspace{1em}
	  $r_1\cdot r_2:= \Phi_f(r_1,r_2)$,
	   where $f:{\Bbb R}^2\rightarrow {\Bbb R}$ with $f(x_1,x_2)=x_1x_2$.
	
	\item[(3)] ({\it scalar multiplication}$\,$) \hspace{1em}
      $a r:= \Phi_f(r)$,
	   where $f:{\Bbb R}\rightarrow {\Bbb R}$ with $f(x)=ax$.	
	
	\item[(4)] ({\it identity elements}$\,$) \hspace{1em}
	 The additive identity element $0:= \Phi_{0^{\prime}}(\emptyset)$,
	  where $0^{\prime}:{\Bbb R}^0\rightarrow {\Bbb R}$ with $\emptyset\mapsto 0$;
	 and the multiplicative identity element $1:= \Phi_{1^{\prime}}(\emptyset)$,
	  where $1^{\prime}:{\Bbb R}^0\rightarrow {\Bbb R}$ with $\emptyset\mapsto 1$.
  \end{itemize}
} \end{definition}

\bigskip

\begin{definition} {\bf [$C^k$-ring homomorphism].} {\rm  (Cf.\ [Joy: Definition 2.1].)
   A {\it $C^k$-ring homomorphism} between $C^k$-rings
    $(R, (\Xi_f)_{f\in \cup_{n\ge 0}C^k({\Bbb R}^n)})$ and
	$(S, (\Upsilon_f)_{f\in \cup_{n\ge 0}C^k({\Bbb R}^n)})$
	is a map $\psi:R\rightarrow S$ such that
	$$
	    \Upsilon_f(\psi(r_1),\,\cdots\,,\,\psi(r_n))\;
		 =\;  \psi(\Xi_f(c_1,\,\cdots\,,\,c_n))
	$$
	 for all $C^k$-functions $f:{\Bbb R}^n\rightarrow {\Bbb R}$,
	  $c_1,\,\cdots\,,\,c_r\in R$, and $n\ge 0$.
  Note that $\psi$ is then automatically an ${\Bbb R}$-algebra homomorphism from $R$ to $S$.
}\end{definition}

\bigskip

\begin{definition}
{\bf [(algebraic) ideal, $C^k$-normal (algebraic)  ideal,
               and quotient, $C^k$-normal quotient of $C^k$-ring].} {\rm
  %(Cf.\ [????].)
 An {\it ideal}   (or {\it algebraic ideal} to distinguish with Definition~2.1.4 next) $I$
    in a $C^k$-ring $R$
  is an ideal $I\subset R$ as a commutative ${\Bbb R}$-algebra.
 $I$ is called {\it $C^k$-normal}
  if  the $C^k$-ring structure on the quotient commutative ${\Bbb R}$-algebra $R/I$
  defined by $\Xi^I_f:(R/I)^{\times n}\rightarrow R/I$ with
  $$
    (\Xi^I_f(r_1+I,\,\cdots\,,\,r_n+I))\;
	   =\; \Xi_f(r_1,\,\cdots\,,\,r_n)\,+\, I\,,
  $$
   where $f:{\Bbb R}^n\rightarrow {\Bbb R}$ are $C^k$-functions,
   is independent of the choices of the representative $r_i$ for $r_i+I$.
 In this case, $R/I$   is called a {\it $C^k$-normal quotient} of $R$.
 Note that for $I$ $C^k$-normal, the quotient $C^k$-ring structure on $R/I$
    is compatible with the quotient ${\Bbb R}$-algebra structure
}\end{definition}

\bigskip

\begin{definition}{\bf [$C^k$-ideal].} {\rm
 An ideal $I$ of a $C^k$-ring $R$, in the sense of Definition~2.1.3 previously,  is called
  a {\it $C^k$-ideal} if the following additional condition is satisfied:
  \begin{itemize}
   \item[{\Large $\cdot$}]
     If $r_1,\,\cdots\,,\,r_n\in I$,
	 then $\Xi_f(r_1,\,\cdots\,,\,r_n)\in I$
	     for all $f\in C^k({\Bbb R}^n)$ such that $f(0,\,\cdots\,,\,0)=0$.	
  \end{itemize}
 If $I$ is also $C^k$-normal, then it is called a {\it $C^k$-normal $C^k$-ideal}.
}\end{definition}

\bigskip

An algebraic ideal in general may not be a $C^k$-ideal.
The notion of the latter kind of ideals of a $C^k$-ring is motivated by the following observation,
 whose proof follows by definition:

\bigskip

\begin{lemma}
 {\bf [kernel is $C^k$].}
  Let $\psi:R\rightarrow S$ be a $C^k$-ring homomorphism.
  Then its kernel $\Ker(\psi)$ is a $C^k$-ideal of $R$.
\end{lemma}

\bigskip

\begin{remark}
$[\,$$C^k$-ring as ${\Bbb R}$-algebra$\,]$.  {\rm
 Note that every $C^k$-ring has a built-in ${\Bbb R}$-algebra structure,
   with its ideal an ${\Bbb R}$-vector subspace of the underlying ${\Bbb R}$-algebra.
}\end{remark}

\bigskip

\begin{definition}
{\bf [nilradical, reduced $C^k$-ring].} {\rm
 The {\it nilradical} of a $C^k$-ring $R$ is the $C^k$-ideal $N$ that is $C^k$-generated
   by all nilpotent elements of $R$.
 If it is normal,
  then $R/N$ with the quotient $C^k$-ring structure is called
   the {\it reduced $C^k$-ring} associated to $R$.
} \end{definition}

\bigskip

\begin{definition} {\bf [generators of $C^k$-ring, finite generatedness].}  {\rm
 Let $R$ be a $C^k$-ring.
  A set of elements $\{\hat{r}_{\beta}\}_{\beta\in B}\subset I$
      is called a set of {\it ($C^k$-)generators} of $R$
  if for every $r\in I$, there exist
    $\hat{r}_{\beta_1},\,\cdots\,,\, \hat{r}_{\beta_n}$ from the set  and
    an $f\in C^k({\Bbb R}^n)$,
	for some $n\in{\Bbb Z}_{\ge1}$,
  such that 	$r=\Xi_f(\hat{r}_{\beta_1},\,\cdots\,,\,\hat{r}_{\beta_n})$.		
 Furthermore,
  if $B$ is a finite set, then $R$ is said to a {\it ($C^k$-)finitely generated $C^k$-ring}.
}\end{definition}

\bigskip

\begin{definition}{\bf [generators of ideal, finite generatedness].} {\rm
 Given an (algebraic) ideal $I$ of a $C^k$-ring $R$,
   a set of elements $\{\hat{r}_{\alpha}\}_{\alpha\in A}\subset I$
   is called a set of {\it generators} of $I$
  if for every $r\in I$, there exists $r_{\alpha_1},\,\cdots\,,\, r_{\alpha_n}\in R$,
     $n\in {\Bbb Z}_{\ge 1}$,
   such that
    $r = r_{\alpha_1}\hat{r}_{\alpha_1}
	         +\,\cdots\,+r_{\alpha_n}\hat{r}_{\alpha_n}$.
 Furthermore,
   if $A$  is a finite set, then  $I$ is said to be a {\it finitely generated ideal} of $R$.			
			
 If  $I$ is in addition a $C^k$-ideal, then the above set of elements in $I$ will also be called
  a set of {\it algebraic generators} of $I$  when in need of distinction.
 In this case,
   a set of elements $\{\hat{r}_{\beta}\}_{\beta\in B}\subset I$
      is called a set of {\it $C^k$-generators} of $I$
  if for every $r\in I$, there exist
    $\hat{r}_{\beta_1},\,\cdots\,,\, \hat{r}_{\beta_n}$ from the set  and
    an $f\in C^k({\Bbb R}^n)$ with $f(0,\,\cdots\,,\,0)=0$,
	for some $n\in{\Bbb Z}_{\ge1}$,
  such that 	$r=\Xi_f(\hat{r}_{\beta_1},\,\cdots\,,\,\hat{r}_{\beta_n})$.		
 Furthermore,
  if $B$ is a finite set, then $I$ is said to a {\it $C^k$-finitely generated $C^k$-ideal} of $R$.
   
 {To} distinguish, we'll write
  $I=\langle  r_{\alpha}\,:\,\alpha\in A \rangle$ for the former case  and
  $I=\langle  \hat{r}_{\beta}\,:\, \beta\in B \rangle_{C^k} $ for the latter case.
}\end{definition}

\bigskip

A set of algebraic generators of a $C^k$-ideal is automatically a set of $C^k$-generators;
 but the converse is in general not true.

\bigskip

\begin{definition}
 {\bf [localization of $C^k$-ring].} {\rm  ([Joy3: Sec.~2.3], [M-R: Sec.~1.1].)
  Let $R$ be a $C^k$-ring and $S$ a subset of $R$.
  The {\it localization} of  $R$ at $S$, denoted by $R[S^{-1}]$, is the $C^k$-ring
   with a $C^k$-ring homomorphism $\zeta: R\rightarrow R[S^{-1}]$
   that is characterized by the following two properties:
	\begin{itemize}
	 \item[(1)]
	   $\zeta(s)$ is invertible in $R[S^{-1}]$ for all $s\in S$.
	
	 \item[(2)]
	  If $R^{\prime}$ is a $C^k$-ring and
	     $\psi:R \rightarrow R^{\prime}$ is a $C^k$-ring homomorphism
	    such that $\psi(s)$	is invertible for all $s\in S$,
	   then there exists a unique $C^k$-ring homomorphism
	     $\psi^{\prime}:R[S^{-1}]\rightarrow R^{\prime}$
         with $\psi=\psi^{\prime}\circ \zeta$.		
	  $$
	    \xymatrix{
		 & R \ar[rr]^-{\zeta} \ar[dr]_-{\psi}
		    && R[S^{-1}]    \ar @{.>} [dl]^-{\psi^{\prime}}\\
		 &&    R^{\prime}&&.		  		
		}
	  $$			
	\end{itemize}
}\end{definition}

\bigskip

\begin{definition} {\bf [${\Bbb R}$-point of  $C^k$-ring].}
{\rm (Cf.\ [Joy3: p.10]; see also [M-R].)
 Let $R$ be a $C^k$-ring.
 An {\it ${\Bbb R}$-point} of $R$ is a $C^k$-ring homomorphism $p:R\rightarrow {\Bbb R}$,
  where ${\Bbb R}$ is regarded as the $C^k$-ring $C^k(\{0\})$.
}\end{definition}

\bigskip

\begin{definition} {\bf  [localization at ${\Bbb R}$-point].} {\rm
 Let $R$ be a $C^k$-ring and $p:R\rightarrow {\Bbb R}$ be an ${\Bbb R}$-point of $R$.
 Define the {\it localization of $R$ at $p$}, in notation $R_{(p)}$,
   to be $R[ (p^{-1}({\Bbb R}-\{0\}))^{-1} ]$.
 It is equipped with a built-in $C^k$-ring homomorphism
  $\zeta_{(p)}:R\rightarrow R_{(p)}$.
 (Later when we introduce the ringed topological space $\Spec R$,
     whose points are ${\Bbb R}$-points of $R$, one should think of
	 $\zeta_{(p)}$ as the restriction map to a neighborhood of $p$.)
}\end{definition}

%------------------------------------------------------------------------------------------------------------------
%
% \bigskip
%
% \begin{definition} {\bf [germ of $C^k$-ring at ${\Bbb R}$-point].} {\rm
%    ???????????.
% }\end{definition}
%
% \bigskip
%
% \begin{definition}{\bf [fair $C^k$-ring].} {\rm
%  fair = finitely generated + germ-determined. \\
%  ???????????????????.
% }\end{definition}
%
%=========---------------------------------============----------------------============

\bigskip

\begin{flushleft}
{\bf Affine $C^k$-schemes in $C^k$-algebraic geometry}
\end{flushleft}
Readers are referred to, e.g.,
  [B-T: {\S}10], [Dim: Chapter 2], [G-H: 0.3], [Ha: II.1], [K-S: Chapter~II], and [Wa: 5.1 -- 5.15]
 for basic definitions and facts concerning presheaves and sheaves in topology and geometry.

\bigskip
   
\begin{definition}{\bf [affine $C^k$-scheme $(X,{\cal O}_X)$].}   {\rm
(Cf.\ [Joy3: Definition 4.12]; see also [Du1], [Hart], [M-vQ-R].)
 Let $R$ be a $C^k$-ring.
 Define $X:= \Spec(R)$ to be the set of all ${\Bbb R}$-points $p$ of $R$.
 Then each $r\in R$ defines a map
   $r_{\ast}:X \rightarrow {\Bbb R}$ by setting $p\mapsto p(r)$.
 Equip $X$ with the smallest topology ${\cal T}_R$
   such that $r_{\ast}$ is continuous for all $r\in R$.
   (That is, ${\cal T}_R$ is generated by the open sets $r_{\ast}^{-1}(U)$
          for all $r\in R$ and $U\subset {\Bbb R}$ open.)
 Then note that $(X,{\cal T}_R)$ is a Hausdorff topology space. 		
 
 For each open set $U\subset X$,
  define ${\cal O}_X(U)$ to be the $C^{\infty}$-ring below:
   \begin{itemize}
     \item[{\Large $\cdot$}]
	  ${\cal O}_X(U)$ is the set of functions
	   $$
	      s\; :\;   U\;  \longrightarrow\;  \coprod_{p\in U}\,R_{(p)}
	   $$
	   such that
         \begin{itemize}
		   \item[{\Large $\cdot$}]
		     $\; s(p)\in R_{(p)}\;\;$ for all $p\in U$;
			
		   \item[{\Large $\cdot$}]	
  		     $U$ can be covered by open sets $\{V_{\alpha}\}_{\alpha\in A}$
			  such that
 			   on each $V_{\alpha}$ there exist $r$, $r^{\prime}\in R$
			      with $p(r^{\prime})\ne 0$ for all $p\in V_{\alpha}$
                      and $s(p)= \zeta_{(p)}(r)\zeta_{(p)}(r^{\prime})^{-1}$
				for all $p\in V_{\alpha}$.  			
         \end{itemize}
	
	 \item[{\Large $\cdot$}]
	  Define the $C^k$-ring structure on ${\cal O}_X(U)$
	   pointwise using the $C^k$-ring structure on $R_{(p)}$.
   \end{itemize}
   If $U_2\subset U_1\subset X$ are open sets, then the restriction map
     $\rho_{12}:{\cal O}_X(U_1)\rightarrow {\cal O}_X(U_2)$,
     	 with $s\mapsto s|_{U_2}$, is a $C^k$-ring homomorphism.
 The functor ${\cal O}_X(\,\cdot\,)$ defines a sheaf
    ${\cal O}_X$ on $X$, called the {\it structure sheaf} of $X$,
     whose stalk ${\cal O}_{X,p}$ at $p$ is $R_{(p)}$.	  	
 The locally ringed space $(X,{\cal O}_X)$   is called
    the {\it spectrum} of $R$ or the {\it  affine $C^k$-scheme} associated to $R$.
	
 By construction, there is a natural $C^k$-ring homomorphism
   $$
      R\;\longrightarrow\;  {\cal O}_X(X) \, =:\, \varGamma({\cal O}_X)\,,
   $$
  though, in this generality, it may not be an isomorphism.	
}\end{definition}

%---------------------------------------------------------------------------------------------------------------------
%
% \bigskip
%
% \begin{definition}  {\bf [locally finite sum in $C^k$-ring].} {\rm
%   ??????????????????.
% }\end{definition}
%
% \bigskip
%
% \begin{definition}{\bf [complete $C^k$-ring].} {\rm
%    ??????????????????????.
% }\end{definition}
%
% \bigskip
%
% \begin{proposition}{\bf [fair $\Rightarrow$ complete]. }
%   ?????????????????????.
% \end{proposition}
%
%========------------------------------------========------------------------------===========

\bigskip

\begin{example} {\bf [smooth manifold as affine $C^{\infty}$-scheme].} {\rm
 ([Joy3: Example 2.2, Example 4.9, Example 4.14].)
 Let $X$ be a smooth manifold without boundary (i.e.\ locally modelled on some ${\Bbb R}^n$)
  and $C^{\infty}(X)$ be the ring of smooth functions on $X$.
 Then $C^{\infty}(X)$ is naturally equipped with a $C^{\infty}$-ring structure.
 Furthermore, $\Spec(C^{\infty}(X))$
   as defined in Definition~2.1.13 is canonically homeomorphic to $X$.
                      % Definition [affine $C^k$-scheme $(X,{\cal O}_X)$]
 And, for an open set $U\subset X$,  ${\cal O}_X(U)$ is simply the $C^{\infty}$-ring
  $C^{\infty}(U)$ of smooth functions on $U$.
 The stalk ${\cal O}_{X,p}$ of ${\cal O}_X$ at $p\in X$ is simply the $C^{\infty}$-ring of
   germs of smooth functions  at $p$.
 The natural $C^{\infty}$-ring homomorphism
   $C^{\infty}(X)\rightarrow \varGamma({\cal O}_X):= {\cal O}_X(X)$
   is an isomorphism in this case.
}\end{example}

\bigskip

\begin{example} {\bf [$C^k$-manifold as affine $C^k$-scheme].} {\rm
 In genetral,
  let $X$ be a $C^k$-manifold without boundary (i.e.\ locally modelled on some ${\Bbb R}^n$)
  and $C^k(X)$ be the ring of $C^k$ functions on $X$.
 Then $C^k(X)$ is naturally equipped with a $C^k$-ring structure.
 Furthermore, $\Spec(C^k(X))$
   as defined in Definition~2.1.13 is canonically homeomorphic to $X$.
                      % Definition [affine $C^k$-scheme $(X,{\cal O}_X)$]
 And, for an open set $U\subset X$,  ${\cal O}_X(U)$ is simply the $C^k$-ring $C^k(U)$
   of $C^k$ functions on $U$.
 The stalk ${\cal O}_{X,p}$ of ${\cal O}_X$ at $p\in X$ is simply the $C^k$-ring of
   germs of smooth functions  at $p$.
 The natural $C^k$-ring homomorphism
   $C^k(X)\rightarrow \varGamma({\cal O}_X):= {\cal O}_X(X)$
   is an isomorphism in this case.
}\end{example}
 
\bigskip

\noindent
In the above example, the claim that $\Spec(C^k(X))\simeq X$ follows
 the proof of Proposition~3.4.2  in Sec.~3.4.
                 % Proposition [fundamental:
			     %         Every differentiable map from Azumaya/matrix point is of algebraic type]
Other claims are the consequence of the existence of a partition of unity on $X$.
 
\bigskip

\begin{remark}$[\,$on the notion of $C^k$-scheme$\,]$.  {\rm
 In this note and its follow-ups, since there seems to have no satisfying definitions for
    $\Spec$  of  a noncommutative ring and morphisms of such objects that fit well
     to describe D-branes in string theory,
   our focus are more on morphisms of function rings themselves.
  So essentially there is no need to introduce $\Spec$ of a $C^k$-ring here.
  However, as it is quite obvious when one pushes along, though not necessarily in use,
   it remains conceptually and geometrically appealing to think
    there is a topological space that realizes a given $C^k$-ring as its function ring.
 This is why we follow [Joy3]  to give Definition~2.1.13 above.
                                                          % Definition [affine $C^k$-scheme $(X,{\cal O}_X)$]
 The $\Spec$ as defined is more like the spectrum of maximal ideals of a ring
    in (commutative or noncommutative) algebraic geometry.
  With the structure sheaf removed,
    it is similar to the notion of a variety in (commutative) algebraic geometry.
 In the $C^{\infty}$ case and for smooth manifolds,
   Proposition~2.1.18 below from [Joy3] and [M-R] shows that this is a very natural definition.
   % Proposition [smooth map between manifolds
   %                       vs.\ $C^{\infty}$-ring homomorphism between function rings]
 However, for a general $C^k$-ring $R$, it is not yet clear to us
   this is the definition for $\Spec R$ that fits best to describe D-branes.
 See also  [M-vQ-R].
}\end{remark}

\bigskip

\begin{remark}$[$$C^{\infty}$-ring$\,]$. {\rm
 For the case of $C^{\infty}$-rings,
  some parts of the settings are redundant.
  For example, all ideals are $C^k$-normal as a consequence of the Hadamard's Lemma.
 We refer readers to [Joy3], [M-R], [Koc] for details.
}\end{remark}

\bigskip

\begin{flushleft}
{\bf $C^k$-manifold diffeomorphism vs.\ $C^k$-function-ring homomorphism contravariantly }
\end{flushleft}
The following proposition answers the guiding question at the beginning of this subsection
 in the $C^{\infty}$ case completely.
Taking it as a guide for the general $C^k$ case gives the conceptual foundation for this note D(11.1).
It is the reason why, when one cannot directly look at the topology
  of a (noncommutative) Azumaya manifold to define maps therefrom in a way
    that does describe D-branes in string theory,
 one may turn to directly look at the (noncommutative) function ring itself of an Azumaya manifold
 and see if it then fit well for D-branes.

\bigskip

\begin{proposition}{\bf [smooth map between manifolds
                                                   vs.\ $C^{\infty}$-ring homomorphism between function rings].}
{\rm ([Joy3: Proposition~3.3], [M-R: I.1.5].)}	
 Let
  $X$ and $Y$ be smooth manifolds without boundary,
  $\Map^{C^{\infty}}(X,Y)$ be the set of all smooth maps from $X$ to $Y$, and
  $Hom^{C^{\infty}}_{\Bbb R}(C^{\infty}(Y),C^{\infty}(X))$
     be the set of all $C^{\infty}$-${\Bbb R}$-algebra homomorphisms
	 from $C^{\infty}(Y)$ to $C^{\infty}(X)$.   	
 Note that a smooth map $f:X\rightarrow Y$ induces
  a contravariant $C^{\infty}$-${\Bbb R}$-algebra homomorphism of function rings
  by pulling back:
  $$
	\begin{array}{crccc}
     & f^{\ast}\; :\; C^{\infty}(Y)    & \longrightarrow    & C^{\infty}(X) \\[1.2ex]
     &                         h\hspace{1.2ex}    & \longmapsto         &  h \circ f	                 &.
	\end{array}
  $$
 Then the correspondence
  $$
    \begin{array}{ccc}
      \Map^{C^{\infty}}(X,Y) &  \longrightarrow
	     &  Hom^{C^{\infty}}_{\Bbb R}(C^{\infty}(Y),C^{\infty}(X))\\[1.2ex]
	  f     & \longmapsto      & f^{\ast}
	\end{array}		
  $$
  is a bijection.
\end{proposition}

\bigskip

\begin{remark} $[\,$comparison with algebraic geometry$\,]$. {\rm
 One should compare the above proposition with the similar statement
  in (commutative) algebraic geometry concerning the contravariant equivalence
  between the category of affine schemes and the category of (commutative) rings.
 Cf.\ [Hart: II Proposition 2.3].
} \end{remark}

\bigskip
  
\subsection{Sheaves in differential topology and geometry with input from algebraic geometry
                    -- with a view toward the Chan-Paton sheaf on a D-brane}
					
Sheaves are part of the building block of modern algebraic geometry.
On the string-theory side,  Chan-Paton bundles/sheaves/modules are part of the constituents of D-branes.
  %   we need to address them as well in the realm of differential topology/geometry
  %      before we can get down to the main themes of this note.
We collect in this subsection basic definitons  that are needed for this note.
The presentation here continues Sec.~2.1
   and is based on [Joy3: Sec.\ 5 and Sec.\ 6] of Dominic Joyce.

\bigskip

\begin{flushleft}
{\bf Basic definitions for modules in $C^k$-commutative algebra}
\end{flushleft}
As modules and sheaves do not seem to be a standard topic in differential topology and geometry
 (cf.\ [Hir], [Gu-P]; [K-N]) in contrast to algebraic geometry (cf.\ [Hart]),
for the clarity of the note we fix in this subsection a few notions and terminologies
 concerning sheaves on a differentiable manifold with input from algebraic geometry.

\bigskip

\begin{definition} {\bf [module].} {\rm  (Cf.\ [Joy3: Definition 5.1] ; also [A-M], [Ei].)
  Let $R$ be a $C^k$-ring.
  A {\it module $M$ over $R$}, or {\it $R$-module}, is a module over $R$
    in the ordinary sense,
	that is, with $R$ regarded as a commutative ${\Bbb R}$-algebra:
	\begin{itemize}
	  \item[{\Large $\cdot$}]
	    a vector space $M$ over ${\Bbb R}$, together with
		    an operation $\mu: R\times M \rightarrow M$ such that\\
		 $$
		   \begin{array}{ll}
    		 \mu(r, m_1 + m_2)\;=\; \mu(r,m_1)\,+\, \mu(r,m_2)\,,  \hspace{2em}
		       & \mu(r_1 + r_2, m)\;=\; \mu(r_1,m)\,+\, \mu(r_2,m)\,,  \\[1.2ex]
		     \hspace{1.7em}\mu(r_1r_2, m)\; =\;  \mu(r_1, \mu(r_2, m))\,,
		       & \hspace{2.5em}\mu(1,m)\;=\; m
		  \end{array}
		 $$
	   for all $r$, $r_1$, $r_2\in R$  and $m$, $m_1$, $m_2\in M$.    	 	
	\end{itemize}
	Most often we denote $\mu(r,m)$ simply by $\,r\cdot m\,$ or just $\,rm$.
		
  The notion of
  
      \medskip
	
    {\Large $\cdot$}
	   {\it homomorphism} $M_1 \rightarrow M_2$ of $R$-modules,
	
	{\Large $\cdot$}
       {\it submodule} $M_1 \hookrightarrow M_2$, (cf.\ {\it monomorphism}),
	
	{\Large $\cdot$}
       {\it quotient module}  $M_1 \twoheadrightarrow M_2$, (cf.\ {\it epimorphism}), 	
	
	{\Large $\cdot$}
	   {\it direct sum}  $M_1\oplus M_2$ of $R$-modules,
	
    {\Large $\cdot$}
      {\it tensor product} $M_1\otimes_RM_2$ of $R$-modules,
	
    {\Large $\cdot$}
	    {\it finitely generated}:
    		if $\oplus_lR\simeq R\otimes_{\Bbb R}{\Bbb R}^l \twoheadrightarrow M$
		     exists for some $l$,
			
    {\Large $\cdot$}
        {\it finitely presented}:
           if $R\otimes_{\Bbb R}{\Bbb R}^{l^{\prime}}
		          \rightarrow R\otimes_{\Bbb R}{\Bbb R}^l\rightarrow M\rightarrow 0$
              is exact for some $l$, $l^{\prime}$				
	
       \medskip
	
	\noindent
    are all defined in the ordinary way, as in e.g.\  [A-M], [Ei] for commutative algebras.
}\end{definition}

\bigskip

\begin{definition} {\bf [$C^k$-derivation, $C^k$-cotangent module of $C^k$-ring, $C^k$-differential].}
{\rm  (Cf.\ [Joy3: Definition 5.10].)
  Let $R$ be a $C^k$-ring and $M$ an $R$-module.
  An ${\Bbb R}$-linear map
    $$
	   d\; :\;  R\;  \rightarrow\;  M
	$$
    is called a {\it $C^k$-derivation}, $k\in{\Bbb Z}_{\ge 1}\cup{\infty}$,
  if
   $$
      d\Psi_f(r_1,\,\cdots\,,\,r_n) \;
		=\; \sum_{i=1}^n\Psi_{\partial_if}(r_1,\,\cdots,\,r_n)\cdot dr_i
   $$
   for all $f\in \cup_{n}C^k({\Bbb R}^n)$	 and $r_i\in R$.
  Here
    $\partial_if$ is the partial derivative of $f\in C^k({\Bbb R}^n)$
    with respect  to the $i$-th coordinate of ${\Bbb R}^n$.
  An $R$-module $M$ with a $C^k$-derivation $d:R\rightarrow M$
    is called the {\it $C^k$-cotangent module} of $R$
   if it satisfies the following universal property:
   \begin{itemize}
    \item[{\Large $\cdot$}]
	 For any $R$-module $M^{\prime}$  and
	   $C^k$-derivation $d^{\prime}:R\rightarrow M^{\prime}$,
	 there exists a unique homomorphism of $R$-modules $\psi:M\rightarrow M^{\prime}$
	  such that $\;d^{\prime}= \psi \circ d\,$.	
  	  $$
	    \xymatrix{
		 & R \ar[rr]^-{d} \ar[dr]_-{d^{\prime}}  && M   \ar @{.>} [dl]^-{\psi}\\
		 &&    M^{\prime}&&.		  		
		}
	  $$			
   \end{itemize}
   (Thus, $M$ is unique up to a unique $R$-module isomorphism.)
  We denote this $M$ with $d: R\rightarrow M$ by $\Omega_R$,
   with the built-in $C^k$-derivation $d:R\rightarrow \Omega_R$ understood.
}\end{definition}

\bigskip

\begin{remark}$[\,$explicit construction of $\Omega_R$$\,]$.  {\rm
(Cf.\ [Joy3: Definition 5.10].)
 The $C^k$-cotangent module $\Omega_R$ of a $C^k$-ring $R$ can be constructed explicitly
   from the $R$-module generated by the set
   $$
     \{ d(r)\,|\,  r\in R   \}\,,
   $$
   subject to the relations
   $$
    \begin{array}{llll}	
     \mbox{(${\Bbb R}$-linearity)}
	  &&& d(a_1r_1+a_2r_2)\;
	            =\; a_1\,d(r_1)\,+\,a_2\,d(r_2)\;
				=\; d(r_1)a_1\,+\, d(r_2)\,a_2 				\,,   \\[.6ex]
	 \mbox{(Leibniz rule)}
	  &&&  d(r_1r_2)\;=\;  d(r_1)\, r_2\,+\, r_1\, d(r_2) \,, \\[.6ex]
	 \mbox{($R$-commutativity)}
	  &&&  d(r_1)\,r_3\;=\; r_3\,d(r_1)
    \end{array}
   $$	
   for all $a_1,\,a_2\in {\Bbb R}$, $r_1,\,r_2,\, r_3 \in R$,  and
  $$
   \begin{array}{llll}
    \mbox{(chain rule)}\hspace{4em}
      &&&  d(h(r_1,\,\cdots\,,\,r_s))   \\
	  &&&  \hspace{1em}
                 =\;  \partial_1 h(r_1,\,\cdots\,,\,r_s)\, d(r_1)\;
	                     +\; \cdots\; +\;
			           \partial_s h(r_1,\,\cdots\,,\,r_s)\, d(r_s)	 \hspace{1em}
   \end{array}			
  $$
  for all $h\in C^k({\Bbb R}^s)$, $s\in {\Bbb Z}_{\ge 1}$, and $r_1,\,\cdots\,,\,r_s\in R$.
 Denote the image of $d(r)$ under the quotient by $dr$.
 Then, by definition, the built-in map
  $$
    \begin{array}{ccccc}
	 d & :&  R & \longrightarrow  & \Omega_R\\[.6ex]
	    &&    r  & \longmapsto        & dr
	\end{array}
  $$
  is a $C^k$-derivation from $R$ to $\Omega_R$.
}\end{remark}

\bigskip

It follows from the explicit construiction of cotangent modiles that

\bigskip

\begin{lemma}{\bf [induced map on cotangent modules].}
 Let $\rho: R\rightarrow S$ be a $C^k$-ring homomorphism.
 Then there is a canonically induced $S$-module  homomorphism
    $$
	   \rho_{\ast}\;  : \;   S\otimes_R\Omega_R\;  \longrightarrow\;   \Omega_S\,.
	$$
\end{lemma}

\bigskip

As
  the notion of $C^k$-derivations is exactly the notion of ${\Bbb R}$-derivation
    in the usual sense of differentiable topology and geometry  and
  it  is the only kind of derivation we'll use,
we'll call a $C^k$-derivation simply an {\it ${\Bbb R}$-derivation} or just {\it derivation}.
  
%------------------------------------------------------------------------------------------------------
%
% \bigskip
%
% \begin{remark}
% $[\,$algebraic derivation, algebraic cotangent module, and K\"{a}hler differential$\,]$.  {\rm
%
% \bigskip
%    ?????????????????.
%  \bigskip
%
% }\end{remark}
%
%==========------------------------------===========-----------------------------

\bigskip

\begin{definition}{\bf [localization of module].}   {\rm
 (Continuing Definition~2.1.10.)
                    % Definition [localization of $C^k$-ring]
 Let
    $\zeta:R\rightarrow R[S^{-1}]$  be the localization of a $C^k$-ring $R$ at $S$   and
	$M$ be an $R$-module.
 Then, the {\it localization of $M$ at $S$}, denoted by $M[S^{-1}]$,  is the $R[S^{-1}]$-module defined by
     $$
	   M[S^{-1}]\; :=\;  R[S^{-1}]\otimes_R M\,.
	 $$
 By construction,  it is equipped with an $R$-module-homomorphism $\zeta_M: M \rightarrow M[S^{-1}]$.
}\end{definition}

\bigskip
 
$C^k$-cotangent modules behave well under localization:

\bigskip

\begin{proposition} {\bf [localization of cotangent module].} {\rm (Cf. [Joy3: Proposition 5.14].)}
 Let
  $R$ be a $C^k$-ring,  $r\in R$,  and
  $\zeta_r:R\rightarrow R[r^{-1}]$ be the localization of $R$ at $r$.
 Then the natural $R[r^{-1}]$-module homomorphism
  $\zeta_{r\ast} : R[r^{-1}]\otimes_R\Omega_R \rightarrow \Omega_{R[r^{-1}]}$
  is an isomorphism.
\end{proposition}

\bigskip

\begin{flushleft}
{\bf Quasi-coherent sheaves of modules on an affine $C^k$-scheme}
\end{flushleft}
Let $R$ be a $C^k$-ring.
Recall from Definition~2.1.13
                % Definition [affine $C^k$-scheme $(X,{\cal O}_X)$]
  the affine $C^k$-scheme $X:=\Spec(R)$, with the structure sheaf ${\cal O}_X$.
Let $M$ be an $R$-module.
Then, the assignment  $U\mapsto   {\cal O}_X(U)\otimes_R M$,
  with the restriction map $\Id_M\times \rho_{UV}$ for $V\subset U$,
   where $\rho_{UV}:{\cal O}_X(U)\rightarrow {\cal O}_X(V)$
    is the restriction map of ${\cal O}_X$,
  is a presheaf on $X$.
 Let $M^{\sim}$ be its sheafification.
 
\bigskip

\begin{definition} {\bf [quasi-coherent sheaf and coherent sheaf on $\Spec R$].}
{\rm
 The sheaf $M^{\sim}$ of ${\cal O}_X$-modules on $X$ thus obtained
  from the $R$-module $M$ is called a {\it quasi-coherent sheaf} on $X$.
 If furthermore, $M$ is finitely presented,
  then $M^{\sim}$ is called a {\it coherent sheaf} on $X$.
}\end{definition}
 
\bigskip

\begin{convention}$[$convention from algebraic geometry$\,]$. {\rm
 Though $C^k$-algebraic geometry is meant to be differential topology and geometry
    remade in terms of  algebraic geometry,
  our notation follows the convention in algebraic geometry as much as we can.
 In particular,
%  Stalk vs.\ germ vs. fiber of a sheaf ${\cal F}$ at a point $p\in M$.
 %
 \begin{itemize}
  % \item[{\Large $\cdot$}]
  %   When a sheaf ${\cal F}$ is realized as the sheaf of sections of a vector bundle  $F$ over $X$
  %   (cf.\ [Hus] and [Ste] ),
  %	then the restriction of  ${\cal F}$ to a $C^k$-subscheme $Z\subset X$ is defined to be
  %	the sheaf of sections of $F|_Z$ over $Z$ and is denoted by ${\cal F}|_Z$.
  %   In particular, when $Z$ is a point $p\in M$, we'll denote ${\cal F}|_p$ simply by ${\cal F}_p$
  %     and call is the {\it fiber} of ${\cal F}$ at $p$.
  
  \item[{\Large $\cdot$}]
   For a general  ${\cal O}_X$-module ${\cal F}$ and $Z$ is a closed $C^k$-subscheme of $X$
     described by an ideal sheaf ${\cal I}_Z$ in ${\cal O}_X$,
    we shall define
     ${\cal F}|_Z$ to be the quotient sheaf ${\cal F}/({\cal I}_Z\cdot{\cal F})$.	
      
  \item[{\Large $\cdot$}]
   The {\it stalk} of ${\cal F}$ at $p\in X$  is denoted by ${\cal F}_{(p)}$
     while the {\it fiber} ${\cal F}|_p$  of ${\cal F}$ at $p\in X$ is denoted also
	 by ${\cal F}_p$.
 \end{itemize}
}\end{convention}
 
%------------------------------------------------------------------------------------------------------------ 
% \bigskip
%
% \begin{definition} {\bf [pull-back and push-forward].}  {\rm
%   ?????????????
% }\end{definition}
%
% \bigskip
%
%======--------------------------------========----------------------------------------======

\bigskip

\noindent	
Having going through this section,
one sees that
  various standard objects and notions in algebraic geometry
  have their precise counter objects and notions in differential geometry
  through the language of synthetic differential topology/geometry and $C^k$-algebraic geometry.
This is what we will adopt to understand D-branes along the line of Definition-Prototype~1.4.
                                             % Definition-Prototype [D-brane as morphism from Azumaya/matrix space]

%-------------------------------------------------------------------------------------------------------------------------------------- 
% Having said above and
%   advertised our belief that synthetic differential topology/geometry
%     is in line with the study of D-branes in the context of differential/symplectic geometry,
%  we intentionally choose not to use this mathematical language in the main text
%    to avoid adding to this beginning D(11.1)
%    another layer of abstractions and technical formats
%     that may obscure the underlying simple ideas
% 	  that were already laid out in [L-Y1] (D(1))
% 	    and are most relevant to D-branes
% 	    from the viewpoint of Grothendieck's approach to algebraic geometry.
% Instead, we keep the presentation as much similar to ordinary differential geometry as can be
%      except the necessary use  of function rings.
% Remarks from the aspect of synthetic differential topology/geometry and $C^{\infty}$-schemes
%  are added when appropriate.
%   %
%   \marginpar{\raggedright\tiny\vspace{-3em} $\bullet$ Tentatively only. To be revised in the end.}
%
%============--------------------------------------=============-------------------------------------------===

\bigskip

\section{The case of D0-branes on ${\Bbb R}^n$:
  Differentiable maps from an Azumaya point with a fundamental module
  to the real differentiable manifold ${\Bbb R}^n$}

Let
 $(p^{A\!z},
    {\Bbb C}^{\oplus r}):=(p, \End_{\Bbb C}({\Bbb C}^{\oplus r}),
	{\Bbb C}^{\oplus r})$
   be a fixed Azumaya point  with a fundamental module   and
 ${\Bbb R}^n$ be the real $n$-dimensional Euclidean space as a $C^k$-manifold, $0\le k\le \infty$.
We address in this section the following two questions:
  \begin{itemize}
    \item[{\bf Q1.}]{\it
	   What is a $k$-times-differentiable (i.e. $C^k$-)map $\varphi$
 	   from $(p^{A\!z},{\Bbb C}^{\oplus r})$ to ${\Bbb R}^n$?}
		
   \item[{\bf Q2.}]	{\it
  	 Does this notion suit well to describe D0-branes and their deformations on ${\Bbb R}^n$?}
  \end{itemize}	
 
\smallskip

As in  [L-Y1] (D(1)),
 the intended contravariant equivalence between category of spaces
         and the category of function rings on spaces   and
  the guide from the behavior of D-branes under deformations, reviewed in Sec.~1,
  together with the lesson from $C^k$-algebraic geometry, reviewed in Sec.~2,
  suggest that:
  
\bigskip 	

\noindent
{\bf Definition 3.0.1. [$C^k$-admissible ring-homomorphism to a possibly noncommutative ring].}
 Let $R$ be a $C^k$-ring and $S$ be an associative unital ring, which may not be commutative.
 A ring-homomorphism $\psi:R\rightarrow S$ is said to be  {\it $C^k$-admissible}
 if the image $\psi(R)$, as a commutative subring of $S$, has a $C^k$-ring structure
  induced from the ring-epimorphism $\psi:R\rightarrow \psi(R)$.
% end-definition

\bigskip

\noindent
{\bf  Definition 3.0.2. [$C^k$-map from Azumaya/matrix point].}
 A {\it $C^k$-map}
   $$
      \varphi\; :\;
	    (p, \End_{\Bbb C}({\Bbb C}^{\oplus r}),{\Bbb C}^{\oplus r})\;
                    \longrightarrow\;  {\Bbb R}^n
   $$
  from an Azumaya point with a fundamental module to  ${\Bbb R}^n$
  is defined contravariantly by a $C^k$-admissible ring-homomorphism
  $\varphi^{\sharp}:
     C^k({\Bbb R}^n)\rightarrow \End_{\Bbb C}({\Bbb C}^{\oplus r})$
  over ${\Bbb R}\hookrightarrow{\Bbb C}$.
% end-definition

\bigskip

\noindent
{\bf Definition 3.0.3. [D0-brane on ${\Bbb R}^n$ as a fundamental object].}
 In the region of Wilson's theory-space for string theory where a simple D0-brane mass is relatively small,   
 a {\it D0-brane} on ${\Bbb R}^n$ is a $C^k$-map
 $\varphi:(p, \End_{\Bbb C},{\Bbb C}^{\oplus r})\rightarrow {\Bbb R}^n$
 in the sense of Definition~3.0.2, for some $0\le k\le \infty$.
                       % Definition [$C^k$-map from Azumaya/matrix point]
% end-definition
  
\bigskip

This gives our proto-answer to Q1. In other words,
 \begin{itemize}
   \item[{\bf A1.}] {\bf [proto].} {\it
    Any  ring-homomorphism
         $\varphi^{\sharp}:
		   C^k({\Bbb R}^n)\rightarrow \End_{\Bbb C}({\Bbb C}^{\oplus r})$
	       over ${\Bbb R}\hookrightarrow {\Bbb C}$
	     that defines a $C^k$-ring-homomorphism
	      $C^k({\Bbb R}^n)\rightarrow \Image(\varphi^{\sharp})$
         defines  a $C^k$-map
		 $\varphi: (p^{Az},{\Bbb C}^{\oplus r})\rightarrow {\Bbb R}^n$.}
  \end{itemize}
 We shall examine
  whether
    the above proto-answer to Q1 and
	Definition~3.0.3 that defines D0-branes as fundamental objects in string theory  and
                                    	means to answer Q2 affirmatively
               % Definition [D0-brane on ${\Bbb R}^n$ as a fundamental object]										
 really make sense
     both mathematically and from the aspect of D-branes in string theory.
 We recall first in Sec.~3.1
    the reference/warm-up case of  morphisms from an Azumaya point with a fundamental module
        to the real affine space ${\Bbb A}^n_{\Bbb R}$ in algebraic geometry,
        based on [L-Y1] (D(1)),
	and then enter the new territory
      first for the $C^{\infty}$ case in Sec.~3.2,
      then for the $C^0$ case in Sec.~3.3,  and
      finally for the general $C^k$	 case in Sec.~3.4.
 Lessons learned from this case study will guide us toward the general notion of
   differentiable maps from an Azumaya/matrix manifold with a fundamental module to a real manifold
   in Sec.~5  and
  whether such notion fits to describe general D-branes as fundamental objects in string theory.

\bigskip
  
\subsection{Warm-up:
        Morphisms from an Azumaya point with a fundamental\\ module
        to the real affine space ${\Bbb A}^n_{\Bbb R}$ in algebraic geometry}
  
Following [L-Y1](D(1)),
 a morphism
 $$
    \varphi\; :\; (p,\End_{\Bbb C}({\Bbb C}^{\oplus r}),{\Bbb C}^{\oplus r})\;
	  \longrightarrow\; {\Bbb A}^n_{\Bbb R}:=\Spec {\Bbb R}[y^1,\,\cdots\,,\,y^n])	
 $$
 from the fixed Azumaya point with a fundamental module to the real affine space ${\Bbb A}^n_{\Bbb R}$
 is defined by a ring-homomorphism
 $$
    \varphi^{\sharp}\;:\;  {\Bbb R}[y^1,\,\cdots\,,\,y^n] \; \longrightarrow\;
	  \End_{\Bbb C}({\Bbb C}^{\oplus r})
 $$
 over ${\Bbb R}\hookrightarrow {\Bbb C}$.
The fundamental  $\End_{\Bbb C}({\Bbb C}^{\oplus r})$-module
  ${\Bbb C}^{\oplus r}$ is turned to a ${\Bbb R}[y^1,\,\cdots\,,\,y^n]$-module via $\varphi^{\sharp}$.
This defines a $0$-dimensional coherent sheaf
  $\varphi_{\ast}({\Bbb C}^{\oplus r})=:\Image\varphi$
 on ${\Bbb A}^n_{\Bbb R}$ of ${\Bbb C}$-length $r$.
The scheme-theoretical support
   $\Supp(\varphi_{\ast}({\Bbb C}^{\oplus r}))$
     of $\varphi_{\ast}({\Bbb C}^{\oplus r})$
   is a $0$-dimensional subscheme of ${\Bbb A}^n_{\Bbb R}$ of ${\Bbb R}$-length $\le r$,
   defined by the ideal $\Ker(\varphi^{\sharp})$ of ${\Bbb R}[y^1,\,\cdots\,,\,y^n]$.
In general,
 $\Supp(\varphi_{\ast}({\Bbb C}^{\oplus r}))$ contains
  both ${\Bbb R}$-points and ${\Bbb C}$-points of ${\Bbb A}^n_{\Bbb R}$.

\bigskip

\begin{example} {\bf [map from $p^{Az}$ to ${\Bbb A}^1_{\Bbb R}$].} {\rm
(Cf.\ [L-Y1: Sec.~4.1] (D(1)).)
 Let ${\Bbb A}^1_{\Bbb R}:= \Spec({\Bbb R}[y])$ be the affine ${\Bbb R}$-line.
 A morphism
   $\varphi:(p,\End_{\Bbb C}({\Bbb C}^{\oplus r}),{\Bbb C}^{\oplus r})
                    \rightarrow {\Bbb A}^1_{\Bbb R}$
  is then defined by a ring-homomorphism
    $\varphi^{\sharp}:
	    {\Bbb R}[y] \rightarrow \End_{\Bbb C}({\Bbb C}^{\oplus r})$
    over ${\Bbb R}\hookrightarrow {\Bbb C}$.
 Since ${\Bbb R}[y]$ is generated by $y$ as a ring over ${\Bbb R}$,
  $\varphi$  is determined by
      $m_{\varphi}
	     := \varphi^{\sharp}(y)\in \End_{\Bbb C}({\Bbb C}^{\oplus r})$.
 Let 	
 $$
     \left[
       \begin{array}{ccc}
        A_1 &         & 0 \\
            & \ddots  &   \\
        0   &         & A_l
       \end{array}
      \right]
       \hspace{1em}
        \mbox{with each $A_i\in \End_{\Bbb C}({\Bbb C}^{r_i})$
		 of the form}\hspace{1em}
      \left[
        \begin{array}{ccc}
        J^{(\lambda_i)}_{r_{i1}} &         &  0\\
                  & \ddots  &              \\
           0     &         &  J^{(\lambda_i)}_{r_{il_i}}
       \end{array}
      \right]\,,
 $$
 where
   $J^{(\lambda)}_j\in \End_{\Bbb C}({\Bbb C}^j)$ is the matrix
   $${\scriptsize
    \left[
     \begin{array}{cccc}
      \lambda   &          &        & 0       \\
      1         & \lambda  &        &         \\
                 &  \ddots  & \ddots &         \\
      0         &          &   1    & \lambda
     \end{array}
    \right]_{j\times j}    }\,,
   $$
   be the Jordan form of $m_{\varphi}$,
     with
	  $\lambda_1,\,\cdots\,,\, \lambda_s\in{\Bbb R}$,
	  $\lambda_{s+1},\,\cdots\,,\,\lambda_k\in {\Bbb C}-{\Bbb R}$,  and
      $r_{i1}\ge \,\cdots\,\ge\,r_{il_i} $  for  $i=1,\,\ldots\,,\,l$.	
 {From} the given notation, $r=r_1+\,\cdots\,+r_l$  and
   $r_i=r_{i1}+\,\cdots\,+r_{il_i}$, for $i=1,\,\ldots,\,l$. 	
 Then
   $$
     \Ker(\varphi^{\sharp})\;
	  =\; \left(\mbox{$\prod_{i=1}^s(y-\lambda_i)^{r_{i1}}
	            \cdot
	            \prod_{j={s+1}}^l
				  (y^2 -(\lambda_j+\overline{\lambda_j})y
				           + \lambda_j\overline{\lambda_j})^{r_{j1}}$}   \right)\;
	           \in\; {\Bbb R}[y]
   $$
  and
  $\Supp(\varphi_{\ast}({\Bbb C}^{\oplus r}) ) \subset {\Bbb A}^1_{\Bbb R}$
   consists of a finite collection of $0$-dimensional subschemes associated respectively to the ideals
     $$
	    \left((y-\lambda_i)^{r_{i1}}  \right)\,,\;   i=1,\,\ldots\,,s\,,
	 $$
     and
	 $$
	    \left( (y^2 -(\lambda_j+\overline{\lambda_j})y
				           + \lambda_j\overline{\lambda_j})^{r_{j1}}  \right)\,,\;   j=s+1,\,\ldots\,,l\,.
	 $$
  Each of the former contains an ${\Bbb R}$-point,
      namely the point associated to the prime ideal  $(y-\lambda_i)$ of ${\Bbb R}[y]$,
 	  for $i=1,\,\ldots\,,s$,
  while each of the latter contains a ${\Bbb C}$-point,
	   namely the point associated to the prime ideal
	   $(y^2 -(\lambda_j+\overline{\lambda_j})y
				           + \lambda_j\overline{\lambda_j})$  of ${\Bbb R}[y]$,
		  for $j=s+1,\,\ldots\,,\,l$.
		  	
 As $m_{\varphi}$ and hence $\varphi^{\sharp}$ and $\varphi$ vary,
  the push-forward $\varphi_{\ast}({\Bbb C}^{\oplus r})$  also varies along.
 This produces a Higgsing/un-Higgsing phenomenon of D$0$-branes on ${\Bbb A}^1_{\Bbb R}$,
    realized as $\varphi_{\ast}({\Bbb C}^{\oplus r})$.
	
 Cf.~{\sc Figure}~3-4-1.
              % Figure: higg-un.pdf
     
\noindent\hspace{15.7cm}$\square$
}\end{example}

\bigskip

\subsection{Smooth maps from an Azumaya point with a fundamental module
        to ${\Bbb R}^n$ as a smooth manifold}

We consider first in Sec.~3.2.1
  the notion of smooth maps from a fixed Azumaya point with a fundamental module
        to ${\Bbb R}^1$ as a smooth manifold  and
 then in Sec.~3.2.2 to the general ${\Bbb R}^n$.
The discussion here redo the related study in [Koc: III.5] of Anders Kock.

Before proceeding, recall the following technical lemma from calculus:

\bigskip

\noindent
{\bf Lemma 3.2.0.1. [Taylor expansion with remainder term - smooth case].} {\it
 Given an integer $s\ge 1$,
  let $f\in C^{\infty}({\Bbb R}^n)$,
  $\mbox{\boldmath $a$}=(a^1,\,\cdots\,,\,a^n)$ be a point on ${\Bbb R}^n$,  and
  $\mbox{\boldmath $y$}=(y^1,\,\cdots\,,\,y^n)$ be the tuple of coordinate-functions for ${\Bbb R}^n$.
 Then there exist $h_{i_1\,\cdots\,i_s}\in C^{\infty}({\Bbb R}^n)$,
     $1\le i_1,\,\cdots\,, i_s \le n$,
   with 	
 $$
    h_{i_1\,\cdots\, i_s}(\mbox{\boldmath $a$})
            = \frac{\partial^s\!f}{\partial y^{i_1}\,\cdots\,\partial y^{i_s}}
			      (\mbox{\boldmath $a$})\,,
	 \hspace{2em}\mbox{for all $\;\;1\le i_1,\,\cdots\,,\, i_s\le n$}\,,				
 $$
   such that
   \begin{eqnarray*}
   \lefteqn{
     f(\mbox{\boldmath $y$})\; =\; \sum_{s^{\prime}=0}^{s-1}\;
	         \sum_{1\le i_1,\,\cdots\,, i_{s^{\prime}} \le n}\;
			  \frac{1}{s^{\prime}!}\,
	           \frac{\partial^{s^{\prime}}\!\!f}
			     {\partial y^{i_1}\cdots \partial y^{i_{s^{\prime}}}}(\mbox{\boldmath $a$})\,
				     (y^{i_1}-a^{i_1})\cdots (y^{i_{s^{\prime}}}- a^{i_{s^{\prime}}}) }
					                                                                                              \\[.6ex]
      && \hspace{2em}					
            +\; \frac{1}{s!}\,
	  		         \sum_{1\le i_1,\,\cdots\,, i_s\le n}
			           h_{i_1\,\cdots\,i_s}({\mbox{\boldmath $y$}})\,
					       (y^{i_1}-a^{i_1})\cdots (y^{i_s}-a^{i_s})\,.			
			\hspace{8em}	
   \end{eqnarray*}
  Explicitly, $h_{i_1\,\cdots\,i_s}$ can be chosen to be the function defined through $f$ by
    $$
	  h_{i_1\,\cdots\,i_s}(\mbox{\boldmath $y$})\;
	  =\;
	    \left(\prod_{s^{\prime}=1}^ss^{\prime}!\right)
	    \int_0^1\cdots\int_0^1
		   \frac{\partial^s\!f}{\partial y^{i_1}\cdots\partial y^{i_s}}
			    \left(\rule{0em}{1em}
				   t_1\cdots t_s\,
				   (\mbox{\boldmath $y$}\,-\,\mbox{\boldmath $a$})\,+\, \mbox{\boldmath $a$}
		        \right)\,
			    \left(\prod_{i=1}^{s-1}t_i^{s-i}\right)dt_1\,\cdots\,\,dt_s\,.
	$$	
}% end-lemma}

 \begin{proof}
  This follows from an iteration of the Fundamental Theorem of Calculus and the chain rule,
      together with an adjustment of the factorial factors along the way:
   $$
     g(\mbox{\boldmath $y$})-g(\mbox{\boldmath $a$})\;
	 =\; \int_0^1\frac{d}{dt}
	        g\left(\rule{0em}{1em}
			    t(\mbox{\boldmath $y$}- \mbox{\boldmath $a$})
	                                                              + \mbox{\boldmath $a$}\right) dt\;
     =\; \sum_{i=1}^n h_i({\mbox{\boldmath $y$}})(y^i-a^i)\,,
   $$
   where
   $$
      h(\mbox{\boldmath $y$})\;
	    =\; \int_0^1 \frac{\partial\!f}{\partial y^i}		
		       \left(\rule{0em}{1em}
			    t(\mbox{\boldmath $y$}- \mbox{\boldmath $a$})
	                                                              + \mbox{\boldmath $a$}
			   \right)dt\,.
   $$
   
 \end{proof}

\bigskip

\subsubsection{Smooth maps from an Azumaya point with a fundamental module to ${\Bbb R}^1$.}
 
Let $\varphi:(p,\End_{\Bbb C}({\Bbb C}^{\oplus r}),{\Bbb C}^{\oplus r})
                          \rightarrow {\Bbb R}^1$
  be a smooth map from a fixed Azumaya point with a fundamental module to ${\Bbb R}^1$
  defined by a $C^{\infty}$-admissible ring-homomorphism
  $\varphi^{\sharp}:
      C^{\infty}({\Bbb R}^1)\rightarrow \End_{\Bbb C}({\Bbb C}^{\oplus r})$.
Let $y\in C^{\infty}({\Bbb R}^1)$ be a coordinate function on ${\Bbb R}^1$   and
  $m_{\varphi}:=\varphi^{\sharp}(y)\in \End_{\Bbb C}({\Bbb C}^{\oplus r})$.
Up to a change of bases of ${\Bbb C}^{\oplus r}$, we may assume that
 $m_{\varphi}$ coincides with its Jordan form:
 $$
   m_{\varphi}\;=\;
     \left[
       \begin{array}{ccc}
        A_1 &         & 0 \\
            & \ddots  &   \\
        0   &         & A_l
       \end{array}
      \right]
       \hspace{1em}
        \mbox{with each $A_i\in \End_{\Bbb C}({\Bbb C}^{r_i})$
		 of the form}\hspace{1em}
      \left[
        \begin{array}{ccc}
        J^{(\lambda_i)}_{r_{i1}} &         &  0\\
                  & \ddots  &              \\
           0     &         &  J^{(\lambda_i)}_{r_{il_i}}
       \end{array}
      \right]\,,
 $$
 where
   $J^{(\lambda)}_j\in \End_{\Bbb C}({\Bbb C}^j)$ is the matrix
   $${\scriptsize
    \left[
     \begin{array}{cccc}
      \lambda   &          &        & 0       \\
      1         & \lambda  &        &         \\
                 &  \ddots  & \ddots &         \\
      0         &          &   1    & \lambda
     \end{array}
    \right]_{j\times j}    }\,.
   $$
 After relabelling, we may assume that
   $r_{i1}\ge \,\cdots\,\ge\,r_{il_i} $.
 {From} the given notation, $r=r_1+\,\cdots\,+r_l$  and
   $r_i=r_{i1}+\,\cdots\,+r_{il_i}$, for $i=1,\,\ldots,\,l$.
 Since $C^{\infty}({\Bbb R}^1)$ is commutative,
  the commutator
  $$
     [\, \varphi^{\sharp}(f)\,,\,m_{\varphi}\,]\; =\;   0
  $$
  for all $f\in C^{\infty}({\Bbb R}^1)$.
 It follows that $\varphi^{\sharp}(f)$ must be of the block-diagonal form
  $$
   \left[
    \begin{array}{ccc}
     B_1 &         & 0 \\
         & \ddots  &   \\
     0   &         & B_l
    \end{array}
   \right]
  $$
  with each $B_i\in \End_{\Bbb C}({\Bbb C}^{r_i})$ of the block form
   $\left[ B_{i,\alpha\beta}\right]_{l_i\times l_i}$,  $1\le \alpha,\beta\le l_i$,
   where
   $$
    B_{i,\alpha\beta}\;
	   =\; T^{(b_{i,\alpha\beta; 1}, \,\cdots\,,\, b_{i,\alpha\beta; r_{i\alpha}})}
                                                _{r_{i\alpha}\times r_{i\beta}}
     \hspace{1ex}\mbox{for $\alpha\ge \beta$},\hspace{1em}
    B_{i,\alpha\beta}\;
	   =\; T^{(b_{i,\alpha\beta; 1}, \,\cdots\,,\, b_{i,\alpha\beta; r_{i\beta}})}
                                               _{r_{i\alpha}\times r_{i\beta}}
     \hspace{1ex}\mbox{for $\alpha<\beta$,}
   $$
   with
   $${\scriptsize
    \begin{array}{c}
     \\[1ex]
    T^{(b_1,\,\cdots\,,\, b_i)}_{i\times j}\\[1em]
    (i\le j)
    \end{array}\;
    =\;
    \left[
     \begin{array}{cccccc}
       b_1    &         &         &      &         \\[.6ex]
       b_2    & b_1     &         &      &         \\
              & b_2     & \ddots  &      &     & \hspace{1em}0    \\
       \vdots & \ddots  & \ddots  & b_1  &         \\[.6ex]
       b_i    & \cdots  &         & b_2  & b_1 &
     \end{array}
    \right]_{i\times j}\,,
    \hspace{1em}
    \begin{array}{c}
     \\[1ex]
    T^{(b_1,\,\cdots\,,\, b_j)}_{i\times j}\\[1em]
    (i\ge j)
    \end{array}\;
    =\;
    \left[
     \begin{array}{ccccc}
              &         &    0                     \\[1em]
       b_1    &         &         &      &         \\[.6ex]
       b_2    & b_1     &         &      &         \\
              & b_2     & \ddots  &      &         \\
       \vdots & \ddots  & \ddots  & b_1  &         \\[.6ex]
       b_j    & \cdots  &         & b_2  & b_1
     \end{array}
    \right]_{i\times j}\,.      }
   $$
(In all the matrix forms above,  omitted entries are zero.)
 Here, all the entries of $\varphi^{\sharp}(f)$ depends on $f$.
 It follows that  $\varphi^{\sharp}$ factors to the composition of ring-homomorphisms
  (over ${\Bbb R}\hookrightarrow{\Bbb C}$ and ${\Bbb C}$ respectively)
   $$
     \xymatrix{
      C^{\infty}({\Bbb R}^1)
                   \ar[rr]^-{(\varphi^{\sharp}_1\,,\,\cdots\,,\,\varphi^{\sharp}_l)}
			   \ar @/_4ex/[rrrr]_-{\varphi^{\sharp}}
        &&  \End_{\Bbb C}({\Bbb C}^{r_1})
                 \oplus\,\cdots\, \oplus \End_{\Bbb C}({\Bbb C}^{\oplus r_l})\;
                        \ar @{^{(}->}[rr]^-{\iota}
		&&  \End_{\Bbb C}({\Bbb C}^{\oplus r})\,,
	  }
   $$
   induced by the decomposition
     ${\Bbb C}^{\oplus r}
	   ={\Bbb C}^{\oplus r_1}\oplus\,\cdots\,\oplus{\Bbb C}^{\oplus r_l}$
	 specified by $m_{\varphi}$.
   
 Suppose that $\lambda_j\in {\Bbb C}-{\Bbb R}$ for some $j\in\{1,\,\cdots\,,\,l\}$.
 Then,
   $$
	 \left((y^2 -(\lambda_j+\overline{\lambda_j})y
				           + \lambda_j\overline{\lambda_j})^{r_{j1}}\right)   \;
     \subset\;    \Ker(\varphi^{\sharp}_j)\;
	 \subset\;    C^{\infty}({\Bbb R}^1)\,.
   $$
  Since
   $y^2-(\lambda_j+\overline{\lambda_j})y+\lambda_j\overline{\lambda_j}$
	 is invertible in $C^{\infty}({\Bbb R}^1)$,
   $\left((y^2 -(\lambda_j+\overline{\lambda_j})y
				           + \lambda_j\overline{\lambda_j})^{r_{j1}}\right)
      = C^{\infty}({\Bbb R}^1)$   and, hence,
   $\varphi^{\sharp}_j$ is the zero-map and $r_j=0$.
  This contradicts with the fact that all $r_j>0$ from the setting.
  Furthermore, in the above discussion,
   only the commutativity of $C^{\infty}({\Bbb R}^1)$ is used   and
   we may replace $m_{\varphi}$ by any $\varphi^{\sharp}(g)$
   for $g\in C^{\infty}({\Bbb R}^1)$.
  This proves that

  \bigskip

\begin{sslemma} {\bf [eigenvalues all real].}
 Let $\varphi^{\sharp}:
         C^{\infty}({\Bbb R}^1)\rightarrow \End_{\Bbb C}({\Bbb C}^{\oplus r})$
   be a ring-homomorphism over ${\Bbb R}\hookrightarrow {\Bbb C}$.
 Then,  for any $g\in C^{\infty}({\Bbb R}^1)$,
    all the eigenvalues of $\varphi^{\sharp}(g)\in\End_{\Bbb C}({\Bbb C}^{\oplus r})$
   	are real.
\end{sslemma}
  
\bigskip

\noindent
It follows that
  
\bigskip
  
\begin{sslemma}
 {\bf [push-forward $\varphi_{\ast}({\Bbb C}^{\oplus r})$ on ${\Bbb R}^1$].}
 The push-forward $\varphi_{\ast}({\Bbb C}^{\oplus r})$ on ${\Bbb R}^1$ is well-defined
 (i.e.\  supported on the real line ${\Bbb R }^1$, as a $C^{\infty}$ real manifold,
   without having to add additional ${\Bbb C}$-points,
  cf.\ Example~3.1.1)
                     % Example [map from $p^{Az}$ to ${\Bbb A}^1_{\Bbb R}$]
  as an $C^{\infty}({\Bbb R})^{\Bbb C}$-module of ${\Bbb C}$-length $r$.
\end{sslemma}		
   
\bigskip
     
Observe that
   $$
     \varphi_{\ast}({\Bbb C}^{\oplus r})\;\simeq\;
    	{\Bbb C}^{\oplus r}   \otimes_{\Bbb R}C^{\infty}({\Bbb R}^1)
		     \left /\left(
			     (\Id_r\otimes f-\varphi^{\sharp}(f)\otimes 1\,:\, f\in C^{\infty}({\Bbb R}^1))
		           \cdot  ({\Bbb C}^{\oplus r}
				   \otimes_{\Bbb R}C^{\infty}({\Bbb R}^1))\right)\right. \,.
   $$
 Let
  $$
    Z\;  := \; \{\lambda_1,\,\cdots\,,\lambda_l\}
  $$
    be the set of eigenvalues of $m_{\varphi}$   and
  $$
    U\; :=\;  {\Bbb R}^1  - \{\lambda_1,\,\cdots\,,\,\lambda_l\}
  $$
    be its open complement.
  Then	
   $\Id_r\otimes y- m_{\varphi}\otimes 1$
      is an invertible endomorphism of
	  ${\Bbb C}^{\oplus r}\otimes_{\Bbb R}C^{\infty}(U)$.
  It follows thus
     from the above presentation of $\varphi_{\ast}({\Bbb C}^{\oplus r})$
   that
    $(\varphi_{\ast}({\Bbb C}^{\oplus r}))|_U=0$.
  Together with the decomposition
    $\varphi^{\sharp}
      =\iota\circ(\varphi^{\sharp}_1,\,\cdots\,,\,\varphi^{\sharp}_l)$,
   this proves the following two lemmas:

\bigskip

\begin{sslemma}
{\bf  [dependence of  $\varphi^{\sharp}$ on germs at $Z$].}
  For $f\in C^{\infty}({\Bbb R}^1)$,
   $\varphi^{\sharp}(f)\in\End_{\Bbb C}({\Bbb C}^{\oplus r})$
   depends only on the germs of $f$ at the finite set $Z$.
\end{sslemma}
		
\bigskip
 
\begin{sslemma}
 {\bf  [$\varphi_{\ast}({\Bbb C}^{\oplus r})$ supported on $Z$].}
 (1)
    $\varphi_{\ast}({\Bbb C}^{\oplus r})$ is set-theoretically supported
    on the finite set $Z$.
 (2)
    $Z$ is naturally equipped with
	  the structure sheaf  ${\cal O}_Z$ associated to the quotient $C^{\infty}({\Bbb R}^1)$-algebra
    $C^{\infty}({\Bbb R}^1)/\Ker(\varphi^{\sharp})$,
      which, after $\otimes_{\Bbb R}{\Bbb C}$, is isomorphic to the commutative subalgebra
	${\Bbb C}\langle \Image\varphi^{\sharp} \rangle$
     	of $\End_{\Bbb C}({\Bbb C}^{\oplus r})$
	  that is generated by $\Image\varphi^{\sharp}$ and the center ${\Bbb C}\cdot\Id$
	 of $\End_{\Bbb C}({\Bbb C}^{\oplus r})$.
    By construction,
	  ${\cal O}_Z$ has ${\Bbb R}$-length $\le r$ and may contain nilpotent elements.
\end{sslemma}

\bigskip

\noindent
By construction, $\varphi_{\ast}({\Bbb C}^{\oplus r})$
 is naturally an ${\cal O}_Z^{\,\Bbb C}$ ($:= {\cal O}_Z\otimes_{\Bbb R}{\Bbb C}$)-module.
Deformations of $\varphi^{\sharp}$ (and hence $\varphi$)
 produce Higgsing/un-Higgsing phenomena of D$0$-branes on ${\Bbb R}^1$,
 realized as $\varphi_{\ast}({\Bbb C}^{\oplus r})$,
 similar to that in Example~3.1.1.
                        % Example [map from $p^{Az}$ to ${\Bbb A}^1_{\Bbb R}$]

So far we have been trying to push our understanding of $\varphi$ purely algebraically.
Lemma~3.2.0.1 from calculus enables us
            % Lemma [Taylor expansion with remainder term -- smooth case]
  to refine Lemma~3.2.1.3 further:
            % Lemma [dependence of  $\varphi^{\sharp}$ on germs at $Z$]

\bigskip

\begin{ssnotation} {\bf [Taylor expansion/jet].} {\rm
 Recall the coordinate function $y\in C^{\infty}({\Bbb R}^1)$.
 For $f\in C^{\infty}({\Bbb R}^1)$, $q\in {\Bbb R}^1$, and $d\in {\Bbb Z}_{\ge 0}$,
  let
    $$
	  T^{(y, q,d)}(f)\;
	    :=\;   \sum_{j=0}^d
				     \frac{1}{j!}\frac{\partial^j\!f}{\partial y^j}(y(q))(y-y(q))^j \;
					\in\; C^{\infty}({\Bbb R}^1)
	$$
	be the Taylor expansion of $f$ at $q$ up to order/degree $d$.
  If the value $a=y(q)$ is known, we also denote $T^{(y,q,d)}(f)$ by $T^{(y,a,d)}(f)$. 	
}\end{ssnotation}
			
\bigskip

\begin{ssproposition}
 {\bf [dependence of $\varphi^{\sharp}$ on $(r-1)$-jet at $Z$].}
  Continuing the discussion and notations.
  Let $f\in C^{\infty}({\Bbb R}^1)$.
  Then
	\begin{eqnarray*}
	  \varphi^{\sharp}(f)
	   & =
	   & \iota\circ (\varphi^{\sharp}_1(T^{(y, \lambda_1, r_{11}-1)}(f)),\,
	         \cdots\,,\,
			 \varphi^{\sharp}_l(T^{(y, \lambda_l, r_{l1}-1)}(f)))  \\
       & = 			
	   & \iota\circ (\varphi^{\sharp}_1(T^{(y, \lambda_1, r-1)}(f)),\,
	         \cdots\,,\,
			 \varphi^{\sharp}_l(T^{(y, \lambda_l, r-1)}(f)))\,.  	
    \end{eqnarray*}			
  In particular,
   $\varphi^{\sharp}:
      C^{\infty}({\Bbb R}^1)\rightarrow \End_{\Bbb C}({\Bbb C}^{\oplus r})$	
	is uniquely determined by its value $m_{\varphi}= \varphi^{\sharp}(y)$
	at the specified coordinate function $y$ of ${\Bbb R}^1$.
\end{ssproposition}			

\begin{proof}
 With the coordinate on ${\Bbb R}^1$ given by $y\in C^{\infty}({\Bbb R}^1)$,
 let $\rho_a$ be a $C^{\infty}$ bell-shaped cut-off function on ${\Bbb R}^1$
    that takes values in $[0,1]$
	with the value $1$ on $(a-\varepsilon, a+\varepsilon)$  and
	        the value $0$ on ${\Bbb R}^1-(a-2\varepsilon, a+2\varepsilon)$.
  Assume that $\varepsilon>0$ is small enough
   so that the intervals
    $[\lambda_i-2\varepsilon, \lambda_i+2\varepsilon]\subset {\Bbb R}^1$, $i=1,\,\ldots\,,\,l$,	
	are all disjoint from each other.
  It follows from
    Lemma~3.2.1.3 and the linearity of $\varphi^{\sharp}$ that
            % Lemma [dependence of  $\varphi^{\sharp}$ on germs at $Z$]
   $$
     \varphi^{\sharp}(f)\;
	   =\;  \varphi^{\sharp}(\rho_{\lambda_1}f\,  +\, \cdots\, +\, \rho_{\lambda_l}f)\;
	   =\; \varphi^{\sharp}(f_1)\,+\,\cdots\,+\,\varphi^{\sharp}(f_l)\,,
   $$
   where $f_i:=\rho_{\lambda_i}f$.
 Since $f_i\in C^{\infty}({\Bbb R}^1)$ is supported in the interval
   $(\lambda_i-2\varepsilon, \lambda_i+2\varepsilon)$
   that is disjoint from the intervals $(\lambda_j-2\varepsilon, \lambda_j+2\varepsilon)$ for all $j\ne i$,  
  $\varphi^{\sharp}(f_i)$ acts only on the direct summand
     ${\Bbb C}^{\oplus r_i}$ of ${\Bbb C}^{\oplus r}$
   specified by $m_{\varphi}$. I.e.\
  $$
   \varphi^{\sharp}(f_i)\;=\;
      \iota(0,\,\cdots\,,\,0,\,\varphi^{\sharp}_i(f_i),\,0,\,\cdots\,,0)
  $$
  and hence
  $$
   \varphi^{\sharp}(f)\;=\;
      \iota(\varphi^{\sharp}_1(f_1),\,\cdots\,,\, \varphi^{\sharp}_l(f_l))\,.
  $$
 For each $f_i$,
   consider the decomposition
   $$
     f_i\; =\;  \rho_{\lambda_i}T^{(y, \lambda_i, r_{i1}-1)}(f)
		             \,+\, h_i\,,
   $$
    in $C^{\infty}({\Bbb R}^1)$, where
	$$
	   h_i\; :=\;  f_i- \rho_{\lambda_i}T^{(y, \lambda_i, r_{i1}-1)}(f))\;
	              =\;  f_i- \rho_{\lambda_i}T^{(y, \lambda_i, r_{i1}-1)}(f_i))\,.
	$$
 By construction,
   both $\rho_{\lambda_i}T^{(y, \lambda_i, r_{i1}-1)}(f)$ and $h_i$ are supported in
	  $(\lambda_i-2\varepsilon, \lambda_i+2\varepsilon)$
	and
   $$
       \frac{\partial ^j h_i}{\partial y^j}(\lambda_i)\;=\; 0\,,
	     \hspace{2em}\mbox{for $j=0,\,\ldots\,,\,r_{i1}-1$}\,.
   $$
 Since
   $\varphi^{\sharp}_i:
      C^{\infty}({\Bbb R}^1)\rightarrow \End_{{\Bbb C}}({\Bbb C}^{r_i})$
	  is a ring-homomorphism over ${\Bbb R}\hookrightarrow {\Bbb C}$,
  it follows from
    Lemma~3.2.1.3 and
           % Lemma [dependence of  $\varphi^{\sharp}$ on germs at $Z$]
    Lemma~3.2.0.1
	      % Lemma [Taylor expansion with remainder term - smooth case]	
	%-------------------------------
	%  the Malgrange Preparation Theorem (1)
    %             % Theorem [Malgrange Preparation Theorem] (1)
    %=================			
  that
   $$
	 \varphi^{\sharp}_i(h_i)\; =\; 0\, .
   $$			
 Consequently,
  by Lemma~3.2.1.3 again and the property of $m_{\varphi}$ read off from its Jordan form,
                  % Lemma [dependence of  $\varphi^{\sharp}$ on germs at $Z$]
  $$
	\varphi^{\sharp}_i(f_i)\;
	 =\; \varphi^{\sharp}_i
	         (\rho_{\lambda_i}T^{(y, \lambda_i, r_{i1}-1)}(f_i))\;
	 =\; \varphi^{\sharp}_i(T^{(y, \lambda_i, r_{i1}-1)}(f))\;
	 =\; \varphi^{\sharp}_i(T^{(y, \lambda_i, r-1)}(f))\,.
  $$
 The proposition follows.
  
\end{proof}

\bigskip

\begin{ssremark}
{$[$algebraic nature of $C^{\infty}({\Bbb R}^1)$ under $\varphi^{\sharp}$$\,]$.}  {\rm
In other words, from the aspect of $\varphi^{\sharp}$,  $C^{\infty}({\Bbb R}^1)$
 would behave much like the polynomial ring ${\Bbb R}[y]$
 except that $C^{\infty}({\Bbb R}^1)$ contains more abundantly invertible elements,
 which renders $\varphi^{\sharp}$ more restrictive than the case in Sec.~3.1.
} \end{ssremark}

\bigskip

\subsubsection{Smooth maps from an Azumaya point with a fundamental module to ${\Bbb R}^n$}

While the study in Sec.~3.2.1 is ad hoc by hand,
 it gives us a guide to understanding
 a smooth map
  $\varphi:(p,\End_{\Bbb C}({\Bbb C}^{\oplus r}),{\Bbb C}^{\oplus r})
                   \rightarrow {\Bbb R}^n$
  from a fixed Azumaya point with a fundamental module to ${\Bbb R}^n$
  defined by a $C^{\infty}$-admissible ring-homomorphism
  $\varphi^{\sharp}:
     C^{\infty}({\Bbb R}^n)\rightarrow \End_{\Bbb C}({\Bbb C}^{\oplus r})$.

\bigskip

\begin{flushleft}
{\bf Reduction to the case
          when $\Supp(\varphi_{\ast}({\Bbb C}^{\oplus r}))$ is connected}
\end{flushleft}
Let
  $y^1,\,\cdots\,,\, y^n \in C^{\infty}({\Bbb R}^n)$
    be a set of coordinate functions on ${\Bbb R}^n$   and
  $$
     m_{\varphi}^i \;
 	 :=\;   \varphi^{\sharp}(y^i)\;  \in\;   \End_{\Bbb C}({\Bbb C}^{\oplus r})\,.
  $$
The same argument as for Lemma~3.2.1.1 implies that all the eigenvalues of  $m_{\varphi}^i$ are real.
                                           % Lemma [eigenvalues all real]
And it follows inductively from the standard form of a  pair of communiting $r\times r$ matrices,
 used already in Sec.~3.2.1,  that
up to a change of basis of ${\Bbb C}^{\oplus r}$, there exists a decomposition
 $$
   {\Bbb C}^{\oplus r}\;
    =\; {\Bbb C}^{\oplus r_1}\oplus \,\cdots\,\oplus {\Bbb C}^{\oplus r_l}
 $$
 such that
  \begin{itemize}
   \item[(1)]
   $\varphi^{\sharp}$ factors accordingly to the composition of ring-homomorphisms
  (over ${\Bbb R}\hookrightarrow{\Bbb C}$ and ${\Bbb C}$ respectively)
   $$
     \xymatrix{
      C^{\infty}({\Bbb R}^n)
                   \ar[rr]^-{(\varphi^{\sharp}_1\,,\,\cdots\,,\,\varphi^{\sharp}_l)}
			   \ar @/_4ex/[rrrr]_-{\varphi^{\sharp}}
        &&  \End_{\Bbb C}({\Bbb C}^{\oplus r_1})
                 \oplus\,\cdots\, \oplus \End_{\Bbb C}({\Bbb C}^{\oplus r_l})\;
                        \ar @{^{(}->}[rr]^-{\iota}
		&&  \End_{\Bbb C}({\Bbb C}^{\oplus r})\,;
	  }
	$$
	
   \item[(2)]
    for each $1\le l^{\prime}\le l$,
    there exist $\lambda_{l^{\prime}}^i\in {\Bbb R}$, $1\le i\le n$,
	such that
     $\varphi^{\sharp}_{l^{\prime}}(y^i)$ has the unique eigenvalue
	 $\lambda_{l^{\prime}}^i$ of multiplicity $r_{l^{\prime}}$.
  (I.e.\
       $\left(\varphi^{\sharp}_{l^{\prime}}(y^i)
	      -\lambda_{l^{\prime}}^i\cdot\Id_{r_{l^{\prime}}\times r_{l^{\prime}}}
		  \right)^{r_{l^{\prime}}}=0$.)	
 \end{itemize}
For later use,  denote the tuple
 $$
   (\lambda^1_{l^{\prime}},\,\cdots\,,\,\lambda^n_{l^{\prime}})\;
      =:\;  \mbox{\boldmath $\lambda$}_{l^{\prime}}\,.
 $$
 
This reduces the study of $\varphi^{\sharp}$ to the study of each $\varphi_{l^{\prime}}$.
For the simplicity of notations and without loss of generality, one may assume that $l=1$,
  which corresponds to the case
            when $\Supp(\varphi_{\ast}({\Bbb C}^{\oplus r}))$ is connected,
  and drop the subscript $l^{\prime}$.

\bigskip

\begin{flushleft}
{\bf The case when $\Supp(\varphi_{\ast}({\Bbb C}^{\oplus r}))$ is connected}
\end{flushleft}
Since $C^{\infty}({\Bbb R}^n)$ is commutative,
 the set $\{\varphi^{\sharp}(f)\,|\, f\in C^{\infty}({\Bbb R}^n)\}$
 forms a family of commuting $r\times r$-matrices over ${\Bbb C}$
  and hence can be simultaneously triangulated.
Thus, subject to a change of basis of ${\Bbb C}^{\oplus r}$,
 we may assume that  all $\varphi^{\sharp}(f)$, $f\in C^{\infty}({\Bbb C}^n)$,
 are lower triangular.
In this form,
 the diagonal entries of $\varphi^{\sharp}(f)$ then correspond to
  eigenvalues of $\varphi^{\sharp}(f)$.
By the same argument as in Sec.~3.2.1, they must be all real.
Thus, after the post-composition of $\varphi^{\sharp}$ with the ${\Bbb C}$-algebra-epimorphism
  from the ring of $r\times t$-lower-triangular matrices over ${\Bbb C}$
  to the rings of $r\times r$-diagonal matrices over ${\Bbb C}$,
 one obtains a $C^{\infty}$-ring-homomorphism
 $$
   C^{\infty}({\Bbb R}^n)\; \longrightarrow \;
                           \underbrace{{\Bbb R}\times \cdots \times {\Bbb R}}
	                                                   _{\scriptsize  \mbox{($r$-many times)}}\,.
 $$
Note that the latter product ring is the $C^{\infty}$-function ring of the $0$-dimensional manifold
  that consists of  $r$-many distinct points.

Let ${\Bbb R}\hookrightarrow {\Bbb R}\times\cdots\times {\Bbb R}$
 be the diagonal ring-homomorphism.
Then, the requirement that
   $\varphi^{\sharp}:C^{\infty}({\Bbb R}^n)\rightarrow \Image\varphi^{\sharp}$
     be a $C^{\infty}$-ring-homomorphism    and
	that $\lambda^i$ be the unique eigenvalue of each $m_{\varphi}^i$
 implies that the above induced 	$C^{\infty}$-ring-homomorphism
 $C^{\infty}({\Bbb R}^n)\rightarrow {\Bbb R}\times \cdots \times {\Bbb R}$
  factors through a composition of $C^{\infty}$-ring-homomorphisms
  $$
    C^{\infty}({\Bbb R}^n)\; \longrightarrow\;  {\Bbb R}\; 	
	 \hookrightarrow\;  {\Bbb R}\times \cdots \times {\Bbb R}\,.
  $$
It follows that, for all $f\in C^{\infty}({\Bbb R}^n)$,
 $\varphi^{\sharp}(f)$ has the unique eigenvalue $f(\lambda^1,\,\cdots\,,\lambda^n)$
 with multiplicity $r$. That is,
 $$
    \left(  \varphi^{\sharp}(f)
	  - f(\lambda^1,\,\cdots\,,\,\lambda^n)\cdot\Id_{r\times r}  \right)^r\;=\;0\,.
 $$
 
Let $s\ge n(r-1)+1$ and
denote the tuple $(\lambda^1,\,\cdots,\,\,\lambda^r)$ by {\boldmath $\lambda$},
 regarded also as a point in ${\Bbb R}^n$ with coordinates
 $\mbox{\boldmath $y$} :=(y^1,\,\cdots\,,\,y^n)$.
Then, it follows from Lemma~3.2.0.1 that
                               % Lemma  [Taylor expansion with remainder term - smooth case]
 there exist $h_{i_1\,\cdots\,i_s}\in C^{\infty}({\Bbb R}^n)$,
     $1\le i_1,\,\cdots\,, i_s \le n$,
   with 	
 $$
    h_{i_1\,\cdots\, i_s}(\mbox{\boldmath $\lambda$})
            = \frac{\partial^s\!f}{\partial y^{i_1}\,\cdots\,\partial y^{i_s}}
			      (\mbox{\boldmath $\lambda$})\,,
	 \hspace{2em}\mbox{for all $\;\;1\le i_1,\,\cdots\,,\, i_s\le n$}\,,				
 $$
   such that
   \begin{eqnarray*}
   \lefteqn{
     f(\mbox{\boldmath $y$})\; =\; \sum_{s^{\prime}=0}^{s-1}\;
	         \sum_{1\le i_1,\,\cdots\,, i_{s^{\prime}} \le n}\;
			  \frac{1}{s^{\prime}!}\,
	           \frac{\partial^{s^{\prime}}\!\!f}
			     {\partial y^{i_1}\cdots \partial y^{i_{s^{\prime}}}}
				                                                                            (\mbox{\boldmath $\lambda$})\,
				     (y^{i_1}-\lambda^{i_1})
					     \cdots (y^{i_{s^{\prime}}}- \lambda^{i_{s^{\prime}}}) }
					                                                                                              \\[.6ex]
      && \hspace{2em}					
            +\; \frac{1}{s!}\,
	  		         \sum_{1\le i_1,\,\cdots\,, i_s\le n}
			           h_{i_1\,\cdots\,i_s}({\mbox{\boldmath $y$}})\,
					       (y^{i_1}-\lambda^{i_1})\cdots (y^{i_s}-\lambda^{i_s})\,.			
			\hspace{8em}	
   \end{eqnarray*}
It follows that
   \begin{eqnarray*}
    \lefteqn{
     \varphi^{\sharp}(f) \; =\; \sum_{s^{\prime}=0}^{s-1}\;
	         \sum_{1\le i_1,\,\cdots\,, i_{s^{\prime}} \le n}\;
			  \frac{1}{s^{\prime}!}\,
	           \frac{\partial^{s^{\prime}}\!\!f}
			     {\partial y^{i_1}\cdots \partial y^{i_{s^{\prime}}}}
				                                                                            (\mbox{\boldmath $\lambda$})\,
				     (m_{\varphi}^{i_1}-\lambda^{i_1}\cdot\Id_{r\times r})
					     \cdots (m_{\varphi}^{i_{s^{\prime}}}
						                      - \lambda^{i_{s^{\prime}}}\cdot\Id_{r\times r}) }
					                                                                                              \\[.6ex]
      &&   \hspace{2.4em}
  	  +\; \frac{1}{s!}\,
	  		         \sum_{1\le i_1,\,\cdots\,, i_s\le n}
			           \varphi^{\sharp}(h_{i_1\,\cdots\,i_s})\,
					       (m_{\varphi}^{i_1}-\lambda^{i_1}\cdot\Id_{r\times r})
						      \cdots (m_{\varphi}^{i_s}-\lambda^{i_s}\cdot\Id_{r\times r})\,.	
             \hspace{4em} 							
  \end{eqnarray*}
Now the product
 $$
  (m_{\varphi}^{i_1}-\lambda^{i_1}\cdot\Id_{r\times r})
						      \cdots (m_{\varphi}^{i_s}-\lambda^{i_s}\cdot\Id_{r\times r})
 $$
 in each term in the last summation
 contains $s$-many factors,
  each of the form $m_{\varphi}^i-\lambda^i\cdot\Id_{r\times r}$ for some $i=1,\,\cdots\,, n$.
Since $s\ge n(r-1)+1$, there must be at least one $i$ that occurs for  more than $(r-1)$-many
 times.
Since $(m_{\varphi}^i-\lambda^i\cdot\Id_{r\times r})^r\;=\; 0$  for all $i$,
 the last  summation must vanish.
This proves the following proposition:
 
\bigskip

\begin{ssproposition}
{\bf [dependence of $\varphi^{\sharp}$ on $n(r-1)$-jet at {\boldmath $\lambda$}].}
 $$
    \varphi^{\sharp}(f) \; =\; \sum_{s^{\prime}=0}^{n(r-1)}\;
	         \sum_{1\le i_1,\,\cdots\,, i_{s^{\prime}} \le n}\;
			  \frac{1}{s^{\prime}!}\,
	           \frac{\partial^{s^{\prime}}\!\!f}
			     {\partial y^{i_1}\cdots \partial y^{i_{s^{\prime}}}}
				                                                                            (\mbox{\boldmath $\lambda$})\,
				     (m_{\varphi}^{i_1}-\lambda^{i_1}\cdot\Id_{r\times r})
					     \cdots (m_{\varphi}^{i_{s^{\prime}}}
						                      - \lambda^{i_{s^{\prime}}}\cdot\Id_{r\times r})\,.
 $$
  In other words,
    $\varphi^{\sharp}(f)$ depends only  on
      $\varphi^{\sharp}(y^1),\,\cdots\,,\,\varphi^{\sharp}(y^n)$,  and
     the Taylor expansion of $f$ at the tuple of eigenvalues $(\lambda^1,\,\cdots\,,\,\lambda^n)$
      of $\varphi^{\sharp}(y^1),\,\cdots\,,\,\varphi^{\sharp}(y^n)$
	 up to order $n(r-1)$.
\end{ssproposition}

\bigskip

\begin{flushleft}
{\bf The general case}
\end{flushleft}
Resume now to the general case, i.e.\ the case when $l\ge 2$.
As in Sec.~3.2.1,
let $\mbox{\boldmath $q$} = (q^1,\,\cdots\,, q^n)$ be a point in ${\Bbb R}^n$
  with tuple of coordinate functions $\mbox{\boldmath $y$} = (y^1,\,\cdots\,y^n)$,
and denote
 $$
   T^{(\mbox{\scriptsize\boldmath $y$}, \mbox{\scriptsize\boldmath $q$}, d)}(f)\;
    :=\;  \mbox{the Taylor expansion of $f\in C^{\infty}({\Bbb R}^n)$ at {\boldmath $q$}
                             up to order/degree $d$	}\,.
 $$
Let
 $$
   Z\;  =\;
    \{\, \mbox{\boldmath $\lambda$}_1\,,\, \cdots\,,\, \mbox{\boldmath $\lambda$}_l\,\}\;
	\subset\; {\Bbb R}^n\,.
 $$
Then, it follows immediately from the above discussion that:

\bigskip

\begin{ssproposition}    {\bf [dependence of $\varphi^{\sharp}$ on $n(r-1)$-jet at $Z$].}
 Continuing the notations above and at the beginning of the current subsubsection.
 Let $f\in C^{\infty}({\Bbb R}^n)$.
 Then
   \begin{eqnarray*}
	  \varphi^{\sharp}(f)
	   & =
       & \iota \circ
	          (\varphi^{\sharp}_1(
	                 T^{(\mbox{\scriptsize\boldmath $y$},
					              \mbox{\scriptsize\boldmath $\lambda$}_1, n(r_1-1))}(f)),\,
	         \cdots\,,\,
			 \varphi^{\sharp}_l(
			         T^{(\mbox{\scriptsize\boldmath $y$},
					              \mbox{\scriptsize\boldmath $\lambda$}_l, n(r_l-1))}(f)))  \\
       & =
       & \iota \circ
	          (\varphi^{\sharp}_1(
	                 T^{(\mbox{\scriptsize\boldmath $y$},
					              \mbox{\scriptsize\boldmath $\lambda$}_1, n(r-1))}(f)),\,
	         \cdots\,,\,
			 \varphi^{\sharp}_l(
			         T^{(\mbox{\scriptsize\boldmath $y$},
					              \mbox{\scriptsize\boldmath $\lambda$}_l, n(r-1))}(f)))\,.
   \end{eqnarray*}								
 In particular,
   $\varphi^{\sharp}:C^{\infty}({\Bbb R}^n)
       \rightarrow \End_{\Bbb C}({\Bbb C}^{\oplus r})$	
	is uniquely determined by its values $m_{\varphi}^i := \varphi^{\sharp}(y^i)$
	at any specified tuple of coordinate functions $\mbox{\boldmath $y$}=(y^1,\,\cdots\,,\,y^n)$
	of ${\Bbb R}^n$.
\end{ssproposition}

\bigskip

In the last theme
   `The meaning of a $C^k$-map
    $\varphi: (p^{A\!z},{\Bbb C}^{\oplus r}) \rightarrow {\Bbb R}^n$
    being of algebraic type'
  of Sec.~3.4,
 the algebraic nature of $C^{\infty}({\Bbb R}^n)$ under $\varphi^{\sharp}$
  the $0$-dimension sheaf $\varphi_{\ast}({\Bbb C}^{\oplus r})$  on ${\Bbb R}^n$
  as shown in this subsection will be explained more clearly.

\bigskip

\subsection{Continuous maps from an Azumaya point with a fundamental module
	 to ${\Bbb R}^n$ as a topological manifold}

Let $\varphi:(p,\End_{\Bbb C}({\Bbb C}^{\oplus r}),{\Bbb C}^{\oplus r})
           \rightarrow {\Bbb R}^n$
  be a continuous map from a fixed Azumaya point with a fundamental module to ${\Bbb R}^n$
  defined by a $C^0$-admissible ring-homomorphism
  $\varphi^{\sharp}: C^0({\Bbb R}^n)
      \rightarrow \End_{\Bbb C}({\Bbb C}^{\oplus r})$.
Note that,
  with $C^{\infty}({\Bbb R}^n)$ replaced by $C^0({\Bbb R}^n)$
          and $C^{\infty}$-ring-homomorphisms replaced by $C^0$-ring-homomorphisms,
the part of the otherwise-purely-algebraic discussion in Sec.~3.2 for the $C^{\infty}$ case
   before bringing in the notion of Taylor expansions of functions on ${\Bbb R}^n$
 remains valid for the $C^0$ case.
We recall it in the first theme below and then employ a special feature of $C^0({\Bbb R}^n)$,
     not shared by all other $C^k({\Bbb R}^n)$ with $k>0$,
  to characterize $\varphi^{\sharp}$.

\bigskip

\begin{flushleft}
{\bf  Reduction to the case
           when $\Supp(\varphi_{\ast}({\Bbb C}^{\oplus r}))$ is connected}
\end{flushleft}
Let
  $y^1,\,\cdots\,,\, y^n \in C^0({\Bbb R}^n)$
    be a set of coordinate functions on ${\Bbb R}^n$   and
  $$
     m_{\varphi}^i \;
	 :=\;   \varphi^{\sharp}(y^i)\;  \in\;   \End_{\Bbb C}({\Bbb C}^{\oplus r})\,.
  $$
Then,
 up to a change of basis of ${\Bbb C}^{\oplus r}$, there exists a decomposition
 $$
   {\Bbb C}^{\oplus r}\;
   =\; {\Bbb C}^{\oplus r_1}\oplus \,\cdots\,\oplus {\Bbb C}^{\oplus r_l}
 $$
 such that
  \begin{itemize}
   \item[(1)]
   $\varphi^{\sharp}$ factors accordingly to the composition of ring-homomorphisms
  (over ${\Bbb R}\hookrightarrow{\Bbb C}$ and ${\Bbb C}$ respectively)
   $$
     \xymatrix{
      C^0({\Bbb R}^n)
                   \ar[rr]^-{(\varphi^{\sharp}_1\,,\,\cdots\,,\,\varphi^{\sharp}_l)}
			   \ar @/_4ex/[rrrr]_-{\varphi^{\sharp}}
        &&  \End_{\Bbb C}({\Bbb C}^{\oplus r_1})
                 \oplus\,\cdots\, \oplus \End_{\Bbb C}({\Bbb C}^{\oplus r_l})\;
                        \ar @{^{(}->}[rr]^-{\iota}
		&&  \End_{\Bbb C}({\Bbb C}^{\oplus r})\,;
	  }
	$$
	
   \item[(2)]
    for each $1\le l^{\prime}\le l$,
    there exist $\lambda_{l^{\prime}}^i\in {\Bbb R}$, $1\le i\le n$,
	such that
     $\varphi^{\sharp}_{l^{\prime}}(y^i)$ has the unique eigenvalue
	 $\lambda_{l^{\prime}}^i$ of multiplicity $r_{l^{\prime}}$.
  (I.e.\
       $\left(\varphi^{\sharp}_{l^{\prime}}(y^i)
	      -\lambda_{l^{\prime}}^i\cdot\Id_{r_{l^{\prime}}\times r_{l^{\prime}}}
		  \right)^{r_{l^{\prime}}}=0$.)	
 \end{itemize}
Denote the tuple
 $$
   (\lambda^1_{l^{\prime}},\,\cdots\,,\,\lambda^n_{l^{\prime}})\;
      =:\;  \mbox{\boldmath $\lambda$}_{l^{\prime}}\,.
 $$
This reduces the study of $\varphi^{\sharp}$ to the study of each $\varphi_{l^{\prime}}$.
For the simplicity of notations and without loss of generality, one may assume that $l=1$,
  which corresponds to the case
     when $\Supp(\varphi_{\ast}({\Bbb C}^{\oplus r}))$ is connected,
  and drop the subscript $l^{\prime}$.
In this case,
 the fact that $\varphi^{\sharp}$ is $C^0$-admissible
  implies that
   $\varphi^{\sharp}(f)$ has the unique eigenvalue $f(\lambda^1,\,\cdots\,,\lambda^n)$
  with multiplicity $r$. That is,
 $$
    \left(  \varphi^{\sharp}(f)
	  - f(\lambda^1,\,\cdots\,,\,\lambda^n)\cdot\Id_{r\times r}  \right)^r\;=\;0\,.
 $$
Denote the tuple $(\lambda^1,\,\cdots\,,\,\lambda^n)$ by {\boldmath $\lambda$}.
Then, it follows that
 $\varphi^{\sharp}(f-f(\mbox{\boldmath $\lambda$}))$  is nilpotent
  in $\End_{\Bbb C}({\Bbb C}^{\oplus r})$,
 with $(\varphi^{\sharp}(f-f(\mbox{\boldmath $\lambda$})))^r=0$.

\bigskip
 
\begin{flushleft}
{\bf A special feature of $C^0({\Bbb R}^n)$  and the characterization of $\varphi^{\sharp}$}
\end{flushleft}
\begin{terminology} $[$non-negative function, $s$-root$\,]$. {\rm
 For an $h\in C^0({\Bbb R}^n)$, we say that $h$ is {\it non-negative},   denoted by
  $$
     h\;  \ge \;  0\,,
  $$
  if the value of $h$ at every point of ${\Bbb R}^n$ is greater than or equal to $0$.
 Then
   \begin{itemize}
     \item[$\cdot$]  {\bf [feature of $C^0({\Bbb R}^n)$]}\hspace{1em}
	  For any $h\in C^0({\Bbb R}^n)$ such that $h\ge 0$,
	   its {\it $s$-root} $h^{\frac{1}{s}}$,
	     i.e.\ the non-negative function on ${\Bbb R}^n$ whose $s$-power is $h$,
	   is also in $C^0({\Bbb R}^n)$,
	   for all $s\in {\Bbb Z}_{\ge 1}$.
   \end{itemize}
 Note that this is a special feature of $C^0({\Bbb R}^n)$ not shared with any other
   $C^k({\Bbb R}^n)$ with $k>0$.
}\end{terminology}

\bigskip

Continue the discussion of $\varphi^{\sharp}$.
Let
  $$
    g_{(f,\mbox{\scriptsize\boldmath $\lambda$})}\;
	  :=\;  f-f(\mbox{\boldmath $\lambda$}) \,,  \hspace{2em}
    g_{(f,\mbox{\scriptsize\boldmath $\lambda$})}^+\;
      :=\;  \max\{ g_{(f,\mbox{\scriptsize\boldmath $\lambda$})}  , 0\}\,,
	    \hspace{2em}\mbox{and}\hspace{2em}
   g_{(f,\mbox{\scriptsize\boldmath $\lambda$})}^-\;
      :=\;  \max\{ - g_{(f,\mbox{\scriptsize\boldmath $\lambda$})}, 0\}\,.
  $$
Then
 both
    $g_{(f,\mbox{\scriptsize\boldmath $\lambda$})}^+$,
	$g_{(f,\mbox{\scriptsize\boldmath $\lambda$})}^-\in C^0({\Bbb R}^n)$
	are non-negative, and
  $$
     g_{(f,\mbox{\scriptsize\boldmath $\lambda$})}\;
      =\;  g_{(f,\mbox{\scriptsize\boldmath $\lambda$})}^+
	         - g_{(f,\mbox{\scriptsize\boldmath $\lambda$})}^-
	 \hspace{2em}\mbox{with}\hspace{2em}
     g_{(f,\mbox{\scriptsize\boldmath $\lambda$})}(\mbox{\boldmath $\lambda$})\;
      =\;  g_{(f,\mbox{\scriptsize\boldmath $\lambda$})}^+(\mbox{\boldmath $\lambda$})\;
	  =\;  g_{(f,\mbox{\scriptsize\boldmath $\lambda$})}^-(\mbox{\boldmath $\lambda$})\;
      =\; 0\,.
  $$	
It follows that
  both $\varphi^{\sharp}(g_{(f,\mbox{\scriptsize\boldmath $\lambda$})}^+)$ and
	       $\varphi^{\sharp}(g_{(f,\mbox{\scriptsize\boldmath $\lambda$})}^-)$
   are also nilpotent.		
 Furthermore,
  since
    $g_{(f,\mbox{\scriptsize\boldmath $\lambda$})}^+$ and
	$g_{(f,\mbox{\scriptsize\boldmath $\lambda$})}^-$
	are non-negative,
  both of the $s$-roots
    $(g_{(f,\mbox{\scriptsize\boldmath $\lambda$})}^+)^{\frac{1}{s}}$ and
	$(g_{(f,\mbox{\scriptsize\boldmath $\lambda$})}^-)^{\frac{1}{s}}$,
       $s\in {\Bbb Z}_{\ge 1}$,
	exist in $C^0({\Bbb R}^n)$
 and they satisfy
 $$
    (g_{(f,\mbox{\scriptsize\boldmath $\lambda$})}^+)^{\frac{1}{s}}
	      (\mbox{\boldmath $\lambda$})\;
	  =\;  (g_{(f,\mbox{\scriptsize\boldmath $\lambda$})}^-)^{\frac{1}{s}}
	      (\mbox{\boldmath $\lambda$})\;
      =\; 0\,,
 $$
 for all $s\in{\Bbb Z}_{\ge 1}$.
Thus,
   both  $\varphi^{\sharp}
               ((g_{(f,\mbox{\scriptsize\boldmath $\lambda$})}^+)^{\frac{1}{s}})$  and
	  $\varphi^{\sharp}
	     ((g_{(f,\mbox{\scriptsize\boldmath $\lambda$})}^-)^{\frac{1}{s}})$
    are nilpotent as well, for all $s\in {\Bbb Z}_{\ge 1}$.			

Now let $s\ge r$.
Then
 $$
  \varphi^{\sharp}(  g_{(f,\mbox{\scriptsize\boldmath $\lambda$})}^+)\;
   =\; \left(
          \varphi^{\sharp}(
		    (g_{(f,\mbox{\scriptsize\boldmath $\lambda$})}^+)^{\frac{1}{s}})
         \right)^s\;
   =\;0
     \hspace{1em}\mbox{and}\hspace{1em}
   \varphi^{\sharp}(  g_{(f,\mbox{\scriptsize\boldmath $\lambda$})}^-)\;
   =\; \left(
          \varphi^{\sharp}(
		    (g_{(f,\mbox{\scriptsize\boldmath $\lambda$})}^-)^{\frac{1}{s}})
         \right)^s\;
   =\;0\,.
 $$
It follows that
 $$
    \varphi^{\sharp}(g_{(f,\mbox{\scriptsize\boldmath $\lambda$})})\;=\; 0
 $$
 and
 $$
   \varphi^{\sharp}(f)\;=\;  f(\mbox{\boldmath $\lambda$})\cdot\Id_{r\times r}\,.
 $$
This characterizes $\varphi^{\sharp}$  for the case when $l=1$.

Resuming the setting and notations in the beginning of the current subsection, let
 $$
     Z\; :=\;   \{\,
	     \mbox{\boldmath $\lambda$}_1\,,\,\cdots\,,\,    \mbox{\boldmath $\lambda$}_l\, \}\,.
  $$
Then the characterization of $\varphi^{\sharp}$ in the case when $l\ge 2$ follows immediately
 from that in the case when $l=1$:

\bigskip

\begin{proposition} {\bf [dependence of $\varphi^{\sharp}$ on evaluation at $Z$].}
 Let $f\in C^0({\Bbb R}^n)$.
 Then,
  $$
	  \varphi^{\sharp}(f)\;
	  =\;  \iota \circ
	          (f(\mbox{\boldmath $\lambda$}_1)\cdot\Id_{r_1\times r_1}\,,
	               \cdots\,,\,  f(\mbox{\boldmath $\lambda$}_l)\cdot\Id_{r_l\times r_l} )\,.
 $$	
 In particular,
   $\varphi^{\sharp}:
       C^0({\Bbb R}^n)\rightarrow \End_{\Bbb C}({\Bbb C}^{\oplus r})$	
	is uniquely determined by its values $m_{\varphi}^i := \varphi^{\sharp}(y^i)$
	at any specified tuple of coordinate functions $\mbox{\boldmath $y$}=(y^1,\,\cdots\,,\,y^n)$
	of ${\Bbb R}^n$.
\end{proposition}

\bigskip

In the last theme
   `The meaning of a $C^k$-map
    $\varphi: (p^{A\!z},{\Bbb C}^{\oplus r}) \rightarrow {\Bbb R}^n$
     being of algebraic type'
 of Sec.~3.4,
 the algebraic nature of $C^0({\Bbb R}^n)$ under $\varphi^{\sharp}$
  the $0$-dimension sheaf $\varphi_{\ast}({\Bbb C}^{\oplus r})$  on ${\Bbb R}^n$
  as shown in this subsection will be explained more clearly.

\bigskip

\subsection{$C^k$-maps from an Azumaya point with a fundamental module to ${\Bbb R}^n$
                            as a $C^k$-manifold}

With lessons learned from the $C^{\infty}$ case in Sec.~3.2 and the $C^0$ case in Sec.~3.3,
 we now proceed to understand $\varphi^{\sharp}$ in the general $C^k$ case.
 
\bigskip

\begin{definition} {\bf [$C^k$-map of algebraic type].} {\rm
 A $C^k$-map
      $\varphi: (p, \End_{\Bbb C}({\Bbb C}^{\oplus r}),{\Bbb C}^{\oplus r})
          	             \rightarrow {\Bbb R}^n$,
    defined by a $C^k$-admissible ring-homomorphism
      $\varphi^{\sharp}:
	    C^k({\Bbb R}^n)\rightarrow \End_{\Bbb C}({\Bbb C}^{\oplus r})$
      over ${\Bbb R}\hookrightarrow{\Bbb C}$, 	  	
   is said to be {\it of algebraic type}
  if there exist
      a finite set of points $\{p_1,\,\cdots\,,\, p_l\}\subset {\Bbb R}^n$ for some $l$  and
	  a non-negative integer $d\le k$
   such that
     $\varphi^{\sharp}(f)$ depends only on the Taylor expansion of $f$ at $\{p_1,\,\cdots\,,\, p_l \}$
	  up to order $d$
     for all $f \in C^k({\Bbb R}^n)$.	
} \end{definition}

\bigskip

{From} the investigation in Sec.~3.2 and  Sec.~3.3, using ad hoc method in hand,
we see that
 \begin{itemize}
  \item[$\cdot$]
 {\it Every smooth (i.e. $C^{\infty}$-)map
    $\varphi:  (p,\End_{\Bbb C}({\Bbb C}^{\oplus r}),{\Bbb C}^{\oplus r})
	                     \rightarrow {\Bbb R}^n$
	is of algebraic type.}
	
  \item[$\cdot$]
 {\it Every continuous (i.e.\ $C^0$-)map
    $\varphi:(p, \End_{\Bbb C}({\Bbb C}^{\oplus r}),{\Bbb C}^{\oplus r})
	                  \rightarrow {\Bbb R}^n$
	is of algebraic type.}
 \end{itemize}
It turns that the above two extreme cases can be generalized
 to the general $k$-times-differentiable (i.e.\ $C^k$) case:

\bigskip

\begin{proposition}
{\bf [fundamental: Every differentiable map from Azumaya/matrix point is of algebraic type].}
 Every $C^k$-map
    $\varphi:(p, \End_{\Bbb C}({\Bbb C}^{\oplus r}),{\Bbb C}^{\oplus r})
	                   \rightarrow {\Bbb R}^n$
    from an Azumaya point with a fundamental module to the $C^k$-manifold	${\Bbb R}^n$
  is of algebraic type.
\end{proposition}	

\bigskip

\noindent
This gives a structural description of maps of class $C^k$ in question.
This subsection is mainly devoted to the proof of this fundamental proposition,
 using the more robust link between equations on functionals on $C^k({\Bbb R}^n)$
   associated to $\varphi^{\sharp}$
 and differential operators acting on $C^k({\Bbb R}^n)$.

\bigskip

\begin{flushleft}
{\bf  Proof of Proposition~3.4.2 [fundamental]}
\end{flushleft}
Let
  $\varphi^{\sharp}:C^k({\Bbb R}^n)
     \rightarrow \End_{\Bbb C}({\Bbb C}^{\oplus r})$
    be the underlying $C^k$-admissible ring-homomorphism over ${\Bbb R}\hookrightarrow {\Bbb C}$
 that defines $\varphi$.

\bigskip

\noindent
{\it $(a)$ Reduction to the case
                        when $\Supp(\varphi_{\ast}({\Bbb C}^{\oplus r}))$ is connected}

\medskip

\noindent
Exactly the same as in the $C^{\infty}$ case and the $C^0$ case,
let
  $y^1,\,\cdots\,,\, y^n \in C^k({\Bbb R}^n)$
    be a set of coordinate functions on ${\Bbb R}^n$   and
  $$
     m_{\varphi}^i \;
	 :=\;   \varphi^{\sharp}(y^i)\;  \in\;   \End_{\Bbb C}({\Bbb C}^{\oplus r})\,.
  $$
Then,
 up to a change of basis of ${\Bbb C}^{\oplus r}$, there exists a decomposition
 $$
   {\Bbb C}^{\oplus r}\;
    =\; {\Bbb C}^{\oplus r_1}\oplus \,\cdots\,\oplus {\Bbb C}^{\oplus r_l}
 $$
 such that
  \begin{itemize}
   \item[(1)]
   $\varphi^{\sharp}$ factors accordingly to the composition of ring-homomorphisms
  (over ${\Bbb R}\hookrightarrow{\Bbb C}$ and ${\Bbb C}$ respectively)
   $$
     \xymatrix{
      C^k({\Bbb R}^n)
                   \ar[rr]^-{(\varphi^{\sharp}_1\,,\,\cdots\,,\,\varphi^{\sharp}_l)}
			   \ar @/_4ex/[rrrr]_-{\varphi^{\sharp}}
        &&  \End_{\Bbb C}({\Bbb C}^{r_1})
                 \oplus\,\cdots\, \oplus \End_{\Bbb C}({\Bbb C}^{r_l})\;
                        \ar @{^{(}->}[rr]^-{\iota}
		&&  \End_{\Bbb C}({\Bbb C}^{\oplus r})\,;
	  }
	$$
	
   \item[(2)]
    for each $1\le l^{\prime}\le l$,
    there exist $\lambda_{l^{\prime}}^i\in {\Bbb R}$, $1\le i\le n$,
	such that
     $\varphi^{\sharp}_{l^{\prime}}(y^i)$ has the unique eigenvalue
	 $\lambda_{l^{\prime}}^i$ of multiplicity $r_{l^{\prime}}$.
  (I.e.\
       $\left(\varphi^{\sharp}_{l^{\prime}}(y^i)
	      -\lambda_{l^{\prime}}^i\cdot\Id_{r_{l^{\prime}}\times r_{l^{\prime}}}
		  \right)^{r_{l^{\prime}}}=0$.)	
 \end{itemize}
 
 Denote the tuple
 $$
   (\lambda^1_{l^{\prime}},\,\cdots\,,\,\lambda^n_{l^{\prime}})\;
      =:\;  \mbox{\boldmath $\lambda$}_{l^{\prime}}\,.
 $$
This reduces the study of $\varphi^{\sharp}$ to the study of each $\varphi_{l^{\prime}}$.
For the simplicity of notations and without loss of generality, one may assume that $l=1$,
  which corresponds to the case
     when $\Supp(\varphi_{\ast}({\Bbb C}^{\oplus r}))$ is connected,
  and drop the subscript $l^{\prime}$.
In this case,
 the fact that $\varphi^{\sharp}$ is $C^k$-admissible
  implies that
   $\varphi^{\sharp}(f)$ has the unique eigenvalue $f(\lambda^1,\,\cdots\,,\lambda^n)$
  with multiplicity $r$. That is,
 $$
    \left(  \varphi^{\sharp}(f)
	  - f(\lambda^1,\,\cdots\,,\,\lambda^n)\cdot\Id_{r\times r}  \right)^r\;=\;0\,.
 $$
Denote the tuple $(\lambda^1,\,\cdots\,,\,\lambda^n)$ by {\boldmath $\lambda$}.
Then, it follows that
 $\varphi^{\sharp}(f-f(\mbox{\boldmath $\lambda$}))$  is nilpotent
  in $\End_{\Bbb C}({\Bbb C}^{\oplus r})$,
 with $(\varphi^{\sharp}(f-f(\mbox{\boldmath $\lambda$})))^r=0$.
For notational simplicity, by taking $y^i-\lambda^i$, $i=1,\,\ldots\,,\,n$,
  as the new coordinate functions (still denoted by $y^i$) on ${\Bbb R}^n$,
 one may assume that $\mbox{\boldmath $\lambda$}=\mbox{\boldmath $0$}=(0,\,\cdots\,,\,0)$.
Recall also that
  in the derivation of the eigenvalues of $\varphi^{\sharp}(f)$,
  we've rendered all $\varphi^{\sharp}(f)$ lower triangular, i.e.\
  $$
   \varphi^{\sharp}(f)\;
    =\; \left(
	        \begin{array}{ccccc}
			  f({\mathbf 0})\\
			  a_{21}(f)   & f({\mathbf 0})        &&& \hspace{-6ex} \mbox{\Large $0$}\\
			  \cdot                    & \cdot     & \cdot    \\
			  \cdot                    & \cdot     & \cdot   & \cdot   \\
			  a_{r1}(f)   & \cdot     & \cdot   & a_{r,r-1}(f)    & f({\mathbf 0})
			\end{array}
          \right)\,.
  $$

 \bigskip
  
 \noindent
 {\it $(b)$
 {From} ring-homomorphism
    $C^l({\Bbb R}^n)\stackrel{\varphi^{\sharp}}{\longrightarrow} T^-_r({\Bbb C})$
  to differential operators on $C^l({\Bbb R}^n)$}
 
 \medskip
 
 \noindent
 Being a ring-homomorphism,
  $$
   \varphi^{\sharp}(\alpha f+\beta g)\;=\; \alpha \varphi^{\sharp}(f)+ \beta \varphi^{\sharp}(g)\,,
   \hspace{2em}
   \varphi^{\sharp}(fg)\;=\; \varphi^{\sharp}(f)\,\varphi^{\sharp}(g)
  $$
   for $\alpha$, $\beta\in {\Bbb R}$ and $f$, $g\in C^k({\Bbb R}^n)$;
   that is,
  $$
    \begin{array}{lcl}
	  a_{ij}(\alpha f+\beta g)\;=\;  \alpha a_{ij}(f)+\beta a_{ij}(g)
	      &\hspace{1em}& \mbox{({\it linearity})} \\[1.8ex]
	  \begin{array}{l}	
	    \hspace{-1ex}a_{ij}(fg)\\[.2ex]
	     \hspace{1em}=\;  a_{ij}(f)g({\mathbf 0})\,+\, a_{i,j+1}(f)a_{j+1,j}(g)\\[.2ex]
		 \hspace{3em}   +\,\cdots\,  +\, a_{i,j+s}(f)a_{j+s,j}(g)\,+\,\cdots \\[.2ex]
		 \hspace{3em}   +\, a_{i,i-1}(f)a_{i-1,j}(g)+ f({\mathbf 0})a_{ij}(g)
	  \end{array}	
	      && \mbox{({\it inductive higher-order Leibniz rule})}
    \end{array}
  $$
  for $1\le j<i\le r$.
 In particular, for entries $a_{i,i-1}$, $2\le i\le r$,   in the first lower sub-diagonal,
  $$
    a_{i,i-1}(fg)\;=\; a_{i,i-1}(f)g({\mathbf 0})\,+\, f({\mathbf 0})a_{i,i-1}(g)\,,	
  $$
  which is the ordinary Leibniz rule and
  determines each linear functional $a_{i,i-1}:C^k({\Bbb R}^n)\rightarrow {\Bbb C}$
   as a derivation with complex coefficients on $C^k({\Bbb R}^n)$,
   evaluated at ${\mathbf 0}\in{\Bbb R}^n$.
 For entries $a_{i,i-s}$, $s+1\le i\le r$, in the $s$-th lower sub-diagonal,
 the above inductive product rule for $a_{i,i-s}(fg)$
  determines inductively
   the linear functional $a_{i,i-s}:C^k({\Bbb R}^n)\rightarrow {\Bbb C}$ on $C^k({\Bbb R}^n)$
   as a differential operator of order $i$ with complex coefficients on $C^k({\Bbb R}^n)$,
   evaluated at ${\mathbf 0}\in{\Bbb R}^n$,
 if it is not the zero-functional.
   
{To} see explicitly what these higher-order differential operators are,   	
  let $s\ge 2$  and
  consider the restriction of $\varphi^{\sharp}$ to the $C^l$-subring
   $C^l({\Bbb R}^n)$ in $C^k({\Bbb R}^n)$, now regarded as a $C^l$-ring, for all $l\ge k+s$.
  
 \bigskip
 
 \begin{lemma} {\bf [Taylor expansion with remainder term].}
  With $l$ and $s$ as above,
   let $f\in C^l({\Bbb R}^n)$,
   $\mbox{\boldmath $a$}=(a^1,\,\cdots\,,\,a^n)$ be a point on ${\Bbb R}^n$,  and
   $\mbox{\boldmath $y$}=(y^1,\,\cdots\,,\,y^n)$ be the tuple of coordinate-functions for ${\Bbb R}^n$.
  Then there exist $h_{i_1\,\cdots\,i_s}\in C^{l-s}({\Bbb R}^n)$,
      $1\le i_1,\,\cdots\,, i_s \le n$,
    with 	
	 $$
	    h_{i_1\,\cdots\, i_s}(\mbox{\boldmath $a$})
	            = \frac{\partial^s\!f}{\partial y^{i_1}\,\cdots\,\partial y^{i_s}}
				      (\mbox{\boldmath $a$})\,,
		 \hspace{2em}\mbox{for all $\;\;1\le i_1,\,\cdots\,,\, i_s\le n$}\,,				
	 $$
    such that
    \begin{eqnarray*}
	 \lefteqn{
      f(\mbox{\boldmath $y$})\; =\; \sum_{s^{\prime}=0}^{s-1}\;
	         \sum_{1\le i_1,\,\cdots\,, i_{s^{\prime}} \le n}\;
			  \frac{1}{s^{\prime}!}\,
	           \frac{\partial^{s^{\prime}}\!\!f}
			     {\partial y^{i_1}\cdots \partial y^{i_{s^{\prime}}}}(\mbox{\boldmath $a$})\,
				     (y^{i_1}-a^{i_1})\cdots (y^{i_{s^{\prime}}}- a^{i_{s^{\prime}}}) }
					                                                                                              \\[.6ex]
      && \hspace{2em}					
            +\; \frac{1}{s!}\,
	  		         \sum_{1\le i_1,\,\cdots\,, i_s\le n}
			           h_{i_1\,\cdots\,i_s}({\mbox{\boldmath $y$}})\,
					       (y^{i_1}-a^{i_1})\cdots (y^{i_s}-a^{i_s})\,.			
			\hspace{8em}	
   \end{eqnarray*}
  Explicitly, $h_{i_1\,\cdots\,i_s}$ can be chosen to be the function defined through $f$ by
    $$
	  h_{i_1\,\cdots\,i_s}(\mbox{\boldmath $y$})\;
	  =\;
	    \left(\prod_{s^{\prime}=1}^ss^{\prime}!\right)
	    \int_0^1\cdots\int_0^1
		   \frac{\partial^s\!f}{\partial y^{i_1}\cdots\partial y^{i_s}}
			    \left(\rule{0em}{1em}
				   t_1\cdots t_s\,
				   (\mbox{\boldmath $y$}\,-\,\mbox{\boldmath $a$})\,+\, \mbox{\boldmath $a$}
		        \right)\,
			    \left(\prod_{i=1}^{s-1}t_i^{s-i}\right)dt_1\,\cdots\,\,dt_s\,.
	$$	
 \end{lemma}

 \begin{proof}
 This follows  from the same proof as that for Lemma~3.2.0.1 in the beginning of Sec.~3.2.
                                                                         % Lemma [Taylor expansion with remainder term - smooth case].]
 
 \end{proof}
 
\bigskip
 
With $\mbox{\boldmath $a$}=\mbox{\boldmath $0$}$ in our disussion,
 one has thus
 \begin{eqnarray*}
  \lefteqn{a_{i,i-s}(f)}\\[1.2ex]
   && =\; \sum_{s^{\prime}=0}^{s-1}\;
	         \sum_{1\le i_1,\,\cdots\,, i_{s^{\prime}} \le n} \;
			  \frac{1}{s^{\prime}!}\,
	           \frac{\partial^{s^{\prime}}\!\!f}
			     {\partial y^{i_1}\cdots \partial y^{i_{s^{\prime}}}}({\mathbf 0})\,
				     a_{i,i-s}(y^{i_1}\cdots y^{i_{s^{\prime}}})                                           \\
       && \hspace{1.6em}
      	     +\; \frac{1}{s!}\,\sum_{1\le i_1,\,\cdots\,, i_s\le n}
		        	            a_{i,i-s}(  h_{i_1\,\cdots\,i_s}\, y^{i_1}\cdots y^{i_s})\,.
 \end{eqnarray*}
 Consider now terms in the last summation in the above expression  and
  their expansion by the inductive higher-order Leibniz rule on a product,
  \begin{eqnarray*}
   \lefteqn{
    a_{i,i-s}(h_{i_1\,\cdots\,i_s}\, y^{i_1}\,\cdots\,y^{i_s})}\\
	&&
	  =\;   a_{i,i-s}(h_{i_1\,\cdots\,i_s})
	          \cdot (y^{i_1}\,\cdots\,y^{i_s})
			  |_{\mbox{\scriptsize\boldmath $y$}=\mbox{\scriptsize\boldmath $0$}} \;
			 +\;  \sum_{s^{\prime}=1}^{s-1}
			           a_{i,i-s+s^{\prime}}(h_{i_1\,\cdots\,i_s})\,
					   a_{i-s+s^{\prime},i-s}(y^{i_1}\,\cdots\,y^{i_s})  \\
    &&	\hspace{2em}				
			+\; 	h_{i_1\,\cdots\,i_s}(\mbox{\boldmath $0$})
			          \cdot a_{i,i-s}(y^{i_1}\,\cdots\,y^{i_s})   \\
     &&  =\;  \frac{\partial^s\!f}{\partial y^{i_1}\,\cdots\,\partial y^{i_s}}
				      (\mbox{\boldmath $0$})\cdot   a_{i,i-s}(y^{i_1}\,\cdots\,y^{i_s})  \;
					 +\;   \sum_{s^{\prime}=1}^{s-1}
			                 a_{i,i-s+s^{\prime}}(h_{i_1\,\cdots\,i_s})\,
					         a_{i-s+s^{\prime},i-s}(y^{i_1}\,\cdots\,y^{i_s})\,.
  \end{eqnarray*}
 By induction,
  $a_{i,i-s+s^{\prime}}$, $1\le s^{\prime}\le s-1$, are differential operators of order $s-s^{\prime}$,
   evaluated at ${\mathbf 0}$.
 It follows that
  $$
    a_{i-s+s^{\prime},i-s}(y^{i_1}\,\cdots\,y^{i_s})\;=\;0
  $$
  and thus
  $$
     a_{i,i-s}(h_{i_1\,\cdots\,i_s}\, y^{i_1}\,\cdots\,y^{i_s})\;
	  =\;    \frac{\partial^s\!f}{\partial y^{i_1}\,\cdots\,\partial y^{i_s}}
				      (\mbox{\boldmath $0$})\cdot   a_{i,i-s}(y^{i_1}\,\cdots\,y^{i_s})\,.
  $$
      
 Consequently,
  $$
   a_{i,i-s}(f)\;
   =\; \sum_{s^{\prime}=0}^s\;
	         \sum_{1\le i_1,\,\cdots\,, i_{s^{\prime}} \le n} \;
			  \frac{1}{s^{\prime}!}\,
	           \frac{\partial^{s^{\prime}}\!\!f}
			     {\partial y^{i_1}\cdots \partial y^{i_{s^{\prime}}}}({\mathbf 0})\,
				     a_{i,i-s}(y^{i_1}\cdots y^{i_{s^{\prime}}})
  $$
  and
  \begin{eqnarray*}
    a_{i,i-s}
     & =  &   \sum_{s^{\prime}=0}^s\;
	                   \sum_{1\le i_1,\,\cdots\,, i_{s^{\prime}} \le n} \;
			            \frac{1}{s^{\prime}!}\,
			  	        a_{i,i-s}(y^{i_1}\cdots y^{i_{s^{\prime}}})\,
	                    \frac{\partial^{s^{\prime}}}
		       	         {\partial y^{i_1}\cdots \partial y^{i_{s^{\prime}}}}\,,
                      \hspace{1em}\mbox{evaluated at ${\mathbf 0}$}                            \\
	 & =  &   \sum_{s^{\prime}=1}^s\;
	                   \sum_{1\le i_1,\,\cdots\,, i_{s^{\prime}} \le n} \;
			            \frac{1}{s^{\prime}!}\,
			  	        a_{i,i-s}(y^{i_1}\cdots y^{i_{s^{\prime}}})\,
	                    \frac{\partial^{s^{\prime}}}
		       	         {\partial y^{i_1}\cdots \partial y^{i_{s^{\prime}}}}\,,
                      \hspace{1em}\mbox{evaluated at ${\mathbf 0}$}\,.                      				  
  \end{eqnarray*}	
 Note that such realization of $a_{i,i-s}$  as a differential operator, evaluated at ${\mathbf 0}$,
   is compatible with the inclusions
   $C^{l^{\prime}}({\Bbb R}^n)\subset C^l({\Bbb R}^n)$ for $l^{\prime}> l\ge k+s$.
 Furthermore,
  though we restrict to $C^l({\Bbb R}^n)$, $l\ge k+s$, to solve for
    the system of equations on the functionals  $a_{ij}$ on $C^k({\Bbb R}^n)$
	  for $2\le i\le r$ and $i-s\le j\le i-1$,
  the uniqueness of the solution for $a_{ij}$ associated to $\varphi^{\sharp}$  and
  the nature of differential operators as a limit of a composition of a rational combination of difference operators  and
  the fact that these difference operators are already defined on
     $C^0({\Bbb R}^n)$, which contains $C^k({\Bbb R}^n)$,
   imply that the differential-operator form of $a_{ij}$, as a functional, prolongs to
    on the whole $C^k({\Bbb R}^n)$ if $a_{ij}$ is not the zero-functional.
 This proves the following lemma:
 
 \bigskip
 
 \begin{lemma}
 {\bf [vanishing].}
  Let $h\in C^k({\Bbb R}^n)$ be a function that satisfies the vanishing conditions
    $$
       \frac{\partial^s h}
			     {\partial y^{i_1}\cdots \partial y^{i_s}}({\mathbf 0})\;=\; 0\,,
    $$	
    for all $\, 0\le s \le \min\{r-1,k\}$,    $1\le i_1\,, \cdots\,,\, i_s \le n$.
  Then 	
   $$
     \varphi^{\sharp}(h)\;=\; 0\,.
   $$
 \end{lemma}
 
 \bigskip
 
 As a consequence,
 
 \bigskip
 
 \begin{lemma}
 {\bf [dependence of $\varphi^{\sharp}$ on jet at {\boldmath $0$} of order $\min\{r-1,k\}$].}
  For $f\in C^k({\Bbb R}^n)$,
    \begin{eqnarray*}
      \varphi^{\sharp}(f)
        &  = &  \sum_{s=0}^{\min\{r-1, k\}}\;
	                     \sum_{1\le i_1,\,\cdots\,, i_s \le n} \;
			              \frac{1}{s!}\,
			  	           \varphi^{\sharp}(y^{i_1}\cdots y^{i_s} )\,
	                       \frac{\partial^s\!f}
		       	            {\partial y_{i_1}\cdots \partial y_{i_s}}
						              (\mbox{\boldmath $0$})  \\
	    & = &  \hspace{1.4em}
		                 \sum_{s=0}^k\;
	                     \sum_{1\le i_1,\,\cdots\,, i_s \le n} \;
			              \frac{1}{s!}\,
			  	           \varphi^{\sharp}(y^{i_1}\cdots y^{i_s} )\,
	                       \frac{\partial^s\!f}
		       	            {\partial y_{i_1}\cdots \partial y_{i_s}}
						              (\mbox{\boldmath $0$})  \,.
    \end{eqnarray*} 								
 \end{lemma}	

 \begin{proof}
  For any $\min\{r-1,k\}\le d \le k$,
   one can decompose  $f\in C^k({\Bbb R}^n)$ into a summation
  $$
    f\;=\;  h\,+\,
	            \sum_{s=0}^d\;
	                   \sum_{1\le i_1,\,\cdots\,, i_s \le n} \;
			            \frac{1}{s!}\,			  	
	                    \frac{\partial^s\!f}
		       	         {\partial y_{i_1}\cdots \partial y_{i_s}}
						          (\mbox{\boldmath $0$})\,
                    y^{i_1}\cdots y^{i_s}\,.      								  	
  $$
  Then $h\in C^k({\Bbb R}^n)$ and satisfies the vanishing conditions in Lemma~3.4.4.
                                                                                                                                    % Lemma [vanishing]  
  The lemma follows.																			
 
 \end{proof}

 \bigskip
 
 \noindent
 {\it $(c)$ The complete expression of $\varphi^{\sharp}$, which implies that $\varphi$ is of algebraic type}

 \medskip
 
 \noindent
 Resume now to the general case, i.e.\ the case when $l\ge 2$.
 As in Sec.~3.2.2,
 let $\mbox{\boldmath $q$} = (q_1,\,\cdots\,, q^n)$ be a point in ${\Bbb R}^n$
   with tuple of coordinate functions $\mbox{\boldmath $y$} = (y^1,\,\cdots\,y^n)$,
 and denote
 $$
   T^{(\mbox{\scriptsize\boldmath $y$}, \mbox{\scriptsize\boldmath $q$}, d)}(f)\;
    :=\;  \mbox{the Taylor expansion of $f\in C^k({\Bbb R}^n)$ at {\boldmath $q$}
                             up to order/degree $d\le k$	}\,.
 $$
Let
 $$
   Z\;  =\;
    \{\, \mbox{\boldmath $\lambda$}_1\,,\, \cdots\,,\, \mbox{\boldmath $\lambda$}_l\,\}\;
	\subset\; {\Bbb R}^n\,.
 $$
Then, it follows immediately from the discussions in Part (a) and Part (b) that:

 \bigskip

 \begin{proposition}    {\bf [dependence of $\varphi^{\sharp}$ on $k$-jet at $Z$].}
 Continuing the notations above and at the beginning of the current subsubsection.
 Let $f\in C^k({\Bbb R}^n)$.
 Then
   \begin{eqnarray*}
	 \varphi^{\sharp}(f)
	  & =
	  &  \iota \circ
	            (\varphi^{\sharp}_1(
	                 T^{(\mbox{\scriptsize\boldmath $y$},
					              \mbox{\scriptsize\boldmath $\lambda$}_1, \min\{k, r_1-1\})}(f)),\,
	         \cdots\,,\,
			 \varphi^{\sharp}_l(
			         T^{(\mbox{\scriptsize\boldmath $y$},
					              \mbox{\scriptsize\boldmath $\lambda$}_l,  \min\{k, r_l-1\})}(f)))\\
	  & =
	  &  \iota \circ
	            (\varphi^{\sharp}_1(
	                 T^{(\mbox{\scriptsize\boldmath $y$},
					              \mbox{\scriptsize\boldmath $\lambda$}_1, k)}(f)),\,
	         \cdots\,,\,
			 \varphi^{\sharp}_l(
			         T^{(\mbox{\scriptsize\boldmath $y$},
					              \mbox{\scriptsize\boldmath $\lambda$}_l,  k)}(f)))\;.
   \end{eqnarray*}								
  In particular,
   $\varphi^{\sharp}:
      C^k({\Bbb R}^n)\rightarrow \End_{\Bbb C}({\Bbb C}^{\oplus r})$	
	is uniquely determined by its values $m_{\varphi}^i := \varphi^{\sharp}(y^i)$
	at any specified tuple of coordinate functions $\mbox{\boldmath $y$}=(y^1,\,\cdots\,,\,y^n)$
	of ${\Bbb R}^n$.
 \end{proposition}

 \bigskip
 This shows that $\varphi$ is of algebraic type.
 And we conclude the proof of Proposition~3.4.2.
                                                 % Proposition [fundamental]
                  
\noindent\hspace{40.7em}$\square$ 				

\bigskip

\bigskip

\begin{flushleft}			
{\bf The meaning of a $C^k$-map
          $\varphi: (p^{A\!z},{\Bbb C}^{\oplus r}) \rightarrow {\Bbb R}^n$
		  being of algebraic type}
\end{flushleft}	
Let
 {\boldmath$y$}$=(y^1,\,\cdots\,,\, y^n)$ be a tuple of coordinate functions on ${\Bbb R}^n$,
  $p_i:= (a_1^1,\,\cdots\,,\, a_i^n)$, $i=1,\,\ldots\,,\, l$,  be distinct points on ${\Bbb R}^n$,
  $Z:=\{p_1,\,\cdots\,,\, p_l\}\subset {\Bbb R}^n$ as an ordered set of points,  and
 $d\le k$ be a non-negative integer.
Then
 the product ring
  $$
    R^{(\mbox{\scriptsize\boldmath $y$},Z,d)}	   \;
     	   :=\;   \frac{{\Bbb R}[y^1-a_1^1,\,\cdots\,,\,  y^n-a_1^n]}
		                     {  (y^1-a_1^1,\,\cdots \,,\, y^n-a_1^n )^d   }\,
		            \times\,\cdots\, \times
		            \frac{{\Bbb R}[y^1-a_l^1,\,\cdots\,,\,  y^n-a_l^n]}
		                     {  (y^1-a_l^1,\,\cdots \,,\, y^n-a_l^n )^d   }  \\[2ex]
  $$
   is naturally a $C^k$-ring  and
 the ring-epimorphism
 $$
    \begin{array}{ccccc}	
	  T^{(\mbox{\scriptsize\boldmath $y$},Z,d)}  & :  &  C^k({\Bbb R}^n)& \longrightarrow
  	   &  R^{(\mbox{\scriptsize\boldmath $y$},Z,d)}	   \\[1.2ex]
      && f   &  \longmapsto
	    & \begin{array}{l}
		      \mbox{Taylor expansion of $f$ at $Z$ to order $d$}\\[.6ex]
			  \left(  T^{(\mbox{\scriptsize\boldmath $y$},\, p_1,\, d)}(f)\,,\,
                           \cdots\,,\, 	  T^{(\mbox{\scriptsize\boldmath $y$},\, p_l,\, d)}(f) \right)
		    \end{array}	
	\end{array}
 $$	
 is automatically a $C^k$-ring-homomorphism over ${\Bbb R}$.
A $C^k$-admissible ring-homomorphism
    $\varphi^{\sharp}:
	    C^k({\Bbb R}^n)\rightarrow \End_{\Bbb C}({\Bbb C}^{\oplus r})$
     over ${\Bbb R}\hookrightarrow{\Bbb C}$
 is of algebraic type 	
  if and only if
   it factors through $T^{(\mbox{\scriptsize\boldmath $y$},Z,d)}$
   $$
     \xymatrix{
	  & C^k({\Bbb R}^n)\ar[rrr]^-{\varphi^{\sharp}}
	      \ar[d]_-{T^{(\mbox{\tiny\boldmath $y$},Z,d)}\;\;  }
 	      &&& \End_{\Bbb C}({\Bbb C}^{\oplus r}) \\
	  & R^{(\mbox{\scriptsize\boldmath $y$},Z,d)}
	      \ar[rrru]_-{\underline{\varphi}^{\sharp}}	 	&&&& .
	  }
   $$
   
Now, $R^{(\mbox{\scriptsize\boldmath $y$},Z,d)}$
   is the function ring of a collection of fat ${\Bbb R}$-points, i.e.\ a $0$-dimensional $C^k$-scheme $\hat{Z}$,
   and the quotient $C^k$-ring-homomorphism
   $T^{(\mbox{\scriptsize\boldmath $y$},Z,d)}  :
      C^k({\Bbb R}^n)\rightarrow    R^{(\mbox{\scriptsize\boldmath $y$},Z,d)}$
  defines an embedding $\iota: \hat{Z}\rightarrow {\Bbb R}^n$
    as a $0$-dimensional subscheme of ${\Bbb R}^n$
  with $\iota(\hat{Z})_{\, \redscriptsize}=Z$.
Thus,
 the $C^k$-map
   $\varphi: (p, \End_{\Bbb C}({\Bbb C}^{\oplus r}),{\Bbb C}^{\oplus r})
                        \rightarrow {\Bbb R}^n$
    defined by $\varphi^{\sharp}$
 is of algebraic type if and only if
  $\varphi$ factors through $\iota$
  $$
   \xymatrix{
     & (p^{A\!z},{\Bbb C}^{\oplus r})\ar[rrr]^{\varphi} \ar[rrrd]_-{\underline{\varphi}}
    	 &&& {\Bbb R}^n \\
	 &  &&&	 \rule{0ex}{1.2em}\hat{Z}  \ar@{^{(}->}[u]_-{\iota}  &.
   }
  $$
Having the function ring a finite-dimensional ${\Bbb C}$- or ${\Bbb R}$-algebra,
 both $ (p^{A\!z},{\Bbb C}^{\oplus r})$ and $\hat{Z}$ are algebraic in nature;
 thus, so does the map
  $\underline{\varphi}:(p^{A\!z},{\Bbb C}^{\oplus r})\rightarrow \hat{Z}$.
This explains the name for $\varphi$ being `of algebraic type'.

Proposition~3.4.2
                  % Proposition [fundamental: Every differentiable map from Azumaya/matrix point is of algebraic type]
 now says that any $C^k$-map
     $\varphi:(p, \End_{\Bbb C}({\Bbb C}^{\oplus r}),{\Bbb C}^{\oplus r})
	                    \rightarrow {\Bbb R}^n$
  is algebraic in nature.
In particular,
 except the main difference
   that all eigenvalues of $\varphi^{\sharp}(f)\in M_{r\times r}({\Bbb C})$,
      $f\in C^k({\Bbb R}^n)$
   and that the fuzziness of $\varphi(p^{A\!z})$ is controled not only by $r$ but also by $k$,
	  the order of differentiability,
 what one learns about D0-brane in [L-Y1: Sec.~4] (D(1))
  crosses over to the current case with ${\Bbb R}^n$ target, with only mild revision if necessary.
In particular,
 the $0$-dimension push-forward sheaf
    $\varphi_{\ast}({\Bbb C}^{\oplus r})$  on ${\Bbb R}^n$
 behaves as in [L-Y1: Sec.~4] (D(1)) under deformations of $\varphi$,
 resembling the Higgsing/un-Higgsing behavior of D0-branes on ${\Bbb R}^n$.
Cf.~{\sc Figure}3-4-1.
%
%  \marginpar{\raggedright\tiny $\bullet$
%        {\sc Figure}:  \\  higg-un.pdf}

 \begin{figure} [htbp]
  \bigskip
  \centering
  \includegraphics[width=0.80\textwidth]{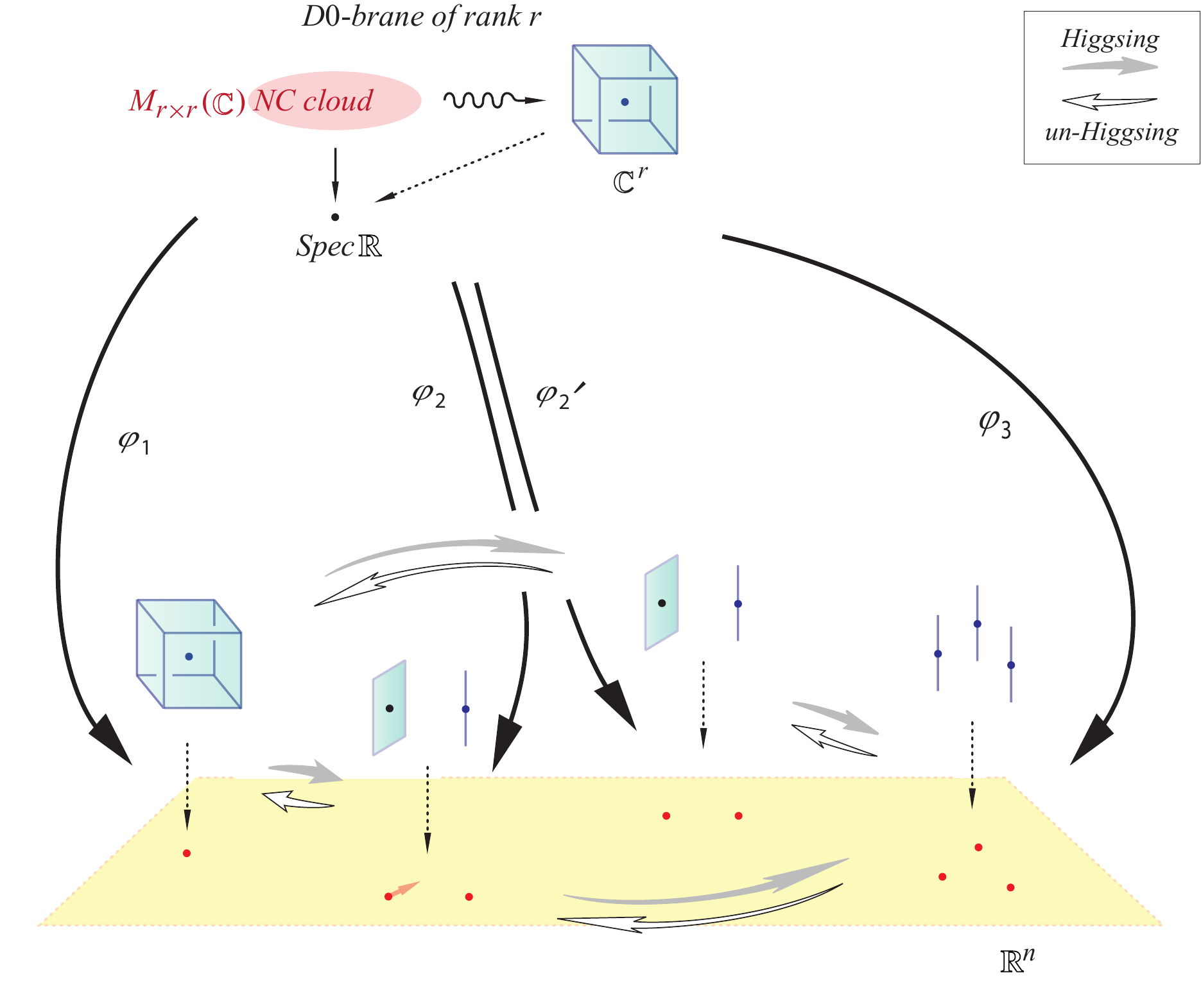}
 
  \bigskip
  \bigskip
  \centerline{\parbox{13cm}{\small\baselineskip 12pt
   {\sc Figure}~3-4-1. (Cf.\ [L-Y5: {\sc Figure}~2-1-1] (D(6)).)
   Any $C^k$-map $(p^{A\!z},{\Bbb C}^{\oplus r})\rightarrow {\Bbb R}^n$
      from an Azumaya/matrix point with a fundamental module $(p^{A\!z},{\Bbb C}^{\oplus r})$
	   to the real manifold ${\Bbb R}^n$ is of algebraic type.
   Similar to the algebraic case in [L-Y1] (D(1)),
	 despite that {\it Space}$\,M_{r\times r}({\Bbb C})$ may look only one-point-like,
     under a $C^k$-map the Azumaya/matrix  ``noncommutative cloud" $M_{r\times r}({\Bbb C})$
     over {\it Space}$\,M_{r\times r}({\Bbb C})$ can ``split and condense"
     to various image $0$-dimensional $C^k$-schemes with a rich geometry.
   The latter image $C^k$-schemes can even have more than one component.
   The Higgsing/un-Higgsing behavior of the Chan-Paton module of
     D$0$-branes on ${\Bbb R}^n$ occurs
     due to the fact that
       when a $C^k$-map
	     $\varphi: (p^{A\!z},{\Bbb C}^{\oplus r}) \rightarrow {\Bbb R}^n$ deforms,
       the corresponding push-forward $\varphi_{\ast}{\Bbb C}^{\oplus r}$
       of the fundamental module ${\Bbb C}^{\oplus r}$
  	   on $(p^{A\!z},{\Bbb C}^{\oplus r})$ can also change/deform.
   These features generalize to $C^k$-maps
     from Azumaya manifolds with a fundamental module to a real $C^k$-manifold $Y$.
   Despite its simplicity, the case study in Sec.~3 already hints at a richness of Azumaya-type noncommutative
     differential geometry (interpreted via or joined with $C^k$ algebraic geometry).
    In the figure, a module over a $C^k$-scheme is indicated by
      a dotted arrow $\xymatrix{ \ar @{.>}[r] &}$.	
       }}
  \bigskip
 \end{figure}

\bigskip

\subsection{Lessons from the case study}

Before moving on, let us recapitulate some conceptual lessons behind the messy details
 in the case study of D$0$-branes on ${\Bbb R}^n$ as maps
 from Azumaya/matrix points
 $(p^{A\!z},{\Bbb C}^{\oplus r})$, $r\in {\Bbb Z}_{\ge 1}$, to ${\Bbb R}^n$.

%--------------------------------------------------------------------------------------------
%
% \bigskip
% \begin{flushleft}
% {\bf Example:  Deformation of disjoint simple D0-branes into a T D0-brane on ${\Bbb R}^1$}
% \end{flushleft}
%
%
% \begin{example}
% {\bf [filtered D0-branes on ${\Bbb R}^1$].} {\rm
%   ?????????????? (Cf.\ T D$0$-brane).
%     %
%     \marginpar{\raggedright\tiny $\bullet$
%      Explicit examples of
% 	  $\varphi^{\sharp}:
%          C^{\infty}({\Bbb R}^1)\rightarrow End_{\Bbb C}({\Bbb C}^{\oplus r})$
% 	  over ${\Bbb R}\hookrightarrow{\Bbb C}$ to be given here.\\
% 	 To be completed.
% 	 }
% 	
%  \bigskip
%
%  \noindent $\bullet$
%  Examples of $\varphi^{\sharp}$ via the composition
%   $$
%     C^{\infty}({\Bbb R}^1)\;\longrightarrow\;
% 	{\Bbb R}[y]/(f)\;\longrightarrow\; \End_{\Bbb C}({\Bbb C}^{\oplus r})\,,
%   $$
%   where the first ${\Bbb R}$-algebra-homomorphism is constructed by taking truncations of Taylor series
%     at a finite subset of ${\Bbb R}^1$.
%
%   \bigskip
%
%   \noindent $\bullet$
%   ???????????????????????.
%
% \bigskip
%
% \noindent\hspace{15.7cm}$\square$
% }\end{example}
%
% \bigskip
% ===-----------------------------------------------------=================-------------------------------------------------

\bigskip

\begin{lesson} {\bf [map between spaces vs.\ homomorphism between function rings].}
{\rm
   Naively, from the notion of real schemes in (commutative) algebraic geometry as in Example~3.1.1,
                                                                     % Example [map from $p^{Az}$ to ${\Bbb A}^1_{\Bbb R}$]
   one may expect that
      Grothendieck's theory of schemes should need to be injected to the notion of (real) $C^k$-manifolds
      to create a theory  of $C^k$-schemes
    so that these ``missing ${\Bbb C}$-points" from the aspect of homomorphisms between function rings involved
	can be recovered and,
    at least locally, the notion of maps between spaces is contravariantly equivalent to the notiion of homomorphisms
    between rings.
  However, from what one learns in this section,
    this is not the case.
  $C^k({\Bbb R}^n)$	 is much larger than ${\Bbb R}[y^1,\,\cdots\,,\, y^n]$.
  This renders irreducible polynomials in any ${\Bbb R}[y^i]$ that potentially correspond to ${\Bbb C}$-points on
   $y^i$-coordinate axis invertible.
  The fact that all the eigenvalues of  $\varphi^{\sharp}(f)$ are real  for
    any admissible ring-homomorphism
	 $\varphi^{\sharp}:
	    C^k({\Bbb R}^n)\rightarrow \End_{\Bbb C}({\Bbb C}^{\oplus r})$
	   over ${\Bbb R}\hookrightarrow{\Bbb C}$
	  and $f\in C^k({\Bbb R}^1)$
  says that
   \begin{itemize}
    \item[$\cdot$] \parbox[t]{38em}{\it
 	The usual ${\Bbb R}^n$ in the context of differential topology and geometry  is enough
           to describe D0-branes on ${\Bbb R}^n$ correctly in line of [L-Y1] (D(1)).}
  \end{itemize}	
}\end{lesson}	

\bigskip

\begin{lesson} {\bf [support of sheaf].} {\rm
 The notion of {\it scheme-theoretical-like support} $\Supp({\cal F})$
     of a $C^k(N)$-module ${\cal F}$ on a $C^k$-manifold $N$
  remains definable as the complement of locus $\{p\in N: \,\mbox{stalk $F_{(p)} =0$} \}$
  with the {\it structure sheaf
       ${\cal O}_{Supp({\cal F})}
                      := C^k(N)/\Ker(C^k(N)\rightarrow \Endsheaf_{C^k(N)}({\cal F}))$}.
 While $C^k(N)$ contains no nilpotent elements,
   its quotient ${\cal O}_{Supp({\cal F})}$ can have nilpotent elements.
 The latter encodes more details of ${\cal F}$ as a $C^k(N)$-module.
 It is in this way maps from Azumaya/matrix manifolds with a fundamental module to a real manifold $Y$
   encode and store the information of D-branes on $Y$.
 In other owrds,
  \begin{itemize}
   \item[$\cdot$] \parbox[t]{38em}{\it
    Our interpretation of the `information-preserving geometry' on D-branes as explored in the work
    	  [G-S] of T\'{o}mas G\'{o}mez and Eric Sharpe
     remains valid for D-branes in the realm of differential	topology and geometry.}
  \end{itemize}
 See [L-Y5: Sec.~2.4: Item (4)] (D(6)) for more explanations.
}\end{lesson}

\bigskip

Following the philosophy
 that smearing D$0$-branes along  a $p$-cycle $Z$ in $C^k$-manifold renders $Z$ a D$p$-brane,
 once it is justified enough that D$0$-branes on the $C^k$-manifold ${\Bbb R}^n$
  are well-described
  as $C^k$-maps from Azumaya/matrix points with a fundamental module to ${\Bbb R}^n$
  defined contravariantly by $C^k$-admissible ring-homomorphisms of the function-rings involved,
one expects that
 the setting in [L-Y1] (D(1)) and [L-L-S-Y] (D(2)) of the project,
    which lay down the foundation to their sequels [L-Y2] (D(3)) -- [L-Y10] (D(10.2)),
   remains to work to describe D-branes in string theory in the realm of differential topology and geometry.
The rest of the current note is to justify this natural expectation.
   
\bigskip

\begin{remark} {$[$generalization to differentiable maps from other noncommutative points$\,]$.} {\rm
Since any finite-dimensional (associative, unital) ${\Bbb C}$-algebra
embeds in an $M_{\times r}r({\Bbb C})$
 for some $r$, Proposition~3.4.2 implies that
                     % Proposition [fundamental: Every differentiable map from Azumaya/matrix point is of algebraic type]
 %
 \begin{itemize}
  \item[$\cdot$] \parbox[t]{38em}{\it
  {\bf [every $C^k$-map from finite-type noncommutative points is of algebraic type].}
   Let
     $S$ be a finite-dimensional (associative, unital) ${\Bbb C}$-algebra   and
	 $k\in {\Bbb Z}_{\ge 0}\cup\{\infty\}$.
   Then any $C^k$-map $\varphi:(p,S)\rightarrow {\Bbb R}^n$
    (with ${\Bbb R}^n$ regarded as a $C^k$-manifold)
	defined by a $C^k$-admissible ring-homomorphism $\varphi^{\sharp}:C^k({\Bbb R}^n)\rightarrow S$
	  over ${\Bbb R}\hookrightarrow {\Bbb C}$
	is of algebraic type.}
 \end{itemize}
}\end{remark}

\bigskip

\begin{remark} {\it $[$differentiability of map vs.\ size of nilpotent cloud of image$\,]$.} {\rm
 For a first-timer, it may look a little bit odd that
   \begin{itemize}
    \item[$\cdot$] \parbox[t]{38em}{\it
     In general it is not true that
       if $\varphi:(p^{A\!z},{\Bbb C}^{\oplus r})\rightarrow {\Bbb R}^n$ is $C^k$,
       then $\varphi$ is $C^{k^{\prime}}$ for all $\,0\le k^{\prime}<k$.}
   \end{itemize}
 The order of nonreducedness of the $0$-dimensional subscheme $\varphi(p^{A\!z})$ of ${\Bbb R}^n$,
   in the sense of $C^k$-algebraic geometry, is dominated not only by $r$ but also by $k$.
 While we have no intention to elaborate it here, this phenomenon gives us an anticipation that
  \begin{itemize}
   \item[$\cdot$] \parbox[t]{38em}{{\bf
     [D-brane probe to singularity and differentiable structure]}\hspace{1em}
    In the realm of differential topology and geometry,
    D-branes can probe not only the singularities of  the target manifold-with-singularities $Y$
    --- as is well-studied since the work [Do-M] (cf.\ [L-Y8] (D(9.1))) ---	
    but also the differentiability and differentiable structures of $Y$.}
  \end{itemize}
}\end{remark}

%-------------------------------------------------------------------------------------------------------------------------------
% \bigskip
%
% \begin{remark} {\it $[$factorization through algebraic setting$\,]$.} {\rm
%   ?????????????????????.
%   %
%   \marginpar{\raggedright\tiny $\bullet$
%            Draw a figure for this. \\ To be completed.}
%   %
%   This is the basis and the main reason for the fact we'll see
%    that the more Grothendieck's-algebraic-geometry-oriented settings and discussions
%      in [L-Y1] (D(1)) and [L-L-S-Y] (D(2)) extend naturally
% 	 to the realm of differential or symplectic geometry without changing anything fundamentally
% 	 beside fitting into the language of $C^k$-algebraic geometry.
% }\end{remark}
%-----------------------============-----------------------------------------------------------------------------------------

\bigskip
  
\section{Differential calculus and geometry of Azumaya manifolds with a fundamental module}

Let
 \begin{itemize}
   \item[$\cdot$]
    $X$ be a (real $m$-dimensional) $C^k$-manifold, $k\in {\Bbb Z}_{\ge 0}\cup\{\infty\}$,
	
   \item[$\cdot$]	
    $E$ be a $C^k$ complex vector bundle of (complex) rank $r$ over $X$,
	
   \item[$\cdot$]	
    $E^{\vee}$ be the complex dual vector bundle of $E$,
	
   \item[$\cdot$]	
    $\End_X(E):= E\otimes_{C^k(X)^{\Bbb C}}E^{\vee}$
     	be the endomorphism bundle of $E$,

   \item[$\cdot$]		
    $C^k(E)$ be the $C^k(X)^{\Bbb C}$-module of $C^k$-sections of $E$, and

   \item[$\cdot$]
    $C^k(\End_X(E))$   be the $C^k(X)^{\Bbb C}$-algebra
	  of $C^k$-sections of $\End_X(E)$.
 \end{itemize}
Note that:
 \begin{itemize}
  \item[$\cdot$]
   $\End_X(E)$ acts on $E$ tautologically via the fundamental representation.

  \item[$\cdot$]
   By definition, $C^k(\End_X(E))$  is an Azumaya algebra over its center $C^k(X)^{\Bbb C}$.
 \end{itemize}
Equivalently,  in terms of sheaves on $X$, let
 \begin{itemize}
  \item[$\cdot$]
   ${\cal O}_X$ be the sheaf of $C^k$-functions on $X$, i.e.\ the structure sheaf of $X$,\\ and
   ${\cal O}_X^{\,\Bbb C}$ be its complexifications,
      
  \item[$\cdot$]
   ${\cal E}$ be a locally free ${\cal O}_X^{\,\Bbb C}$-module of rank $r$;
    equivalently, the sheaf of local $C^k$-section of $E$,
	
  \item[$\cdot$]
   ${\cal E}^{\vee}
      := \Homsheaf_{{\cal O}_X^{\,\Bbb C}}({\cal E},{\cal O}_X^{\,\Bbb C})$
	  be the complex dual sheaf of ${\cal E}$;
	equivalently,  the sheaf of local $C^k$-sections of $E^{\vee}$,  and
	
  \item[$\cdot$]	
   $\Endsheaf_{{\cal O}_X^{\,\Bbb C}}({\cal E})$
    be the sheaf of ${\cal O}_X^{\,\Bbb C}$-endomorphisms of ${\cal E}$; equivalently,
   the sheaf of local $C^k$-sections of $\End_X(E)$.	 	
 \end{itemize}
By construction,
 \begin{itemize}
  \item[$\cdot$]
  $\Endsheaf_{{\cal O}_X^{\,\Bbb C}}({\cal E})$ acts on ${\cal E}$ tautologically
    via the fundamental representation; and

  \item[$\cdot$]
  $\Endsheaf_{{\cal O}_X^{\,\Bbb C}}({\cal E}))$  is  a sheaf of  Azumaya algebras
   over its center ${\cal O}_X^{\,\Bbb C}$.	
 \end{itemize}
This motivates the following definition:

\bigskip

\noindent
{\bf Definition 4.0.1. [Azumaya/matrix $C^k$-manifold with a fundamental module].}  {\rm
 The $C^k$-manifold $X$ with the enhanced structure sheaf
   ${\cal O}_X^{A\!z}:= \Endsheaf_{{\cal O}_X^{\,\Bbb C}}({\cal E})$
    of noncommutative function-rings from the endomorphism algebras of ${\cal E}$
   is called a {\it  (complex-)Azumaya (real $m$-dimensional) $C^k$-manifold}\footnote{That is,
                    a real $m$-dimensional $C^k$-manifold with a complex Azymaya structure for its function ring.
                   In contrast, if one begins with a real $C^k$ vector bundle $E$ over $X$,
				    then the resulting $X^{\!A\!z}$ shall called
					a {\it real-Azumaya(real $m$-dimensional) $C^k$-manifold} over $X$.}
    over $X$; 	
  in notation, $X^{\!A\!z}\!:=(X, \Endsheaf_{{\cal O}_X^{\,\Bbb C}}({\cal E}))$.
 The triple
   $$
     (X,{\cal O}_X^{A\!z}:=\Endsheaf_{{\cal O}_X^{\,\Bbb C}}({\cal E}),{\cal E})
   $$
  is called an {\it Azumaya $C^k$-manifold with a fundamental module}.
 With respect to a local trivialization of ${\cal E}$,
  ${\cal O}_X^{\!A\!z}$  is a sheaf of $r\times r$-matrix algebras with entries
    complexified local $C^k$-functions on $X$.
 For that reason and to fit better with the terminology in quantum field and string theory,
  we shall call  	
   $(X,{\cal O}_X^{A\!z}:=\Endsheaf_{{\cal O}_X^{\,\Bbb C}}({\cal E}),{\cal E})$
   also as a {\it matrix $C^k$-manifold with a fundamental module},
   particularly in a context that is more directly related to quantum field and string theory.
}%end-definition

\bigskip

\noindent
{\it Remark 4.0.2. $[$$X^{\!A\!z}$ as the space associated to a gluing system of noncommutative rings$\,]$.}
{\rm
 Following the study in [L-Y1] (D(1)), let
  \begin{itemize}
   \item[$\cdot$]
    $\{U_{\alpha}\}_{\alpha\in A}$ be an $C^k$-atlas of $X$ \hspace{2em}
	  (here, the gluings, i.e.\ transition functions,
	     $$
		   U_{\alpha_1}\; \supset\;
	       U_{\alpha_1}\cap U_{\alpha_2} =:  U_{\alpha_1\alpha_2}:\;
		   \stackrel{\sim}{\longrightarrow}\;
		   U_{\alpha_2\alpha_1}:= U_{\alpha_2}\cap U_{\alpha_1}	\;
 		      \subset\; U_{\alpha_2}
		 $$
		that satisfy the cocycle conditions are implicit in the notation);
	
  \item[]	
	passing to a refinement of the atlas, we may assume without loss of generality that\\
	the restriction of $E$ to each $U_{\alpha}$, in notation $E_{U_{\alpha}}$, is trivial and trivialized.
  \end{itemize}
  The associated gluing data of bundles
  $$
     (\{E_{U_{\alpha}}\}_{\alpha\in A}\,,\,
	      \{E_{U_{\alpha_1\alpha_2}}
		   \stackrel{\sim}{\rightarrow}
		      E_{U_{\alpha_2\alpha_1}}\}_{\alpha_1,\alpha_2\in A})\,.
  $$
  gives rise to a {\it gluing system of noncommutative rings --  in our case Azumaya/matrix algebras}
  $$
   {\cal R}\;
     :=\;  \left(\{  \End_{U_{\alpha}}(E_{U_{\alpha}})\}_{\alpha\in A}\,,\,
                 \{ \End_{U_{\alpha_1\alpha_2}}(E_{U_{\alpha_1\alpha_2}})
	                     \stackrel{\sim}{\rightarrow}
			                 \End_{U_{\alpha_2\alpha_1}}(E_{U_{\alpha_2\alpha_1}})
	                 \}_{\alpha_1,\alpha_2\in A}			  \right )\,.
  $$
 A refinement of the atlas $\{U_{\alpha}\}_{\alpha\in A}$ of $X$
  gives rise to a refinement of the gluing system ${\cal R}$ of rings
   via localizations of $\End_{U_{\alpha}}(E_{\alpha})$, $\alpha\in A$,
   using only central elements.
 This defines an equivalence relations on the set of such gluing systems of rings.
 We should regard $X^{\!A\!z}$ as the ``space" that is associated to an equivalence
  class of gluing system of rings,
   with the structure sheaf ${\cal O}_X^{A\!z}$ from central localizations of the Azumaya algebras
   in $\{\End_{U_{\alpha}}(E_{U_{\alpha}})\}_{\alpha\in A}$.
  
 Conceptually, the noncommutative space $X^{\!A\!z}$ sits over the commutative space $X$
  via a built-in surjection $X^{\!A\!z}\rightarrow X$
  that is defined through the tautological inclusion ${\cal O}_X\hookrightarrow {\cal O}_X^{A\!z}$
  of structure sheaves.
 When in need of a point-set picture, one may still think of $X^{\!A\!z}$
   as the topological space $X$ whose local function-rings are specified by ${\cal O}_X^{A\!z}$.
 However, {\sc Figure}~1-2, though only in the $0$-dimensional case, illustrates the fact
                         % Figure [surrogate.pdf]
   that $X^{\!A\!z}$ can have very rich contents, despite topologically the same as $X$.
}%end-remark

\bigskip
\bigskip

The purpose is this note is to study the question
  \begin{itemize}
   \item[{\bf Q.}] \parbox[t]{37em}{{\it
    What is the correct notion of a `$k$-times-differentiable map'  from such a noncommutative manifold}
	to a(n ordinary real) manifold, in particular a Calabi-Yau manifold,
  {\it that can fit in to describing D-branes as fundamental (as opposed to solitonic) objects in string theory}?}
  \end{itemize}	
Before attempting this (cf.\ Sec.~5),
it is very basic to understand the noncommutative space $X^{\!A\!z}$ itself
  in the context of differential topology and geometry first.
It turns out that, except the terminologies,
 such study of basics of matrix manifolds in the context of noncommutative differential geometry
 was already made earlier in the works
  [DV-K-M1], [DV-M], [Mas]
    of {\it Michel Dubois-Violette}, {\it Richard Kerner}, {\it John Madore}, {\it Thierry Masson}
 and  the works 	
  [S\'{e}1], [S\'{e2}]
    of {\it Emmanuel S\'{e}ri\'{e}}.
See also the related works, e.g., [Co1], [GB-V-F], [Mad], and [M-M-M].
We review the needed parts of these works in this section, with adaptations,
  as the second of the two beginning building blocks of this note.

\bigskip
					
\subsection{Differential calculus on noncommutative rings with center a $C^k$-ring}

Basic notions like derivations, differentials, and differential forms
  can be defined for a noncommutative ring purely algebraically.
When the center of the ring in question is a $C^k$-ring,
 additional related non-algebraic features should be added to the notion to correctly reflect
 the nature of the space, if any, associated to that ring.
This prepares us toward the differential calculus on Azumaya differentiable manifolds
 with a fundamental module in Sec.~4.2.

\bigskip

\begin{flushleft}
{\bf Derivations and  differentials of a noncommutative ring with center a $C^k$-ring}
\end{flushleft}

\begin{definition} {\bf [derivation on ring].} {\rm
 Let $R$ be a(n associative, unital) algebra with center $Z_R$
   over a ground ring $S$.
 (In particular, $R\supset Z_R\supset S$.)
 An {\it $S$-derivation $\Theta$} on $R$  is a map
    $$
      \Theta\;:\; R\;\longrightarrow\; R
    $$
   that satisfies
    $$
    \begin{array}{llll}	
     \mbox{($S$-linearity)}
	  &&& \Theta(a_1r_1+a_2r_2)\;
	            =\; a_1\,\Theta(r_1)\,+\,a_2\,\Theta(r_2)   \,,   \\[.6ex]
	 \mbox{(Leibniz rule)}
	  &&& \Theta(r_1r_2)\;=\;  \Theta(r_1)\, r_2\,+\, r_1\, \Theta(r_2)
    \end{array}
   $$	
   for all $a_1,\,a_2\in S$ and $r_1,\,r_2\in R$.
 When $Z_R$ is a $C^k$-ring, $k\ge 1$, and $S\subset Z_R$ a $C^k$-subring,
  we require $\Theta$ to satisfy in addition
  $$
   \begin{array}{l}
    \hspace{-1em}\mbox{(chain rule)}\\[.6ex]
    \Theta(h(r_1,\,\cdots\,,\,r_s))\;
	  =\;  \partial_1 h(r_1,\,\cdots\,,\,r_s)\,\Theta(r_1)\;
	         +\; \cdots\; +\;
			 \partial_s h(r_1,\,\cdots\,,\,r_s)\,\Theta(r_s)	
   \end{array}			
  $$
  for all $h\in C^k({\Bbb R}^s)$, $s\in {\Bbb Z}_{\ge 1}$, and $r_1,\,\cdots\,,\,r_s\in Z_R$.

 Denote by $\Der_S(R)$ the set of all $S$-derivations on $R$.
 Then,  not that  if $\Theta\in \Der_S(R)$, then so does $r\Theta=\Theta r$,
     with $(r\Theta)(\,\cdot\,) := r(\Theta(\,\cdot\,))$ and
	         $(\Theta r)(\,\cdot\,):= (\Theta(\,\cdot\,))r$,
	 for $r\in Z_R$.
 Thus,
 {\it $\Der_S(R)$ is naturally a $Z_R$-module},
 with $r\cdot \Theta := r\Theta = \Theta r =: \Theta\cdot r$ for $\Theta\in \Der_S(R)$, $r\in Z_R$.
 
 Furthermore, if $\Theta_1,\, \Theta_2\in \Der_S(R)$, then so does the {\it Lie bracket}
   $$
     [\Theta_1, \Theta_2]\;:=\; \Theta_1\Theta_2\,-\,\Theta_2\Theta_1
   $$
   of $\Theta_1$ and $\Theta_2$.
 Thus, {\it $\Der_S(R)$   is naturally a Lie $S$-algebra}.
}\end{definition}

\bigskip

\begin{example} {\bf [inner derivation of ring].} {\rm
 (Continuing Definition~4.1.1.)
                    % Definition [derivation on ring]
 Every $r\in R$ defines an $S$-derivation $\Theta_r:R\rightarrow R$
    with $r^{\prime}\mapsto [r,r^{\prime}]:=rr^{\prime}-r^{\prime}r$.
 This defines a $Z_R$-module homomorphism $R\rightarrow \Der_S(R)$. 	
 An element in the image of this homomorphism is called an {\it inner derivation} of $R$.
}\end{example}

\bigskip

Except the Lie algebra structure, Definition~4.1.1 is a special case of the following:
                                                     % Definition [derivation on a ring]

\bigskip

\begin{definition}  {\bf [derivation with value in a module].} {\rm
 Let
    $R$ be a(n associative, unital) algebra with center $Z_R$ over a ground ring $S$  	
	and
	$M$ be a bi-$R$-module.
	(I.e.\ $M$ is both a left and a right $R$-module, with $rm=mr$ for all $r\in Z_R$ and $m\in M$.)
 An {\it $S$-derivation $\Theta$} from $R$ to $M$  is a map
    $$
      \Theta\;:\; R\;\longrightarrow\; M
    $$
   that satisfies
   $$
    \begin{array}{llll}	
     \mbox{($S$-linearity)}
	  &&& \Theta(a_1r_1+a_2r_2)\;
	            =\; a_1\,\Theta(r_1)\,+\,a_2\,\Theta(r_2) 	\,,   \\[.6ex]
	 \mbox{(Leibniz rule)}
	  &&& \Theta(r_1r_2)\;=\;  \Theta(r_1)\, r_2\,+\, r_1\, \Theta(r_2)
    \end{array}
   $$	
   for all $a_1,\,a_2\in S$ and $r_1,\,r_2\in R$.
 When $Z_R$ is a $C^k$-ring, $k\ge 1$, and $S\subset Z_R$ a $C^k$-subring,
 we require $\Theta$ to satisfy in addition
  $$
   \begin{array}{l}
    \hspace{-1em}\mbox{(chain rule)}\\[.6ex]
    \Theta(h(r_1,\,\cdots\,,\,r_s))\;
	  =\;  \partial_1 h(r_1,\,\cdots\,,\,r_s)\,\Theta(r_1)\;
	         +\; \cdots\; +\;
			 \partial_s h(r_1,\,\cdots\,,\,r_s)\,\Theta(r_s)	
   \end{array}			
  $$
  for all $h\in C^k({\Bbb R}^s)$, $s\in {\Bbb Z}_{\ge 1}$, and $r_1,\,\cdots\,,\,r_s\in Z$.

 Denote by $\Der_S(R,M)$ the set of all $S$-derivations from $R$ to $M$.
 Then, similar to the case of $\Der_S(R)$,
  {\it $\Der_S(R,M)$ is naturally a $Z_R$-module},
    with $r\cdot \Theta :=  r \Theta = \Theta r =: \Theta\cdot r$
   for $\Theta\in \Der_S(R,M)$, $r\in Z_R$.
}\end{definition}

\bigskip

\begin{definition} {\bf [module of differentials of ring].} {\rm
 Let $R$ be a(n associative, unital) $S$-algebra,
  with the center $Z_R$ of $R$ a $C^k$-ring that contains $S$ as a $C^k$-subring, $k\ge 1$
 Then,
  the {\it module of differentials}, denoted by $\Omega_{R/S}$, is the bi-$R$-module
   generated by the set
   $$
     \{ d(r)\,|\,  r\in R   \}
   $$
   subject to the relations
   $$
    \begin{array}{llll}	
     \mbox{($S$-linearity)}
	  &&& d(a_1r_1+a_2r_2)\;
	            =\; a_1\,d(r_1)\,+\,a_2\,d(r_2)\;
				=\; d(r_1)a_1\,+\, d(r_2)\,a_2 				\,,   \\[.6ex]
	 \mbox{(Leibniz rule)}
	  &&&  d(r_1r_2)\;=\;  d(r_1)\, r_2\,+\, r_1\, d(r_2) \,, \\[.6ex]
	 \mbox{($Z_R$-commutativity)}
	  &&&  d(r_1)\,r_3\;=\; r_3\,d(r_1)
    \end{array}
   $$	
   for all $a_1,\,a_2\in S$, $r_1,\,r_2\in R$, and $r_3\in Z_R$,  and
  $$
   \begin{array}{llll}
    \mbox{(chain rule)}\hspace{4em}
      &&&  d(h(r_1,\,\cdots\,,\,r_s))   \\
	  &&&  \hspace{1em}
                 =\;  \partial_1 h(r_1,\,\cdots\,,\,r_s)\, d(r_1)\;
	                     +\; \cdots\; +\;
			           \partial_s h(r_1,\,\cdots\,,\,r_s)\, d(r_s)	 \hspace{1em}
   \end{array}			
  $$
  for all $h\in C^k({\Bbb R}^s)$, $s\in {\Bbb Z}_{\ge 1}$, and $r_1,\,\cdots\,,\,r_s\in Z_R$.
 Denote the image of $d(r)$ under the quotient by $dr$.
 Then, by definition, the built-in map
  $$
    \begin{array}{ccccc}
	 d & :&  R & \longrightarrow  & \Omega_{R/S}\\[.6ex]
	    &&    r  & \longmapsto        & dr
	\end{array}
  $$
  is an $S$-derivation from $R$ to $\Omega_{R/S}$.
}\end{definition}

\bigskip

\noindent
Note that the Leibniz rule and the $Z_R$-commutativity imply that
 $$
   dr_3\,r_1\;=\; r_1\,dr_3  \hspace{4em}
	 \mbox{for all $r_1\in R$ and $r_3\in Z_R$.}
 $$
Note also that, using $dr_1\,r_2= d(r_1r_2)-r_1 dr_2$ for all $r_1,\,r_2\in R$,
 the bi-$R$-module $\Omega_{R/S}$ can be regarded as a left $R$-module generated by
 $\{dr\,|\,r\in R\}$.
 
By construction, $\Omega_{R/S}$ has the following universal property,
 which determines $\Omega_{R/S}$ uniquely up to a unique isomorphism:

\bigskip

\begin{lemma} {\bf [universal property of $\Omega_{R/S}$].}
 Continuing the setting in Definition~4.1.3.
                                        % Definition [derivation with value in a module]
 Let $\Theta: R\rightarrow M$ be an $S$-derivation from $R$ to $M$.
 Then there exists a unique bi-$R$-module homomorphism $h:\Omega_{R/S}\rightarrow M$
  such that the following diagram commutes
  $$
   \xymatrix{
    & &&& \Omega_{R/S}\ar[dd]^-{h} \\
    & R \ar[urrr]^-{d}   \ar[drrr]_-{\Theta}                \\
    & &&& M  &.
   }
  $$
\end{lemma}

\begin{proof}
 The map $d(r)\mapsto \Theta(r)$, for all $r\in R$,
 descends to the unique $h:\Omega_{R/S}\rightarrow M$ that makes the diagram commute.

\end{proof}

\bigskip

Taking $M$ to be $R$ as a bi-$R$-module,
 then $Hom_{\,\mbox{\scriptsize bi-}R}(\Omega_{R/S},R)$ is naturally a $Z_R$-module
  (with left and right identical)
  and one has the following identification:
  
\bigskip

\begin{lemma} {\bf [derivation vs.\ differential].}
Continuing the setting in Definition~4.1.3.
                                       % Definition [derivation with value in a module]
 There is a canonical isomorphism
 $$
   \Der_S(R)\; \stackrel{\sim}{\longrightarrow}\;  Hom_{\,\mbox{\scriptsize bi-}R}(\Omega_{R/S},R)
 $$
 as $Z_R$-modules.
\end{lemma}

\bigskip

This defines a pairing
 $$
   \Der_S(R)\otimes_{Z_R} \Omega_{R/S}\; \longrightarrow\;   R\,
 $$
 that is non-degenerate in the $\Der_S(R)$-component.
By construction, it can be expressed as:

\bigskip

\begin{lemma} {\bf [$Z_R$-bilinear pairing].}
 The above pairing is identical to the following well-defined, $Z_R$-bilinear pairing
  $$
  \begin{array}{llll}
   \Der_S(R)\otimes_{Z_R} \Omega_{R/S}  &  \longrightarrow  &  \hspace{1.2em} R\\[.6ex]
    \hspace{3.1em}\Theta\otimes \sum r_1dr_2r_3               &  \longmapsto
	   &  \sum r_1\Theta(r_2)r_3\;=:\; \alpha(\Theta) &,
  \end{array}
  $$
  where $\alpha:=\sum r_1dr_2 r_3$ represents an element of $\,\Omega_{R/S}$.
\end{lemma}

\bigskip

We will call $\alpha(\Theta)$ in the above pairing
 the {\it evaluation} of $\alpha$ at $\Theta$.
 
%---------------------------------------------------------------------------------------------------------------------------- 
% \bigskip
%
% \begin{remark}  {$[$K\"{a}hler differential$\,]$.} {\rm
%  ??????????????????
%  %
%  \marginpar{\raggedright\tiny $\bullet$ To be completed.}
% }\end{remark}
%
%===========------------------------------------------===========----------------------------------------

\bigskip

\begin{remark}  {$[$Definition~4.1.4 as a generalization from the commutative case$\,]$.} {\rm
                                        % Definition [module of differentials of ring]
  When $R$ is commutative,
   $\Omega_{R/S}$ in Definition~4.1.4
                                          % Definition [module of differentials of ring]
    reduces to the usual definition of differentials for commutative rings.
  {From} this aspect, Definition~4.1.4 is a natural extension of the notion of differentials
                                    % Definition [module of differentials of ring]
   from the commutative case  to the noncommutative case.
 {\it However},
  for a general $S$-algebra-homomorphism  $\rho:R \rightarrow R^{\prime}$
     with $R^{\prime}$ noncommutative,
   $\rho(Z_R)$ may not be contained in $Z_{R^{\prime}}$     and,  hence,
   the natural correspondence
     $r_1\cdot dr_2\cdot r_3 \mapsto \rho(r_1)\cdot d\rho(r_2)\cdot\rho(r_3) $
	is defined from $\Omega_{R/S}$ {\it only to a quotient of} $\Omega_{R^{\prime}/S}$
	  by enforcing the additional $\rho(Z_R)$-commutativity condition to $\Omega_{R^{\prime}/S}$.
 Cf.~Example~4.1.20 and Sec.~6.3.
    % Example [pull-back of differential under $\varphi$]
}\end{remark}

\bigskip

\begin{flushleft}
{\bf Differential graded algebra associated to ring with center a $C^k$-ring}
\end{flushleft}
\begin{definition} {\bf [DG-algebra associated to ring with center $C^k$-ring].} {\rm
 Continuing the setting in Definition~4.1.1.
                                        % Definition [derivation on ring]
 The {\it differential graded algebra} (in short, {\it DG-algebra})
  associated to the ring $R/S$,
  in notation $\bigwedge^{\bullet}\Omega_{R/S}$,  consists of the following data:
  \begin{itemize}
    \item[$\cdot$]
	 $\bigwedge^0\Omega_{R/S}\;:=\; R$.
	
	\item[$\cdot$]
     $\bigwedge^1\Omega_{R/S}\;:=\; \Omega_{R/S}$.

    \item[$\cdot$]
	 For general $l\in {\Bbb Z}_{\ge 2}$,
	  define $\bigwedge^l\Omega_{R/S}$ to be the bi-$R$-module of $Z_R$-multilinear antisymmetric
	  maps,  generated by
	  $${\small
	   \begin{array}{ccccc}
	     \alpha_1\wedge\,\cdots\,\wedge \alpha_ l  & :
		    & \overbrace{\Der_S(R)\times \,\cdots\,\times \Der_S(R)}
		                                                                           ^{\mbox{\scriptsize\it $l$-times}}
            & \longrightarrow     &  R                                                                                           \\[1.2ex]
	     &&
		  (\Theta_1,\,\cdots\,,\Theta_l)
		     & \longmapsto
		     & \sum_{\sigma\in\scriptsizeSym_l }\,
		        (-1)^\sigma\,
                  \alpha_1(\Theta_{\sigma(1)})\, \cdots\,  \alpha_l(\Theta_{\sigma(l)})
        \end{array}					 } % end-small
	  $$
	  for all $\alpha_1,\,\cdots\,,\alpha_l\in \Omega_{R/S}$.
	 Here, $\Sym_l$  is the permutation group of a set of $l$-many elements.
	
   \item[$\cdot$]
    A product
	 $$
	   \wedge\; :\;
	     \mbox{$\bigwedge^l\Omega_{R/S}$}
		   \times \mbox{$\bigwedge^{l^{\prime}}\Omega_{R/S}$}\;
		  \longrightarrow\; \mbox{$\bigwedge^{l+l^{\prime}}\Omega_{R/S}$}
	 $$
	 that is defined by extending $S$-linearly
	   the $S$-algebra structure on $R$,
	   the bi-$R$-module structure on each $\bigwedge^l\Omega_{R/S}$,  and
	   the following maps
	 $$
	   (\alpha_1\wedge\,\cdots\, \wedge \alpha_l\,,\,
	        \beta_1\wedge\,\cdots\,\wedge\beta_{l^{\prime}})\;
		 \longmapsto\;
		   \alpha_1\wedge\,\cdots\, \wedge \alpha_l
		    \wedge \beta_1\wedge\,\cdots\,\wedge\beta_{l^{\prime}}\,.
	 $$
	
   \item[$\cdot$]	
    A degree-$1$ $S$-linear map
     $$	
	   d\; :\; \mbox{$\bigwedge^{\bullet}\Omega_{R/S}$}\;
	         \longrightarrow\;  \mbox{$\bigwedge^{\bullet +1}\Omega_{R/S}$}
     $$
	 that is defined as follows:
	  \begin{itemize}
	    \item[$\cdot$]
          $d^{\,0}=d:R\rightarrow \Omega_{R/S}$
		  as defined in Definition~4.1.4.
		                     % Definition [module of differentials of ring]
							
	    \item[$\cdot$]
         $d^{\,1}: \Omega_{R/S}\rightarrow   \bigwedge^2\Omega_{R/S}$
		  is defined by extending $S$-linearly
		  $$
		    r^{\prime}drr^{\prime\prime}\; \longmapsto\;
			  dr^{\prime}\wedge dr r^{\prime\prime}\,-\, r^{\prime}dr dr^{\prime\prime}\,.
		  $$				
		
        \item[$\cdot$]
         The general
          $d^{\,l}: \bigwedge^l \Omega_{R/S}\rightarrow   \bigwedge^{l+1}\Omega_{R/S}$,
           $l\in {\Bbb Z}_{\ge 2}$,
		  is defined by extending $S$-linearly	
		  $$
		    \alpha_1 \wedge\, \cdots\, \wedge \alpha_l\;
			 \longmapsto\;
			  \sum_{i=1}^l\,
			    (-1)^{i+1}\,
				   \alpha_1\wedge\,\cdots\,
				   \wedge \alpha_{i-1} \wedge  d\alpha_i \wedge \alpha_{i+1}
				   \wedge\,\cdots\,\wedge \alpha_l\,.
		  $$				 		 				
	  \end{itemize}
  \end{itemize}
}\end{definition}

\bigskip

The following lemma is straightforward to check, though part of it is very tedious:

\bigskip

\begin{lemma}  {\bf [properties of $(\bigwedge^{\bullet}\Omega_{R/S}, d)$].}
 (1)
     $\omega \wedge r\omega_2= (\omega r)\wedge \omega_2$
	 for all $r\in R$ and $\omega_1,\,\omega_2\in\bigwedge^{\bullet}\Omega_{R/S}$.
	In particular, $ r \omega = \omega r $
	   for all $r\in Z_R$ and $\omega \in\bigwedge^{\bullet}\Omega_{R/S}$.

 (2)
  The graded algebra $\bigwedge^{\bullet}\Omega_{R/S}$ defined in Definition~4.1.9
                                 % Definition [DG-algebra associated to ring with center $C^k$-ring]
   is a differential graded algebra over $S$;
   namely, it is a graded $S$-algebra that satisfies
    \begin{itemize}
      \item[$\cdot$]  {$($underlying differential complex$)$}\hspace{1em}
	   $d\circ d\;=\; 0\,$,
	
	 \item[$\cdot$]   {$($graded Leibniz rule$)$} \hspace{1em}
      $$
	    d(\omega_1 \wedge \omega_2 )\;
	     =\; (d\omega) \wedge \omega_2 + (-1)^l  \omega_1\wedge  d\omega_2
      $$		
	  for all $\omega_1\in \bigwedge^l\Omega_{R/S}$
	       and $\omega_2\in \bigwedge^{\bullet}\Omega_{R/S}$.
   \end{itemize}
 
  (3)
    The differential $d$ of the graded differential algebra $\bigwedge^{\bullet}\Omega_{R/S}$
	 has the following generator-independent intrinsic form:
	  \begin{eqnarray*}
	  (d\omega)(\Theta_1,\,\cdots\,,\Theta_{l+1})\;
	     & =
		  & \sum_{i=1}^{l+1}
		       \Theta_i \omega(\Theta_1,\,\cdots\,,\widehat{\Theta_i}\,,\,\cdots\,,\,\Theta_{l+1}) \\
	   && \hspace{2em}
              \sum_{1\le i<j\le l+1}\,(-1)^{i+j}\,
                 \omega([\Theta_i,\Theta_j]\,,\,
				           \Theta_1,\,\cdots,\,
						    \widehat{\Theta_i}\,,\, \cdots\,,\, \widehat{\Theta_j}\,,\,
							\cdots\,,\,\Theta_{l+1})\,.			  			
	  \end{eqnarray*}
	 for all $\omega\in \bigwedge^l\Omega_{R/S}$  and
	           $\Theta_1,\,\cdots\,,\, \Theta_{l+1}\in \Der_S(R)$.
	 Here, $\widehat{(\,\cdot\,)}$ indicates an omitted term.
\end{lemma}

\bigskip

\begin{remark}  {$[$simplified expression$]$.} {\rm
  Using Lemma~4.1.10 (1) and the notes after Definition~4.1.4,
         % Lemma [properties of $(\bigwedge^{\bullet}\Omega_{R/S}, d)$]
		 % Definition [module of differentials of ring]
  an element in $\bigwedge^l\Omega_{R/S}$ can be expressed as an $S$-linear combination of  forms
    $r_0dr_1\wedge\,\cdots\,\wedge dr_l$, where $r_0,\, r_1,\, \cdots\,,\, r_l\in R$.		
}\end{remark}

\bigskip

\begin{remark} {$[$on graded commutativity$]$.} {\rm
  For $R$ noncommutative, there is no universal statement relating
     $\omega_1\wedge\omega_2$ and $\omega_2\wedge\omega_1$
	 for $\omega_1$, $\omega_2$ of pure degrees.
  However,
   \begin{itemize}
    \item[$\cdot$]
	{\bf $[$induced graded commutativity from $Z_R$]} \hspace{1em} {\it
    for $\alpha$	of pure degree and
    in a $Z_R$-linear combination of forms
	$dr_1\wedge\,\cdots\,\wedge dr_l$ with $r_1,\,\cdots\,,\, r_l\in Z_R$,
   one has the graded commutativity property:
   $$
     \alpha \wedge\omega \;
	  =\; (-1)^{\scriptsizedegree(\alpha)\,\scriptsizedegree(\omega)}\,\omega\wedge\alpha
   $$
    for all $\omega\in\bigwedge^{\bullet}\Omega_{R/S} $  of pure degree.}
  \end{itemize}
 In particular, when $R$ is commutative,
  $(\bigwedge^{\bullet}\Omega_{R/S}, d)$ as defined is a differential graded commutative $S$-algebra.
}\end{remark}

\bigskip

\begin{remark} {$[$permutation of derivations vs.\ permutation of differentials$\,]$.} {\rm
 For $R$ noncommutative and general $\alpha_1,\,\cdots\,,\,\alpha_l\in\Omega_{R/S}$,
  $\alpha_1\wedge\,\cdots\,\wedge \alpha_l$
   not the same as
  $\sum_{\sigma\in\scriptsizeSym_l}
       \alpha_{\sigma(1)}\otimes\,\cdots\,\otimes\alpha_{\sigma(l)}$.
 Indeed, in general the latter does not define an antisymmetric map from
  $\Der_S(R)\times\,\cdots\,\times\Der_S(R)$ to $R$ through the natural evaluation.	
}\end{remark}

\bigskip

\begin{remark} {$[$algebraic vs.\ geometric notions$\,]$.} {\rm
 Though the above algebraic formulation is unavoidable, it is instructive to
 keep the following correspondence in mind:
  $$
   \begin{array}{clcccll}
     \mbox{\boldmath $\cdot$}
	  & \parbox[t]{11em}{\it ${\Bbb R}$-algebra $R$}
	     && \Longleftrightarrow &&  \parbox[t]{9em}{\it space $X$} \\[.6ex]		
	\mbox{\boldmath $\cdot$}
      & \parbox[t]{11em}{{\it derivation} on $R$}
		&&   \Longleftrightarrow
        && \parbox[t]{9em}{{\it vector field} on $X$}    \\[.6ex]
	\mbox{\boldmath $\cdot$}
	 & \parbox[t]{11em}{{\it differential} of $R$}
	    && \Longleftrightarrow &&  \parbox[t]{9em}{{\it $1$-form} on $X$}\\[.6ex]
     \mbox{\boldmath $\cdot$}
	 & \parbox[t]{11em}{{\it degree-$l$ differential} of $R$}
	    && \Longleftrightarrow &&  \parbox[t]{9em}{{\it $l$-form} on $X$}	 &.
   \end{array}
   $$
 However, it should be noted that
  when $R$ is noncommutative, $\Der_{\Bbb R}(R)$ is only a $Z_R$-module.
 Let $X_{Z_R}$ be the space associated to the center $Z_R$ of $R$.
 Then, a derivation $\Theta\in\Der_{\Bbb R}(R)$ is more adequately thought of
  as  an $R$-valued vector field on $X_{Z_R}$
  acting on a sheaf of ${\cal O}_{X_{Z_R}}$-algebras associated to $R$.
}\end{remark}

\bigskip

\begin{flushleft}
{\bf Differential calculus on Azumaya/matrix points, push-pulls under $C^k$-maps to ${\Bbb R}^n$,
         and infinitesimal deformations of such maps}
\end{flushleft}
The function-ring of an Azumaya/matrix point $p^{A\!z}$ is isomorphic to the matrix algebra
  $\End_{\Bbb C}({\Bbb C}^{\oplus r})$
     $= M_{r\times r}({\Bbb C})$ over ${\Bbb C}$ for some $r$,
 whose center is given by the subalgebra ${\Bbb C}\cdot\Id_{r\times r}\simeq {\Bbb C}$.
Up to a relabelling,
 the differential calculus on $p^{A\!z}$ is by definition the differential calculus on the matrix ring
  $M_{r\times r}({\Bbb C})$.
We review it in the following two examples to illustrate the notions introduced in this subsection and
  to prepare for Sec.~4.2.
    
\bigskip

\begin{example} {\bf [derivations \& differentials of Azumaya/matrix point].} {\rm
 Up to a relabelling, derivations and differentials of $p^{A\!z}$ are then defined to be
  derivations and differentials of the noncommutative algebra $M_{r\times r}({\Bbb C})$
  over ${\Bbb C}$.
 Let $(e_1,\,\cdots\,,\,e_r)$ be the standard basis of ${\Bbb C}^{\oplus r}$ and
     $(e^1,\,\cdots\,,\,e^r)$ be the dual basis for $({\Bbb C}^{\oplus r})^{\vee}$.
 
 All derivations on $M_{r\times r}({\Bbb C})$ over ${\Bbb C}$ are inner.
 Thus,
  $$
   \Der_{\Bbb C}(M_{r\times r}({\Bbb C})) \;
    \simeq\; M_{r\times r}({\Bbb C})/({\Bbb C}\cdot\Id_{r\times r})\;
    \simeq\; \slLie_r({\Bbb C})\; \subset\; M_{r\times r}({\Bbb C})
  $$
  as Lie algebras over ${\Bbb C}$.
 In particular, $\dimm_{\Bbb C}(\Der_{\Bbb C}(M_{r\times r}({\Bbb C})))=r^2-1$.
   
 For differentials, let $e_i\!\,^j:=e_i\otimes e^j$, $1\le i,\,j\le r$,
  be the $r\times r$-matrix with $1$ for the $(i,j)$-entry and $0$ elsewhere.
 Then,
  $$
    e_i\!\,^j\,e_{i^{\prime}}\!\,^{j^{\prime}}\;
	 =\; \delta^j_{i^{\prime}}\, e_i\!\,^{j^{\prime}}
  $$
 and, hence,
  $$
   de_i\!\,^j\,e_{i^{\prime}}\!\,^{j^{\prime}}\;
	 =\; -\, e_i\!\,^j\,de_{i^{\prime}}\!\,^{j^{\prime}}\,
	     +\, \delta^j_{i^{\prime}}\, de_i\!\,^{j^{\prime}}\,.
  $$
 Thus, the bi-$M_{r\times R}({\Bbb C})$-module $\Omega_{M_{r\times r}({\Bbb C})/{\Bbb C}}$
  can be converted naturally to a purely left $M_{r\times r}({\Bbb C})$-module,
  generated by $de_i\!\,^j$, $1\le i,\,j\le r$, with the relation
  $de_1\!\,^1+\,\cdots\,+ de_i\!\,^i+\,\cdots\,+ de_r\!\,^r\,=\,0\,$.
 It follows from
  the pairing
   $\Der_{\Bbb C}(M_{r\times r}({\Bbb C}))
     \otimes_{\Bbb C}\Omega_{M_{r\times r}({\Bbb C})/{\Bbb C}}
	  \rightarrow M_{r\times r}({\Bbb C})$
  and a dimension count that
  $$
   \Omega_{M_{r\times r}({\Bbb C})/{\Bbb C}}\;
    \stackrel{\sim}{\longrightarrow}\;
	M_{r\times r}({\Bbb C})\otimes_{\Bbb C} (\slLie_r({\Bbb C}))^{\vee}
  $$
  (i.e.\ $M_{r\times r}({\Bbb C})$-valued ${\Bbb C}$-linear functional on $\slLie_r({\Bbb C})$ )
  with
  $$
    m_0\,dm_1\;\; \longmapsto\;\;  (\,\cdot\;\mapsto\; m_0\, [\,m_1,\,\cdot\,] )\,, \hspace{2em}
	  \mbox{for $m_0$, $m_1\in M_{r\times r}({\Bbb C})$}\,.
 $$
 Here,
   $(\slLie_r({\Bbb C}))^{\vee}$ is the dual ${\Bbb C}$-vector space of
      $\slLie_r({\Bbb C})$ as a ${\Bbb C}$-vector space    and
  the bi-$M_{r\times r}({\Bbb C})$-module structure
   on  $M_{r\times r}({\Bbb C})\otimes_{\Bbb C}(\slLie_r({\Bbb C}))^{\vee}$
   is given by $m_1\cdot (m\otimes f )\cdot m_2 := (m_1mm_2)\otimes f$.
 The built-in derivation
   $d:M_{r\times r}({\Bbb C})\rightarrow\Omega_{M_{r\times r}({\Bbb C})/{\Bbb C}}$
   is realized as
   $$
    \begin{array}{ccccc}
     d  & :     &  M_{r\times r}({\Bbb C})  &  \longrightarrow
         & M_{r\times r}({\Bbb C})\otimes_{\Bbb C} (\slLie_r({\Bbb C}))^{\vee}\\[.6ex]
     &&  m   &  \longmapsto    &  (\,\cdot\;\mapsto\; [\,m,\, \cdot\,] )\,.
   \end{array}
  $$
\noindent\hspace{15.7cm}$\square$
}\end{example}

\bigskip

\begin{example}
{\bf [differential graded algebra of Azumaya/matrix point].} {\rm
 It follows from Example~4.1.15 that
                     % Example [derivations \& differentials of Azumaya/matrix point]
 the differential graded algebra
  $\bigwedge^{\bullet}\Omega_{M_{r\times r}({\Bbb C})/{\Bbb C}}$
  of the matrix ring $M_{r\times r}({\Bbb C})$ over ${\Bbb C}$
  is given by
  $$
    \mbox{$\bigwedge^{\bullet}\Omega_{M_{r\times r}({\Bbb C})/{\Bbb C}}$}\;
     \simeq\; M_{r\times r}({\Bbb C})
	                     \otimes_{\Bbb C}\mbox{$\bigwedge^{\bullet}(\slLie_r({\Bbb C}))^{\vee}$}\,.
  $$
 Here,
  \begin{itemize}
    \item[$\cdot$]
	  the bi-$M_{r\times r}({\Bbb C})$-module structure of
	  $M_{r\times r}({\Bbb C})
	                     \otimes_{\Bbb C}\bigwedge^{\bullet}(\slLie_r({\Bbb C}))^{\vee}$
	    comes from the $M_{r\times r}({\Bbb C})$-factor as in Example~4.1.15,
                                                                      % Example [derivations \& differentials of Azumaya/matrix point]
																	
    \item[$\cdot$]
     $\bigwedge^{\bullet}(\slLie_r({\Bbb C}))^{\vee}$	
	   is the exterior algebra (e.g.\ [Gu-P] and [Wa]) of
	   $(\slLie_r({\Bbb C}))^{\vee}$ as a ${\Bbb C}$-vector space,
	
    \item[$\cdot$]	
     $d: \bigwedge^{\bullet}\Omega_{M_{r\times r}(\Bbb C)/{\Bbb C}}
	        \rightarrow  \bigwedge^{\bullet+1}\Omega_{M_{r\times r}({\Bbb C})/{\Bbb C}}$
     is realized on
	 $M_{r\times r}({\Bbb C})
	                     \otimes_{\Bbb C}\bigwedge^{\bullet}(\slLie_r({\Bbb C}))^{\vee}$
	as the ${\Bbb C}$-linear extension of
     $$
        d(m\otimes  \omega)\;=\;  dm\wedge \omega
     $$	 	
	 for all $m\in M_{r\times r}({\Bbb C})$ and
	           $\omega\in\bigwedge^{\bullet}(\slLie_r({\Bbb C}))^{\vee}$,
	      with $dm$ as defined in Example~4.1.15.
	                                            % Example [derivations \& differentials of Azumaya/matrix point]
  \end{itemize}		
\noindent\hspace{15.7cm}$\square$
}\end{example}

\bigskip

Combined now with the study in Sec.~3,
let
 $$
   \varphi\; :\; (p,\End_{\Bbb C}({\Bbb C}^{\oplus r}),{\Bbb C^{\oplus r}})\;
                            \longrightarrow\; {\Bbb R}^n
 $$
   be a $C^k$-map, $k\ge 1$, from a fixed Azumaya point with a fundamental module to ${\Bbb R}^n$
   defined by a $C^k$-admissible ring-homomorphism
   $\varphi^{\sharp}:C^k({\Bbb R}^n)
      \rightarrow \End_{\Bbb C}({\Bbb C}^{\oplus r})=M_{r\times r}({\Bbb C})$
   over ${\Bbb R}\hookrightarrow{\Bbb C}$.
  
\bigskip

\begin{example} {\bf [push-forward of derivation under $\varphi$].} {\rm
 Let $\Theta\in \Der_{\Bbb C}(M_{r\times r}({\Bbb C}))$
  be a derivation on $M_{r\times r}({\Bbb C})$ over ${\Bbb C}$.
 Then $\Theta$ acts on $C^k({\Bbb R}^n)$ via
   $$
     (\varphi_{\ast}\Theta)(f)\; :=\; \Theta(\varphi^{\sharp}(f))\,.
   $$
 Which satisfies a Leibniz rule of the form
  $$
   (\varphi_{\ast}\Theta)(fg)\;
     =\;  (\varphi_{\ast}\Theta)(f)\,\varphi^{\sharp}(g)\,
	          +\, \varphi^{\sharp}(f)\,(\varphi_{\ast}\Theta)(g) \,.
  $$
 Thus it makes sense to regard
   $\varphi_{\ast}\Theta$ thus defined
   as a {\it
     derivation on $C^k({\Bbb R}^n)$ with values in $M_{r\times r}({\Bbb C})$ through $\varphi$},
 though {\it caution} that,
    under the ring-homomorphism $\varphi^{\sharp}$,
     $M_{r\times r}({\Bbb C})$ as a left-$C^k({\Bbb R}^n)$-module
      is not identical to $M_{r\times r}({\Bbb C})$ as a right-$C^k({\Bbb C}^n)$   and, hence,
   that {\it there is no map
     $\Omega_{C^k({\Bbb R}^n)/{\Bbb R}}\rightarrow M_{r\times r}({\Bbb C})$
	 that makes the following diagram commute}
	 $$
       \xymatrix{
        & &&& \Omega_{C^k({\Bbb R}^n)/{\Bbb R}}\ar[dd] \\
        & C^k({\Bbb R}^n) \ar[urrr]^-{d}   \ar[drrr]_-{\varphi_{\ast}\Theta}                \\
        & &&& M_{r\times r}({\Bbb C})  &.
       }
     $$
 Define the ${\Bbb C}$-module
  $$
    \varphi^{\ast}\Der_{\Bbb R}(C^k({\Bbb R}^n))\;
	:=\;  \left\{\left.\!\begin{array}{l}
			      \mbox{${\Bbb R}$-linear map} \\[.6ex]
			      \Xi: C^k({\Bbb R}^n)\rightarrow M_{r\times r}({\Bbb C})			
				                 \end{array}
			 \right|\;
			 \Xi(fg)\;=\; \Xi(f)\,\varphi^{\sharp}(g)\,+\, \varphi^{\sharp}(f)\,\Xi(g)\,
			\right\}
  $$ 	
 Then, on the domain-of-$\varphi$ side,
  the correspondence $\Theta \mapsto \varphi_{\ast}\Theta $ defines a homomorphism
  $$
    \varphi_{\ast}\; :\;
	  \Der_{\Bbb C}(M_{r\times r}({\Bbb C}))\;
	   \longrightarrow\;
	   \varphi^{\ast}\Der_{\Bbb R}(C^k({\Bbb R}^n))  		
  $$
  as ${\Bbb C}\,$($=Z_{M_{r\times r}({\Bbb C})}$)-modules
  
\noindent\hspace{15.7cm}$\square$
}\end{example}

\bigskip

\begin{remark}
 $[ M_{r\times r}({\Bbb C})
	   \otimes_{\varphi^{\sharp}, C^k({\Bbb R}^n)}\Der_{\Bbb R}(C^k({\Bbb R}^n)) ]$.
{\rm
 Caution	that, unlike in the commutative case, in general and as ${\Bbb C}$-modules,
  $$
    \varphi^{\ast}\Der_{\Bbb R}(C^k({\Bbb R}^n))\;
     \not\simeq\;	
     M_{r\times r}({\Bbb C})
	   \otimes_{\varphi^{\sharp}, C^k({\Bbb R}^n)}\Der_{\Bbb R}(C^k({\Bbb R}^n))
  $$
  for $\varphi$ with $\varphi^{\sharp}(C^k({\Bbb R}^n))$ not contained in the center
  ${\Bbb C}\cdot\Id_{r\times r}$ of $M_{r\times r}({\Bbb C})$.
}\end{remark}		

\bigskip

\begin{remark} {$[$infinitesimal deformation of $C^k$-map$\,]$.} {\rm
 Given a real $1$-parameter family
  $$
   \varphi_t\;:\; (p,\End_{\Bbb C}({\Bbb C}^{\oplus r}),{\Bbb C}^{\oplus r})\;
    \longrightarrow\; {\Bbb R}^n
  $$
  of $C^k$-maps that are defined by a real $1$-parameter family of $C^k$-admissible ring-homomorphisms
  $\varphi^{\sharp}_t:C^k({\Bbb R}^n)\rightarrow M_{r\times r}({\Bbb C})$,
  $t\in (-\varepsilon, \varepsilon)$,
 the {\it infinitesimal deformation of  $\varphi_t$ at $t=0$}
  $$
    \left. \frac{d}{dt}\right|_{t=0}\varphi^{\sharp}\;:
	  C^k({\Bbb R}^n)\; \longrightarrow\; M_{r\times r}({\Bbb C})
  $$
  defines a $M_{r\times r}({\Bbb C})$-valued derivation
  $\Xi   \in   \varphi^{\ast}\Der_{\Bbb R}(C^k({\Bbb R}^n))$
  of $C^k({\Bbb R}^n)$ in the sense of Example~4.1.17.
                                                                           % Example [push-forward of derivation under $\varphi$]
 The converse statement
  \begin{itemize}
   \item[$\cdot$] \parbox[t]{38em}{\it
    {\bf [derivation vs.\ infinitesimal deformation].}
     Let
       $\varphi:(p,\End_{\Bbb C}({\Bbb C}^{\oplus r}),{\Bbb C}^{\oplus r})
	         \rightarrow {\Bbb R}^n$
	  	be a $C^k$-map  and
       $\Xi   \in  \varphi^{\ast}\Der_{\Bbb R}(C^k({\Bbb R}^n))$					
        be an $M_{r\times r}({\Bbb C})$-valued derivation on $C^k({\Bbb R}^n)$ through $\varphi$.
     Then there exists a $1$-parameter family of  $C^k$-admissible ring-homomorphisms
       $\varphi^{\sharp}_t:C^k({\Bbb R}^n)\rightarrow M_{r\times r}({\Bbb C})$,
        $t\in (-\varepsilon, \varepsilon)$,
	   over ${\Bbb R}\hookrightarrow{\Bbb C}$
       with $\varphi^{\sharp}_0=\varphi^{\sharp}$
      such that  $\Xi=\left.\frac{d}{dt}\right|_{t=0}\varphi^{\sharp}_t$.
	  }
  \end{itemize}
 should also be true, following a matrix exponential-map argument.
 But we'll leave this topic for another work.
}\end{remark}

\bigskip

\begin{example}
{\bf [pull-back of differential under $\varphi$].} {\rm
 (Cf.\ Remark~4.1.8; and recall some notation from Example~4.1.15.)
        % Remark [Definition~??? as a generalization from the commutative case]
		% Example [derivations \& differentials of Azumaya/matrix point]
 {To} be concrete,  consider the $C^k$-map
   $\varphi:(p.\End_{\Bbb C}({\Bbb C}^2),{\Bbb C}^4)\rightarrow {\Bbb R}^1$
   defines by
   $$
    \begin{array}{ccccccl}
     & \varphi^{\sharp} &: & C^k({\Bbb R}^1)   & \longrightarrow  & M_{2\times 2}({\Bbb C})\\[.6ex]
	 &  && y   & \longmapsto   & e_2\!\,^1    &.
	\end{array}
   $$
 Then,  {\it naively but naturally},   $\varphi$ induces a map
  $$
    \begin{array}{ccccccl}
     & \varphi^{\sharp}_{\ast}=\varphi^{\ast} &:    & \Omega_{C^k({\Bbb R}^1)/{\Bbb R}}
 	 & \longrightarrow     & \Omega_{M_{2\times 2}({\Bbb C})/{\Bbb C}}                    \\[.6ex]
	 &  && dy   & \longmapsto   & de_2\!\,^1     &.
	\end{array}
  $$
 Which would imply, for example,
  $$
    \begin{array}{cl}
      \varphi^{\ast}(y\,dy)  & =\; e_2\!\,^1\,de_2\!\,^1  \\[.2ex]
	   \|    \\
	  \varphi^{\ast}(dy\,y)  & =\; de_2\!\,^1\, e_2\!\,^1\;  =\;  -\, e_2\!\,^1\,de_2\!\,^1\, .
	\end{array}
  $$
 That is, $e_2\!\,^1\,de_2\!\,^1=0$ in $\Omega_{M_{2\times 2}({\Bbb C})/{\Bbb C}}$,
   which is a contradiction.
 Thus, {\it
    there is no natural map
    	from $\Omega_{C^k({\Bbb R}^1)/{\Bbb R}}$
		to $\Omega_{M_{2\times 2}({\Bbb C})/{\Bbb C}}$ over $\varphi^{\sharp}$,
    or from $M_{2\times 2}({\Bbb C})
	                   \otimes_{\varphi^{\sharp},C^k({\Bbb R}^1)}\Omega_{C^k({\Bbb R}^1)/{\Bbb R}}$
        to $\Omega_{M_{2\times 2}({\Bbb C})/{\Bbb C}}$
		   as bi-$M_{2\times 2}({\Bbb C})$-modules}.
 This illustrates the fact that
    in general a $C^k$-map
     $\varphi:(p,\End_{\Bbb C}({\Bbb C}^{\oplus r}),{\Bbb C}^{\oplus r})
	     \rightarrow {\Bbb R}^n$
    does not induce naturally a map
    	from $\Omega_{C^k({\Bbb R}^n)/{\Bbb R}}$
		to $\Omega_{M_{r\times r}({\Bbb C})/{\Bbb C}}$ over $\varphi^{\sharp}$,\\
    or from $M_{r\times r}({\Bbb C})
	                   \otimes_{\varphi^{\sharp},C^k({\Bbb R}^n)}\Omega_{C^k({\Bbb R}^n)/{\Bbb R}}$
        to $\Omega_{M_{r\times r}({\Bbb C})/{\Bbb C}}$
		   as bi-$M_{r\times r}({\Bbb C})$-modules.
 This is a general phenomenon when defining a map $\varphi$  from a noncommutative space to a commutative space
    via its contravariant notion of a ring-homomorphism $\varphi^{\sharp}$
	 from the function-ring of the target of $\varphi$ to the function-ring of the domain of $\varphi$.\footnote{Alert
	                                       readers may ask legitimately,
										   ``{\it Is there something wrong  in either our notion of $C^k$-map $\varphi$
	         									  or our definition of the module $\Omega_{R/S}$ of differentials
										          that causes such a consequence?" }
                                           Indeed, it is very true that the issue would be gone if in defining the notion of $\varphi$
										      via $\varphi^{\sharp}$, we require that {\it Im}$\,\varphi^{\sharp}$
											  lie in the center of the ring in question.
										   However, as already explained in [L-Y1] (D(1)) and tested against string-theory literature,
 										    cf.\ [L-Y1] (D(1)), [L-Y2] (D(3)), [L-Y3] (D(4)), and [L-Y5] (D(6)),
                                            such more restrictive notion of morphisms, though mathematically acceptable, is not adequate
                                              to describe D-branes in full.
										  Thus, we choose to stay on our string-theory-oriented-'n-tested notion of morphisms
										    and try to let other accompanying notions fall into it. 	
										  This leads us to the second part of the question concerning the defintion of
     										 $\Omega_{R/S}$ in Definition~4.1.4.
											                                         % Definition [module of differentials of ring]
                                          															
                                          {To} see clearer that this is really an issue more on how $\Omega_{R/S}$
  										    is defined when $R$ is commutative, rather than when $R$ is noncommutative,
										  let us resume the algebraic case of module of K\"{a}hler differentials and note 	
										   that in this case one couid define $\Omega_{R/S}$ as follows:
										   ($R$ being commutative or noncommutative)
											\begin{itemize}
											 \item[$\cdot$] \parbox[t]{14.2cm}{{\bf
											  Definition~4.1.4{\boldmath $^{\prime}$}.
											    [module of (K\"{a}hler) differentials of ring].}
                                              Let $R$ be a(n associative, unital) ring over a commutative ring $S\subset Z_R$.
                                              Then,
                                                 the {\it module of (K\"{a}hler) differentials}, denoted by $\hat{\Omega}_{R/S}$,
												 is the bi-$R$-module generated by the set
                                                  $$
                                                     \{ d(r)\,|\,  r\in R   \}
                                                  $$
                                                 subject to the relations
                                                 $$
                                                   \begin{array}{llll}	
                                                    \mbox{($S$-linearity)}
	                                                 &&& d(a_1r_1+a_2r_2)\;
	                                                         =\; a_1\,d(r_1)\,+\,a_2\,d(r_2)\;
				                                             =\; d(r_1)a_1\,+\, d(r_2)\,a_2 				\,,   \\[.6ex]
	                                                \mbox{(Leibniz rule)}
	                                                 &&&  d(r_1r_2)\;=\;  d(r_1)\, r_2\,+\, r_1\, d(r_2)
                                                   \end{array}
                                                 $$	
                                                 for all $a_1,\,a_2\in S$ and $r_1,\,r_2\in R$.												 
                                              Denote the image of $d(r)$ under the quotient by $dr$.
                                              Then, by definition, the built-in map
                                                $$
                                                  \begin{array}{ccccc}
	                                               d & :&  R & \longrightarrow  & \hat{\Omega}_{R/S}\\[.6ex]
	                                               &&    r  & \longmapsto        & dr
	                                              \end{array}
                                                $$
                                               is an $S$-derivation from $R$ to $\hat{\Omega}_{R/S}$.
                                              }% end-parbox											
											\end{itemize}
                                          With this definition,
   									       let $R$, $R^{\prime}$ be $S$-algebras with
										      $S\hookrightarrow Z_R$ and $S\hookrightarrow Z_{R^{\prime}}$.
                                          Then any $S$-algebra-homomorphism $\rho:R\rightarrow R^{\prime}$
										   with $\rho(S)\subset Z_{R^{\prime}}$ induces canonically a map
										   $\rho_{\ast}: \hat{\Omega}_{R/S}
										                                 \rightarrow \hat{\Omega}_{R^{\prime}/S}$
										   with $r_1drr_2\mapsto \rho(r_1)d\rho(r)\rho(r_2)$.
										
                                         An issue now comes in when $R$ is {\it commutative}.
										 There in the standard treatment
										   $\Omega_{R/S}$ is defined as in Definition~4.1.4$^{\prime}$
                                                                                  % Definition [module of (K"{a}hler) differentials of ring]
                                           but subject to {\it an additional set of relations}
                                           $$
                                             \begin{array} {llll}
                                               \mbox{(commutativity)}
                                                &&& 	 r_1\, d(r_2)\; =\;  d(r_2)\, r_1 \hspace{16.6em}
                                             \end{array}
                                           $$
                                           for all $r_1,\, r_2\in R$.										   								
                                         By definition, there is a built-in quotient bi-$R$-module epimorphism
                                          $$
                                            \hat{\Omega}_{R/S}\;\longrightarrow\; \Omega_{R/S}\,.
                                          $$
                                        (Here, the right $R$-module structure on $\Omega_{R/S}$ is, by definition, identical
                                            to the left $R$-module structure on $\Omega_{R/S}$.)
                                         In this case when $\rho:R\rightarrow R^{\prime}$ is as above
										     with $R$ commutative and $R^{\prime}$ noncommutative,
										   one has the natural map
										    $\hat{\Omega}_{R/S}\rightarrow \hat{\Omega}_{R^{\prime}/S}$ as before
											but in general it doesn't descend to a map
											from $\Omega_{R/S}$ to $\hat{\Omega}_{R^{\prime}/S}$.
                                         This is the fundamental/underlying/built-in reason
										   our $\varphi$ doesn't pull back differentials in general.
										
										 Our definition of $\Omega_{R/S}$ in Definition~4.1.4
										                                                                    % Definition [module of differentials of ring]                                       																											
										  is guided by and designed to make $\Omega_{R/S}$ as a module of functionals
  										   on {\it Der}$_S(R)$.
										 Since the latter is canonically a $Z_R$-module with identical left- and right-module structure,
   										   $Z_R$-commutativity relations are added in.
                                         With this, it naturally resumes the usual definition of $\Omega_{R/S}$ when $R$ is commutative.
                                         However, it doesn't resolve the issue of the existence of a canonically induced map on differentials
										   from a ring-homomorphism.
                                         We simply have to live with this ``defect" (from the aspect of commutative geometry)
										  and find other means when we do need such a notion.
                                         Very fortunately, as we will see in Sec.~6.3,
										  passing to surrogates patches up this defect in a natural way:
										 {\it While in general
										    differentials of a commutative space cannot be pulled back to a noncomutative space,
										     it can still be pulled back to commutative surrogates of that noncommutative space.}
 										 This solves the problem as long as the goal of this project is concerned. 	
												 } %end-footnote 		
 
 In Sec.~6.3, we will discuss how the notion of `pull-back of differentials or tensors under $\varphi$'
  can still be naturally and usefully formulated to accommodate the notion of $C^k$-map $\varphi$ in our study
  by passing to the `surrogate' of the noncommutative space in question under $\varphi$,
  cf.~Lemma/Definition~5.3.1.7.
    % Definition [surrogate of $X^{\!A\!z}$ specified by $\varphi$]
		
\noindent\hspace{15.7cm}$\square$
}\end{example}

\bigskip
					
\subsection{Differential calculus on Azumaya differentiable manifolds with a fundamental module}
 
Recall the introduction of  the current Sec.~4 and
let
 $$
  (X^{\!A\!z}, {\cal E})\;
	 :=\;  (X,{\cal O}_X^{A\!z}:=\Endsheaf_{{\cal O}_X^{\,\Bbb C}}({\cal E}),{\cal E})
 $$
 be an Azumaya $C^k$-manifold with a fundamental module,  as introduced in Definition~4.0.1.
                                                              % Definition [Azumaya/matrix $C^k$-manifold with a fundamental module]
The general setting and differential-calculus-related notions in Sec.~4.1 apply then
  to Azumaya/matrix algebras over a $C^k$-ring in our situation.
A local-to-global gluing of related modules gives rise to respective associated sheaves on $X$ or $X^{\!A\!z}$.
This provides a basis for the differential calculus on $(X^{\!A\!z},{\cal E})$.
The study in this subsection can be regarded also
 as a central extension of the differential calculus on Azumaya points  $p^{A\!z}$,
 as summarized/reviewed in Example~4.1.15 and Example~4.1.16.
                       % Example [derivations \& differentials of Azumaya/matrix point]
			           % Example [differential graded algebra of Azumaya/matrix point]
This subsection is an adaptation of relevant parts of  the work [DV-M] of Michel Dubois-Violette and Thierry Masson,
  making the terminology and notation more amicable to our project  and
   bringing out more to-the-front of the $C^k$-ring structure of the center of the noncommutative rings involved.
Readers are referred to ibidem and references therein for more details and other pre-1995, physical motivations
 that lead these and other authors to study calculus on matrix differentiable manifolds as well.

\bigskip

\begin{flushleft}
{\bf Basic objects in differential calculus on Azumaya manifolds with a fundamental module}
\end{flushleft}
The following lemma follows from standard differential topology:

\bigskip

\begin{lemma} {\bf [differential calculus on $C^k(U)$ for open set of $C^k$-manifold].}
 For the $C^k$-ring $C^k(U)$, where $U$ is an open set of a $C^k$-manifold,
 $Der_{\Bbb R}(C^k(U))$,
 $\Omega_{C^k(U)/{\Bbb R}}$,
  and $\bigwedge^{\bullet}\Omega_{C^k(U)/{\Bbb R}}$
    as defined in Sec.~4.1
  coincide with the usual $C^k(U)$-module of
    $C^k$ tangent vector fields (cf.~$C^k(T_{\ast}U)$),
    $C^k$ $1$-forms (cf.~$C^k(T^{\ast}U)$),     and
    $C^k$ differential forms  (cf.~$C^k(\bigwedge^{\bullet}T^{\ast}U)$)
  on $U$, as defined in differential topology.
\end{lemma}

\bigskip

\begin{definition} {\bf [tangent sheaf, cotangent sheaf, and sheaf of $l$-forms].}  {\rm
 Recall the notation from the introduction of the current Sec.~4.
 (1)
 The presheaf on $X$  that associates to an open set $U\subset X$
  the $C^k(U)^{\Bbb C}$-module $\Der_{\Bbb C}(\End_U(E_U))$ is a sheaf on $X$,
  denoted by ${\cal T}_{\ast}X^{\!A\!z}$
   and called interchangeably the {\it tangent sheaf} of $X^{\!A\!z}$ or
    the {\it sheaf of derivations on ${\cal O}_X^{A\!z}$}.
 Caution that ${\cal T}_{\ast}X^{\!A\!z}$ is an ${\cal O}_X$-module,
  but not an ${\cal O}_X^{A\!z}$-module.
  
 (2)
 The presheaf on $X$  that associates to an open set $U\subset X$
  the bi-$\End_U(E_U)$-module $\Omega_{\scriptsizeEnd_U(E_U)/{\Bbb C}}$ is a sheaf on $X$,
  denoted by ${\cal T}^{\ast}X^{\!A\!z}$
   and called interchangeably the {\it cotangent sheaf} of $X^{\!A\!z}$ or
    the {\it sheaf of differentials of ${\cal O}_X^{A\!z}$}.
 Note that ${\cal T}^{\ast}X^{\!A\!z}$ is a bi-${\cal O}_X^{A\!z}$-module.
 By construction, there is a canonical map
  $d:{\cal O}_X^{A\!z}\rightarrow {\cal T}^{\ast}X^{\!A\!z}$
   as sheaves of ${\Bbb C}$-vector spaces on $X$.
 
 (3)
 The presheaf on $X$  that associates to an open set $U\subset X$
  the bi-$\End_U(E_U)$-module $\bigwedge^l\Omega_{\scriptsizeEnd_U(E_U)/{\Bbb C}}$
   is a sheaf on $X$,
  denoted by $\bigwedge^l{\cal T}^{\ast}X^{\!A\!z}$
   and called
    the {\it sheaf of $l$-forms on $X^{\!A\!z}$}.
 Note that $\bigwedge^l{\cal T}^{\ast}X^{\!A\!z}$ is a bi-${\cal O}_X^{A\!z}$-module.
 
 (4) Continuing (3),
 the differential graded algebras
  $(\bigwedge^{\bullet}\Omega_{\scriptsizeEnd_U(E_U)/{\Bbb C}}, d)$,
    $U\stackrel{\mbox{\it\tiny open}}{\hookrightarrow}X$, 	
   as defined in Definition~4.1.9,
                      % Definition [DG-algebra associated to ring with center $C^k$-ring]
  glue to a {\it sheaf of differential graded algebras}
  $(\bigwedge^{\bullet} {\cal T}^{\ast}X^{\!A\!z}, d)$ on $X$.		
}\end{definition}
		
\bigskip

Let $U\subset X$ be an open set
 on which the complex vector bundle $E$ of rank-$r$ is trivial and is trivialized
 $$
   E_U\;\simeq\; U\times {\Bbb C}^r\,.
 $$
With respect the trivialization, one has
 $$
   {\cal O}_X^{A\!z}(U)\;\simeq\;  M_{r\times r}(C^k(U)^{\Bbb C})
 $$
 the ${\Bbb C}$-algebra of $r\times r$-matrices with entries ${\Bbb C}$-valued $C^k$-functions on $U$.
A central extension of the discussion and results in Example~4.1.15 and Example~4.1.16
                                         % Example [derivations \& differentials of Azumaya/matrix point]
				                         % Example [differential graded algebra of Azumaya/matrix point]
 from ${\Bbb C}$ ($=C^k(p)$) to $C^k(U)^{\Bbb C}$
 gives the local expression for
   ${\cal T}_{\ast}X^{\!A\!z}$,   ${\cal T}^{\ast}X^{\!A\!z}$,    and
   $\bigwedge^l{\cal T}^{\ast}X^{\!A\!z}$, and
   $(\bigwedge^{\bullet}{\cal T}^{\ast}X^{\!A\!z}, d)$
 as shown in the following lemma:
 
\bigskip

\begin{lemma} {\bf [expression under local trivialization of ${\cal E}$].}
 {\rm (Cf.\ [DV-M: Sec.~1.3].)}
 When restricted to $U$,
 \begin{itemize}
  \item[$\cdot$]   $\hspace{3em}$\\[-6.8ex]
  \item[]
   $\begin{array}{llll}
     {\cal T}_{\ast}X^{\!A\!z}(U)
          & \simeq 	
		  &   \Der_{\Bbb C}(M_{r\times r}(C^k(U)^{\Bbb C}))\;\;
		         \simeq\;\;
     				 \Der_{\Bbb C}\left(C^k(U)^{\Bbb C}
				                    \otimes_{\Bbb C}M_{r\times r}({\Bbb C})\right) \\[.2ex]
	 & \simeq
	   &   \Der_{\Bbb C}(C^k(U)^{\Bbb C})\otimes_{\Bbb C}\Id_{r\times r}\,
	          \oplus\,  C^k(U)^{\Bbb C}
			                     \otimes_{\Bbb C}\Der_{\Bbb C}(M_{r\times r}({\Bbb C}))       \\[.2ex]
       & \simeq        &     \Der_{\Bbb C}(C^k(U)^{\Bbb C})\otimes_{\Bbb C}\Id_{r\times r}\,
	                                     \oplus\,  C^k(U)^{\Bbb C}\otimes_{\Bbb C} \slLie_r({\Bbb C})   &;
	 \end{array}$
 \end{itemize}
 \begin{itemize}
  \item[$\cdot$]  $\hspace{3em}$\\[-6.8ex]
  \item[]
   $\begin{array}{llll}
      {\cal T}^{\ast}X^{\!A\!z}(U)
	   & \simeq
	   &\Omega_{C^k(U)^{\Bbb C}/{\Bbb C}}
		                     \otimes_{\Bbb C}M_{r\times r}({\Bbb C})\,
						  \oplus\, C^k(U)^{\Bbb C}\otimes_{\Bbb C}
						                   \Omega_{M_{r\times r}({\Bbb C})/{\Bbb C}}\\[.2ex]
	 &  \simeq
       & \Omega_{C^k(U)^{\Bbb C}/{\Bbb C}}
		                     \otimes_{\Bbb C}M_{r\times r}({\Bbb C})\,
						  \oplus\, C^k(U)^{\Bbb C}\otimes_{\Bbb C}(\slLie_r({\Bbb C}))^{\vee}&;
     \end{array}$
 \end{itemize}
\begin{itemize}
  \item[$\cdot$]  $\hspace{3em}$\\[-6.8ex]
  \item[]
   $\begin{array}{llll}
      (\bigwedge^{\bullet}{\cal T}^{\ast}X^{\!A\!z}(U),d)
	     & \simeq
		 & (\bigwedge^{\bullet}\Omega_{C^k(U)^{\Bbb C}/{\Bbb C}},d^{\prime})
		         \otimes_{\Bbb C}
		          (\bigwedge^{\bullet}\Omega_{M_{r\times r}(\Bbb C)/{\Bbb C}},
		                                            d^{\prime\prime}) \\[.2ex]
   	  & \simeq
		 & (\bigwedge^{\bullet}\Omega_{C^k(U)^{\Bbb C}/{\Bbb C}},d^{\prime})
		         \otimes_{\Bbb C}
		          (\bigwedge^{\bullet} (\slLie_r({\Bbb C}))^{\vee},
		                                            d^{\prime\prime}) \\[.2ex]											
     \end{array}$\\[1.2ex]
	with $d=d^{\prime}\otimes\Id + 1\otimes d^{\prime\prime}$.
	Here,
	 $d^{\prime}$, $d^{\prime\prime}$
	   are the degree-$1$ map $d$ in the respective graded differentioal algebras  for the purpose of distinguishing  and
     $\Id$ is the identity map.	
	In particular,
	  $\bigwedge^l{\cal T}^{\ast}X^{\!A\!z}(U)
	     \simeq
   		  \sum_{i=0}^l
		    (\bigwedge^i\Omega_{C^k(U)^{\Bbb C}/{\Bbb C}}
		         \otimes_{\Bbb C}
		            \bigwedge^{l-i} (\slLie_r({\Bbb C}))^{\vee})$.
 \end{itemize}
\end{lemma}

\bigskip

\begin{flushleft}
{\bf Connections on ${\cal E}$ and the structure of the tangent sheaf ${\cal T}_{\ast}X^{\!A\!z}$}
\end{flushleft}
We now take a closer look at the tangent sheaf ${\cal T}_{\ast}X^{\!A\!z}$
  of the Azumaya $C^k$-manifold $X^{\!A\!z}$.
Recall first that
 a connection
  $\nabla:{\cal E}\rightarrow {\cal T}^{\ast}X\otimes_{{\cal O}_X} {\cal E}$ on ${\cal E}$
   induces a connection
    $^{\vee}\nabla:{\cal E}^{\vee}
	     \rightarrow{\cal T}^{\ast}X\otimes_{{\cal O}_X^{\,\Bbb C}} {\cal E}^{\vee}$
		 on ${\cal E}^{\vee}$
   and a connection
    $^{A\!z}\!\nabla:{\cal O}_X^{A\!z}
         \rightarrow {\cal T}^{\ast}X\otimes_{{\cal O}_X^{\,\Bbb C}} {\cal O}_X^{A\!z}$
	   on ${\cal O}_X^{A\!z}:=\Endsheaf_{{\cal O}_X^{\,\Bbb C}}({\cal E})$
	  from the canonical isomorphism
	   $\Endsheaf_{{\cal O}_X}({\cal E})
	      \simeq {\cal E}\otimes_{{\cal O}_X}{\cal E}^{\vee}$,
  e.g.\ [Hu-L] and [Kob].

\bigskip

\begin{lemma} {\bf [naturalness of $^{A\!z\!}\nabla$ from $\nabla$: Leibniz rule].}
  Let
    $\;\cdot:{\cal O}_X^{A\!z}\times_X{\cal O}_X^{A\!z} \rightarrow {\cal O}_X^{A\!z} $
     with $(\alpha,\beta)\mapsto \alpha\cdot\beta=:\alpha\beta$
	  be the multiplication map from the sheaf-of-rings structure of ${\cal O}_X^{A\!z}$   and
   ${\cal O}_X^{A\!z}\times_X{\cal E}\rightarrow {\cal E}$
    with $(\alpha, s)\mapsto \alpha(s)$
     be the multiplication map from the ${\cal O}_X^{A\!z}$-module structure of ${\cal E}$.
 Then
  $$
   ^{A\!z}\!\nabla(\alpha\cdot \beta)\;
    =\; ^{A\!z}\!\nabla\alpha \cdot  \beta\, +\, \alpha \cdot  ^{A\!z}\!\nabla\beta
	\hspace{2em}\mbox{and}\hspace{2em}
  \nabla(\alpha(s))\;
    =\;  (^{A\!z}\!\nabla\alpha)(s)\,+\,\alpha(\nabla s)\,.
  $$
\end{lemma}

\bigskip

The proof is straightforward. This is the basis of the following splitting of ${\cal T}_{\ast}X^{\!A\!z}$:

\bigskip

\begin{lemma} {\bf [splitting of ${\cal T}_{\ast}X^{\!A\!z}$].}
{\rm ([DV-M: Proposition 3].)}
 Let $\Innsheaf({\cal O}_X^{A\!z})$ be the sheaf of inner derivations of ${\cal O}_X^{A\!z}$.
 Then there is a natural exact sequence of ${\cal O}_X^{\,\Bbb C}$-modules
  $$
    0\; \longrightarrow\; \Innsheaf({\cal O}_X^{A\!z})\;
	     \longrightarrow \; {\cal T}_{\ast}X^{\!A\!z}\;
         \longrightarrow\; {\cal T}_{\ast}X^{\Bbb C}\;  \longrightarrow\; 0\,.
  $$
 Furthermore, any connection
  $\nabla:{\cal E}\rightarrow {\cal T}_{\ast}X\otimes_{{\cal O}_X^{\,\Bbb C}}{\cal E}$
  on ${\cal E}$ induces an embedding
  $$
    \xymatrix{
	 \iota^{\nabla}\;:\; {\cal T}_{\ast}X^{\Bbb C}\; \ar@{^{(}->}[r]    & \;{\cal T}_{\ast}X^{\!A\!z}
	 }
  $$
  as ${\cal O}_X^{\,\Bbb C}$-modules,
   with  $\xi\mapsto ^{A\!z}\!\nabla_{\xi}$,
  that splits the above short exact sequence.
\end{lemma}

\bigskip

Note that the local statement of the above short exact sequence and the splitting from a flat connection
 are already manifest in Lemma~4.2.3.
                                     % Lemma [expression under local trivialization of ${\cal E}$]
Let us now turn to the proof of the lemma.

\begin{proof}		
 Since both the statements follow from the study in the local, we only need to investigate what happens locally.
 Let $U\subset X$ be an open set and consider
   the Azumaya algebra $C^k(\End_U(E_U))$ over $C^k(U)^{\Bbb C}$.
 Let $\Theta\in \Der_{\Bbb C}(C^k(\End_U(E_U)))$ and
  $f\in C^k(U)^{\Bbb C}\subset C^k(\End_U(E_U))$ as the center.
 Then $\Theta(f)\in C^k(U)^{\Bbb C}$ as well.
 Thus there is a tautological homomorphism of $C^k(U)^{\Bbb C}$-modules
   $\varsigma: \Der_{\Bbb C}(C^k(\End_U(E)))
                           \rightarrow \Der_{\Bbb C}(C^k(U)^{\Bbb C})$.
 
 If $\varsigma(\Theta)=0$,
  then, for any $\alpha\in C^k(\End_U(E_U))$,
   the value $\Theta(\alpha)(u)$ of $\Theta(\alpha)$ at each $u\in U$
   depends only on the value $\alpha(u)$ of $\alpha$ at $u$, rather than the $1$-jet of $\alpha$ at $u$.
 It follows that
  $\Theta$ is defined through derivations on each fiber Azumaya algebra of $\End_U(E_U)$ over $U$.
 This implies that $\Theta$ must be an inner derivation of $C^k(\End_U(E_U))$.
 This proves the existence of the natural sequence of $C^k(U)^{\Bbb C}$-module-homomorphisms
  $$
    0\; \longrightarrow\;  \Inn(C^k(\End_U(E_U)))\;
	      \longrightarrow\;  \Der_{\Bbb C}(C^k(\End_U(E_U)))\;
		  \stackrel{\varsigma}{\longrightarrow}\;  \Der_{\Bbb C}(C^k(U)^{\Bbb C})\,.
  $$
 
 Now bring into consideration the given connection
    $\nabla:{\cal E}\rightarrow
	  {\cal T}^{\ast}X\otimes_{{\cal O}_X^{\,\Bbb C}}{\cal E}$ on ${\cal E}$.
 Let $\xi\in\Der_{\Bbb C}(C^k(U)^{\Bbb C})$ be a derivation on $C^k(U)$,
   realized as a complex-valued $C^k$ vector field on $U$.
 Then it follows from Lemma~4.2.4
                               % Lemma [naturalness of $^{A\!z\!}\nabla$ from $\nabla$: Leibniz rule]
  that $^{A\!z}\!\nabla_{\xi}\in \Der_{\Bbb C}(C^k(\End_U(E_U)))$.
 Since $^{A\!z}\!\nabla_{\xi}$ is non-zero unless $\xi=0$,
 this gives an embedding
    $\iota^{\nabla}: \Der_{\Bbb C}(C^k(U)^{\Bbb C})
	    \rightarrow \Der_{\Bbb C}(C^k(\End_U(E_U)))$
	of $C^k(U)^{\Bbb C}$-modules.
 By construction, $\varsigma \circ\iota^{\nabla}$
      is the identity map on $\Der_{\Bbb C}(C^k(U)^{\Bbb C})$.
 Since a connection on ${\cal E}$ always exists,
  this proves that $\varsigma$ is surjective  and, furthermore,
  that $\iota^{\nabla}$ gives a splitting of the now short exact sequence
  $$
     0\; \longrightarrow\;  \Inn(C^k(\End_U(E_U)))\;
	      \longrightarrow\;  \Der_{\Bbb C}(C^k(\End_U(E_U)))\;
		  \stackrel{\varsigma}{\longrightarrow}\;  \Der_{\Bbb C}(C^k(U)^{\Bbb C})\;
		  \longrightarrow\; 0\,.
  $$
  
 This proves the lemma.
 									
\end{proof}
									
%-----------------------------------------------------------------------------------
% \bigskip
%
% \noindent $\bullet$
% Just in case we may still need them.
% Put other parts of [DV-M] that seem to be  less relevant to D(11) into a sequence of remarks.
% %
% \marginpar{\raggedright\tiny\vspace{-3em} $\bullet$
%         Or perhaps postpone them till the future else when in need.}
%
% \bigskip
%==============================================

\bigskip
 
\subsection{Metric structures on an Azumaya differentiable manifold with a fundamental module}

The notion of a metric structure in noncommutative geometry generalizing that in commutative differential geometry
 is not unique.
It is thus a key question as to which one is most compatible or useful as long as D-branes are concerned.
For the current note D(11.1) we do not yet need this notion.
However, for the conceptual completeness,
 we introduce in this subsection a weakest such notion, adapted from the works
 [S\'{e}1: Sec.~4.2] and [S\'{e}2: Sec.~3.1.1] of Emmanuel S\'{e}ri\'{e},
 that is in line with the standard setting in (commutative) differential geometry
 and let future work to decide what revision, if any, is required for the study of dynamics of D-branes in a space-time.
Readers are referred also to [Mad: Sec.~3.4] of John Madore for a dual approach that favors instead
  the cotangent sheaf of the noncommutative space in question.
 
\bigskip

\begin{definition} {\bf [metric tensor].} {\rm
 (Cf.\ [S\'{e}1: Definition:4.2.1], [S\'{e}2: Sec.~3.1.1: Definition 3.1].)
 A {\it metric tensor} on the Azumaya $C^k$-manifold $X^{\!A\!z}:=(X,{\cal O}_X^{A\!z})$
  is a nondegenerate symmetric ${\cal O}_X^{\,\Bbb C}$-bilinear map
  $$
    g\;:\; {\cal T}_{\ast}X^{\!A\!z}
	            \otimes_{{\cal O}_X^{\,\Bbb C}}{\cal T}_{\ast}X^{\!A\!z}\;
				\longrightarrow\; {\cal O}_X^{\,\Bbb C}\,.
  $$
 Here, we say that $g$ is {\it nondegenerate}  if
  \begin{itemize}
   \item[(1)]
    the induced module-homomorphism
	  $\hat{g}:{\cal T}_{\ast}X^{\!A\!z}\rightarrow  {\cal T}^{\ast}X^{\!A\!z}$
	  over the built-in ${\cal O}_X\hookrightarrow {\cal O}_X^{A\!z}$,
	 which sends a local section $\Theta$ to the functional $\,\cdot\,\mapsto\,g(\Theta,\,\cdot\,)$,
	is injective and
    	
   \item[(2)]
    ${\cal T}^{\ast}X^{\!A\!z}$ is generated by $\Image\hat{g}$
	both as a left-${\cal O}_X^{A\!z}$-module and as a right-${\cal O}_X^{A\!z}$-module.
 \end{itemize}
}\end{definition}

\bigskip

Recall from Lemma~4.2.5 that
                % Lemma [splitting of ${\cal T}_{\ast}X^{\!A\!z}$]
  ${\cal T}_{\ast}X^{\!A\!z}
         \simeq \Innsheaf({\cal O}_X^{A\!z})\oplus {\cal T}_{\ast}X^{\Bbb C}$.
Thus, $g$ has a noncanonical $2\times 2$ block-decomposition
 with the diagonal blocks
   a metric tensor on the $\slLie_r({\Bbb C})$-sheaf $\Innsheaf({\cal O}_X^{A\!z})$ over $X$
   and a metric tensor on the complexified tangent sheaf  ${\cal T}_{\ast}X^{\Bbb C}$ of $X$ respectively.

As in the case of commutative differential geometry,
the metric tensor $g$ on ${\cal T}_{\ast}X^{\!A\!z}$ induces
 \begin{itemize}
  \item[$\cdot$]
  a  {\it pairing}
     ${\cal T}^{\ast}X^{\!A\!z}
	     \otimes_{{\cal O}_X^{A\!z}}{\cal T}^{\ast}X^{\!A\!z}
		  \rightarrow {\cal O}_X^{A\!z}$, \hspace{2em}
	 which extends to
	
 \item[$\cdot$]	
  a {\it pairing}		
    $\bigwedge^{\bullet}{\cal T}^{\ast}X^{\!A\!z}
	    \otimes_{{\cal O}_X^{A\!z}}\bigwedge^{\bullet}{\cal T}^{\ast}X^{\!A\!z}
		  \rightarrow {\cal O}_X^{A\!z}$.
 \end{itemize}
The notion of
 \begin{itemize}
  \item[$\cdot$]
   the {\it integration over $X^{\!A\!z}$} and
	
  \item[$\cdot$]	
   the {\it Hodge $\ast$-operator on $\bigwedge^{\bullet}{\cal T}^{\ast}X^{\!A\!z}$ associated to $g$}
 \end{itemize}
 can also be defined; cf.\ [S\'{e}2].
We will come back to these themes for a more detailed examination
 when we study in a sequel
 the dynamics of D-branes along our line of maps from Azumaya manifolds with a fundamental module.

\bigskip
 
With Sec.~1 -- Sec.~4 as motivations,  preliminaries, and guides,
 we are now finally ready to address the main theme of the current note D(11.1).

\bigskip
  
\section{Differentiable maps from an Azumaya/matrix manifold with a fundamental module
                    to a real manifold}

In this section we generalize the notion of a $k$-times-differentiable (i.e.\  $C^k$-)map
  $$
     \varphi\; :\; (p,\End_{\Bbb C}({\Bbb C}^{\oplus}),{\Bbb C}^{\oplus r})\;
	   \longrightarrow\;  {\Bbb R}^n
  $$
  studied in Sec.~3
  to the notion of $k$-times-differentiable (i.e.\ $C^k$-)map
  $$
     \varphi\; :\; (X,{\cal O}_X^{A\!z}:=\Endsheaf_{{\cal O}_X}({\cal E}), {\cal E})\;
      \longrightarrow\;  (Y,{\cal O}_Y)
  $$
 from an Azumaya manifold with a fundamental module to a real manifold $Y$.

\bigskip

\subsection{A generalization of Sec.\ 2.1 and Sec.\ 3 to Azumaya/matrix manifolds
         with a fundamental module  -- local study}
		
Having seen in Sec.~2.1
   how differential topology and geometry can proceed in the spirit of algebraic geometry
   once the notion of $C^k$-rings to capture the non-algebraic structure of the ring $C^k(Y)$
    of $C^k$-functions of a manifold $Y$ is introduced,
 we would like to generalize it to the case that involves Azumaya algebras as well.
However, the latter type of algebras are noncommutative
   and hence cannot be made $C^k$-rings directly.
{To} deal with this issue,
  we look deeperly into the meaning of a `differentiable map' in terms of the function ring
 of its graph and compare that with the setting in [L-Y1] (D(1)) and [L-L-S-Y] (D(2))
 to find a mathematically sound setup for the case that involves Azumaya algebras,
 which is the focus of this subsection,
 {\it and then} check against concrete examples to make certain that the setup matches with
 the behavior of D-branes in string theory,
  which is the focus of Sec.\ 7.2 and sequels to the current note.

\bigskip

\begin{flushleft}
{\bf Differentiable map and the function-ring of its graph}
\end{flushleft}
Continuing the discussion of the previous subsection.
Another aspect  of a $C^k$-map $f:U\rightarrow V$ is that
 it defines a $C^k$-submanifold $\tilde{f}: \varGamma_f\hookrightarrow U\times V$
  that makes the following diagram commute:
 $$
   \xymatrix{
     \varGamma_f    \ar @{^{(}->}[dr]^(.55){\tilde{f}}
	                                \ar @/^{2ex}/ [drrrr]^-{f}
	                                \ar@/_{2ex}/  [ddr]_-{\pi_f}                      \\
     &  U\times V      \ar [rrr]_{pr_V} \ar [d]^-{pr_U}
     &&& V \\	
     &  U
   	}
 $$
 with all arrows $C^k$-maps of $C^k$-manifolds  and
    the arrow $\varGamma_f\rightarrow U$  from the composition $\pr_U\circ\tilde{f}$
	a $C^k$-isomorphism.
 Here, $\pr_U:U\times V\rightarrow U$ and $\pr_V:U\times V\rightarrow V$
   are the projection maps to each factor.
 In terms of $C^k$-function rings, this is the following diagram
 $$
   \xymatrix{
      C^k(\varGamma_f)\\
     &  C^k(U\times V)\ar @{->>}[ul]_(.45){\tilde{f}^{\sharp}}
     &&& \rule{0ex}{3ex}\hspace{1ex}
               C^k(V)\ar @/_{2ex}/[ullll]_-{f^{\sharp}}
                                \ar @{_{(}->}[lll]^{pr_V^{\sharp}}\\	
     &  \rule{0ex}{3ex}
      C^k(U)\ar @{^{(}->}[u]_-{pr_U^{\sharp}}
                       \ar@/^{2ex}/ [uul]^-{\pi_f^{\sharp}}
   	}
 $$
 of $C^k$-ring-homomorphisms.
 
Before attempting to generalize this to a notion of a $C^k$-map
  from a $C^k$-Azumaya manifold, in notation $\varphi: U^{\!A\!z}\rightarrow V$,
 the above diagrams already suggest some immediate generalizations
  by replacing the requirement
      that $\pr_U\circ\tilde{f}:\varGamma_f\rightarrow V$ be a $C^k$-isomorphism:
 \begin{itemize}
  \item[{\Large $\cdot$}]
   $[\,${\it multi-valued $C^k$-map}$\,]$ \hspace{2em}
   If we replace it by the requirement that
	  {\it $\pr_U\circ\tilde{f}$ be a $C^k$-surjection of $C^k$-manifolds
	            of the same dimension},
	then $\varGamma_f$ defines a multiple-valued $C^k$-map $f:U\rightarrow V$.
  This generalizes the notion of a $C^k$-map from single-valued to multi-valued. 	

 \item[{\Large $\cdot$}]	
  $[\,${\it $C^k$-correspondence}$\,]$ \hspace{2em}
  If we replace it by the requirement that
    {\it $\tilde{f}:\varGamma_f\rightarrow U\times V$
	        be a $C^k$-submanifold map},
    then $\varGamma_f$ defines a $C^k$-correspondence $f:U\rightarrow V$.
  This covers the previous case.	
	
 \item[{\Large $\cdot$}]	
   $[\,${\it more general $C^k$-correspondence}$\,]$ \hspace{2em}
	Once the notion of $C^k$-schemes and $C^k$-subschemes of a $C^k$-scheme are defined
  	in a natural/functorial way, then we may require
	 $\tilde{f}$ realize $\varGamma_f$ as a $C^k$-subscheme of $U\times V$
	 (as a $C^k$-scheme),
	then $\varGamma_f$ defines an even more general correspondence $f:U\rightarrow V$.
   This covers all the previous cases.
 \end{itemize}

While a noncommutative algebra $A$ itself cannot be made a $C^k$-ring
  and hence there is no way to define the notion of a $C^k$-ring-homomorphism
  $C^k(V)\rightarrow A$,
 commutative subalgebras of $A$ may.
Combined with the above generalizations of the notion of  a `map' between $C^k$-manifolds
   and lessons learned from the case study of D$0$-branes on ${\Bbb R}^n$ in Sec.~3,
 this gives us a hint of what we should do in our case that involves Azumaya algebras.

\bigskip
 
\begin{flushleft}
{\bf A generalization to ring-homomorphisms to Azumaya/matrix algebras}
\end{flushleft}
Let
  $U\subset {\Bbb R}^m$ and $V\subset {\Bbb R}^n$ be open subsets  and
  $E$ be a complex $C^k$ vector bundle of rank $r$ on $U$.
Let
  $E^{\vee}$ be the dual complex vector bundle to $E$ on $U$  and
  $\End_U(E):= E\otimes E^{\vee}$
     be the complex endomorphism bundle of $E$.
Each fiber of $\End_U(E)$ is isomorphic to the Azumaya algebra
 $\End_{\Bbb C}({\Bbb C}^{\oplus r})$
   (i.e.\ the matrix algebra $M_{r\times r}({\Bbb C})$ of rank $r$) over ${\Bbb C}$.
Let $C^k(\End_U(E))$ be the $C^{k}(U)^{\Bbb C}$-algebra
  of  $C^k$-sections of $\End_U(E)$.
It is the function ring of the noncommutative $C^k$-manifold $U^{\!A\!z}$.
	
Let $C^{-{\infty}}(\End_U(E))$ be the ${\Bbb C}$-algebra of sections
  of the endomorphism-bundle $\End_U(E)\rightarrow U$ as a map between sets.
Then, $C^{-{\infty}}(\End_U(E))\supset C^k(\End_U(E))$
  for all $k\in {\Bbb Z}_{\ge 0}\cup\{\infty\}$.
Let
  $$
   \xymatrix{
     C^k(\End_U(E)) &&  C^k(V)\ar[ll]_-{\varphi^{\sharp}}
	}
  $$
  be a ring-homomorphism over ${\Bbb C}\hookleftarrow{\Bbb R}$.
Then $\varphi^{\sharp}$ extends canonically to a ring-homomorphism
  $$
   \xymatrix{
     C^{-{\infty}}(\End_U(E))
	    &&  C^k(U\times V)\ar[ll]_-{\tilde{\varphi}^{\sharp}}  \\	
	   \rule{0ex}{2.6ex}	
	  C^k(\End_U(E))\ar  @{^{(}->}   [u]
	    &&  \rule{0ex}{2.6ex}
		       C^k(V)\ar[ll]_-{\varphi^{\sharp}}\ar @{^{(}->} [u]
	}
  $$
  over ${\Bbb C}\hookleftarrow{\Bbb R}$,
    where both vertical inclusions are tautological,
  as follows:
 \begin{itemize}
  \item[{\Large $\cdot$}]
    Associated to each $f\in C^k(U\times V)$ is the subset
 	$$
	   U^{f;1}\; := \;
	       \{(u, f|_{\{u\}\times V})\,:\, u\in U\}\; \subset\; U\times C^k(V)\,.
    $$
	
  \item[{\Large $\cdot$}]	
   The map
      $\Id_U\times \varphi^{\sharp}:
	       U\times C^k(V)\rightarrow U\times C^k(\End_U(E))$
	sends $U^{f;1}$ to the subset
	$$
	  U^{f;2}\;=\;
	     \{(u,   \varphi^{\sharp}(f|_{\{u\}\times V}))\,:\, u\in U\}\;
     		 \subset\; U\times C^k(\End_U(E))\,.
	$$
	
  \item[{\Large $\cdot$}]	
    Which produces a section of $\End_U(E)\rightarrow U$ as a map between sets:
	 $$
	   s_f\; =\;
	      \{(u,
		        (\varphi^{\sharp}(f|_{\{u\}\times V}))|_{\End_u(E_u)})\,:\, u\in U\}\;
     		 \in\; C^{-\infty}(\End_U(E))\,.
     $$
	
  \item[{\Large $\cdot$}]
   $\tilde{\varphi}^{\sharp}:C^k(U\times V)\rightarrow C^{-\infty}(\End_U(E))\,$
     is now defined by $\,f\mapsto s_f\,$.
 \end{itemize}
By construction,
 $\tilde{\varphi}^{\sharp}$  is a ring-homomorphism
   over ${\Bbb R}\hookrightarrow{\Bbb C}$   and
 it makes the following diagram of ring-homomorphisms commute:
 $$
   \xymatrix{
    \hspace{1ex}
     C^{-\infty}(\End_U(E))
	     &  \hspace{1ex}C^k(\End_U(E))\ar @{_{(}->}[l]   \\
	& &  C^k(U\times V)\ar @/^2ex /  [ull]^(.6){\tilde{\varphi}^{\sharp}}
	   &&& \rule{0ex}{3ex}\hspace{1ex}
	            C^k(V)\ar @/_{2ex}/[ullll]_-{\varphi^{\sharp}}
	                            \ar @{_{(}->}[lll]^{pr_V^{\sharp}}       \; , \\	
    & &  \rule{0ex}{3ex}
	     C^k(U)\ar @{^{(}->}[u]_-{pr_U^{\sharp}}
		                   \ar@/^{2ex}/ @{_{(}->}[uul]
	}
  $$
  where $\pr_U: U\times V\rightarrow U$ and $\pr_V: U\times V\rightarrow V$
    are the projection maps and
  $C^k(U)\hookrightarrow C^k(\End_U(E))$
    is the inclusion of the center of $C^k(\End_U(E))$.

\bigskip

\begin{definition}
 {\bf [canonical extension $\tilde{\varphi}^{\sharp}$ of $\varphi^{\sharp}$].} {\rm
   The ring-homomorphism  $\tilde{\varphi}^{\sharp}$ defined above
     is called the {\it canonical extension} of $\varphi^{\sharp}$.	
}\end{definition}	
	
\bigskip	
	
The notion of a $C^k$-map from an Azumaya point in Sec.\ 3
  motivates and is generalized to the following three key definitions of this note.

\bigskip

\begin{definition}{\bf [admissible homomorphism from $C^k$-ring to Azumaya algebra].}
{\rm
 With the above notation,
   let
     $$
      \xymatrix{
        C^k(\End_U(E)) &&  C^k(V)\ar[ll]_-{\varphi^{\sharp}}
    	}
     $$
    be a ring-homomorphism over ${\Bbb C}\hookleftarrow{\Bbb R}$.
 Then $\varphi^{\sharp}$ is said to be {\it $C^k$-admissible}
  if the following two conditions are satisfied:
   \begin{itemize}
     \item[(1)]
      Its canonical extension
       $$
        \xymatrix{
          C^{-\infty}(\End_U(E))
		  &&  C^k(U\times V)\ar[ll]_-{\tilde{\varphi}^{\sharp}}
         }
      $$
	   has image $\Image\tilde{\varphi}^{\sharp}$ contained in $C^k(\End_U(E))$.

     \item[(2)]
	  Under (1), the quotient map
	   $\tilde{\varphi}^{\sharp}:C^k(U\times Y)
	        \rightarrow \Image\tilde{\varphi}^{\sharp}$
	   realizes $\Image\tilde{\varphi}^{\sharp}$ as a $C^k$-normal quotient of $C^k(U\times Y)$,
	   with the induced $C^k$-ring structure.
	
	 \item[(3)]
      Except the arrow
	    $C^k(\End_U(E))\; \hookleftarrow\; \Image\tilde{\varphi}^{\sharp}$, 	
	  the arrows of the induced diagram under (1)
	  $$
        \xymatrix{
          **[l]C^k(\End_U(E))\; \hookleftarrow\; \Image\tilde{\varphi}^{\sharp}\\
	      &  C^k(U\times V)\ar[ul]_(.45){\tilde{\varphi}^{\sharp}}
	      &&& \rule{0ex}{3ex}\hspace{1ex}
	                C^k(V)\ar @/_{2ex}/[ullll]_-{f_{\varphi}^{\sharp}}
	                                 \ar @{_{(}->}[lll]^{pr_V^{\sharp}}\\	
         &  \rule{0ex}{3ex}
	         C^k(U)\ar @{^{(}->}[u]_-{pr_U^{\sharp}}
		                       \ar@/^{2ex}/ @{_{(}->}[uul]^-{\pi_{\varphi}^{\sharp}}
    	}
      $$
	  are all $C^k$-ring homomorphisms.
	Here,
	  $f_{\varphi}^{\sharp}$ is the same as $\varphi^{\sharp}$
	     but regarded now as a ring-homomorphism to the (commutative) subring
		   $\Image\tilde{\varphi}^{\sharp}$ of $C^k(\End_U(E))$,
	  $C^k(U)$, $C^k(V)$, and $C^k(U\times V)$ are with their built-in $C^k$-ring structure
	   and
	  $\Image\tilde{\varphi}^{\sharp}$ is equipped with the quotient $C^k$-ring structure
        from the ${\Bbb R}$-algebra-epimorphism
       $\tilde{\varphi}^{\sharp}:C^k(U\times V)
	     \longrightaarrow \Image\tilde{\varphi}^{\sharp}$.	
   \end{itemize}
   %
 %-------------------------------------------------------------------------------------
 % As a consequence of Lemma~???,
 %     %
 %		\marginpar{\raggedright\tiny $\bullet$
 %		        {\bf Lemma:}\\    {\it  As a $C^k$-ring,
 %    		      $C^k(U\times V)$ is geberated by $C^k(U)$ and $C^k(V)$.}}
 %		%
 %=====----------------------------------=======------------------------------
   When $\varphi^{\sharp}$ is $C^k$-admissible,
     denote also
	  $$
	    C^k(U)\langle\Image\varphi^{\sharp}\rangle\;
		  :=\;    \mbox{the image $\Image\tilde{\varphi}^{\sharp}$
                      of $\tilde{\varphi}^{\sharp}$ in $C^k(\End_U(E))$}
     $$
  and call it {\it the $C^k$-subalgebra of $C^k(\End_U(E))$
    generated by $\Image\varphi^{\sharp}$ over $C^k(U)$}.
}\end{definition}

\bigskip
  
\begin{definition}
{\bf [local model of maps from Azumaya manifolds in differential topology/geometry].} {\rm
 Continuing the discussion.
 A {\it $C^k$-map}
  $$
     \varphi\; :\; U^{\!A\!z}\;  \longrightarrow\; V
  $$
  is defined contravariantly in terms of function-rings in the following diagram
  $$
   \xymatrix{
    & C^k(\End_U(E))   \\
    &    {\cal A}_{\varphi}\;:=\,\rule{0ex}{3ex}
         		C^k(U)\langle\Image\varphi^{\sharp}\rangle \ar@{^{(}->}[u]
                   				&&& C^k(V)\ar[lllu]_-{\varphi^{\sharp}}
								                           \ar[lll]^-{f_{\varphi}^{\sharp}}\\
    &  C^k(U)\rule{0ex}{3ex}  \ar@{^{(}->}[u]_-{\pi_{\varphi}^{\sharp}}
         &&&&,
   }
  $$
  where
   $\varphi^{\sharp}$ is $C^k$-admissible.
 By definition,
	  ${\cal A}_{\varphi}\;:=\, C^k(U)\langle\Image\varphi^{\sharp}\rangle$
	     is a $C^k$-ring with a built-in $C^k$-ring-epimorphism
		$\tilde{\varphi}^{\sharp}: C^k(U\times V)\longrightaarrow {\cal A}_{\varphi}$.	
} \end{definition}

\bigskip

\begin{definition} {\bf [surrogate of $U^{\!A\!z}$ specified by $\varphi$].} {\rm
  Continuing the discussion.
  the (commutative) $C^k$-scheme
  $U_{\varphi}:=\Spec({\cal A}_{\varphi})$,
  with the built-in $C^k$-maps
  $$
   \xymatrix{
       U^{\!A\!z}\ar[rrrd]^-{\varphi}\ar@{->>}[d]   \\
       U_{\varphi}\ar[rrr]_-{f_{\varphi}}  \ar@{->>}[d]^-{\pi_{\varphi}} &&& V \\
	   U
    }
  $$
  is called the {\it surrogate of} (the noncommutative) $U^{\!A\!z}$
   {\it specified by $\varphi:U^{\!A\!z}\rightarrow V$}.
 In this diagram, only the maps $f_{\varphi}$ and $\pi_{\varphi}$
    have the interpretation as morphisms between ringed spaces in the usual sense;
  the `maps'  $\varphi$ and $U^{\!A\!z}\twoheadrightarrow U_{\varphi}$ are
  defined only through $\varphi^{\sharp}$ and
    ${\cal A}_{\varphi}\hookrightarrow C^k(\End_U(E))$ respectively.
 Caution that in general there is {\it no} $C^k$-map $U\rightarrow V$	
   that makes the diagram commute.
} \end{definition}

\bigskip

\begin{remark}$[$algebraic geometry vs.\ $C^k$-algebraic geometry$\,]$.  {\rm
 Compared with [L-Y1] (D(1)) and [L-L-S-Y] (D(2)), highlighted in Sec.~1,
 this is exactly the same diagram of spaces there, with $X$ replaced by $U$ and $Y$ replaced by $V$,
  {\it except that}:
  \begin{itemize}
   \item[{\Large $\cdot$}]
    $\varphi^{\sharp}: C^k(V)\rightarrow \End_U(E)$
	  cannot be just arbitrary ring-homomorphism
      over ${\Bbb R}\hookrightarrow{\Bbb C}$.
    Rather, $\varphi^{\sharp}$ 	 has to, in addition, be compatible
	 with the $C^k$-ring structures involved in order to be adequately regarded as
	 defining a $C^k$-map $\varphi:U^{\!A\!z}\rightarrow V$.
  \end{itemize}
 This is the main $C^k$-revised building block that is needed to convert
   the essentially algebraic-geometry-oriented discussions/settings
     in [L-Y] (D(1)) and [L-L-S-Y] (D(2))
   to the situations of D-branes that lie in the realm of differential/symplectic topology/geometry.
 Beyond this, all are essentially the same at the formal level only that one has to adjust the various parts
  accordingly to fit in $C^k$-manifold and $C^k$-algebraic geometry language.
 For details, however, as $C^k(U\times V)$ is no longer Noetherian,
  statements in algebraic geometry for the case of Noetherian schemes, e.g., [Ha],
  cannot be copied.
 All has to be re-done, re-checked, and/or re-formatted properly.
}\end{remark}

\bigskip

\begin{remark} $[\,$$k=\infty$$\,]$. {\rm
 For the case $k=\infty$, Condition (2)  in Definition~5.1.2 is redundant.
                   % Definition [admissible homomorphism from $C^k$-ring to Azumay aalgebra]
}\end{remark}

\bigskip

\begin{remark}$[\,$the role of $E$$\,]$. {\rm
 The vector bundle $E$ plays also an important role
   in realizing $\varphi$ as a model for a D-brane in string theory.
 It is the Chan-Paton bundle on a D-brane that encodes nondynamical degrees of freedom
  at an end-point of an open string that is attached to and moving inside that D-brane.
 We'll illuminate the role of $E$ in the next section,
     based on the review of sheaves in $C^k$-algebraic geometry in Sec.~2.2.
}\end{remark}

%------------------------------------------------------------------------------------------------------------------
%
% \bigskip
%
% \begin{remark} $[$for differential geometers with hard-analysis background$\,]$. {\rm
%  Naively, one ??????????.
%  %
%  \marginpar{\raggedright\tiny $\bullet$ Some comments to be added.}
% }\end{remark}
%
% \bigskip
%
% \begin{remark} $[$for string-theorists$\,]$. {\rm
%  ????????????????????
% }\end{remark}
%
%
% \bigskip
% 	
% With this preliminary subsection,
%  we hope the various definitions and constructions that follow
%  make more sense to readers, particularly those more mathematically-oriented string-theorists.
%
%=========---------------------------------------===========----------------------------------------

\bigskip

\subsection{The role of the bundle $E$: The Chan-Paton bundle on a D-brane\\(continuing Sec.\ 2.2)}

When D-branes are stacked together, the enhancement of its Chan-Paton bundle
  and the enhancement of the matrix-type noncommutative structure thereupon go hand in hand.
The former is modeled by the vector bundle $E$ on the manifold $U$ (the world-volume of D-brane)
  and the latter is modeled by the endomorphism bundle $\End_U(E)$,
  whose ring of sections -- with a $C^k$-ring structure on its center --
   give the new noncommutative function-ring on $U$, enhancing it to the noncommutative $U^{\!A\!z}$.
In Sec.~5.1, we have focused on this noncommutative ``space'' $U^{\!A\!z}$
  and the notion of differentiable maps $\varphi$ from $U^{\!A\!z}$ to a manifold $V$.
In this subsection we continue the local study in Sec.~5.1
   but now focus on the complex vector bundle $E$ over $U$,
 from the aspect of modules and sheaves in $C^k$-algebraic geometry.

\bigskip

\begin{flushleft}
{\bf From the aspect of function-rings and modules}
\end{flushleft}
Let $C^k(E)$ be the ${\Bbb C}$-vector space of $C^k$-sections
 of the vector bundle $E$ over $U$.
Then $C^k(E)$ is tautologically and simultaneously
   \begin{itemize}
    \item[$\cdot$]
     a $C^k(U)$-module, denoted by $C^k(E)$ itself,

    \item[$\cdot$]	
     an ${\cal A}_{\varphi}$-module, denoted by $_{{\cal A}_{\varphi}}C^k(E)$,
       and
	
	\item[$\cdot$]
     a $C^k(\End_U(E))$-module,
	  denoted by $_{C^k(\scriptsizeEnd_U(E))}C^k(E)$,
   \end{itemize}
  in a way that is consistent with the built-in inclusions
  $C^k(U)\hookrightarrow {\cal A}_{\varphi}\hookrightarrow C^k(\End_U(E))$;
  that is,
  \begin{itemize}
    \item[(1)]
	 $_{C^k(\scriptsizeEnd_U(E))}C^k(E)$ as an ${\cal A}_{\varphi}$-module
		through ${\cal A}_{\varphi}\hookrightarrow C^k(\End_U(E))$
	  is canonically isomorphic to $_{{\cal A}_{\varphi}}C^k(E)$;
	 in notation,
       $$
                _{{\cal A}_{\varphi}}(_{C^k(\scriptsizeEnd_U(E))}C^k(E))    \;
				=\;  _{{\cal A}_{\varphi}}C^k(E)\,;
       $$	
	
    \item[(2)]	
	 $_{C^k(\scriptsizeEnd_U(E))}C^k(E)$ as a $C^k(U)$-module
		through $C^k(U)\hookrightarrow C^k(\End_U(E))$
	  is canonically isomorphic to $C^k(E)$; in notation,
	   $$
	        _{C^k(U)}(_{C^k(\scriptsizeEnd_U(E))}C^k(E))\;=\; C^k(E)\,;
	   $$
	
	\item[(3)]
     $_{{\cal A}_{\varphi}}C^k(E)$ as a $C^k(U)$-module
	  through $C^k(U)\hookrightarrow {\cal A}_{\varphi}$
	  is canonically isomorphic to $C^k(E)$; in notation,
	   $$
	        _{C^k(U)}(_{{\cal A}_{\varphi}}C^k(E))\; =\;  C^k(E)\,.
	   $$
  \end{itemize}
 Let us  indicate this in the following commutative diagram,
  which also summarizes the construction in Sec.\ 5.1:
  (Here, $\varphi^{\sharp}$ is $C^k$-admissible.)
  $$
   \xymatrix{
    & C^k(E)\\
	&& C^k(\End_U(E))  \ar@{~>}[lu] \\
    &&    {\cal A}_{\varphi}\,:=\,\rule{0ex}{3ex}
         		C^k(U)\langle\Image\varphi^{\sharp}\rangle
				      \ar@{^{(}->}[u]  \ar@{~>}@/^2ex/[luu]
                   				&&& C^k(V)\ar[lllu]_-{\varphi^{\sharp}}
								                           \ar[lll]^-{f_{\varphi}^{\sharp}}
														   \ar@{_{(}->}[d]^-{pr_V^{\sharp}}\\
    &&  \;\;C^k(U)\;\; \rule{0ex}{3ex}  \ar@{^{(}->}[u]_-{\pi_{\varphi}^{\sharp}}
	                                                                 \ar@{~>}@/^5ex/[luuu]
																	 \ar@{^{(}->}[rrr]_-{pr_U^{\sharp}}
         &&&   C^k(U\times V)\ar@{->>}[ulll]^-{\tilde{\varphi}^{\sharp}}&.
   }
  $$
  
Through the above commutative diagram, one observes further that:
  \begin{itemize}
   \item[(4)]{\it
    Through the $C^k$-admissible $\varphi^{\sharp}$, which gives
	 $\tilde{\varphi}^{\sharp}: C^k(U\times V)\twoheadrightarrow {\cal A}_{\varphi}$
	 in the diagram,
	 $C^k(E)$ is also a $C^k(U\times V)$-module,
	 denoted by $_{C^k(U\times V)}C^k(E)$.\footnote{The
	                             $C^k(U\times V)$-module $_{C^k(U\times V)}C^k(E)$								 
								 encodes simultaneously the data of $C^k(E)$
								 and $\varphi$. Thus, mathematically it is the central object to study.}
								 % end-footnote
   Through the inclusion $\pr_U^{\sharp}:C^k(U)\hookrightarrow C^k(U\times V)$,
     one has the canonical exact sequence of $C^k(U\times V)$-modules:}
     $$
	   C^k(E)\otimes_{C^k(U)}C^k(U\times V) \;
	    \longrightarrow\;     _{C^k(U\times V)}C^k(E)\;\longrightarrow\; 0\,.
	 $$								
	  	
   \item[(5)]{\it
    Through $\pr_U^{\sharp}:C^k(U)\hookrightarrow C^k(U\times V)$,
     $_{C^k(U\times V)}C^k(E)$ is realized as a $C^k(U)$-module,
        denoted by $_{C^k(U)}(_{C^k(U\times V)}C^k(E))$, 	
	 which is canonically isomorphic to $C^k(E)$ itself.
	In notation,}
     $$
	   _{C^k(U)}(_{C^k(U\times V)}C^k(E))\;  =\;  C^k(E)\,.
	 $$
	
   \item[(6)]{\it
    Through $\varphi^{\sharp}:C^k(V)\rightarrow C^k(\End_U(E))$,
	  $_{C^k(\scriptsizeEnd_U(E))}C^k(E)$ is realized as a $C^k(V)$-module,
	   denoted by $_{C^k(V)}C^k(E)$.\footnote{The
	                        $C^k(V)$-module $_{C^k(V)}C^k(E)$
                     	   is meant to be the module of $C^k$-sections of a sheaf on $V$
                            that is directly interacting with open strings.
						   Thus, string-theoretically it is the second most important object to study
		                    next to the fundamental module $C^k(E)$ itself.}
	On the other hand,
	  through $f_{\varphi}^{\sharp}:C^k(V)\rightarrow {\cal A}_{\varphi}$,
	    $_{{\cal A}_{\varphi}}C^k(E)$ is realized as another $C^k(V)$-module,
		denoted by $_{C^k(V)}(_{{\cal A}_{\varphi}}C^k(E))$; 	
	    and through $\pr_V^{\sharp}:C^k(V)\hookrightarrow C^k(U\times V)$,
	  $_{C^k(U\times V)}C^k(E)$ is realized as yet another $C^k(V)$-module,
	  denoted by $_{C^k(V)}(_{C^k(U\times V)}C^k(E))$.
     These three $C^k(V)$-modules are 	canonically isomorphic.
     In notation, }
	  $$
	   _{C^k(V)}C^k(E)\;
	     =\; _{C^k(V)}(_{{\cal A}_{\varphi}}C^k(E))\;
	     =\;  _{C^k(V)}(_{C^k(U\times V)}C^k(E))\,.
	  $$
  \end{itemize}
  
These observations serve the basis to understand
 $_{C^k(U\times V)}C^k(E)$ and $_{C^k(V)}C^k(E)\,$:
  
\bigskip

\begin{lemma} {\bf [finite generatedness of modules].}
 Assume that $C^k(E)$ is a finitely generated $C^k(U)$-module. Then,
  \begin{itemize}
   \item[(7)]
     $_{C^k(U\times V)}C^k(E)$ is a finitely generated $C^k(U\times V)$-module.
	
   \item[(8)] 	
    In general, $_{C^k(V)}C^k(E)$ may still not be a finitely generated $C^k(V)$-module.
    However, if in addition ${\cal A}_{\varphi}$ is also finitely generated as a $C^k(V)$-module
	    through $f_{\varphi}^{\sharp}:C^k(V)\rightarrow {\cal A}_{\varphi}$,
    then $_{C^k(V)}C^k(E)$ is a finitely generated $C^k(V)$-module.
  \end{itemize}	
\end{lemma}
 
\begin{proof}
 Property(7) is a consequence of the canonical exact sequence of $C^k(U\times V)$-modules
    $$
	   C^k(E)\otimes_{C^k(U)}C^k(U\times V) \;
	    \longrightarrow\;     _{C^k(U\times V)}C^k(E)\;\longrightarrow\; 0
	$$
   in Property (4).	
 Property (8) is a consequence of the canonical isomorphisms:
    $$
	   _{C^k(V)}C^k(E)\;=\; _{C^k(V)}(_{{\cal A}_{\varphi}}C^k(E))
	$$
   in Property (6).
   
\end{proof}

\bigskip

\begin{flushleft}
{\bf The $C^k(U\times V)$-module $_{C^k(U\times V)}C^k(E)$ made concrete}
\end{flushleft}
The $C^k(U\times V)$-module $_{C^k(U\times V)}C^k(E)$ play an important role
 to our understanding of the $C^k$-map $\varphi$ from the Azumaya $U^{\!A\!z}$
 with a fundamental module $E$ to $V$.
In this theme, we will try to make $_{C^k(U\times V)}C^k(E)$ as concrete as can be
 in a setting that remains general enough for our later study of global objects from gluing local objects.

Let
 $U$ be $C^k$-diffeomorphic to ${\Bbb R}^m$,  with coordinates $x=(x^1,\,\cdots\,,\,x^m)$,
   and
 $V$ be $C^k$-diffeomorphic to ${\Bbb R}^n$,  with coordinates $y=(y^1,\,\cdots\,,\,y^n)$.	
Assume that
 $$
   E\;  \simeq\;  U\times{\Bbb C}^r
 $$
  is a trivialized trivial bundle of rank $r$ over $U$.
Then
 $$
   C^k(\End_U(E))\;
    \simeq\;  \End_{\Bbb C}({\Bbb C}^{\oplus r})\otimes_{\Bbb R}C^k(U)
 $$
 the Azumaya algebra over $C^k(U)$  of rank $r$ and
 $$
   C^k(E)\otimes_{C^k(U)}C^k(U\times V)\;
     \simeq\;  {\Bbb C}^{\oplus r}\otimes_{\Bbb R}C^k(U\times V)
 $$
 the free $C^k(U\times V)^{\Bbb C}$-module of rank $r$.
With the identification of $y_i$ as in $C^k(U\times V)$ through the inclusion
   $\pr_V^{\ast}:C^k(V)\hookrightarrow C^k(U\times V)$,
consider the $C^k$-ideal $I_0$ of $C^k(U\times V)$ with its $C^k$-generators indicated:

\bigskip

\begin{definition}  {\bf [characteristic $C^k$-ideal with respect to coordinate functions on target].} {\rm
 Define the {\it characteristic $C^k$-ideal of $\varphi^{\sharp}$ in $C^k(U\times V)$
    with respect to the coordinate functions $y^1\,,,\,\cdots\,,\,y^n\in C^k(V)$}
  to be
  $$
     I_0\;  :=\;
	    \langle  \determinant( \Id_{r\times r}\otimes y^i - \varphi^{\sharp}(y^i)\otimes 1 )\,
	                            :\, i=1,\,\ldots\,,\,n  \rangle_{C^k}\,.
  $$
}\end{definition}

\bigskip
 
\noindent
Then it follows from the previous discussion and linear algebra that one has the following
  commutative diagram of short exact sequences of $C^k(U\times V)$-modules:
 $$
   \xymatrix{
    \;\;\; 0\;\;  \ar[r]
	  & \;\; M_0\raisebox{-1ex}{\rule{0ex}{2ex}}\;\;  \ar[r]^-{\iota}   \ar @{^{(}->} [d]
      & \;\;{\Bbb C}^{\oplus r}\otimes_{\Bbb R}C^k(U\times V)\;\;	\ar@{=}[d] \ar[r]	
	  & \;\;\Coker(\iota)\;\; \ar@{->>}  [d]    \ar[r]     & \;\; 0\;\;\;    \\
    \;\;\; 0\;\; \ar[r]   & \;\;\Ker(\alpha)\;\; \ar[r]
	  & \;\; C^k(E)\otimes_{C^k(U)}C^k(U\times V)\;\;  \ar[r]^-{\alpha}
	  & \;\; _{C^k(U\times V)}C^k(E)\;\; \ar[r]    & \;\; 0\;\; ,
   }
 $$
 where
 $$
     M_0\;=\;
	    I_0\cdot ({\Bbb C}^{\oplus r}\otimes_{\Bbb R}C^k(U\times V))\;
	 \simeq\;  {\Bbb C}^{\oplus r}\otimes_{\Bbb R}I_0
 $$
 and, hence,
 $$
   \Coker(\iota)\;\simeq \;
     {\Bbb C}^{\oplus r}\otimes_{\Bbb R}\left(  C^k(U\times V)/I_0 \right)\,.
 $$
Note that
  the element $\determinant( \Id_{r\times r}\otimes y^i - \varphi^{\sharp}(y^i)\otimes 1 )$
    in the $C^k$-generating set of $I_0$
  lies in the polynomial ring $C^{k}(U)^{\Bbb C}[y^i]\subset C^k(U\times V)^{\Bbb C}$
   and has degree $r$ in $y^i$.
As is shown in Sec.\ 3, for any $p\in U$,
  all these polynomials in ${\Bbb C}[y^i]$  have only real roots.
With this and after a translation to the language of sheaves in the next theme,
  $_{C^k(U\times V)}C^k(E)$ is geometrically a sheaf $\tilde{\cal E}_{\varphi}$ supported
  within the zero-locus $Z_0\subset U\times V$ described by $I_0$.
As $Z_0$  is finite over $U$, so must be the support $\Supp(\tilde{\cal E}_{\varphi})$
 of $\tilde{\cal E}_{\varphi}$.

\bigskip

\begin{flushleft}
{\bf From the aspect of schemes and sheaves}
\end{flushleft}
All the above discussions/statements have an equivalent and yet  more geometrical description
 in terms of schemes and sheaves,  which we now present.
 
Introduce first the following basic objects:
 \begin{itemize}
   \item[{\Large $\cdot$}]
    ${\cal O}_U$ is the sheaf of germs of $C^k$-functions of $U$;
    denote its complexification ${\cal O}_U\otimes_{\Bbb R}{\Bbb C}$
    	by ${\cal O}_U^{\,\Bbb C}$.
	
   \item[{\Large $\cdot$}]
    ${\cal E}$ is the sheaf of germs of $C^k$-sections
	   of the complex vector bundle $E$ of rank $r$,
	 it is a locally free sheaf of ${\cal O}_U^{\,\Bbb C}$-modules of rank $r$.

   \item[{\Large $\cdot$}]	
    ${\cal O}_U^{A\!z}:=\Endsheaf_{{\cal O}_U}({\cal E})$
	  is the sheaf of endomorphisms of ${\cal E}$ as an ${\cal O}_U^{\,\Bbb C}$-module,
	${\cal O}_U^{A\!z}$ is identical to
	   the sheaf of germs of $C^k$-sections of $\End_U(E)$.
	 It naturally acts on ${\cal E}$,
	    realizing the latter as the fundamental ${\cal O}_U^{A\!z}$-module.
	
   \item[{\Large $\cdot$}]	
    The ringed space $U^{\!A\!z}:= (U,{\cal O}_U^{A\!z})$
	  is a $C^k$-manifold with the sheaf ${\cal O}_U^{A\!z}$ of Azumaya algebras
	   as its structure sheaf.
	 This is a noncommutative space that serves
	   as a local chart for an Azumaya noncommutative manifold.
     Since there is the tautological inclusion
    	 ${\cal O}_U^{\,\Bbb C}\hookrightarrow {\cal O}_U^{A\!z}$,
      which realizes ${\cal O}_U$ as the sheaf of centers of ${\cal O}_U^{A\!z}$,
	 we shall think of an un-spelled-out dominant map
	  $$
	     \pi_U\;:\;  U^{A\!z}\; \longrightarrow\;  U
	  $$
	   if that helps or is conceptually needed. 	
 \end{itemize}

Given a $C^k$-map $\varphi:U^{A\!z}\rightarrow V$ via
   $$
   \xymatrix{
    & C^k(\End_U(E))   \\
    &    {\cal A}_{\varphi}\;:=\,\rule{0ex}{3ex}
         		C^k(U)\langle\Image\varphi^{\sharp}\rangle \ar@{^{(}->}[u]
                   				&&& C^k(V)\ar[lllu]_-{\varphi^{\sharp}}
								                           \ar[lll]^-{f_{\varphi}^{\sharp}}\\
    &  C^k(U)\rule{0ex}{3ex}  \ar@{^{(}->}[u]_-{\pi_{\varphi}^{\sharp}}
         &&&&,
   }
  $$
  as in Definition~5.1.3 in Sec.~5.1,
       % Definition bf [local model of maps from Azumaya manifolds in differential topology/geometry]
 one has the corresponding diagram of maps between spaces:	
  $$
   \xymatrix{
    &  U^{\!A\!z}\ar[rrrd]^-{\varphi}\ar@{->>}[d]^-{\sigma_{\varphi}}   \\
    &  U_{\varphi}\ar[rrr]_-{f_{\varphi}}  \ar@{->>}[d]^-{\pi_{\varphi}} &&& V \\
	&  U   &&&&;
    }
  $$	
 cf.\ Definition~5.1.4.
       % Definition [surrogate of $U^{\!A\!z}$ specified by $\varphi$]
The locally free sheaf ${\cal E}$ as the  fundamental ${\cal O}_U^{A\!z}$-module
  now  fits in to the following extended diagram of maps between spaces:
  $$
   \xymatrix{
    & {\cal E} \ar@{.>}[rd]     \ar@{.>}@/_1ex/[rdd]    \ar@{.>}@/_2ex/[rddd]   \\
    &  & U^{\!A\!z}\ar[rrrd]^-{\varphi}\ar@{->>}[d]^-{\sigma_{\varphi}}   \\
    &  & U_{\varphi}\ar[rrr]_-{f_{\varphi}}
	                                   \ar@{_{(}->} [rrrd]_{\tilde{\varphi}}
                                	   \ar@{->>}[d]^-{\pi_{\varphi}} &&& V \\
	&  & U     &&& U\times V \ar@{->>}[u]_-{pr_V} \ar@{->>}[lll]^-{pr_U}  &.
    }
  $$
Which translates the diagram
  $$
   \xymatrix{
     C^k(E)\\
	& C^k(\End_U(E))  \ar@{~>}[lu] \\
    &    {\cal A}_{\varphi}\,:=\,\rule{0ex}{3ex}
         		C^k(U)\langle\Image\varphi^{\sharp}\rangle
				      \ar@{^{(}->}[u]  \ar@{~>}@/^2ex/[luu]
                   				&&& C^k(V)\ar[lllu]_-{\varphi^{\sharp}}
								                           \ar[lll]^-{f_{\varphi}^{\sharp}}
														   \ar@{_{(}->}[d]^-{pr_V^{\sharp}}\\
    &  \;\;C^k(U)\;\; \rule{0ex}{3ex}  \ar@{^{(}->}[u]_-{\pi_{\varphi}^{\sharp}}
	                                                                 \ar@{~>}@/^5ex/[luuu]
																	 \ar@{^{(}->}[rrr]_-{pr_U^{\sharp}}
         &&&   C^k(U\times V)\ar@{->>}[ulll]^-{\tilde{\varphi}^{\sharp}}
   }
  $$
 in the earlier rings-and-modules discussions.
 
Other parts are translated into the following:\hspace{1em}
Denote ${\cal E}$
  as ${\cal E}$, $_{U_{\varphi}}{\cal E}$,  or $_{U^{\!A\!z}}{\cal E}$,
  depending on whether it is regarded respectively as
    an ${\cal O}_U$-module, an ${\cal O}_{U_{\varphi}}$-module,
	or an ${\cal O}_U^{A\!z}$-module.
 \begin{itemize}
  \item[(1$^{\prime}$)]
    $\;\;\sigma_{\varphi\ast}(_{U^{\!A\!z}}{\cal E})\;
	    =\;   _{U_{\varphi}}{\cal E}\,$
	on $U_{\varphi}$.
    
  \item[(2$^{\prime}$)]
    $\;\;\pi_{U\ast}(_{U^{\!A\!z}}{\cal E})\; =\;  {\cal E}\,$
	on $U$.
  
  \item[(3$^{\prime}$)]
    $\;\; \pi_{\varphi\ast}(_{U_{\varphi}}{\cal E})\; =\; {\cal E}\,$
	on $U$.
  
  \item[(4$^{\prime}$)]
   Denote the sheaf
     $\tilde{\varphi}_{\ast}(_{U_{\varphi}}{\cal E})$ on $U\times V$
	 by $\tilde{\cal E}_{\varphi}$.
   Then there is a canonical exact sequence of ${\cal O}_{U\times V}$-modules:
	$$
	   \pr_U^{\ast}{\cal E}\;\longrightarrow\;
    	   \tilde{\cal E}_{\varphi}\; \longrightarrow \; 0\,.
	$$
  
  \item[(5$^{\prime}$)]
    $\;\; \pr_{U\ast}\tilde{\cal E}_{\varphi}\; =\; {\cal E}\,$
	 on $U$.
  
  \item[(6$^{\prime}$)]
    $\;\; \varphi_{\ast}{\cal E}\;
	   =\; f_{\varphi\ast}(_{U_{\varphi}}{\cal E})\;
	   =\; \pr_{V\ast}\tilde{\cal E}_{\varphi}\,$
	on $V$.
 \end{itemize}
 
\bigskip

\begin{lemma} {\bf [finite generatedness of sheaves].}
 Assume that ${\cal E}$ is a finitely generated ${\cal O}_U$-module. Then,
  \begin{itemize}
   \item[(7$^{\,\prime}$)]
     $\tilde{\cal E}_{\varphi}$ is a finitely generated ${\cal O}_{U\times V}$-module.
	
   \item[(8$^{\,\prime}$)] 	
    In general, $\varphi_{\ast}{\cal E}$
	   may still not be a finitely generated ${\cal O}_V$-module.
    However, if in addition ${\cal O}_{U_{\varphi}}$ is also finitely generated
	      as an ${\cal O}_V$-module
	    through $f_{\varphi}: U_{\varphi}\rightarrow V$,
    then $\varphi_{\ast}{\cal E}$ is a finitely generated ${\cal O}_V$-module.
  \end{itemize}	
\end{lemma}

\bigskip

Properties
(5$^{\prime}$)
     $\,\pi_{U\ast}\tilde{\cal E}_{\varphi}={\cal E}\,$  on $U$  and
(6$^{\prime}$)
     $\,\pi_{V\ast}\tilde{\cal E}_{\varphi}= \varphi_{\ast}{\cal E}\,$ on $V$
  of $\tilde{\cal E}_{\varphi}$ on $U\times V$	 together motivate the following definition:

\bigskip
  
\begin{definition}  {\bf [graph of $\varphi$].} {\rm
  The sheaf  $\tilde{\cal E}_{\varphi}$ on $U\times V$ is called
    the {\it graph of} the $C^k$-map
	$\varphi:(U^{\!A\!z},{\cal E})\rightarrow V\,$.
}\end{definition}

\bigskip
  
\begin{remark} $[\,$taming/regularization/renormalization of push-forwards$\,]$. {\rm
   The push-forward $\varphi_{\ast}{\cal E}$ may not be even finitely generated in general.
   It has to be ``tamed/regularized/renormalized".
   For the $C^k$-map we are most interested in,
     there exists an open dense subset $V$ of $\Image\varphi\subset Y$,
	 over which $\varphi$ is of finite type.
   In this case, $\varphi_{\ast}{\cal E}$ can be regularized.
 Cf.~ Sec.~6.2.
}\end{remark}

\bigskip

\begin{flushleft}
{\bf Example: Matrix strings and D1-branes in ${\Bbb R}^2$}
\end{flushleft}
Some key points of and geometric visions behind this subsection is illustrated below.

\bigskip
 
\begin{example}
{\bf [matrix strings in ${\Bbb R}^2$ via maps from an Azumaya line
           $({\Bbb R}^{1,A\!z}, {\cal E})$].}
{\rm
 Let
  \begin{itemize}
    \item[{\Large $\cdot$}]
     $E\simeq {\Bbb R}^1\times {\Bbb C^3}$ be a trivialized complex vector bundle of rank $3$
       over the real line ${\Bbb R}^1$  (in coordinate $x$)  as a smooth manifold
    	 with the structure sheaf ${\cal O}_{{\Bbb R}^1}$
	       the sheaf of smooth functions on ${\Bbb R}^1$,
		
	\item[{\Large $\cdot$}]	
     ${\cal E}$ be the associated sheaf of smooth sections of $E$,
	 with ${\cal E}\simeq {\cal L}_1\oplus {\cal L}_2\oplus{\cal L}_3$
	   associated to the trivialization of $E\simeq L_1\oplus L_2\oplus L_3$, and
	
	\item[{\Large $\cdot$}]
     ${\Bbb R}^{1,A\!z}
	      =({\Bbb R}^1,
              {\cal O}_{{\Bbb R}^1}^{A\!z}:= \Endsheaf_{{\cal O}_X}({\cal E}))$
	  be the Azumaya noncommutative line of rank $3$,
			
	\item[{\Large $\cdot$}]
     ${\Bbb R}^2$ be the real plane, in coordinate $(y_1,y_2)$.	
  \end{itemize}	
 Then
  \begin{itemize}
   \item[{\Large $\cdot$}]
     $\End_{{\Bbb R}^1}(E)\simeq {\Bbb R}^1\times\End_{\Bbb C}({\Bbb C}^3)$
      a bundle of complex Azumaya algebras over ${\Bbb R}^1$,
	
   \item[{\Large $\cdot$}]	
     $C^{\infty}(\End_{{\Bbb R}^1}(E))
        \simeq \End_{\Bbb C}({\Bbb C}^3)\otimes_{\Bbb R}C^{\infty}({\Bbb R}^1)$
	an Azumaya algebra over $C^{\infty}({\Bbb R}^1)$, 	
        and
		
   \item[{\Large $\cdot$}]		
     any ring-homomorphism
       $\varphi^{\sharp}:C^{\infty}({\Bbb R}^2)
	     \rightarrow C^{\infty}(\End_{{\Bbb R}^1}(E))$
	    over ${\Bbb R}\hookrightarrow{\Bbb C}$ is $C^{\infty}$-admissible   and, hence,
       defines a smooth map $\varphi:({\Bbb R}^{1,A\!z},{\cal E})\rightarrow {\Bbb R}^2$.
  \end{itemize}	
	
 Consider in particular the following $C^{\infty}$-admissible
   ring-homomorphisms $\varphi^{\sharp}$,
     with  its value at $y_1$, $y_2\in C^{\infty}({\Bbb R}^2)$ indicated.
 Each reveals a different type of image-string structure in ${\Bbb R}^2$
   that can occur.\footnote{String-theorists
                              are highly recommend to first jump to
							     {\sc Figure}~5-2-1 (a), {\sc Figure}~5-2-1 (b),
								 {\sc Figure}~5-2-1 (c), {\sc Figure}~5-2-1 (d), and
								 {\sc Figure}~7-2-1,
							  and ask yourself three questions:
						      {\it Aren't these features/phenomena of a matrix string or a stacked D-string
  							     (either itself or under deformation)?};
							  at this beginning level of seeking a definition for the notion of
							     `maps from a matrix string/brane to a spae-time', 	                                 							  
						      {\it Is something over-done here?}, and {\it Is something under-done here?},
                              before following the mathematical explanations.}
							  % end-footnote
							
 \bigskip

 \noindent{\bf (5.2.6.a) [three simple strings, and their self-or-not crossing or overlapping].}
 Let
    $$
	  \xymatrix @R=1ex{
	    \;\;C^{\infty}(\End_{{\Bbb R}^1}(E))\;\;
		   &&& \;\;C^{\infty}({\Bbb R}^2)\ar[lll]_{\varphi^{\sharp}}\;\;\\
	    \;\;\mbox{$ \left[
	         \begin{array}{ccc}
			   f_1(x)             & 0                           & 0                                \\
			   0                           & f_2(x)             & 0                                \\
			   0                           &  0                          & f_3(x)			
			 \end{array}
	                         \right] $}\;\;
		   &&& \hspace{1em}y_1 \hspace{1em}\ar @{{|}->}  [lll] \\
	    \;\;\mbox{$ \left[
	         \begin{array}{ccc}
			   g_1(x)            & 0                            & 0                                \\
			   0                            & g_2(x)            & 0                                \\
			   0                            &  0                           & g_3(x)			
			 \end{array}
	                         \right] $}\;\;
		   &&& \hspace{1em}y_2 \hspace{1em}\ar @{{|}->} [lll]      	
	  }
    $$
     with $f_1(x)$, $f_2(x)$, $f_3(x)$,
	         $g_1(x)$, $g_2(x)$, $g_3(x)\in C^{\infty}({\Bbb R}^1)$.	
  Then,  $\varphi^{\sharp}$ encodes three independent $C^{\infty}$-ring homomorphisms
    $$
        \xymatrix @R=1ex{
	   & \;\;C^{\infty}(\End_{{\Bbb R}^1 }(L_i))\;\;
		   &&& \;\;C^{\infty}({\Bbb R}^2)\ar[lll]_{\gamma^{\sharp}_i}\;\;\\
	   &\;\; f_i(x)\;\;
		   &&& \hspace{1em}y_1 \hspace{1em}\ar @{{|}->}  [lll] \\
	   & \;\; g_i(x)\;\;
		   &&& \hspace{1em}y_2 \hspace{1em}\ar @{{|}->} [lll]      &,	
	  }
    $$	
    for $i=1,\,2,\,3$.
  Each $\gamma_i^{\sharp}$ defines a smooth map
    $\gamma_i:{\Bbb R}^1\rightarrow {\Bbb R}^2$,
	   which can be expressed in the more traditional way in differential topology and geometry
	   as $x\mapsto (f_i(x), g_i(x))$.	
  They may cross or overlap with each other  and
   $$
     \varphi_{\ast}{\cal E}\;
	  \simeq\;  \gamma_{1\ast}{\cal L}_1
	                     \oplus \gamma_{2\ast}{\cal L}_2 \oplus\gamma_{3\ast}{\cal L}_3.
   $$
  In this case, $\Supp(\varphi_{\ast}{\cal E})$ is reduced.
 Thus, while fibers of the sheaf $\varphi_{\ast}{\cal E}$ over the image string
    remain to get enhanced from ${\Bbb C}$ to ${\Bbb C}^2$
	  (with a decomposition data ${\Bbb C}^2\simeq {\Bbb C}\oplus {\Bbb C}$ ),	
	or ${\Bbb C}^3$ (with a decomposition data
     	${\Bbb C}^3\simeq{\Bbb C}\oplus{\Bbb C}\oplus{\Bbb C}$),
	or even higher ${\Bbb C}^l$
	   (with a decomposition data ${\Bbb C}^l\simeq\oplus_l{\Bbb C}$
            if no open-stringy effects are taken into account to mingle them)
	depending on whether it's a double crossing/overlapping or a triple crossing/overlapping
	or an even higher-order crossing/overlapping,
  there is no nilpotent cloud along the image strings of ${\Bbb R}^{1,A\!z}$ in ${\Bbb R}^2$.
  
  Cf.\ {\sc Figure}~5-2-1 (a) below and {\sc Figure}~7-2-1 in Sec.~7.2.

 \begin{figure} [htbp]
  \bigskip
  \centering
  \includegraphics[width=0.80\textwidth]{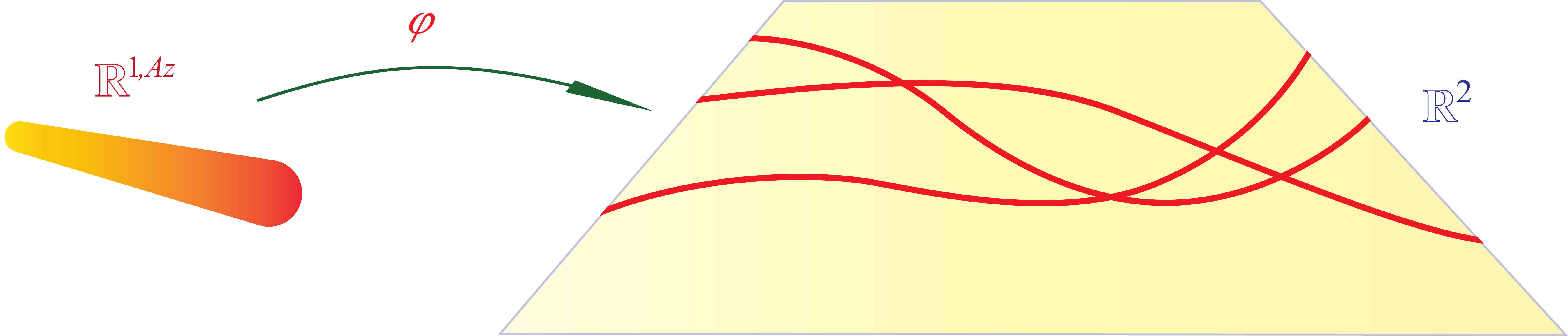}
 
  \bigskip
  \bigskip
   \centerline{\parbox{13cm}{\small\baselineskip 12pt
    {\sc Figure}~5-2-1 (a).
      A smooth map $\varphi:{\Bbb R}^{1,A\!z}\rightarrow {\Bbb R}^2$ in Case (5.2.6.a).
        }}
   \bigskip
  \end{figure}	
    
 \bigskip
 
 \noindent
 {\bf (5.2.6.b) [simple string + string with nilpotent cloud of order $1$,
                                  and their self-or-not crossing or overlapping].}
  Let
    $$
	  \xymatrix @R=1ex{
	    \;\;C^{\infty}(\End_{{\Bbb R}^1}(E))\;\;
		   &&& \;\;C^{\infty}({\Bbb R}^2)\ar[lll]_{\varphi^{\sharp}}\;\;\\
	    \;\;\mbox{$ \left[
	         \begin{array}{ccc}
			   f_1(x)  &   0                         & 0                                \\
			   0                 &   f_2(x)          & 0                                \\
			   0                 &  1                          & f_2(x)			
			 \end{array}
	                         \right] $}\;\;
		   &&& \hspace{1em}y_1 \hspace{1em}\ar @{{|}->}  [lll] \\
	    \;\;\mbox{$ \left[
	         \begin{array}{ccc}
			   g_1(x)  & 0                   & 0                                \\
			   0                  & g_2(x)   & 0                                \\
			   0                  & 0                   & g_2(x)			
			 \end{array}
	                         \right] $}\;\;
		   &&& \hspace{1em}y_2 \hspace{1em}\ar @{{|}->} [lll]
	  }
    $$
	with $f_1(x)$, $f_2(x)$, $g_1(x)$, $g_2(x)\in C^{\infty}({\Bbb R}^1)$.
 Then, $\varphi^{\sharp}$ encodes two independent $C^{\infty}$-ring homomorphisms:
    $$
        \xymatrix @R=1ex{
	      \;\;C^{\infty}(\End_{{\Bbb R}^1}(L_1))\;\;
		   &&& \;\;C^{\infty}({\Bbb R}^2)\ar[lll]_{\gamma_1^{\sharp}}\;\;\\
	      \;\;  f_1(x)\;\;
		   &&& \hspace{1em}y_1 \hspace{1em}\ar @{{|}->}  [lll] \\
	      \;\; g_1(x)\;\;
		   &&& \hspace{1em}y_2 \hspace{1em}\ar @{{|}->} [lll]
	  }
    $$		
   and  	
    $$
	  \xymatrix @R=1ex{
	    & \;\;C^{\infty}(\End_{{\Bbb R}^1}(L_2\oplus L_3))\;\;
		   &&& \;\;C^{\infty}({\Bbb R}^2)\ar[lll]_{\gamma_2^{\sharp}}\;\;\\
	    & \;\;\mbox{$ \left[
	         \begin{array}{cc}
			    f_2(x)           & 0                                \\
			    1                          & f_2(x)			
			 \end{array}
	                         \right] $}\;\;
		   &&& \hspace{1em}y_1 \hspace{1em}\ar @{{|}->}  [lll]     \\
	    & \;\;\mbox{$ \left[
	         \begin{array}{cc}
			   g_2(x)   & 0                                \\
			   0                   & g_2(x)			
			 \end{array}
	                         \right] $}\;\;
		   &&& \hspace{1em}y_2 \hspace{1em}\ar @{{|}->} [lll]        &.
	  }
    $$			
 $\gamma_1^{\sharp}$ defines a smooth map
    $\gamma_1:{\Bbb R}^1\rightarrow {\Bbb R}^2$,
	   which can be expressed in the more traditional way in differential topology and geometry
	   as $x\mapsto (f_1(x), g_1(x))$,
 while $\gamma_i^{\sharp}$ defines a smooth map
    $\gamma_2:
	   ({\Bbb R}^1,
	       {\cal O}_{{\Bbb R}^1}^{A\!z}
		        :=\Endsheaf_{{\cal O}_{{\Bbb R}^1}}({\cal L}_2\oplus{\cal L}_3),
	        {\cal L}_2\oplus {\cal L}_3  )
	     \rightarrow {\Bbb R}^2$.
  It is now a smooth curve ${\Bbb R}^1\rightarrow{\Bbb R}^2$
      expressed in the more traditional way in differential topology and geometry
	   as $x\mapsto (f_2(x), g_2(x))$ {\it together with}
	   a nilpotent cloud,  of order $1$ when no crossing nor overlapping occurs, along it.	 	
  	
  Note that as a $C^{\infty}({\Bbb R}^2)$-module via $\varphi^{\sharp}$,
    the decomposition  $E\simeq L_1\oplus L_2\oplus L_3$ is not preserved,
   a weaker decomposition $E\simeq L_1\oplus E_2$
      (where $E_2=L_2\oplus L_3$  and
     	  whose associated sheaf of sections is denoted by ${\cal E}_2$)
     with a filtration $L_3\subset E_2$ is preserved.
 Such a decomposition-filtration combination passes thus to the push-forward
   $$
      \varphi_{\ast}{\cal E}\;
	    \simeq \; \gamma_{1\ast}{\cal L}_1
		                 \oplus \gamma_{2\ast}{\cal E}_2
      \hspace{2em}
       \mbox{with a filtration
	     $\;\gamma_{2\ast}{\cal L}_3\subset \gamma_{2\ast}{\cal E}_2$}\,.
   $$ 	
  When the two string cross or overlap with themselves or each other,
    this decomposition-filtration of $\varphi_{\ast}{\cal E}$ are respected
	(if no open-stringy effects are taken into account to mingle them).
 
  Cf.\ {\sc Figure}~5-2-1 (b) below and {\sc Figure}~7-2-1 in Sec.~7.2.	
	
 \begin{figure} [htbp]
  \bigskip
  \centering
  \includegraphics[width=0.80\textwidth]{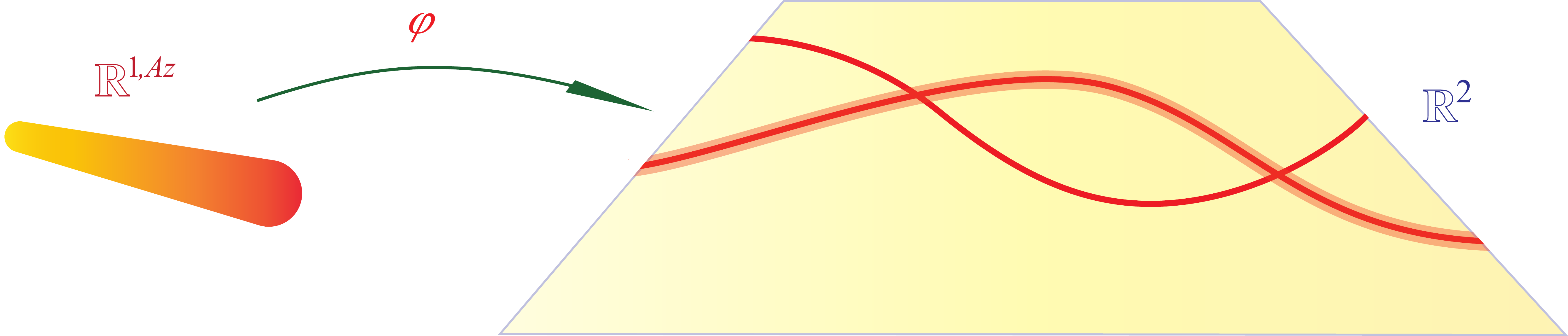}
 
  \bigskip
  \bigskip
   \centerline{\parbox{13cm}{\small\baselineskip 12pt
    {\sc Figure}~5-2-1 (b).
    A smooth map $\varphi:{\Bbb R}^{1,A\!z}\rightarrow {\Bbb R}^2$ in Case (5.2.6.b).
        }}
   \bigskip
  \end{figure}	
  
 \bigskip

 \noindent
 {\bf (5.2.6.c) [string with nilpotent cloud of order $2$, and its self-crossing/overlapping].}
  Let
    $$
	  \xymatrix @R=1ex{
	   \;\;C^{\infty}(\End_{{\Bbb R}^1}(E))\;\;
		   &&& \;\;C^{\infty}({\Bbb R}^2)\ar[lll]_{\varphi^{\sharp}}\;\;\\
	   \;\;\mbox{$ \left[
	         \begin{array}{ccc}
			   f(x)     &  0             & 0                      \\
			   1              &  f(x)    & 0                      \\
			   0              & 1              & f(x)			
			 \end{array}
	                         \right] $}\;\;
		   &&& \hspace{1em}y_1 \hspace{1em}\ar @{{|}->}  [lll] \\
	   \;\;\mbox{$ \left[
	         \begin{array}{ccc}
			   g(x)    & 0            & 0                     \\
			   0              & g(x)  & 0                     \\
			   0              &  0           & g(x)			
			 \end{array}
	                         \right] $}\;\;
		   &&& \hspace{1em}y_2 \hspace{1em}\ar @{{|}->} [lll]
	  }
    $$
 with $f(x)$ and $g(x)\in C^{\infty}({\Bbb R}^1)$.	
 This describes now a smooth curve ${\Bbb R}^1\rightarrow {\Bbb R}^2$,
    expressed in the more traditional way in differential topology and geometry
	   as $x\mapsto (f(x), g(x))$ {\it together with}
	   a nilpotent cloud,  of order $2$ when no crossing nor overlapping occurs, along it.	 	
  
 Similar to Case (5.2.6.b),
  as a $C^{\infty}({\Bbb R}^2)$-module via $\varphi^{\sharp}$,
    the decomposition  $E\simeq L_1\oplus L_2\oplus L_3$ is not preserved;
  only the filtration	
    $$
	   0\; \subset \; L_3\;  \subset\;  E_2 \;(=\,L_2\oplus L_3)\; \subset\; E
	$$
    is preserved.
 The filtration passes then to the push-forward
   $$
       0\; \subset\; \varphi_{\ast}{\cal L}_3\;
	         \subset\; \varphi_{\ast}{\cal E}_2 \;\subset\;   \varphi_{\ast}{\cal E}\,.	
   $$ 	
  When the string crosses or overlaps with itself,
    this filtration of $\varphi_{\ast}{\cal E}$ remains to be respected
	(if no open-stringy effects are taken into account to mingle them).
 
 Cf.\ {\sc Figure}~5-2-1 (c) below and {\sc Figure}~7-2-1 in Sec.~7.2.
 
 \begin{figure} [htbp]
  \bigskip
  \centering
  \includegraphics[width=0.80\textwidth]{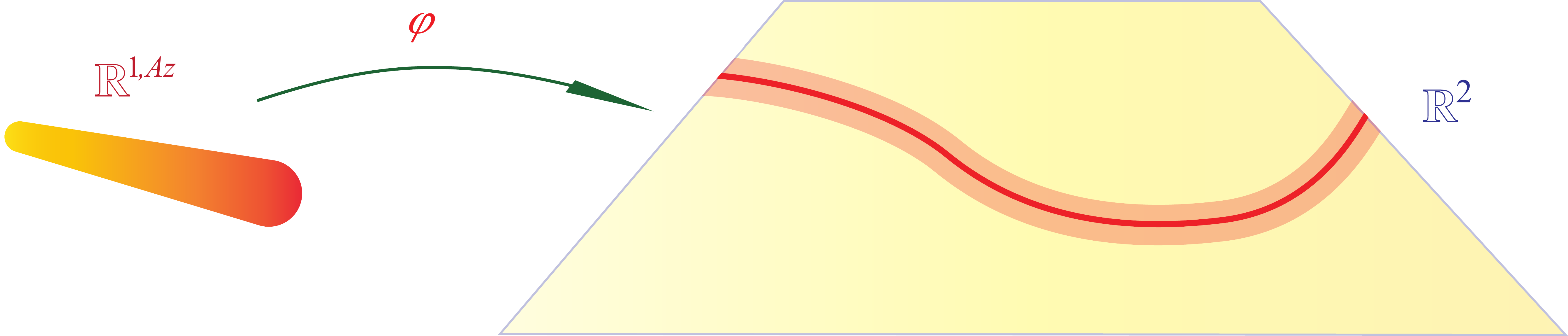}
 
  \bigskip
  \bigskip
   \centerline{\parbox{13cm}{\small\baselineskip 12pt
    {\sc Figure}~5-2-1 (c).
    A smooth map $\varphi:{\Bbb R}^{1,A\!z}\rightarrow {\Bbb R}^2$ in Case (5.2.6.c).
        }}
   \bigskip
  \end{figure}	

  \bigskip
 
  \noindent
 {\bf (5.2.6.d) [general map].}
 For a general smooth map
   $\varphi:({\Bbb R}^{1,A\!z},{\cal E})\rightarrow {\Bbb R}^2$,
 its image $\Image\varphi$  and the push-forward $\varphi_{\ast}{\cal E}$
   has their geometry a mixture of the types shown above.
 When restricted to one interval $I_1^{A\!z}\subset {\Bbb R}^{1,A\!z}$,
  they may look like the situation in, e.g., Case (5.2.6.a),
   while when restricted to another interval $I_2^{A\!z}\subset {\Bbb R}^{1,A\!z}$,
   they may look like the situation in, say, Case (5.2.6.c); and so on.
 Cf.\ {\sc Figure}~5-2-1 (d) below and {\sc Figure}~7-2-1 in Sec.~7.2.
    
 \begin{figure} [htbp]
  \bigskip
  \centering
  \includegraphics[width=0.80\textwidth]{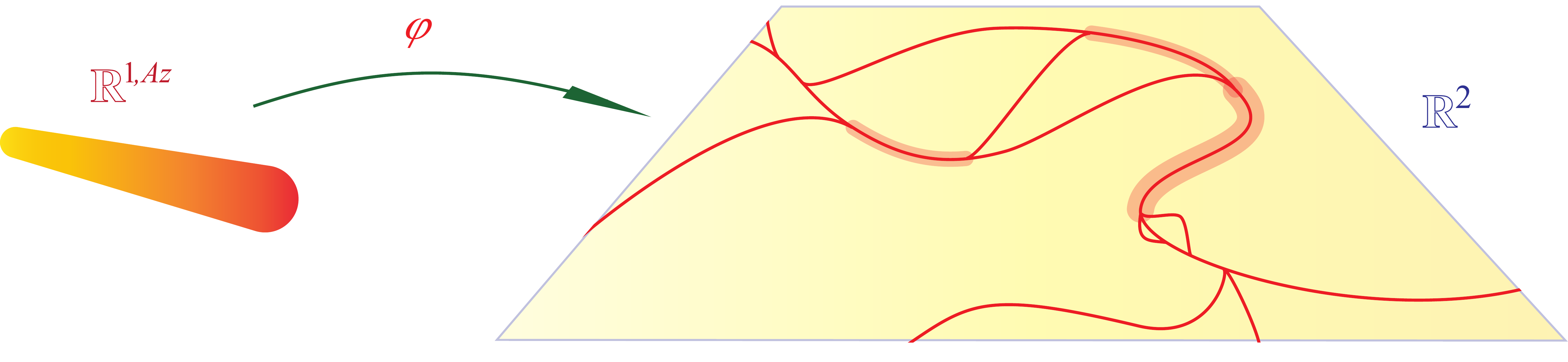}
 
  \bigskip
  \bigskip
   \centerline{\parbox{13cm}{\small\baselineskip 12pt
    {\sc Figure}~5-2-1 (d).
   A general smooth map $\varphi:{\Bbb R}^{1,A\!z}\rightarrow {\Bbb R}^2$.
        }}
   \bigskip
  \end{figure}	
   
\noindent\hspace{15.7cm}$\square$
}\end{example}

\bigskip
 
Before moving on, it deserves to recap a little bit to orient ourselves.
Following Ansatz~1.1 in Sec.~1,
             % Ansatz [D-brane: Azumaya/matrix-type noncommutativity]
D-branes moving in a space-time are described
  by maps from a matrix manifold with a fundamental vector bundle to that space-time.
Sec.~2, Sec.~3, Sec.~4, Sec.~5.1, and Sec.~5.2 together
 provide
   the basis for a mathematical description of such noncommutative manifolds
 and continuous or differentiable maps from them to the space-time.
This fulfills only half of the work to give a prototypical mathematical definition of D-branes in string theory.
The other half we need to address is
 \begin{itemize}
   \item[{\bf Q.}] {\it
     Does our setting reproduce basic D-brane phenomena in string theory?}
 \end{itemize}
Only when we pass the test of this question can we say that we've provided a prototypical mathematical
 definition of D-branes from the aspect of Azumaya noncommutative differential geometry
 with input from algebraic geometry.
We'll come back to this in Sec.~7.2, through Example~7.2.2,
                             % Example [deformation of special Lagrangian branes with a flat bundle
							 %                  in the Calabi-Yau 1-fold ${\Bbb C}^1$]
 after devoting ourselves to some more study of Azumaya manifolds
 in the next section.

\bigskip
		
\subsection{Differentiable maps from an Azumaya manifold with a fundamental\\ module to a real manifold}

In Sec.~3,  it is affirmed that there is no need to add additional points to a real $C^k$-manifold
 to make geometrically compatible sense of a ring-homomorphism of the function ring of the manifold
 to Azumaya algebras.
In Sec.~5.1, a local study is given
  for the notion of $C^k$-maps from an Azumaya manifold with a fundamental module to a $C^k$-manifold.
With these two preparations,
 the notion of differentiable maps from an Azumaya manifold with a fundamental module
  to a $C^k$-manifold can be defined as in the case for algebraic schemes
  in [L-Y1] (D(1)) and [L-L-S-Y] (D(2)).
Similarly to the algebraic case, there are four aspects of such a notion.
The details are given in this subsection.

\bigskip

\subsubsection{Aspect I [fundamental]: Maps as gluing systems of ring-homomorphisms}

The notion of a differentiable map
 $$
   \varphi\;:\; (X,{\cal O}_X^{\!A\!z}:=\Endsheaf_{{\cal O}_X}({\cal E}),{\cal E})\;
     \longrightarrow\;  Y
 $$
 from an Azumaya manifold with a fundamental module to a real manifold 	
 follows from the notion of `morphisms between spaces' studied in
 [L-Y1: Sec.\ 1.2 A noncommutative space as a gluing system of rings] (D(1))
 by adapting the latter to fit in the standard definition of manifolds and fibre/vector bundles
 from gluing local charts and local trivializations respectively (e.g.\ [G-H], [Hus], [K-N], [Ste]).
We give the details below.

\bigskip

\begin{flushleft}
{\bf  The fundamental aspect of $C^k$-maps from Azumaya manifolds}
\end{flushleft}
The following lemma reflects the ``softness, flexibility, and abundanceness" of differentiable functions;
it follows from a partition-of-unity type argument:

\bigskip

\begin{sslemma} {\bf [extension to global differentiable function after shrinking].}
  Let $V_3\subset V_2\subset V_1$ be open subsets of a $C^k$-manifold $M$ such that
   $\overline{V_3}\subset V_2$ and $\overline{V_2}\subset V_1$.
  Then for any $f_2\in C^k(V_2)$, there exists an $f_1\in C^k(V_1)$
     such that $f_2|_{V_3}=f_1|_{V_3}$.
\end{sslemma}

\bigskip

\begin{sslemma} {\bf [determinacy of local from global].}
 Let
   $U_2\subset U_1$ be open subsets of $X$,
   $V_1$ be an open subset of $Y$,    and
   $\phi^{\sharp}_1:   C^k(V_1)
       \rightarrow C^k(\Endsheaf_{{\cal O}_{U_1}}({\cal E}_{U_1}))$
    be a $C^k$-admissible ring-homomorphism over ${\Bbb R}\hookrightarrow{\Bbb C}$.		
 Then
   $\phi^{\sharp}_1$ determines a $C^k$-admissible ring-homomorphism
    $\phi^{\sharp\, \prime}_1: C^k(V_1)
	    \rightarrow  C^k(\Endsheaf_{{\cal O}_{U_2}}({\cal E}_{U_2}))$.  	
 Note that $\phi^{\sharp\, \prime}_1$ renders ${\cal E}_{U_2}$
   a ${\cal O}_{V_1}$-module,
   denoted by $\phi_{1\ast}^{\prime}({\cal E}_{U_2})$.
 Suppose that $V_2$ is an open subset of $V_1$  that contains
    $\Supp(\phi_{1\ast}^{\prime}({\cal E}_{U_2}))$.
 Then $\phi^{\sharp\, \prime}_1$ determines in turn a $C^k$-admissible ring-homomorphism
  $\phi^{\sharp}_2:   C^k(V_2)
       \rightarrow C^k(\Endsheaf_{{\cal O}_{U_2}}({\cal E}_{U_2}))$.
\end{sslemma}

\begin{proof}
 Lemma~5.3.1.1 implies
            % Lemma [extension to global differentiable function after shrinking]
  that the $C^k(V_1)$- and the $C^k(V_2)$-actions on the stalk/germs of the sheaf
   $\phi_{1\ast}^{\prime}({\cal E}_{U_2})$
   is completely determined  by $\phi^{\sharp}_1$.
 The lemma follows.
\end{proof}

\bigskip

\begin{ssdefinition}{\bf [(contravariant) gluing system of ring-homomorphisms].} {\rm
 A {\it (contravariant) gluing system of ring-homomorphisms}
      over ${\Bbb R}\hookrightarrow{\Bbb C}$ related to $(X^{\!A\!z},Y)$
  consists of the following data:
  \begin{itemize}
    \item[{\Large $\cdot$}]
	 ({\it local charts on $X^{\!A\!z}$})\hspace{1.1em}
     an open cover ${\cal U}=\{U_{\alpha}\}_{\alpha\in A}$ on $X$,
	
	\item[{\Large $\cdot$}]
	 ({\it local charts on $Y$})\hspace{2em}
     an open cover ${\cal V}=\{V_{\beta}\}_{\beta\in B}$ on $Y$,
	
	\item[{\Large $\cdot$}]
     a gluing system $\Phi^{\sharp}$ of $C^k$-admissible ring-homomorphisms
       from $\{C^k(V_{\beta})\}_{\beta}$\\
       to $\{C^k(\Endsheaf_{{\cal O}_{U_{\alpha}}}
	                                               ({\cal E}_{U_{\alpha}}))\}_{\alpha}$
       over ${\Bbb R}\hookrightarrow {\Bbb C}$, which consists of
	    \begin{itemize}
	      \item[{\Large $\cdot$}]
	      ({\it specification of a target-chart
		               for each local chart on $X^{\!A\!z}$})\hspace{2em}
	       a map $\sigma:A\rightarrow B$,
	
	      \item[{\Large $\cdot$}]
         ({\it  differentiable map from charts on $X^{\!A\!z}$ to charts on $Y$})\\
	       a $C^k$-admissible ring-homomorphism over ${\Bbb R}\hookrightarrow {\Bbb C}$
	       $$
	        \phi^{\sharp}_{\alpha, \sigma(\alpha)}\;:\;
   	       C^k(V_{\sigma(\alpha)}) \;\longrightarrow\;
		   C^k(\Endsheaf_{{\cal O}_U}({\cal E}_{U_{\alpha}}))
	       $$	
	       for each $\alpha\in A$
		\end{itemize}
	   that satisfy (cf.\ Lemma~5.3.1.2)
	                            % Lemma [determinacy of local from global]
		 \begin{itemize}
      	  \item[{\Large $\cdot$}]
		  ({\it gluing/identification of maps at overlapped charts on $X^{\!A\!z}$})\\
		   for each pair $(\alpha_1$, $\alpha_2)\in A\times A$,
		   \begin{itemize}
		     \item[(G1)]
		      $(\phi_{\alpha, \sigma(\alpha_1)})_{\ast}
		         ({\cal E}_{U_{\alpha_1}\cap\, U_{\alpha_2}})$
			   is completely supported in
			   $V_{\sigma(\alpha_1)}\cap V_{\sigma(\alpha_2)}
			      \subset V_{\sigma(\alpha_1)}$,
						
             \item[(G2)]		
			 recall the $C^k$-admissible ring-homomorphism over ${\Bbb R}\hookrightarrow {\Bbb C}$	
	          $$
		        \phi_{\alpha_1\alpha_2,\, \sigma(\alpha_1)\sigma(\alpha_2)}\;:\;
				   C^k(V_{\sigma(\alpha_1)}\cap V_{\sigma(\alpha_2)})\;
				     \longrightarrow\;
				   C^k(\Endsheaf_{{\cal O}_{U_{\alpha_1}\cap\, U_{\alpha_2}}}
				              ({\cal E}_{U_{\alpha_1}\cap\,U_{\alpha_2}}))	 					
		      $$		
			  induced by $\phi_{\alpha_1,\sigma(\alpha_1)}$,
			then 			
	            $$
		          \phi_{\alpha_1\alpha_2,\,\sigma(\alpha_1)\sigma(\alpha_2)}\;
				   =\; \phi_{\alpha_2\alpha_1,\,\sigma(\alpha_2)\sigma(\alpha_1)}\,.
		        $$
		   \end{itemize}
         \end{itemize}
     \end{itemize}		
} \end{ssdefinition}	
	
\bigskip

\begin{ssdefinition}{\bf [equivalent systems].} {\rm
 A  gluing system $({\cal U}^{\prime},{\cal V}^{\prime},  \Phi^{\prime\sharp})$
   is said to be a
 {\it refinement} of another gluing system $({\cal U},{\cal V},  \Phi^{\sharp})$,
    in notation
	  $({\cal U}^{\prime},{\cal V}^{\prime},  \Phi^{\prime\sharp})
            \preccurlyeq({\cal U},{\cal V},  \Phi^{\sharp})$,	
  if
     \begin{itemize}	
	  \item[{\Large $\cdot$}]
	   ${\cal U}^{\prime}
	       = \{U^{\prime}_{\alpha^{\prime}}\}_{\alpha^{\prime}\in A^{\prime}}$
		is a refinement of  ${\cal U}= \{U_{\alpha}\}_{\alpha\in A}$,
		with a map $\tau:A^{\prime}\rightarrow A$ that labels inclusions
		  $U^{\prime}_{\alpha^{\prime}}\hookrightarrow U_{\tau(\alpha^{\prime})}$; 		
       similarly,		
	   ${\cal V}^{\prime}
	       = \{V^{\prime}_{\beta^{\prime}}\}_{\beta^{\prime}\in B^{\prime}}$
		is a refinement of ${\cal V}= \{V_{\beta}\}_{\beta\in B}$,
		 with a map $\upsilon:B^{\prime}\rightarrow B$  that labels inclusions
		   $V^{\prime}_{\beta^{\prime}}
		        \hookrightarrow V_{\upsilon(\beta^{\prime})}$;
        the maps between the index sets $A$, $B$, $A^{\prime}$, and $B^{\prime}$ satisfy
		  the commutative diagram
		  $$
            \xymatrix{
			 &  A^{\prime} \ar[rr]^-{\sigma^{\prime}} \ar[d]_-{\tau}
			     && B^{\prime}\ar[d]^-{\upsilon} \\
		     &  A \ar[rr]^-{\sigma}  && B  &. 			
			}
          $$		
							
      \item[{\Large $\cdot$}]		(cf.\ Lemma~5.3.1.2)
	                                                              % Lemma  [determinacy of local from global]
	    the $C^k$-admissible ring-homomorphism
		  $$
		    \phi^{\prime\sharp}_{\alpha^{\prime},\, \sigma^{\prime}(\alpha^{\prime})}\; :\;
		     C^k(V_{\sigma^{\prime}(\alpha^{\prime})})\;  \longrightarrow\;
			  C^k(\Endsheaf_{{\cal O}_{U_{\alpha^{\prime}}}}
			           ({\cal E}_{U_{\alpha^{\prime}}})  )
		  $$
			in $\Phi^{\prime\sharp}$
         coincides with the $C^k$-admissible ring-homomorphism
		   $$
    		  C^k(V_{\sigma^{\prime}(\alpha^{\prime})}) \;  \longrightarrow\;
			  C^k(\Endsheaf_{{\cal O}_{U_{\alpha^{\prime}}}}
			           ({\cal E}_{U_{\alpha^{\prime}}})  )
		   $$
		    induced by
	       $$
		     \phi^{\sharp}
		       _{\tau(\alpha^{\prime}),\,\sigma(\tau(\alpha^{\prime}))}  \,=\,
		      \phi^{\sharp}
		       _{\tau(\alpha^{\prime}),\, \upsilon(\sigma^{\prime}(\alpha^{\prime}))}\; :\;
		     C^k(V_{\upsilon(\sigma^{\prime}(\alpha^{\prime}))}) \; \longrightarrow\;
			  C^k(\Endsheaf_{{\cal O}_{U_{\tau(\alpha^{\prime}) }}}
			           ({\cal E}_{U_{\tau(\alpha^{\prime}) }})  )
		   $$
			in $\Phi$ from the inclusions
			 $U_{\alpha^{\prime}}\hookrightarrow U_{\tau(\alpha^{\prime})}$		 and
			 $V_{\sigma^{\prime}(\alpha^{\prime})} \hookrightarrow
			     V_{\upsilon(\sigma^{\prime}(\alpha^{\prime}))  }$.
     \end{itemize}
  Two gluing systems
     $({\cal U}_1,{\cal V}_1,  \Phi_1^{\sharp})$
	  and $({\cal U}_2,{\cal V}_2,  \Phi_2^{\sharp})$	
     are said to be {\it equivalent} if they have a common refinement. 	
}\end{ssdefinition}
  
\bigskip

\begin{ssdefinition}{\bf [differentiable map as equivalence class of gluing systems].}
{\rm
 We denote such an equivalence class compactly as
   $$
     \varphi^{\sharp}\;:\; {\cal O}_Y\; \longrightarrow \;
	   {\cal O}_X^{A\!z}\,:=\, \Endsheaf_{{\cal O}_X}({\cal E})\,.
   $$
  This defines a {\it differentiable map}
   $$
     \varphi\;:\; (X,{\cal O}_X^{\!A\!z}:=\Endsheaf_{{\cal O}_X}({\cal E}),{\cal E})\;
     \longrightarrow\;  Y\,.
   $$
}\end{ssdefinition}

\bigskip

The $C^k$-admissible ring-homomorphism
   $$
	  \phi^{\sharp}_{\alpha, \sigma(\alpha)}\;:\;
   	    C^k(V_{\sigma(\alpha)}) \;\longrightarrow\;
		C^k(\Endsheaf_{{\cal O}_U}({\cal E}_{U_{\alpha}}))
   $$
   renders $C^k({\cal E}_{U_{\alpha}})$  a $C^k(V_{\sigma(\alpha)})$-module.
Passing to germs of $C^k$-sections,
 this defines a  sheaf of ${\cal O}_{V_{\sigma(\alpha)}}$-modules,
 denoted by $(\phi_{\alpha,\,\sigma(\alpha)})_{\ast}({\cal E}_{U_{\alpha}})$.
The following lemma/definition follows by construction:

\bigskip

\begin{sslemma-definition}
 {\bf [push-forward $\varphi_{\ast}({\cal E})$ under $\varphi$].} {\rm
  {\it The collection of sheaves on charts
    $\{\phi_{\alpha,\,\sigma(\alpha)}({\cal E}_{U_{\alpha}})\}_{\alpha\in A}$
	 glue to a sheaf of ${\cal O}_Y$-modules on $Y$.
   It is independent of the contravariant gluing system of
      $C^k$-admissible ring-homomorphisms over ${\Bbb R}\hookrightarrow {\Bbb C}$
     that represents $\varphi$.}
    It is called the {\it push-forward of ${\cal E}$ under $\varphi$} and
	  is denoted by $\varphi_{\ast}({\cal E})$.
}\end{sslemma-definition}

\bigskip

\begin{sslemma-definition}
 {\bf [surrogate of $X^{\!A\!z}$ specified by $\varphi$].} {\rm
   (Cf.~Definition~5.1.4.)
         % Definition [surrogate of $U^{\!A!z}$ specified by $\varphi$]
   {\it The collection of local surrogates $U_{\varphi_{\alpha,\sigma(\alpha)}}$
              of $U_{\alpha}^{\!A\!z}$ specified by $\varphi_{\alpha,\sigma(\alpha)}$
	  glue to a $C^k$-scheme over $X$.
   It is independent of the contravariant gluing system of
      $C^k$-admissible ring-homomorphisms over ${\Bbb R}\hookrightarrow {\Bbb C}$
     that represents $\varphi$.}
    It is called the {\it surrogate of $X^{\!A\!z}$ specified by $\varphi$}
	 and is denoted by $X_{\varphi}$.
}\end{sslemma-definition}

\bigskip

\begin{ssremark} {$[$concerning Definition~5.3.1.5$\,]$.}
                                           % Definition  [differentiable map as equivalence class of gluing systems]
{\rm
 (1) [{\it fundamental}$\,$] \hspace{1em}
   This definition is fundamental in the sense that from the Grothendieck's view of geometry from local to global
    and in terms of function rings of local charts,
  this is exactly what a ``morphism"  between two ``spaces" should mean conceptually, in disregard of the contexts
  in the algebraic or the differentiable geometry or of commutative or noncommutative geometry, only that technically
  its real contents depend on the notion of localizations of a ring, which can be subtle in the noncommutative case.
  Cf.\ E.g.\ [L-Y1: References [Ga], [Goldm], [Jat], [Or1], [Or2], [Sten]] (D(1)).
  
 (2) [{\it Condition (G1) in Definition~5.3.1.3}$\,$] \hspace{1em}
                                % Definition [(contravariant) gluing system of ring-homomorphisms]
   This is the key condition that renders our definition of differentiable maps from an Azumaya manfold
     to a real manifold in Definition~5.3.1.5
	                               % Definition  [differentiable map as equivalence class of gluing systems]
   essentially the same as that for ordinary manifolds in differential topology,
   only that rings of $C^k$-functions are used contravariantly
    instead of the direct covariant point-sets with $C^k$-topology.
 The reason that one can add this condition to our formulation without losing anything is as follows:
   \begin{itemize}
     \item[$\cdot$]
	  A good notion of map/morphism $\varphi$ from an Azumaya manifold with a fundamental module
	   $(X^{\!A\!z},{\cal E})$ to a target-space $Y$
	   should make the notion of push-forward $\varphi_{\ast}({\cal E})$ immediate
	  as the latter is what an open-string moving in $Y$  would ``see"/interact with.
      Combining with the interpretation of $(X^{\!A\!z},{\cal E})$	as a smearing of Azumaya points along $X$
	   and the lesson learned in Sec.~3 in the  case of D0-branes,
	   one anticipates $\varphi_{\ast}({\cal E})$ to come from an $X$-family of $0$-dimensional
	   ${\cal O}_Y$-modules even before one has the detail of how $\varphi$ should be defined.
	  But when $\varphi_{\ast}({\cal E})$ does exist, Condition (G1) must be always arrangeable.	
	  Thus, it's a {\it non-restriction} to the notion of `maps' we have intended to define for the context of D-branes
	    in string theory that lies in the realm of differential/symplectic topology/geometry.	
   \end{itemize}
  Cf.\  Aspect II (Sec.~5.3.2) and Aspect III (Sec.~5.3.3)
   of the notion of differentiable maps in our context for D-branes.
%------------------------------------------------------------------------------------------------------------------------
%  See Remark~???? in Sec.~????.
%                  % Remark [Condition (G2) in Definition~???]
%======-------------------------------------=======-----------------------------------------------------------
    
 (3) [{\it Keep function rings; abandon topologies}$\,$] \hspace{1em}
  As in the algebro-geometric situation in [L-Y1]  (D(1)),
    despite the fact that $X^{\!A\!z}:=(X,{\cal O}_X^{A\!z})$ can be thought of as a ringed-space
     (i.e.\ a topological space $X$ with a structure sheaf of local (noncommutative) function rings),
   our notion of a differentiable map $\varphi: X^{\!A\!z}\rightarrow  Y$ is {\it not} that of ringed-spaces
     as usually defined.
  In particular, {\it in general $\varphi$ does not come from nor induce any map from $X$ to $Y$}.	
  As emphasized in [L-Y1] (D(1)), mathematically one is free to make her/his preferable choice.
  However, to correctly describe D-branes in string theory, one has to {\it abandon
    the data of a map between the underlying topologies} $X\rightarrow Y$
	but {\it just keep the contravariant data of ring-homomorphisms of local function/coordinate rings}.
 The case study of D0-branes in Sec.~3, though simple,  should already make this clear.	
} \end{ssremark}

\bigskip

\begin{flushleft}
{\bf The equivalent affine setting}
\end{flushleft}
For a $C^k$-manifold $Y$, unlike in the case of projective schemes,
the existence of  a partition of unity subordinate to any open covering of $Y$
 implies that $C^k(Y)$ not only separates points on $Y$, it actually generates any germ of $C^k$-functions
 on any open subset of $Y$.
In other words, $Y$ is an affine manifold in the sense of $C^k$-algebraic geometry.
Similarly,  $X^{A\!z}$ can be regarded as an affine Azumaya manifold.
It follows that Definition~5.3.1.5 can be pruned to the following equivalent form:
                   % Definition  [differentiable map as equivalence class of gluing systems]
 
\bigskip

\begin{sslemma-definition}
{\bf [$C^k$-map in affine setting].}
 The equivalence class
  $$
    \varphi^{\sharp}\; :\;  {\cal O}_Y\;
	    \longrightarrow\;   {\cal O}_X^{A\!z}:=  \Endsheaf_{{\cal O}_X}({\cal E})
  $$
   of gluing systems of
      $C^k$-admissible ring-homomorphisms over ${\Bbb R}\hookrightarrow{\Bbb C}$
   in Definition~5.3.1.5 defines a $C^k$-admissible ring-homomorphism, still denoted by $\varphi^{\sharp}$,
   % Definition  [differentiable map as equivalence class of gluing systems]
   $$
     \varphi^{\sharp}\;:\; C^k(Y)\; \longrightarrow\;
	      C^k(X^{\!A\!z}):= C^k(\Endsheaf_{{\cal O}_X}({\cal E}))\,
		                                         =\,  C^k(\End_X(E))
   $$
   over ${\Bbb R}\hookrightarrow{\Bbb C}$.
  Conversely, any $C^k$-admissible ring-homomorphism
      $\varphi^{\sharp}: C^k(Y)\rightarrow  C^k(X^{\!A\!z})$
	     over ${\Bbb R}\rightarrow {\Bbb C}$
	 defines an equivalence class
	 $\varphi^{\sharp}: {\cal O}_Y \rightarrow  {\cal O}_X^{A\!z}$.
  It follows that the notion of a differentiable map
	$$
     \varphi\;:\; (X,{\cal O}_X^{A\!z}:=\Endsheaf_{{\cal O}_X}({\cal E}),{\cal E})\;
     \longrightarrow\;  Y\,.
    $$
     in Definition~5.3.1.5
	 % Definition [differentiable map as equivalence class of gluing systems]
     can be equivalently defined by a $C^k$-admissible ring-homomorphism
      $\varphi^{\sharp}: C^k(Y)\rightarrow  C^k(X^{\!A\!z})$
     over ${\Bbb R}\hookrightarrow{\Bbb C}$.	
\end{sslemma-definition}

\begin{proof}
 A given $C^k$-admissible $\varphi^{\sharp}:{\cal O}_Y \rightarrow {\cal O}_X^{A\!z}$
   induces a $C^k$-admissible ring-homomorphism
     $\varphi^{\sharp}:{\cal O}_Y(Y)\rightarrow {\cal O}_X^{A\!z}(X)$
	  over ${\Bbb R}\hookrightarrow {\Bbb C}$,
	which is exactly a $C^k$-admissible ring-homomorphism
     $\varphi^{\sharp}:C^k(Y)\rightarrow  C^k(X^{\!A\!z})$
	  over ${\Bbb R}\hookrightarrow {\Bbb C}$.
 Conversely, a given $C^k$-admissible ring-homomorphism
     $\varphi^{\sharp}:C^k(Y)\rightarrow  C^k(X^{\!A\!z})$
	  over ${\Bbb R}\hookrightarrow {\Bbb C}$
	 renders $C^k({\cal E})$ a $C^k(Y)$-module.
  By passing to germs, this defines a sheaf ${\cal F}$ of ${\cal O}_Y$-modules on $Y$,
   which can be used to recover a gluing system
   $\varphi^{\sharp}:{\cal O}_Y \rightarrow {\cal O}_X^{A\!z}$
   of $C^k$-admissible ring-homomorphisms over ${\Bbb R}\hookrightarrow{\Bbb C}$.
 The corresponding differentiable map $\varphi:(X^{\!A\!z},{\cal E})\rightarrow Y$
   renders ${\cal F}$ the push-forward $\varphi_{\ast}({\cal E})$.
     
\end{proof}

%----------------------------------------------------------------
% \bigskip
%
% \begin{flushleft}
% {\bf Generalization to locally-finitely-presentable singular spaces}
% \end{flushleft}
%
% \bigskip
%
%  \noindent $\bullet$
%  ??????????????????.
%
% \bigskip
%===----------------------=====----------------------------------------------

\bigskip

\subsubsection{Aspect II: The graph of a differentiable map}

Similar to the studies [L-L-S-Y: Sec.\ 2.2] (D(2)) and [L-Y5:  Sec.\ 2.2]  in the algebro-geometric setting,
the graph of a differentiable map
 $\varphi:(X,{\cal O}_X^{A\!z}:=\Endsheaf_{{\cal O}_X}({\cal E}),{\cal E})
     \rightarrow Y$
  is a sheaf $\tilde{\cal E}_{\varphi}$ of ${\cal O}_{X\times Y}^{\,\Bbb C}$-modules
  on the $C^k$-manifold $X\times Y$ with special properties.	
And $\varphi$ can be recovered from its graph.
We explain the details below in the current differential-topology setting.

\bigskip

\begin{flushleft}
{\bf Graph of a differentiable map $\varphi:(X^{\!A\!z},{\cal E})\rightarrow Y$}
\end{flushleft}
It follows from Sec.~5.1 that
 an equivalence class
  $$
    \varphi^{\sharp}\; :\;  {\cal O}_Y\;
	    \longrightarrow\;   {\cal O}_X^{A\!z}:=  \Endsheaf_{{\cal O}_X}({\cal E})
  $$
   of gluing systems of
      $C^k$-admissible ring-homomorphisms over ${\Bbb R}\hookrightarrow{\Bbb C}$
 extends canonically to an equivalence class
  $$
    \tilde{\varphi}^{\sharp}\; :\;  {\cal O}_{X\times Y}\;
	    \longrightarrow\;   {\cal O}_X^{A\!z}:=  \Endsheaf_{{\cal O}_X}({\cal E})
  $$
   of gluing systems of
      $C^k$-admissible ring-homomorphisms over ${\Bbb R}\hookrightarrow{\Bbb C}$.
In other words,
 a $C^k$-map
  $$
    \varphi\; :\;
    (X,{\cal O}_X^{A\!z}:=\Endsheaf_{{\cal O}_X^{\,\Bbb C}}({\cal E}),{\cal E})\;
      \longrightarrow\;  Y
  $$
 lifts canonically to a $C^k$-map
  $$
    \tilde{\varphi}\; :\;
    (X,{\cal O}_X^{A\!z}:=\Endsheaf_{{\cal O}_X^{\,\Bbb C}}({\cal E}),{\cal E})\;
      \longrightarrow\;  X\times Y\,,
  $$
 making the following diagram commute:
  $$
   \xymatrix{
    & (X,{\cal O}_X^{A\!z}:=\Endsheaf_{{\cal O}_X^{\,\Bbb C}}({\cal E}),{\cal E})
        \ar[rr]^-{\tilde{\varphi}} \ar[rrd]_-{\varphi}
	    && X\times Y \ar[d]^-{pr_Y} \\
	& && Y  &.
	}
  $$
Here $\pr_Y:X\times Y\rightarrow Y$ is the projection map to $Y$.

\bigskip

\begin{ssdefinition} {\bf [graph of $\varphi$].}  {\rm
 The {\it graph} of $\varphi$ is
  a sheaf $\tilde{\cal E}_{\varphi}$ of ${\cal O}_{X\times Y}^{\,\Bbb C}$-modules,
  defined by
  $$
    \tilde{\cal E}_{\varphi}\;:=\; \tilde{\varphi}_{\ast}({\cal E})\,.
  $$
}\end{ssdefinition}

\bigskip

The following basic properties of $\tilde{\cal E}_{\varphi}$ follow directly
  from the local study in Sec.~5.2:
 
\bigskip

\begin{sslemma} {\bf [basic property of graph of $\varphi$].}
 Continuing the above notation and
  let $\pr_X:X\times Y\rightarrow X$ be the projection map to $X$.
 The graph $\tilde{\cal E}_{\varphi}$ of $\varphi$ has the following properties:
 \begin{itemize}
  \item[$(1)$]
    There is a canonical isomorphism of ${\cal O}_X^{\,\Bbb C}$-modules:
     $$
       {\cal E}\; \longrightarrow\;  {\pr_X}_{\ast}(\tilde{\cal E}_{\varphi})\,.
     $$
	
  \item[$(2)$]	
   $\tilde{\cal E}_{\varphi}$ is of relative-dimension $0$ over $X$.
   
  \item[$(3)$]
   Its scheme-theoretical-like  support $\Supp(\tilde{\cal E}_{\varphi})$
    is a sheaf of ${\cal O}_{X\times Y}$-algebras that in general can have nilpotent elements.
 \end{itemize}
\end{sslemma}

\bigskip

It follows from either the local study in Sec.~5.2 or Property (1) above
 that there is a canonical homomorphism
 $$
   \pr_X^{\ast}({\cal E})\; \longrightarrow\; \tilde{\cal E}_{\varphi}
 $$
  of ${\cal O}_{X\times Y}$-modules.
By construction, this is surjective and its kernel can be read off readily:
 
\bigskip

\begin{sslemma} {\bf [presentation of graph of $\varphi$].}
 Continuing the notation.
 The graph $\tilde{\cal E}_{\varphi}$ of $\varphi$
   admits a presentation given by a natural isomorphism
  $$
    \tilde{\cal E}_{\varphi}\;
     \simeq\;
    	 \pr_X^{\ast}({\cal E})
                 \left/  \left(
				      (\pr_Y^{\ast}(f)-\pr_X^{\ast}(\varphi^{\sharp}(f))\,:\, f\in C^k(Y))
					       \cdot \pr_X^{\ast}(\cal E)				
				           \right) \right.\,.
  $$
 Here $\pr_X:X\times Y\rightarrow X$ and $\pr_Y:X\times Y\rightarrow Y$ are the projection maps.
\end{sslemma}

\bigskip

\begin{ssremark}
{$[$presentation of $\tilde{\cal E}_{\varphi}$ in local trivialization of ${\cal E}$$\,]$.} {\rm
Note that
 with respect to a local trivialization ${\Bbb C}^{\oplus r}\otimes_{\Bbb R}C^k(U)$
   of ${\cal E}$ and,
   hence, a local trivialization ${\Bbb C}^{\oplus r}\otimes_{\Bbb R}C^k(U\times Y)$ of
   $\pr_X^{\ast}({\cal E})$ over $X$,
 the subsheaf
 $(\pr_Y^{\ast}(f)-\pr_X^{\ast}(\varphi^{\sharp}(f))\,:\, f\in C^k(Y))
					       \cdot \pr_X^{\ast}(\cal E)$
  in the above lemma is generated (as an ${\cal O}_{X\times Y}^{\,\Bbb C}$-module)
                   % Lemma [presentation of graph of $\varphi$]
  by elements of the form
  $$
     v\otimes f \; -\; (\varphi^{\sharp}(f)(v))\otimes 1\,, \hspace{2em}f\in C^k(Y)\,.
  $$
}\end{ssremark}

\bigskip

\begin{flushleft}
{\bf Recovering $\varphi:(X^{\!A\!z},{\cal E})\rightarrow Y$
          from an ${\cal O}_{X\times Y}^{\,\Bbb C}$-module}
\end{flushleft}
Conversely, let
 $X$ and $Y$ be $C^k$-manifolds   and
 $\tilde{\cal E}$ be a sheaf of  ${\cal O}_{X\times Y}^{\,\Bbb C}$-modules \
  with the following properties:
   \begin{itemize}
    \item[]
	  \begin{itemize}
        \item[(M1)]
		 The annihilator ideal sheaf
		   $\Ker({\cal O}_{X\times Y}\rightarrow
		                          \Endsheaf_{{\cal O}_{X\times Y}}(\tilde{\cal E}))$
		   is $C^k$-normal in ${\cal O}_{X\times Y}$;
          thus, $\Supp(\tilde{\cal E})$ is a $C^k$-subscheme of the $C^k$-manifold $X\times Y$.
		
 	    \item[(M2)]
         The push-forward ${\cal E}:={\pr_X}_{\ast}(\tilde{\cal E})$
	       is a locally free $C^k$ ${\cal O}_X^{\,\Bbb C}$-module,
	      say, of rank $r$.		   		
	 %
     %  \item[(M2)]
     %    The natural homomorphisms of ${\cal O}_{X\times Y}^{\,\Bbb C}$-modules
     %    $$
     %        \pr_X^{\ast}({\cal E})\; \longrightarrow\; \tilde{\cal E}
     %     $$		
	 %    is surjective with kernel of the form $\,{\cal I}\cdot \pr_X^{\ast}{\cal E}\,$,
	 %    where ${\cal I}$ is an ideal sheaf of ${\cal O}_{X\times Y}$,
	 %    locally generated by $C^k$-functions on the related open sets of $X\times Y$.
	 %   Note that the composition of the natural homomorphisms of ${\cal O}_X^{\,\Bbb C}$-modules
	 %    $$
	 %	    {\cal E}\;\longrightarrow\; {\pr_X}_{\ast}(\pr_X^{\ast}({\cal E}))\;
	 %	       \longrightarrow\;   {\pr_X}_{\ast}(\tilde{\cal E})\,=:\,{\cal E}
	 %	 $$
     %    is the identity homomorphism.
	 \end{itemize}	
   \end{itemize}
Consider the noncommutative space
 $(X,{\cal O}_X^{A\!z}:=\Endsheaf_{{\cal O}_X^{\Bbb C}}({\cal E}))$.
Then
 $$
    (X,{\cal O}_X^{A\!z}:=\Endsheaf_{{\cal O}_X^{\Bbb C}}({\cal E}),{\cal E})
 $$
   is a $C^k$ Azumaya manifold with a fundamental module  and
 $\tilde{\cal E}$ defines an equivalence class
  $$
     \varphi^{\sharp}:{\cal O}_Y\;    \longrightarrow\;
          \Endsheaf_{{\cal O}_X^{\,\Bbb C}}({\cal E})=:{\cal O}_X^{A\!z}
  $$
  of gluing systems of $C^k$-admissible ring-homomorphisms over ${\Bbb R}\hookrightarrow{\Bbb C}$
  as follows:
  \begin{itemize}
      \item[(1)]
       Let $U\subset X$ be an open set that lies in an atlas of $X$ such that
   	   $\pr_Y(\Supp(\tilde{\cal E}_U))$ is contained
	      in an open set $V\subset Y$ that lies in an atlas of $Y$.
	    Here, we treat $\tilde{\cal E}$ also as a sheaf over $X$    and
	      $\tilde{\cal E}_U:=\tilde{\cal E}|_{U\times Y}$ is the restriction of $\tilde{\cal E}$
             to over $U$.		
		
      \item[(2)]
        Let $f\in C^k(V)$.
	    Then the multiplication by $\pr_Y^{\ast}(f)\in C^k(U\times V)$
	       induces an endomorphism
		   $\tilde{\alpha}_f:\tilde{\cal E}_U\rightarrow \tilde{\cal E}_U$
		     as an ${\cal O}_{U\times V}^{\,\Bbb C}$-module.
        Since $U\times V\supset \Supp(\tilde{\cal E}_U)$,
          $\alpha_f:={\pr_X}_{\ast}(\tilde{\alpha}_f)$ defines in turn
	      a $C^k$-endomorphism of the ${\cal O}_U^{\,\Bbb C}$-module ${\cal E}_U$;
	      i.e.\ $\alpha_f\in{\cal O}_X^{A\!z}(U)$.
        This defines a ring-homomorphism
	      $\varphi^{\sharp}:C^k(V)\rightarrow {\cal O}_X^{A\!z}(U)$
	      over ${\Bbb R}\hookrightarrow{\Bbb C}$,
	      with $f\mapsto \alpha_f$.
        By construction, $\varphi^{\sharp}$ is $C^k$-admissible.		
		
      \item[(3)]
	   Compatibility of the system of $C^k$-admissible ring-homomorphisms
  	    $\varphi^{\sharp}:C^k(V)\rightarrow {\cal O}_X^{A\!z}(U)$
	     over ${\Bbb R}\hookrightarrow{\Bbb C}$
        with gluings follows directly from the construction.		
  \end{itemize}
In this way, $\tilde{\cal E}$  defines a $C^k$-map
  $\varphi:(X^{\!A\!z},{\cal E})\rightarrow Y$.
  
By construction, the graph $\tilde{\cal E}_{\varphi}$ of the $C^k$-map $\varphi$
 associated to $\tilde{\cal E}$  is canonically isomorphic to $\tilde{\cal E}$.	
This gives an equivalence of the two notions/categories:
 $$
  \begin{array}{c} \\[-1.2ex]
    \fbox{$C^k$-maps $\varphi:(X^{\!A\!z},{\cal E})\rightarrow Y$  }\;
	   \Longleftrightarrow\;
    \fbox{${\cal O}_{X\times Y}^{\,\Bbb C}$-modules $\tilde{\cal E}$
	              that satisfy (M1) and (M2)}	   \\[2ex]
  \end{array}	
 $$
({\sc Figure}~5-3-2-1.)
\begin{figure} [htbp]
 \bigskip
 \centering
 \includegraphics[width=0.80\textwidth]{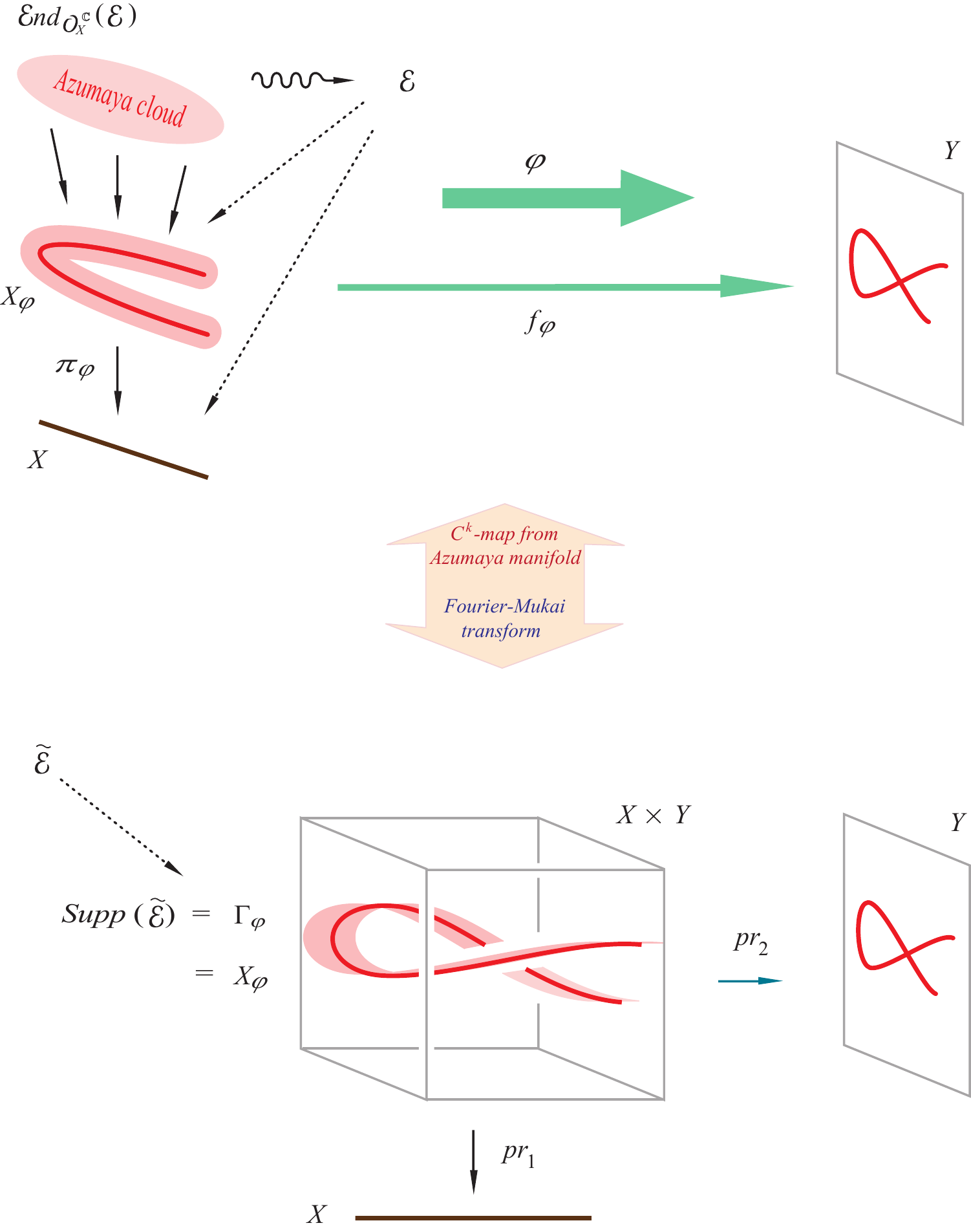}
 
 \vspace{4em}
 \centerline{\parbox{13cm}{\small\baselineskip 12pt
  {\sc Figure}~5-3-2-1.
    The equivalence between
      a $C^k$-map $\varphi$ from an Azumaya $C^k$-manifold with a fundamental module
	    $(X,{\cal O}_X^{Az}:=\Endsheaf_{{\cal O}_X^{\,\Bbb C}}({\cal E}), 
		          {\cal E})$
	    to a $C^k$-manifold $Y$    and
	  a special kind of Fourier-Mukai transform 
	    $\tilde{\cal E}\in \ModCategory^{\Bbb C}(X\times Y)$ from $X$ to $Y$. 		
   Here, $\ModCategory^{\Bbb C}(X\times Y)$
        is the category of ${\cal O}_{X\times Y}^{\,\Bbb C}$-modules. 		
      }}
 \bigskip
\end{figure}

\bigskip

\subsubsection{Aspect III: From maps to the stack of D0-branes}

Aspect II of a $C^k$-map $\varphi:(X^{\!A\!z},{\cal E})\rightarrow Y$
   discussed in Sec.~5.3.2
 brings out a third aspect of $\varphi$, which we now explain.

\bigskip
 
\begin{flushleft}
{\bf An Azumaya manifold  $X^{\!A\!z}$ as a smearing of unfixed Azumaya points over $X$.}
\end{flushleft}
{To} illuminate our point, consider first a case that appears often in complex analysis:

\bigskip

\begin{ssexample} {\bf [real contour $\gamma$ in complex line ${\Bbb C}^1$].} {\rm
 A  differentiable contour in the complex line ${\Bbb C}^1$ with complex coordiate $z=x+\sqrt{-1}y$
   is a differentiable map
   $$
     \gamma\; =\; \gamma_x + \sqrt{-1}\gamma_y\; :\; [0,1]\; \longrightarrow\;  {\Bbb C}^1\,.
   $$
 There is no issue about this, if $\gamma$ is treated as  a map between point-sets with a manifold structure:
    from the interval $[0,1]$ to the underlying real $2$-space ${\Bbb R}^2$ of ${\Bbb C}^1$
	with coordinates $(x,y)$.
 However,  in terms of function rings, some care needs to be taken.
 While there is a built-in ring-homomorphism ${\Bbb R}\hookrightarrow{\Bbb C}$ over ${\Bbb R}$,
   there exists no ring-homomorphism ${\Bbb C}\rightarrow{\Bbb R}$ with $0\mapsto 0$ and $1\mapsto 1$.
 If follows that there is no ring-homomorphism
   $\gamma^{\sharp}: C^k({\Bbb C}^1)^{\Bbb C}\rightarrow C^k([0,1])$,
   where $C^k({\Bbb C}^1)^{\Bbb C}$ is the algebra of complex-valued $C^k$-functions on ${\Bbb C}^1$.
 {To} remedy this, one should first complexify $C^k([0,1])$ to
   $$
      C^k([0,1])^{\Bbb C}\;  :=\;   C^k([0,1])\otimes_{\Bbb R}{\Bbb C}\,;
   $$
   then there is a well-defined algebra-homorphism over ${\Bbb C}$
   $$
	  \gamma^{\sharp}\;:\; C^k({\Bbb C}^1)^{\Bbb C}\;
	       \longrightarrow\;    C^k([0,1])^{\Bbb C}
   $$
   by the pull-back of functions via $\gamma$.
 Here comes the guiding question:
    \begin{itemize}
	  \item[{\bf Q.}]   {\it What
	    is the geometric meaning of the above algebraic operation?}	
	\end{itemize}
 The answer comes from an input to differential topology from algebraic geometry.
 
 By definition, a point with function field ${\Bbb R}$ is an {\it ${\Bbb R}$-point}
  while a point with function field ${\Bbb C}$ is a {\it ${\Bbb C}$-point}.
 Topologically they are the same but algebraically they are different,  as already indicated by
   $$
     {\Bbb R}\; \hookrightarrow\;  {\Bbb C}\,, \hspace{1em}\mbox{while}\hspace{1em}
	   {\Bbb C}\; \hspace{2ex}/\hspace{-3ex}\longrightarrow\;  {\Bbb R}\,,
   $$
   which means algebrao-geometrically, concerning the existence of a map from one to the other,
   $$
       \mbox{\it ${\Bbb C}$-point}\; \longrightarrow\;  \mbox{\it ${\Bbb R}$-point}\,,
	    \hspace{1em}\mbox{while}\hspace{1em}
	   \mbox{\it ${\Bbb R}$-point}\; \hspace{2ex}/\hspace{-3ex}
	     \longrightarrow\;  \mbox{\it ${\Bbb C}$-point}\,.
   $$
  By replacing $C^k([0,1])$ by its complexification $C^k([0,1])^{\Bbb C}$,
   we promote each original ${\Bbb R}$-points on $[0,1]$ to a ${\Bbb C}$-point.
  In other words, we smear ${\Bbb C}$-points along the interval $[0,1]$.
 The map $\gamma$ now simply specifies a $C^k$ $[0,1]$-family of ${\Bbb C}$-points on ${\Bbb C}^1$
    by associating to each ${\Bbb C}$-point on $[0,1]$ a ${\Bbb C}$-point on ${\Bbb C}^1$,
	which is now allowed algebro-geometrically.
 This concludes the example
 
\noindent\hspace{40.7em}$\square$
}\end{ssexample}

\bigskip

Let $p^{A\!z}$ be a point with function ring isomorphic to the endomorphism algebra
 $\End_{\Bbb C}({\Bbb C}^{\oplus r})$
Then, by exactly the same reasoning and geometric pictures as in Example~5.3.3.1,
                               % Example [real contour $\gamma$ in complex line ${\Bbb C}^1$]
   with $(\,\cdots\,)\otimes_{\Bbb R}{\Bbb C}$
            replaced by $(\,\cdots\,)\otimes_{{\Bbb R}}\End_{\Bbb C}(E)$
			  locally\footnote{Here,
	                                  we remain in the case when the class in the Brauer group $\Br(X)$ of $X$
									     associated to $X^{\!A\!z}$ is zero.
		                              See [L-Y4] (D(5)) and references therein for related discussion.},
   where $E$ is a ${\Bbb C}$-vector space,
 one has	{\small
 $$
  \begin{array}{lcl} \\[-1.2ex]
    \fbox{Azumayanized manifold
	  $(X, {\cal O}_X
	                \otimes_{\Bbb R}\End_{\Bbb C}({\Bbb C}^{\oplus r}))$  }\;
	   & \Longleftrightarrow\;
       & \fbox{the smearing of {\it fixed} $p^{A\!z}$'s along $X$}	   \\[2ex]
    \fbox{general Azumaya manifold $(X^{\!A\!z},{\cal E})$}\;
	   & \Longleftrightarrow\;
       & \fbox{a smearing of {\it unfixed} $p^{A\!z}$'s along $X$}   \\[2ex]	
  \end{array}	
 $$   }

\bigskip

\begin{flushleft}
{\bf A $C^k$-map $\varphi:(X^{\!A\!z},{\cal E})\rightarrow Y$
       as smearing D0-branes on $Y$ along $X$}
\end{flushleft}
{To} press on along this line, we have to list two objects that are studied in algebraic geometry and
 yet their counter-objects are much less known/studied in differential topology/geometry:
\begin{itemize}
 \item[(1)] [{\it Quot-schemes}$\,$] \hspace{1em}
   Grothendieck's $\Quot$-scheme $\Quot_Y^{r}(({\cal O}_Y^{\,\Bbb C})^{\oplus r})$
    of $0$-dimensional quotient sheaves of $({\cal O}_Y^{\,{\Bbb C}})^{\oplus r}$
	of complex-length $r$.
  It follows from Sec.~3, Lemma/Definition~5.3.1.9, and Sec.~5.3.2 that
                                     % Lemma/Definiition [$C^k$-map in affine setting]
    this is the parameter space of differentiable maps from the fixed Azumaya point
	  $(p^{A\!z},{\Bbb C}^{\oplus r})$  to $Y$.
   In other words, it parameterizes D0-branes ${\cal F}$ on $Y$
     (where ${\cal F}$ is a complex $0$-dimensional sheaf on $Y$ of complex length $r$)
     that is decorated with an isomorphism
	  ${\Bbb C}^{\oplus r}\stackrel{\sim}{\rightarrow}C^k({\cal F})$ over ${\Bbb C}$. 	
    	
 \item[(2)] [{\it Quotient stacks}$\,$] \hspace{1em}
  The general linear group $\GL_r({\Bbb C})$
    acts on  $\Quot_Y^{r}(({\cal O}_Y)^{\Bbb C})^{\oplus r})$
	by its tautological action on the ${\Bbb C}^{\oplus r}$-factor in the canonical isomorphism	
   $({\cal O}_Y^{\,{\Bbb C}})^{\oplus r}
       \simeq {\cal O}_Y\otimes_{\Bbb R}{\Bbb C}^{\oplus r}$.
 This defines a quotient stack
   $[ \Quot_Y^{r}(({\cal O}_Y^{\,\Bbb C})^{\oplus r})/\GL_r({\Bbb C})  ]$,
   which now parameterizes differentiable maps $\varphi$ from unfixed Azumaya points
	  $(p^{A\!z}, E)$, where $E$ is a ${\Bbb C}$-vector space of rank $r$,  to $Y$.
   In other words,  $[ \Quot_Y^{r}(({\cal O}_Y^{\,\Bbb C})^{\oplus r})/\GL_r({\Bbb C})  ]$
     is precisely the moduli stack ${\frak M}_r^{\,0^{A\!z^{\!f}}}\!\!(Y)$
	 of D0-branes of complex length $r$ on $Y$,
    realized as complex $0$-dimensional sheaves on $Y$ of complex length $r$ via push-forwards
	$\varphi_{\ast}(E)$, from Definition~5.3.1.5.
                                                 % Definition [differentiable map as equivalence class of gluing systems]
\end{itemize}

Recall from Sec.~5.3.2 that a differentiable map $\varphi:(X^{\!A\!z},{\cal E})\rightarrow Y$
  is completely encoded by its graph $\tilde{\cal E}_{\varphi}$ on $X\times Y$.
Over each point $x\in X$, $\tilde{\cal E}_{\varphi}|_{\{x\}\times Y}$
  is simply a $0$-dimensional ${\cal O}_Y^{\,{\Bbb C}}$-module of complex length $r$.
Despite missing the details of these parameter ``spaces", it follows from their definition
 as functors or sheaves of groupoids over the category of $C^k$-manifolds,
   with suitable Grothendieck topology, that
 $$
  \begin{array}{c} \\[-1.2ex]
    \fbox{$\,C^k$-maps $\varphi:(X^{\!A\!z},{\cal E})\rightarrow Y$}\;
	   \Longleftrightarrow\;
    \fbox{admissible maps $X\rightarrow {\frak M}_r^{\,0^{A\!z^{\!f}}}\!\!(Y)$}	   \\[2ex]
  \end{array}	
 $$
This matches perfectly with
   the picture of an Azumaya manifold $(X^{\!A\!z},{\cal E})$
      as a smearing of unfixed Azumaya points $p^{A\!z}$ along $X$
 since, then, a map $X^{\!A\!z}\rightarrow Y$ is nothing but an $X$-family of maps
  $p^{A\!z}\rightarrow Y$, which is exactly the map
  $X\rightarrow {\frak M}_r^{\,0^{A\!z^{\!f}}}\!\!(Y)$.
 Cf.~{\sc Figure}~5-3-3-1.
%
%  \marginpar{\raggedright\tiny $\bullet$
%        {\sc Figure}:  \\  smear-d0.pdf}

\begin{figure}[htbp]
 \bigskip
  \centering
  \includegraphics[width=0.80\textwidth]{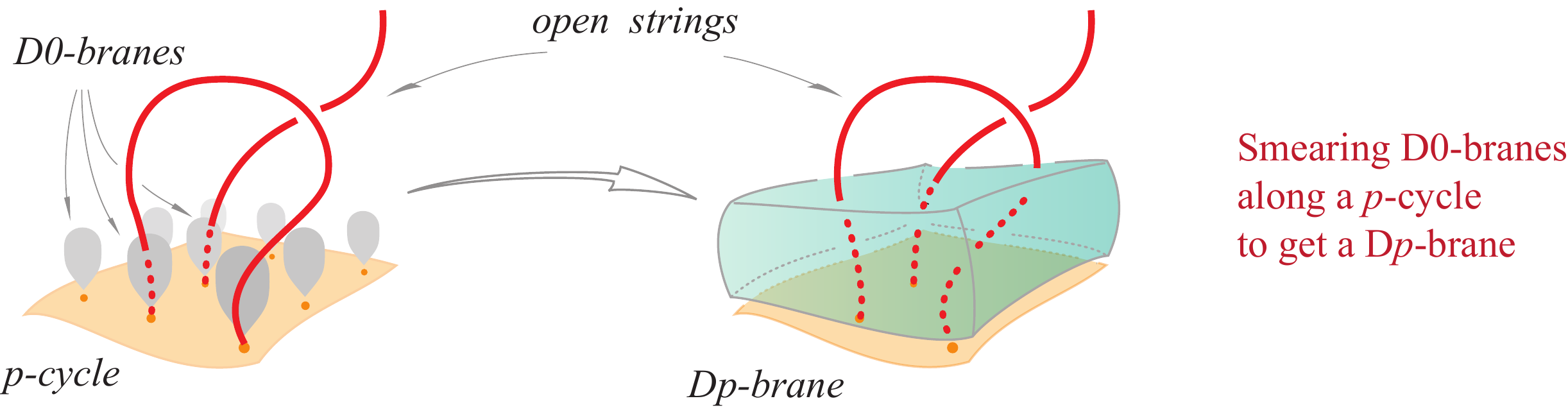}
 
  \bigskip
  \bigskip
 \centerline{\parbox{13cm}{\small\baselineskip 12pt
  {\sc Figure}~5-3-3-1. ([L-L-S-Y: {\sc Figure}~3-1-1].)
  The original stringy operational definition of D-branes as objects
     in the target space(-time) $Y$ of fundamental strings
   where end-points of open-strings can and have to stay
  suggests that smearing D$0$-branes along
    a (real) $p$-dimensional submanifold $X$ in $Y$ renders $X$ a D$p$-brane.
 Such a smearing in our case is realized as
   a map from the manifold $X$ to the stack
   ${\frak M}^{D0}(Y)$ of D$0$-branes on $Y$.
 In the figure, the Chan-Paton sheaf ${\cal E}$ that carries
   the index information on the end-points of open strings
   is indicated by a shaded cloud.
 Its endomorphism sheaf $\Endsheaf_{{\cal O}_X^{\,\Bbb C}}{\cal E}$
  carries the information of the gauge group of the quantum field theory
  on the D-brane world-volume.
  }}
\end{figure}

\bigskip

\subsubsection{Aspect IV: From associated $\GL_r({\Bbb C})$-equivariant maps}

Again, we list a counter-issue in differential topology that remains to be understood:
\begin{itemize}
 \item[(1)] [{\it Fibered product}$\,$] \hspace{1em}
  The notion of the fibered product of stratified singular spaces with a structure sheaf
  and its generalization to stacks needs to be developed.
\end{itemize}
Subject to this missing detail, from the very meaning of a quotient stack, it's natural to anticipate that
 any natural definition of the notion of fibered product should make a map
  $$
    X\;   \longrightarrow\;
   	{\frak M}_r^{\,0^{A\!z^{\!f}}}\!\!(Y)\,
       =\, [ \Quot_Y^{r}(({\cal O}_Y^{\,\Bbb C})^{\oplus r})/\GL_r({\Bbb C})  ]
  $$
 lift to a $\GL_r({\Bbb C})$-equivariant map
  $$
     P_X\;  \longrightarrow\;  \Quot_Y^r(({\cal O}_Y^{\,{\Bbb C}})^{\oplus r})\,,
   $$
 where $P_X$ is a principal $\GL_r({\Bbb C})$-bundle over $X$ from the fibered product
 $$
   P_X\; =\;
    X  \times_{{\frak M}_r^{\,0^{A\!z^{\!f}}}\!\!(Y)}
	   \Quot_Y^{r}(({\cal O}_Y^{\,\Bbb C})^{\oplus r})
 $$
Conversely, any latter map should define a former map.
Together with Aspect III in Sec.~5.3.3, this gives
 $$
  \begin{array}{l} \\[-1.2ex]
    \fbox{$\,C^k$-maps $\varphi:(X^{\!A\!z},{\cal E})\rightarrow Y$} \\[2ex]
	   \hspace{6em}\Longleftrightarrow\;
    \fbox{admissible $\GL_r({\Bbb C})$-equivariant  maps
	           $P_X\rightarrow \Quot_Y^r(({\cal O}_Y^{\,{\Bbb C}})^{\oplus r})$}	   \\[2ex]
  \end{array}	
 $$

%-------------------------------------------------------------------------------------------
%
% \bigskip
%
% \subsubsection{Aspect V: From synthetic differential geometry and $C^{\infty}$-algebraic geometry}
%
% \bigskip
%
% \noindent $\bullet$
%  ???????????????????.
%
% \bigskip
%
%=====---------------------------------======--------------------------------------
  
\bigskip

Before leaving Sec.~5, we should remark that Aspect III and Aspect IV
 of a $C^k$-map in our contents, though as yet cannot be used as a solid tool mathematically,
 remain conceptually important.
They distinguish special roles of D0-branes in string theory.
In view of the related development on the string-theory side, let us pose a guiding question for the future:
 \begin{itemize}
   \item[{\bf Q.}]  {\it Does
    Aspect III have any link to the aspect of M-theory and branes from matrices
	in the revised string theory, e.g.\ [B-F-S-S], [B-S-S], and [B-S]?}
 \end{itemize}

\bigskip

\section{Push-pulls and differentiable maps adapted to additional structures on the target-manifold}	

In this section, we address some basic issues on push-pulls
   that are tailored to the notion of differentiable maps given in Sec.~5.3
 and introduce the notion of various adapted classes of differentiable maps in our setting
  when the target-manifold of the map is equipped with an additional structure.

\bigskip

\subsection{The induced map on derivations, differentials, and tensors}

Let
 $\varphi:
     (X^{\!A\!z},{\cal E})
	     :=(X,{\cal O}_X^{A\!z}:=\Endsheaf_{{\cal O}_X^{\,\Bbb C}}({\cal E}),
		 {\cal E})
       \rightarrow {\cal O}_Y$
 be a $C^k$-differentiable, specified by a $C^k$-admissible
 $\varphi^{\sharp}:{\cal O}_Y\rightarrow{\cal O}_X^{A\!z}$,
 $k\ge 1$.
We address one positive and some negative results concerning the push-pull of various tensors
 under $\varphi$.

\bigskip

\begin{flushleft}
{\bf The induced map on derivations and tangent sheaves}
\end{flushleft}
Exactly as in Example~4.1.17,
                  % Example [push-forward of derivation under $\varphi$]
let $\Theta\in {\cal T}_{\ast}X^{\!A\!z}$ be a derivation on ${\cal O}_Y$.
Then $\Theta$ acts on ${\cal O}_Y$ via
   $$
     (\varphi_{\ast}\Theta)(f)\; :=\; \Theta(\varphi^{\sharp}(f))
   $$
  for $f\in{\cal O}_Y$.
The ${\Bbb R}$-linear map
 $\varphi_{\ast}\Theta:{\cal O}_Y\rightarrow {\cal O}_X^{A\!z}$
 satisfies the Leibniz rule of the form
  $$
   (\varphi_{\ast}\Theta)(fg)\;
     =\;  (\varphi_{\ast}\Theta)(f)\,\varphi^{\sharp}(g)\,
	          +\, \varphi^{\sharp}(f)\,(\varphi_{\ast}\Theta)(g)
  $$
 for $f,\,g\in {\cal O}_Y$.
 
\bigskip

\begin{definition} {\bf [push-forward derivation and $\varphi^{\ast}{\cal T}_{\ast}Y$]. } {\rm
 The $\varphi_{\ast}\Theta$ defined above is called the {\it push-forward} of $\Theta$ under $\varphi$.
 An ${\Bbb R}$-linear map
  $\Xi:{\cal O}_Y\rightarrow {\cal O}_X^{A\!z}$
    that satisfies the Leibniz rule
    $$
      \Xi(fg)\;
       =\;  \Xi(f)\,\varphi^{\sharp}(g)\,
	          +\, \varphi^{\sharp}(f)\,\Xi(g)
    $$
    for $f,\,g\in {\cal O}_Y$
   is called  an {\it ${\cal O}_X^{A\!z}$-valued derivation on ${\cal O}_Y$ through $\varphi$}.
 The set of all such derivations form a sheaf of ${\cal O}_X$-modules,
  denoted by $\varphi^{\ast}{\cal T}_{\ast}Y$.
 The correspondence $\Theta\mapsto \varphi_{\ast}\Theta$
   defines a ${\cal O}_X$-module-homomorphism
  $$
   \varphi_{\ast}\;:\; {\cal T}_{\ast}X^{\!A\!z}\;
     \longrightarrow\; \varphi^{\ast}{\cal T}_{\ast}Y\,.
  $$
}\end{definition}

\bigskip

\begin{remark}
 $[{\cal O}_X^{A\!z}
            \otimes_{\varphi^{\sharp},{\cal O}_Y^{\,\Bbb C}}{\cal T}_{\ast}Y]$.  {\rm
 Caution	that, unlike in the commutative case, in general and as ${\cal O}_X^{\,\Bbb C}$-modules,
  $$
    \varphi^{\ast}{\cal T}_{\ast}Y\;
     \not\simeq\;	
     {\cal O}_X^{A\!z}
	     \otimes_{\varphi^{\sharp}, {\cal O}_Y}{\cal T}_{\ast}Y
  $$
  for $\varphi$ with $\Image\varphi^{\sharp}$ not contained in the center
  ${\cal O}_X^{\,\Bbb C}$ of ${\cal O}_X^{A\!z}$.
Cf.~Remark~4.1.18.
                   % Remark [$ M_{r\times r}({\Bbb C})
	               %                   \otimes_{\varphi^{\sharp}, C^k({\Bbb R}^n)}
				   %                   \Der_{\Bbb R}(C^k({\Bbb R}^n))$]
}\end{remark}

\bigskip

\begin{flushleft}
{\bf The induced map on differentials and cotangent sheaves}
\end{flushleft}
As explained in Example~4.1.20 and footnote~4 in Sec.~4.1,
                       % Example [pull-back of differential under $\varphi$]
 in general there is no notion of the (canonical)  `pull-back of differentials' from $Y$ to $X^{\!A\!z}$ 					
 for a fundamental reason.

\bigskip

\begin{flushleft}
{\bf The induced map on tensors}
\end{flushleft}
Accordingly,
 there is no notion of the (canonical)  `pull-back of differential forms' from $Y$ to $X^{\!A\!z}$
 for general $\varphi$.

\bigskip

\subsection{Remarks on stringy regularizations of the push-forward of a sheaf under\\  a differentiable map}

The push-forward $\varphi_{\ast}{\cal E }$ of the built-in fundamental mudule ${\cal E}$ on $X^{\!A\!z}$
 has already occurred in this note.
For completeness, let us introduce by the same token:

\bigskip

\begin{definition} {\bf [push-forward of ${\cal O}_X^{A\!z}$-module].} {\rm
 Let ${\cal F}$ be an ${\cal O}_X^{A\!z}$-module on $X^{\!A\!z}$.
 Then, ${\cal F}$ is rendered naturally
   as an ${\cal O}_Y^{\,\Bbb C}$-module, denoted by $\varphi_{\ast}{\cal F}$,
    through $\varphi^{\sharp}:{\cal O}_Y\rightarrow{\cal O}_X^{A\!z}$.
 $\varphi_{\ast}{\cal F}$	is called the {\it push-forward} of ${\cal F}$ under $\varphi$.
}\end{definition}

\bigskip
	
While $\varphi_{\ast}{\cal F}$ is well-defined from above,
 in general it doesn't bahave well, as an ${\cal O}_Y^{\,\Bbb C}$-module, under $\varphi$.
In particular,
this already happens on $\varphi_{\ast}{\cal E}$
 when the underling maps from the $\varphi$-specified surrogate
  $   X \stackrel{\pi_{\varphi}}{\leftarrow} X_{\varphi}\stackrel{f_{\varphi}}{\rightarrow} Y$
  has some value $y\in Y$ with $\pi_{\varphi}(f_{\varphi}^{-1}(y))$
  containing a $C^k$-submanifold of $X$ of positive dimension.
For general ${\cal F}$, we have nothing to say.
However,
  for an ${\cal O}_X^{A\!z}$-module ${\cal F}$ that has a meaning in string theorty,
    especially ${\cal E}$,
 one expects there to be a notion of `{\it regularization}' of $\varphi_{\ast}{\cal F}$ 	
  that renders $\varphi_{\ast}{\cal F}$
  a string-theoretically more reasonable ${\cal O}_Y^{\,\Bbb C}$-module.

\bigskip

\subsection{Differentiable maps adapted to additional structures on the target-\\ manifold}

Let us begin with a guiding question:
 \begin{itemize}
  \item[{\bf Q.}]
   {\bf [map adapted to calibration, ..., on target-manifold] } \hspace{1em}
   In the study of symplectic or calibrated geometry,
    one considers the target $Y$ with an additional structure specified by a differential form $\alpha$ on $Y$
    and a map $f:X \rightarrow Y$ is said to be adapted to $\alpha$ if $f^{\ast}\alpha=0$
     (plus some minor conditions on $X$).
   For example, when $(Y,\alpha)$ is a symplectic manifold then such $f$ with $\dimm X=\frac{1}{2}\dimm Y$
   defines a Lagrangian submanifold of $(Y,\alpha)$.
  When trying to generalize this notion to maps $\varphi:(X^{\!A\!z},{\cal E})\rightarrow (Y,\alpha)$
     from Azumaya manifolds,
   one immediately runs into the technical difficulty
     that the notion of the pull-back $\varphi^{\ast}\alpha$ of $\alpha$ is fundamentally undefinable.
   Yet such a notion of adapted maps is required for D-branes in string theory, e.g.\ A-branes and B-branes.
    {\it How shall we deal with this?}
 \end{itemize}
In this section, we present a first and weakest answer to this question and
  bring out a related set of definitions of adapted maps in the related context.
Their possible refinements, further mathematical details and test on string theory are the focus of separate works.

\bigskip

Recall
  the $C^k$-map $\varphi:(X^{\!A\!z},{\cal E})\rightarrow Y$
     from an Azumaya manifold with a fundamental module to a manifold $Y$,
 $X_{\varphi}$ the surrogate of $X$ specified by $\varphi$, and
 $X_{\varphi,\redscriptsize}$ the reduced subscheme of $X_{\varphi}$
 in the sense of $C^k$ algebraic geometry.
Recall also that the $C^k$ map
  $$
    (\pi_{\varphi},f_{\varphi})\; :\; X_{\varphi}\; \longrightarrow\;  X\times Y
  $$
  is an embedding, whose image coincides with the scheme-theoretical support
     $\Supp(\tilde{\cal E}_{\varphi})$
   of the graph $\tilde{\cal E}_{\varphi}$ of $\varphi$.
While $\alpha^{\ast}\alpha$ in the above question is generally undefinable,
 $f_{\varphi}^{\ast}\alpha$ is always defined.
This gives us the basis to the solution.
However, in general $X_{\varphi}$ may be nonreduced.
Mathematically, choices like
 \begin{itemize}
   \item[$\cdot$] {\bf [strong vs.\ weak answer]}\hspace{2em}
   $f_{\varphi}^{\ast}\alpha\;=\; 0\;\;\;\;$ vs.\
   $\;\;\; \left.f_{\varphi}^{\ast}\alpha\right|_{X_{\varphi, \redtiny}}\;=\;0\,$,
 \end{itemize}
 or some conditions in-between,
  where $X_{\varphi,\redscriptsize}$ is the reduced subscheme of $X_{\varphi}$,
 distinguish whether the answer is strong, weak, or intermediate.
In the end, it is fitting in string theory that selects the final correct answer.
For the purpose the current note D(11.1),
 we introduce the weakest answer, leaving room for strengthening in the future.

Definition~6.3.3 below
              % Definition [$(j,J)$-holomorphic map to almost complex manifold $(Y,J)$]
 is phrased more conveniently in terms of Aspect II [graph] of a $C^k$-map $\varphi$,
  in which $X_{\varphi}$ is embedded in $X\times Y$ by $(\pi_{\varphi},f_{\varphi})$.

\bigskip

\begin{definition}  {\bf [map of relative dimension 0].} {\rm
 A $C^k$-map $\varphi:(X^{\!A\!z},{\cal E})\rightarrow Y$
    is said to be {\it of relative dimension $0$}
 if $f_{\varphi}:X_{\varphi}\rightarrow Y$  is of relative dimension $0$;
 i.e.\ for all $y\in Y$,  $f_{\varphi}^{-1}(y)$ is $0$-dimensional if non-empty.
}\end{definition}
 
%--------------------------------------------------------------------------------------------------------------------
% \bigskip
%
% \begin{flushleft}
% {\bf $(Y,g)$ Riemannian manifold}
% \end{flushleft}
% %
% ??????????????.
%
% \bigskip
%
% \noindent $\bullet$
% {\it Isometric immersion} to a Riemannian manifold $(Y,g)$:
%     %
%     \begin{itemize}
% 	  \item[(1)]
% 	   Riemannian structure on $X$.
% 	
% 	  \item[(2)]
% 	   Riemannian structure on $X^{\!A\!z}$ and the induced structure on $X_{\varphi}$.
%
%       \item[(3)]
%        Riemannian structure on $X\times Y$ and the induced structure on $X_{\varphi}$
% 	     as a $C^k$-subscheme of $X\times Y$.
%
%      \item[(4)]	
% 	   Comparison of the two induced structure on $X_{\varphi}$  or $X_{\varphi,\redscriptsize}$.
% 	   ??????????.
%     \end{itemize}
%
% \bigskip
%
% \begin{flushleft}
% {\bf $(Y,g)$ Lorentzian manifold}
% \end{flushleft}
% %
% ????????????????.
%
% \bigskip
%
% \noindent $\bullet$
% Maps into a Lorentzian manifold $(Y,g)$.
%
% \bigskip
%======------------------------------------------========----------------------------------------------------------===
  
\bigskip

\begin{definition} {\bf [Lagrangian map to symplectic manifold].} {\rm
 Let $(Y,\omega)$ be a symplectic manifold.
  A $C^k$-map $\varphi:(X^{\!A\!z},{\cal E})\rightarrow Y$
    is {\it Lagrangian} ({\it in the weakest sense})
 if
  $$
    \mbox{$\dimm X\; =\; \frac{1}{2}\,\dimm Y\,$,
	 $\;\;\; \varphi$ is of relative dimension $0$\,,   $\;\;\;$ and	
     $\;\;\; \left.f_{\varphi}^{\ast}\,\omega\,\right|_{X_{\varphi,\redtiny}}\; =\; 0$}\,.
  $$	
  %----------------------------------------------------------------------------------	 	
  % %
  % \marginpar{\raggedright\tiny $\bullet$  {\bf CAUTION.}
  %          Condition $(\varphi^{\ast}\omega)|_{X_{\varphi}}$
  %		    is too strong as it kills all the nilpotent cloud in the direction transverse to $\varphi$.
  %		   The condition here is the weakest possible. It may have to be supplemented to additional constraint
  %           to reflect the anticipation that the nilpotent cloud arises from Lagrangian deformations.		\\
  %         Similarly in other situations.			 }
  %-------------------===========---------------------------------------=====
}\end{definition}

%------------------------------------------------------------------------------------------------		
% \bigskip
%
% \begin{example}{\bf [?????].}
%   Seidel and Sheridan's example as a Lagrangian map from an Azumaya sphere to ????.
% \end{example}
%====---------------------------------=======------------------------------------------
  
\bigskip
 
\begin{definition}   {\bf [$(j,J)$-holomorphic map to almost complex manifold $(Y,J)$].}
{\rm  Let $(X,j)$, $(Y,J)$ be almost complex manifolds.
  A $C^k$-map $\varphi:(X^{\!A\!z},{\cal E})\rightarrow Y$
    is {\it $(j,J)$-holomorphic}
 if $$
      \mbox{$X_{\varphi}$ is an almost complex subscheme of $(X\times Y, j\times J)$}\,,
   $$
 that is,
    ${\cal T}_{\ast}X_{\varphi}\,
      \subset\, {\cal T}_{\ast}(X\times Y)|_{X_{\varphi}}$
     is invariant under the product almost complex structure $j\times J$ on $X\times Y$.
 When $j$ is understood from the context, $\varphi$ is simply called a {\it $J$-holomorphic map}.
	
 If both $j$ and $J$ are integrable, i.e.\
     $X$ and $Y$ are holomorphic manifolds
      (with respective holomorphic structure sheaves ${\cal O}_X$ and ${\cal O}_Y$),
   ${\cal E}$ is a coherent locally-free holomorphic ${\cal O }_X$-module,  and
   $X_{\varphi}$ holomorphic subscheme of $X\times Y$,
 then $\varphi$ is called a {\it holomorphic map}.
}\end{definition}
  
\bigskip

The following two examples indicate a closed-open pair of new theories
 that serve as the D-string counter theory to the symplectic Gromov-Witten theory
  (cf. [Gr] and [Wi1], [Wi2];
          see [McD-S2] for an exposition in the closed case and
                 [F-O-O-O], [Liu$_{CC}$], [Ye] for the open case),
    which is related to either closed fundamental strings or open fundamental strings.
Their algebraic theory was brought out and studied in [L-Y9] (D(10.1)) and [L-Y10] (D(10.2)).
     
\bigskip
 
\begin{example} {\bf [closed $J$-holomorphic D-curve].} {\rm
 A $J$-holomorphic map $\varphi:(\Sigma^{\!A\!z},{\cal E})\rightarrow (Y,J)$
  from an Azumaya Riemann surface without boundary, with a fundamental module,
  to an almost complex manifold $(Y,J)$ is called a ({\it closed}$\,$) {\it $J$-holomorphic D-curve}.
}\end{example}

\bigskip

\begin{example} {\bf [open $J$-holomorphic D-curve].} {\rm
 Let $Y:= (Y, \omega,J ;L,V_L,\nabla^{V_L})$ be a symplectic manifold $(Y,\omega)$
   with an $\omega$-tame almost complex structure $J$ that is endowed with
   an embedded Lagrangian submanifold $L$ together with a complex vector bundle $V_L$ over $L$
     with a flat connection $\nabla^{V_L}$ on $V_L$.
(Denote the associated sheaf of smooth sections of $V_L$ by ${\cal V}_L$.)
 The notion of {\it open $J$-holomorphic D-curves} on $Y$ is more involved than in the closed case.
 The simplest form of the notion is given by
   \begin{itemize}
    \item[$\cdot$]
     a {\it $J$-holomorphic map}
      $\varphi:(\Sigma^{A\!z},{\cal E})\rightarrow Y$
       from an Azumaya bordered Riemann surface $(\Sigma^{A\!z},{\cal E})$ to $Y$
 	    with $\varphi(\partial\Sigma^{A\!z})_{\redscriptsize}\subset L$   and
	    $\varphi_{\ast}({\cal E}|_{\partial\Sigma})$
	     independent of ${\cal V}_L$ (the {\it free boundary-sheaf condition})
	     or mapped completely into ${\cal V}_L$ (the {\it total-inclusion boundary-sheaf condition})
	     or satisfying a condition between these two extreme conditions
 	    	 (a {\it partial-inclusion boundary-sheaf condition}). 		
   \end{itemize}
 With the D1-D3 brane-systems in Type IIB superstring theory in mind,
 a full version of the notion should involve in addition
   \begin{itemize}
    \item[$\cdot$]
    a {\it connection $\nabla$ on ${\cal E}$} that is a solution to a stringy equation (cf.\ [M-M-M-S])
      with its restriction to the boundary $\nabla|_{\partial\Sigma}$
      intertwined with the connection $\nabla^{V_L}$ on ${\cal V}_L$ through $\varphi$
       in a way that is compatible with the specified boundary-sheaf condition.
   \end{itemize}
 The detail should be studied in its own right.
}\end{example}

\bigskip

\begin{definition} {\bf [special Lagrangian map to Calabi-Yau manifold].}  {\rm
 Let $(Y,J, \omega,\Omega)$ be a Calabi-Yau $n$-fold
   with complex structure $J$,   K\"{a}hler $2$-form $\omega$, and holomorphic $n$-form $\Omega$
   such that $\omega^n/n!=(-1)^{n(n-1)/2}(\sqrt{-1}/2)^n\Omega\wedge\bar{\Omega}$.
 A $C^k$-map $\varphi:(X^{\!A\!z},{\cal E})\rightarrow Y$
    is {\it special Lagrangian} ({\it in the weakest sense})
 if
  $$
   \mbox{$\varphi$ is Lagrangian to $(Y,\omega)$ such that}\;\;
	 \left.f_{\varphi}^{\ast}\,
	               \Real(e^{-\sqrt{-1}\,\theta}\,\Omega)\,\right|_{X_{\varphi,\redtiny}}\; =\; 0
  $$	
  for some locally constant function $\theta$ defined on a (possibly disconnected)  open set $U\subset Y$
    such that $U\cap f_{\varphi}(X_{\varphi})$ is open dense in $f_{\varphi}(X_{\varphi})$.
}\end{definition}

\bigskip

\begin{remark} $[$phase function of special Lagrangian map$\,]$.  {\rm
 Note that we require the phase function $\theta$, after being pulled back to $X_{\varphi}$ via $f_{\varphi}$,
  be only locally constant, rather than constant,    and
  be defined only on an open dense subset of $X_{\varphi}$, rather than the whole $X_{\varphi}$.
 % to reflect the deformations of A-branes in string theory better.
}\end{remark}

\bigskip

See Example~7.2.2 for examples of special Lagrangian maps to the Calabi-Yau $1$-fold ${\Bbb C}^1$.
   
\bigskip

\begin{definition}
{\bf [associative map \& coassociative map to 7-manifold with $G_2$ holonomy $(Y,\eta,g)$].} {\rm
 Let $(Y,\eta,g)$ with positive $3$-form $\eta$ and associated metric $g$
  be a $7$-manifold with $G_2$ holonomy.
 A $C^k$-map $\varphi:(X^{\!A\!z},{\cal E})\rightarrow Y$
    is {\it associative} ({\it in the weakest sense})
 if
  $$
    \dimm X=3\,,\;\;\;
	\mbox{$\varphi$ is of relative dimension $0$}\,,\;\;\;  \mbox{and}\;\;\;
    \left.f_{\varphi}^{\ast}\eta\,\right|_{X_{\varphi,\redtiny}}=0\,.
  $$
 A $C^k$-map $\varphi:(X^{\!A\!z},{\cal E})\rightarrow Y$
    is {\it coassociative} ({\it in the weakest sense})
 if
  $$
    \dimm X=4\,,\;\;\;
	\mbox{$\varphi$ is of relative dimension $0$}\,,\;\;\;  \mbox{and}\;\;\;
    \left. f_{\varphi}^{\ast}(\ast\eta)\,\right|_{X_{\varphi,\redtiny}}=0\,.
  $$
Here, $\ast\eta$ is the Hodge dual $4$-form of $\eta$ with respect to the metric $g$.	
}\end{definition}

\bigskip

\section{Examples of differentiable maps from Azumaya manifolds with a fundamental module}

Two more sets of examples of differentiable maps from Azumaya manifolds with a fundamental module
 are given in this last section of the current note D(11.1).

\bigskip
 
\subsection{Examples generated from branched coverings of manifolds}

This is the class of examples studied in [L-Y5: Sec.~3.2] (D(6)), [L-Y6] (D(7)), and  [L-Y7] (D(8.1)).
Let $X$ and $Y$ be $C^k$-manifolds.
Consider the following data
  $((\hat{X}, \hat{\cal E}), (\hat{c},\hat{f}))\,$:
  $$
   \xymatrix{
     \hat{\cal E}\ar@{.>}[d]   & & & & & \\
    \hat{X} \ar[rd]^{(\hat{c},\hat{f})} \ar@/^2ex/[rrrrrd]^-{\hat{f}}
            \ar@/_/[rddd]_-{\hat{c}}               & &       \\
     & X\times Y \ar[rrrr]_-{pr_2}  \ar[dd]^-{pr_1} & & & & Y\;, \\ \\
     & X
   }
  $$
  where
   \begin{itemize}
    \item[$\cdot$]
	 $\hat{X}$ is  a $C^k$-manifold
    	 with a locally free ${\cal O}_X$-module $\hat{\cal E}$ (of finite rank $\hat{r}$),
	 	
    \item[$\cdot$]
     $\hat{c}:\hat{X}\rightarrow X$ is
       a $C^k$-branched-covering map (of finite degree $\hat{d}$ )
	 such that ${\cal E}:=\hat{c}_{\ast}\hat{\cal E}$  is locally free (of rank $r=\hat{r}\hat{d}$),

    \item[$\cdot$]
     $\hat{f}:\hat{X}\rightarrow Y$ is a $C^k$-map.	
   \end{itemize}
 Denote the associated vector bundle to $\hat{\cal E}$ and ${\cal E}$
   by $\hat{E}$ and $E$ respectively.
Then, associated to this data is a $C^k$-map $\varphi:(X^{\!A\!z},{\cal E})\rightarrow Y$
  defined as follows:
  \begin{itemize}
   \item[$\cdot$]
   $(X^{\!A\!z}, {\cal E})
     =(X,{\cal O}_X^{A\!z}:=\Endsheaf_{{\cal O}_X^{\,\Bbb C}}({\cal E}),
       {\cal E}:= \hat{c}_{\ast}\hat{\cal E} )$
    is the Azumaya $C^k$-manifold with a fundamental module
	associated to $((\hat{X},\hat{\cal E}), \hat{c})$,

  \item[$\cdot$]
   Observe that there are embeddings
    ${\cal O}_X  \subset \hat{c}_{\ast}{\cal O}_{\hat{X}}\subset {\cal O}_X^{A\!z}$
     induced by $\hat{c}$.
   Through this and the fact that $\hat{c}_{\ast}\hat{\cal E}=:{\cal E}$ is tautologically
      a $\hat{c}_{\ast}{\cal O}_{\hat{X}}$-module,
    the ring-homomorphism $\hat{f}^{\sharp}: C^k(Y)\rightarrow C^k(\hat{X}) $
      is pushed forward under $\hat{c}$  to a $C^k$-admissible ring-homomorphism
	    $\varphi^{\sharp}:C^k(Y)\rightarrow C^k(\End_X(E))$
		over ${\Bbb R}\hookrightarrow {\Bbb C}$.
   This defines $\varphi$.	
  \end{itemize}
  
 The graph $\tilde{\cal E}_{\varphi}$ of $\varphi$ on $X\times Y$
   is given by $(\hat{c},\hat{f})_{\ast}\hat{\cal E}$,
   whose scheme-theoretical support $\Supp(\tilde{\cal E}_{\varphi})$ is given by
     the submanifold-with-singularities $(\hat{c},\hat{f})(\hat{X})$ in $X\times Y$.
 The surrogate $X_{\varphi}$ of $X^{\!A\!z}$ specified by $\varphi$
   is isomorphic to  $\Supp(\tilde{\cal E}_{\varphi})=(c,f)(\hat{X})$ and
 the diagram
  $$
   \xymatrix{
    {\cal E}_{\varphi}\ar@{.>}[d]     \\
     \;\Supp({\cal E}_{\varphi})\; \ar@{^{(}->}[r]
     & X\times Y \ar[rrr]^-{pr_2}  \ar[d]^-{pr_1} & &  & Y\;, \\
     & X
   }
  $$
  is translated to the diagram
  $$
   \xymatrix{
    & {\cal E} \ar@{.>}[rd]     \ar@{.>}@/_1ex/[rdd]      \\
    &  & X_{\varphi}\ar[rrr]^-{f_{\varphi}}
                                	    \ar@{->>}[d]^-{\pi_{\varphi}} &&& V \\
	&  & X
    }
  $$
  that underlies $\varphi$.

\bigskip

\begin{remark} {$[$Azumaya sphere and D2- and D3-brane$]$.} {\rm
 Together with the study of knots and links, and low-dimensional topology,
  cf.\ [Al], [He], [Hil], [H-L-M], [Mon], [Ro], [Th],
 one has the following special feature in $2$- and $3$-dimensions:
 \begin{itemize}
  \item[$\cdot$]  {\bf [Azumaya sphere and D2- and D3-brane].} {\it
   Let $Z\subset Y$ be an embedded smooth $2$- (resp.\ $3$-)dimensional submanifold.
   Then there exists a smooth-map  $\varphi:S^{2,A\!z}\rightarrow Y$
      (resp.\ $\varphi:S^{3,A\!z}\rightarrow Y$ )
        from an Azumaya $2$-sphere (resp.\ Azumaya $3$-sphere) to $Y$
      such that the image $\varphi(S^{3,A\!z})$ of $\varphi$ is exactly $Z$.}
 \end{itemize}
 Which may have implications to the construction of a perturbative D2-brane theory
  (resp.\ perturbative D3-brane theory)
  that mimics the construction of the perturbative string theory, cf.\ [B-P].
 See [L-Y6: Sec.~2.4.2] (D(7)) for more explanations.
 Cf.~{\sc Figure}~7-1-1.
 %
 % \marginpar{\raggedright\tiny $\bullet$
 %   {\sc Figure}:\\  morphism-Azumaya-3-sphere.pdf}
  
 \begin{figure} [htbp]
 \bigskip
 \centering
 \includegraphics[width=0.80\textwidth]{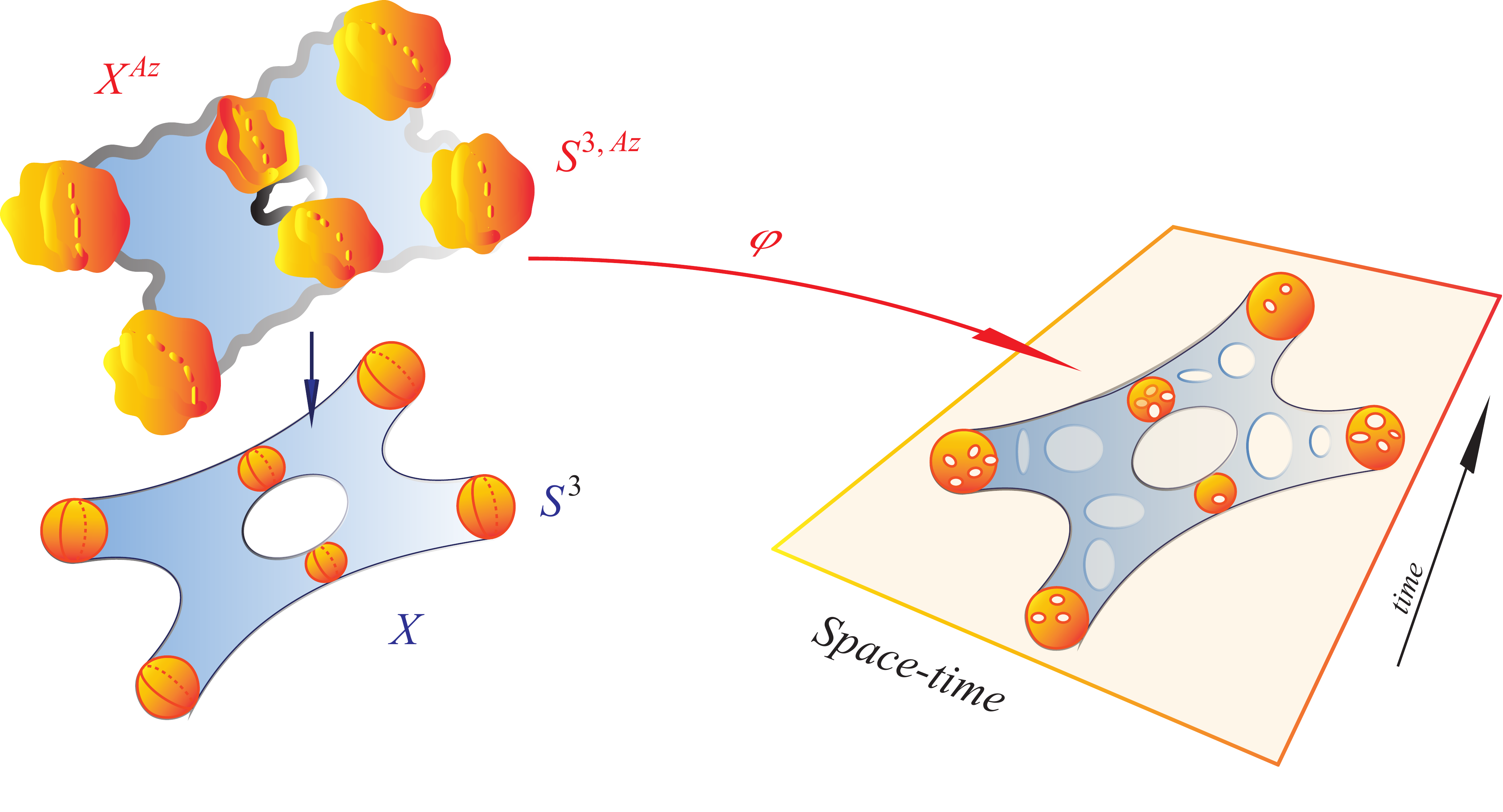}
 
 % \vspace{4em}
 \centerline{\parbox{13cm}{\small\baselineskip 12pt
  {\sc Figure}~7-1-1.
    Here, the world-volume $X$ of a (multiple) D3-brane in a space-time.
	 is a $4$-dimensional graph-manifold that arises from a connected sum of $S^3\times S^1$'s.
    D3-brane theory based only on all such maps $\varphi$
      resembles the perturbative string theory, which is based on the genus expansion of string world-sheets. 		
      }}
 \bigskip
\end{figure}	
}\end{remark}

\bigskip

\subsection{From an immersed special Lagrangian brane with a flat bundle to a fuzzy special Lagrangian brane
           with a flat bundle in the Calabi-Yau 1-fold ${\Bbb C}^1$}

We give in this subsection a local example of
 deformations of immersed special Lagrangian branes with a flat complex line bundle
  to a (possibly fuzzy) special Lagrangian brane with a flat bundle of higher rank.
  
\bigskip

\begin{remark}   $[$Fukaya category of A-branes on Calabi-Yau space$\,]$. {\rm
 While a study of such branes and their deformations\footnote{C.-H.L.\
                             would like to thank (time-ordered)
                               {\it Katrin Wehrheim}
                                  for discussions on the possibility of a notion of scheme-theoretic-like deformations
							 	  of Lagrangian submanifolds in a symplectic manifold,
	                              spring 2007 at Massachusetts Institute of Technology, 	
                               {\it Yng-Ing Lee}
                                  for discussions on wrapping-'n-unwrapping of special Lagrangian submanifolds
  								  in a Calabi-Yau space,
	                              spring 2013 at National Taiwan University,   and    	
                               {\it Siu-Cheong Lau}
							      for discussions on deformations of A-branes and the minimal set of objects that should
								  be  included in the Fukaya category of Lagrangian branes on a Calabi-Yau space
								  to reflect correctly D-branes in string theory
								  during and after his lectures on the work [Joy4] of {\it Dominic Joyce} on
                                    `Conjectures on Bridgeland stability for Fukaya categories of Calabi-Yau manifolds,
                                     special Lagrangians, and Lagrangian mean curvature flow'
                                     (arXiv:1401.4949 [math.DG]), spring 2014 at Harvard University. 	
                                    } % end-footnote
   along the line of the current note
  is the focus of another work,
 the current pedagogical example also means to illustrate
  what should be anticipated to be included in the Fukaya category of $A$-branes on a Calabi-Yau manifold
   from the viewpoint of wrapping-and-unwrapping of immersed special Lagrangian
   submanifolds-with-a-flat-complex-vector-bundle on a Calabi-Yau space.
 See [Joy4] for related discussions.
}\end{remark}

\bigskip

\begin{example} {\bf
 [deformations of special Lagrangian branes with a flat bundle in the Calabi-Yau 1-fold ${\Bbb C}^1$].}
 {\rm
 Let
   $ Y={\Bbb C}^1=({\Bbb R}^2,  J, \omega, \Omega)$
   be a Calabi-Yau $1$-fold,
  with
  \begin{itemize}
   \item[$\cdot$]
    \parbox[t]{3em}{$J\,:$}
	\parbox[t]{13cm}{the complex structure on ${\Bbb R}^2$, with coordinates $(y^1,y^2)$,
	   that renders $({\Bbb R}^2,J)$ the complex line ${\Bbb C}^1$,
	   with coordinates $z=y^1+\sqrt{-1}\,y^2$,}

   \item[$\cdot$]	
    \parbox[t]{3em}{$\omega\,:$}
	\parbox[t]{13cm}{the standard K\:{a}hler form
	    $\frac{\sqrt{-1}}{2}\,dz\wedge d\bar{z}=dy^1\wedge dy^2$
	  on ${\Bbb C}^1$,}
	
   \item[$\cdot$]   	
    \parbox[t]{3em}{$\Omega\,:$}
	\parbox[t]{13cm}{the standard holomorphic $1$-form $dz=dy^1+\sqrt{-1}\, dy^2$
	    on ${\Bbb C}^1$.}
 \end{itemize}
 $(\omega,\Omega)$ defines the phase of a special Lagrangian submanifold of ${\Bbb C}^1$.
 Let $X={\Bbb R}^1$ be the real line, as a $C^{\infty}$ $1$-manifold, with coordinate $x$, and
 \begin{itemize}
  \item[$\cdot$]
   ${\cal E}={\cal O}_{{\Bbb R}^1}\otimes_{\Bbb R}{\Bbb C}^{\oplus 3}$
    be the free sheaf of complex rank $3$ on ${\Bbb R}^1$ and
   
  \item[$\cdot$]
   $\nabla$ be the flat connection on
       ${\cal O}_{{\Bbb R}^1}\otimes_{\Bbb R}{\Bbb C}^{\oplus 3}$
     that coincides with $d$, that is,
	  $\nabla (s^1,s^2,s^3) = (ds^1,ds^2,ds^3)$	
    	for $s=(s^1,s^2,s^3)$ a smooth section of
 	        ${\cal O}_{{\Bbb R}^1}\otimes_{\Bbb R}{\Bbb C}^{\oplus 3}$.
 \end{itemize}
 We will consider smooth special Lagrangian maps
   $$
	 \varphi\; :\; ({\Bbb R^1},{\cal O }_{{\Bbb R}^1}^{A\!z}, {\cal E})\;
 	   \longrightarrow\; {\Bbb C}^1
   $$
   from the Azumay real line with a fundamental module
   $({\Bbb R}^{1,A\!z},{\cal E})$ to ${\Bbb C}^1$
   in the sense of Definition~6.3.6,
                         % Definition [special Lagrangian map to Calabi-Yau manifold]
  their deformations,  and
  how the flat connection $\nabla$ on ${\cal E}$	is pushed-forward to
    a flat connection with singularities on $\varphi_{\ast}{\cal E}$.			
 %-------------------------------------------------------------------------------------------------------------------	
 % Recall that a such map is specified contravariantly by a $C^{\infty}$-admissible ring-homomorphism
 %   $$
 %    \varphi^{\sharp}\; :\;
 %       C^{\infty}({\Bbb C}^1)\;
 %       \longrightarrow\; M_{3\times 3}(C^{\infty}({\Bbb R}^1)^{\Bbb C})
 %   $$
 %    over ${\Bbb R}\hookrightarrow {\Bbb C}$
 %   such that
 %   %
 %   \marginpar{\raggedright\tiny $\bullet$ To be completed.}
 %    $$
 %      ????????????????????\,.
 %    $$
 %-============----------------------------------=========-----------------------------------------
   
 Let $t\in [0,1]$ be a real parameter and
 $\varphi_1$ be the special Lagrangian map at $t=1$ defined by the ring-homomorphism
  $$
   \begin{array}{cccccl}
    \varphi^{\sharp}_1 & :
	  &  C^{\infty}({\Bbb C}^1)
      &  \longrightarrow      &  M_{3\times 3}(C^{\infty}({\Bbb R}^1)^{\Bbb C})\\[1.2ex]
	 &&  y^1  & \longmapsto  &  x \cdot  \Id_{3\times 3}\\[1.2ex]
	 && y^2  & \longmapsto
	    & \left[  \begin{array}{ccc} -x & 0 & 0 \\ 1 & 1  & 0 \\ 0 & 1 & x  \end{array} \right]		   & .	 
   \end{array}
  $$
 Then, over $U:={\Bbb R}^1-\{-1,0,1\}$,
  $\varphi_1$ induces a splitting of the restriction ${\cal E}_U$ of
   the fundamental module ${\cal E}$ to $U$
   into a direct sum of locally free ${\cal O}_U^{\,\Bbb C}$-modules of rank $1$,
   $$
    {\cal E}_U\;=\;
      {\cal O}_U^{\,\Bbb C}\cdot e_1\,
	      \oplus\, {\cal O}_U^{\,\Bbb C}\cdot e_2\,
		  \oplus\, {\cal O}_U^{\,\Bbb C}\cdot e_3 \;
	      =:\;{\cal L}_1 \oplus {\cal L}_2 \oplus {\cal L}_3\,
   $$
   where
  $$
    e_1\;
	=\; \left(
	       \begin{array}{c}1 \\[.6ex]  -\,\frac{1}{x+1}\\[1.2ex]  \frac{1}{2x(x+1)}\end{array}
		  \right)\,,
	 \hspace{2em}	
   e_2\;
	=\; \left(
	       \begin{array}{c}0 \\[.6ex]    1 \\[.6ex]    -\,\frac{1}{x-1}\end{array}
		  \right)\,,
	 \hspace{2em}	
  e_3\;
	=\; \left(
	       \begin{array}{c}0  \\[.6ex]   0 \\[.6ex]    1  \end{array}
		  \right)\,,
  $$
 in the sense that
  \begin{itemize}
   \item[$\cdot$] {\it As
    a singular decomposition of ${\cal E}$,
	the decompotion ${\cal L}_1 \oplus {\cal L}_2 \oplus {\cal L}_3$
	is invariant under the $C^{\infty}({\Bbb C}^1)$-action through $\varphi_1^{\sharp}$.}
  \end{itemize}
 Thus, $\varphi_{1\ast}{\cal L}_i$, $i=1,\,2,\,3$,  make sense and
  $$
   \varphi_{1\ast}{\cal E}_U\;
      =\;  \varphi_{1\ast}{\cal L}_1  \,\oplus\,
                \varphi_{1\ast}{\cal L}_2 \, \oplus\, \varphi_{1\ast}{\cal L}_3\,.
  $$
 Furthermore,
  the eigen-value-function of $f(y^1,y^2)\in C^{\infty}({\Bbb C}^1)$
  on $e_1$, $e_2$, $e_3$ under $\varphi_1^{\sharp}$ are
  $$
    f(x, -x)\,,\hspace{2em}f(x, 1)\,,\hspace{2em}  f(x, x)
  $$
  respectively.
 It follows that
  \begin{itemize}
   \item[$\cdot$]  {\it Let
   $$
    L_{-\frac{\pi}{4}}\;=\; \{(x,-x)\,|\, x\in{\Bbb R}^1 \}\,, \hspace{1.6em}
	L_0\; =\; \{(x,1)\,|\, x\in{\Bbb R}^1 \}\,, \hspace{1.6em}
	L_{\frac{\pi}{4}}\;=\; \{(x,x)\,|\, x\in{\Bbb R}^1 \}
   $$
   be special Lagrangian lines in ${\Bbb C}^1$,
	  of phases $-\pi/4$, $0$, $\pi/4$ respectively as indicated.
  Then,
    the smooth locus $Z_{\smoothscriptsize}$ of the image $Z:=\Image\varphi_1$ of $\varphi_1$
      lies in the union $L_{-\frac{\pi}{4}}\cup L_0\cup L_{\frac{\pi}{4}}$
	and $\varphi_{1\ast}{\cal E}$	is an ${\cal O}_Z^{\,\Bbb C}$-module of rank $1$,
    whose restriction to $Z_{\smoothscriptsize}$ is
	$\varphi_{1\ast}{\cal E}_U
      = \varphi_{1\ast}{\cal L}_1  \oplus
                \varphi_{1\ast}{\cal L}_2  \oplus \varphi_{1\ast}{\cal L}_3$
    with $\varphi_{1\ast}{\cal L}_1$
	(resp.\
        $\varphi_{1\ast}{\cal L}_2$, $\,\varphi_{1\ast}{\cal L}_3$)			
     supported in $L_{-\frac{\pi}{4}}$
	  (resp.\  $L_0$, $\,L_{\frac{\pi}{4}}$).
  }%end-it
 \end{itemize}
 
 The restriction of the flat connection $\nabla$ on ${\cal E}$ to ${\cal E}_U$ does not preserve
   the decomposition ${\cal E}_U={\cal L}_1\oplus{\cal L}_2\oplus {\cal L}_3$.
 However, by post-composition of $\nabla$ with the projection maps
   ${\cal E}_U\rightarrow {\cal L}_1$,  ${\cal E}_U\rightarrow {\cal L}_2$,
   ${\cal E}_U\rightarrow {\cal L}_3$,
 there is a canonically induced  flat connection $\hat{\nabla}$ on ${\cal E}_U$
   that preserves the decomposition.\footnote{{\it String-theoretical remark.}
                                                              This process clearly resembles a Higgs mechanism or a symmetry breaking
  															     induced by $\varphi$.
                                                              The connection $\nabla$ as a field on $X$ splits generically into
								                                 massless components with respect to $\varphi$ and
								                                 massive components with respect to $\varphi$.
								                              $\hat{\nabla}$ keeps the massless part and throw away the massive part.
															  The process can be made consistent
															    only over an open dense subset of ${\Bbb R}^1$.
															  This renders $\hat{\nabla}$	with singularities.
                                                              }% end-footnote
 \begin{itemize}
  \item[$\cdot$] {\it
   As a flat connection with singularities on ${\cal E}$,
   the $\varphi$-induced connetion $\hat{\nabla}$ from $\nabla$
   is pushed forward to a flat connection with singularities, in notation $\varphi_{1\ast}\nabla$,
     on the rank-1 ${\cal O}_Z^{\,\Bbb C}$-module $\varphi_{1\ast}{\cal E}$.}
 \end{itemize}
  	
 Consider now the following four $1$-parameter families $\varphi_t$, $t\in [0,1]$,
   of deformations of the special Lagrangian map $\varphi_1\,$:
  \begin{itemize}
    \item[(1)]
     {\it Family} $\varphi^{(1)}_t$, $t\in [0,1]\;$:
     the $1$-parameter family of special Lagrangian maps defined by the ring-homomorphism
     $$
      \begin{array}{cccccl}
       \varphi^{(1)\, \sharp}_1 & :
 	      &  C^{\infty}({\Bbb C}^1)
          &  \longrightarrow      &  M_{3\times 3}(C^{\infty}({\Bbb R}^1)^{\Bbb C})\\[1.2ex]
	    &&  y^1  & \longmapsto  &  x \cdot  \Id_{3\times 3}\\[1.2ex]
	    && y^2  & \longmapsto
	       & \left[  \begin{array}{ccc} -tx & 0 & 0 \\ t & t  & 0 \\ 0 & t & tx  \end{array} \right]		   & .	
      \end{array}
     $$

    \item[(2)]
     {\it Family} $\varphi^{(2)}_t$, $t\in [0,1]\;$:
     the $1$-parameter family of special Lagrangian maps defined by the ring-homomorphism
     $$
      \begin{array}{cccccl}
       \varphi^{(2)\, \sharp}_1 & :
 	      &  C^{\infty}({\Bbb C}^1)
          &  \longrightarrow      &  M_{3\times 3}(C^{\infty}({\Bbb R}^1)^{\Bbb C})\\[1.2ex]
	    &&  y^1  & \longmapsto  &  x \cdot  \Id_{3\times 3}\\[1.2ex]
	    && y^2  & \longmapsto
	       & \left[  \begin{array}{ccc} -tx & 0 & 0 \\ t & t  & 0 \\ 0 & 1 & tx  \end{array} \right]		   & ,
      \end{array}
     $$
	
   \item[(3)]
     {\it Family} $\varphi^{(3)}_t$, $t\in [0,1]\;$:
     the $1$-parameter family of special Lagrangian maps defined by the ring-homomorphism
     $$
      \begin{array}{cccccl}
       \varphi^{(3)\, \sharp}_1 & :
 	      &  C^{\infty}({\Bbb C}^1)
          &  \longrightarrow      &  M_{3\times 3}(C^{\infty}({\Bbb R}^1)^{\Bbb C})\\[1.2ex]
	    &&  y^1  & \longmapsto  &  x \cdot  \Id_{3\times 3}\\[1.2ex]
	    && y^2  & \longmapsto
	       & \left[  \begin{array}{ccc} -tx & 0 & 0 \\ 1 & t  & 0 \\ 0 & t & tx  \end{array} \right]		   & ,
      \end{array}
     $$
	
  \item[(4)]
     {\it Family} $\varphi^{(4)}_t$, $t\in [0,1]\;$:
     the $1$-parameter family of special Lagrangian maps defined by the ring-homomorphism
     $$
      \begin{array}{cccccl}
       \varphi^{(4)\, \sharp}_1 & :
 	      &  C^{\infty}({\Bbb C}^1)
          &  \longrightarrow      &  M_{3\times 3}(C^{\infty}({\Bbb R}^1)^{\Bbb C})\\[1.2ex]
	    &&  y^1  & \longmapsto  &  x \cdot  \Id_{3\times 3}\\[1.2ex]
	    && y^2  & \longmapsto
	       & \left[  \begin{array}{ccc} -tx & 0 & 0 \\ 1 & t  & 0 \\ 0 & 1 & tx  \end{array} \right]		   & .	
      \end{array}
     $$	 		
 \end{itemize}
 Their individual details follow from a similar discussion as for $\varphi_1$
  and are summarized below.

 \bigskip

 \noindent
 {\it $(1)$ Family $\varphi^{(1)}_t$}, $t\in [0,1]\;$:
  
 \medskip

 \noindent
 Over $U:={\Bbb R}^1-\{-1,0,1\}$ and for $t\in (0,1]$,
 $\varphi^{(1)}_t$ induces a splitting
  ${\cal E}_U
      ={\cal L}^{(1)}_{1,t}\oplus {\cal L}^{(1)}_{2,t}\oplus{\cal L}^{(1)}_{3,t}$
  into a direct sum of ${\cal O}_U^{\,\Bbb C}$-modules of rank $1$,
   generated respectively by the three sections
  $$
    e^{(1)}_{1,t}\;
	=\; \left(
	       \begin{array}{c}1 \\[.6ex]  -\,\frac{1}{x+1}\\[1.2ex]  \frac{1}{2x(x+1)}\end{array}
		  \right)\,,
	 \hspace{2em}	
   e^{(1)}_{2,t}\;
	=\; \left(
	       \begin{array}{c}0 \\[.6ex]    1 \\[.6ex]    -\,\frac{1}{x-1}\end{array}
		  \right)\,,
	 \hspace{2em}	
  e^{(1)}_{3,t}\;
	=\; \left(
	       \begin{array}{c}0  \\[.6ex]   0 \\[.6ex]    1  \end{array}
		  \right)
  $$
  of ${\cal E}_U$.
 The smooth locus $Z^{(1)}_{t,\smoothscriptsize}$ of
   $Z^{(1)}_t := \Image\varphi^{(1)}_t$
   is now supported in the union
    $$
	  L_{t, -\arctan t}\, \cup\,  L_{t,0}\, \cup\,  L_{t,\arctan t}
	$$
	of special Lagrangian lines in ${\Bbb C}^1$,
 where 	
   $$
    L_{t, -\arctan t}\;=\; \{(x,-tx)\,|\, x\in{\Bbb R}^1 \}\,, \hspace{1em}
	L_{t, 0}\; =\; \{(x,t)\,|\, x\in{\Bbb R}^1 \}\,, \hspace{1em}
	L_{t, \arctan t}\;=\; \{(x,tx)\,|\, x\in{\Bbb R}^1 \}
   $$
  with phases $-\arctan t$, $0$, $\arctan t$ respectively.
 The push-forward $\varphi^{(1)}_{t\,\ast}({\cal E},\nabla)$ is a rank-$1$
  ${\cal O}_{Z^{(1)}_t}^{\,\Bbb C}$-module with a flat connection with singularities.
  
 When $t=0$,
  $\varphi^{(1)}_t$ is deformed to $\varphi^{(1)}_0$ defined by the ring-homomorphism
    $$
      \begin{array}{cccccl}
       \varphi^{(1)\, \sharp}_0 & :
 	      &  C^{\infty}({\Bbb C}^1)
          &  \longrightarrow      &  M_{3\times 3}(C^{\infty}({\Bbb R}^1)^{\Bbb C})\\[1.2ex]
	    &&  y^1  & \longmapsto  &  x \cdot  \Id_{3\times 3}\\[1.2ex]
	    && y^2  & \longmapsto
	       &   0_{\,3\times 3}  & ,	
      \end{array}
     $$
  where $0_{\,3\times 3}$ is the $3\times 3$ zero-matrix.	
  $Z^{(1)}_0 := \Image\varphi^{(1)}_0$
   is now reduced and coincides with the real axis
   $$
      L_{0,0}\; =\;  \{(y^1,0)\, |\,  y^1\in{\Bbb R}\}
   $$
   of ${\Bbb C}^1$,   a special Lagrangian line in ${\Bbb C}^1$ of phase $0$.
 The push-forward $\varphi^{(1)}_{0\,\ast}({\cal E},\nabla)$
    is isomorphic to the free ${\cal O}_{L_{0,0}}^{\,\Bbb C}$-module
	${\cal O}_{L_{0,0}}^{\,\Bbb C}\otimes_{\Bbb R}{\Bbb C}^{\oplus 3}$
	of rank $3$ with the flat connection defined by $d$.

 \bigskip

 \noindent
 {\it $(2)$ Family $\varphi^{(2)}_t$}, $t\in [0,1]\;$:
 
 \medskip

 \noindent
 Over $U:={\Bbb R}^1-\{-1,0,1\}$ and for $t\in (0,1]$,
 $\varphi^{(2)}_t$ induces a splitting
  ${\cal E}_U
      ={\cal L}^{(2)}_{1,t}\oplus {\cal L}^{(2)}_{2,t}\oplus{\cal L}^{(2)}_{3,t}$
  into a direct sum of ${\cal O}_U^{\,\Bbb C}$-modules of rank $1$,
   generated respectively by the three sections
  $$
    e^{(2)}_{1,t}\;
	=\; \left(
	       \begin{array}{c}1 \\[.6ex]  -\,\frac{1}{x+1}\\[1.2ex]  \frac{1}{2tx(x+1)}\end{array}
		  \right)\,,
	 \hspace{2em}	
   e^{(2)}_{2,t}\;
	=\; \left(
	       \begin{array}{c}0 \\[.6ex]    1 \\[.6ex]    -\,\frac{1}{t(x-1)}\end{array}
		  \right)\,,
	 \hspace{2em}	
  e^{(2)}_{3,t}\;
	=\; \left(
	       \begin{array}{c}0  \\[.6ex]   0 \\[.6ex]    1  \end{array}
		  \right)
  $$
  of ${\cal E}_U$.
 The smooth locus $Z^{(2)}_{t,\smoothscriptsize}$ of
   $Z^{(2)}_t := \Image\varphi^{(2)}_t$
   is supported also in
    $L_{t, -\arctan t}\cup L_{t,0}\cup L_{t,\arctan t}$.
 The push-forward $\varphi^{(2)}_{t\,\ast}({\cal E},\nabla)$ is a rank-$1$
  ${\cal O}_{Z^{(2)}_t}^{\,\Bbb C}$-module with a flat connection with singularites.
   
 When $t=0$,
   $\varphi^{(2)}_t$ is deformed to $\varphi^{(2)}_0$ defined by the ring-homomorphism
    $$
      \begin{array}{cccccl}
       \varphi^{(2)\, \sharp}_0 & :
 	      &  C^{\infty}({\Bbb C}^1)
          &  \longrightarrow      &  M_{3\times 3}(C^{\infty}({\Bbb R}^1)^{\Bbb C})\\[1.2ex]
	    &&  y^1  & \longmapsto  &  x \cdot  \Id_{3\times 3}\\[1.2ex]
	    && y^2  & \longmapsto
           & \left[  \begin{array}{ccc} 0 & 0 & 0 \\  0 & 0  & 0 \\ 0 & 1 & 0  \end{array} \right]		 & .
      \end{array}
     $$
  It follows that
   ${\cal E}$ splits to a direct sum of
     $C^{\infty}({\Bbb C}^1)$-invariant free ${\cal O}_{{\Bbb R}^1}^{\,\Bbb C}$-modules,
    in notation ${\cal E}= {\cal L}\oplus {\cal F}$, with ${\cal L}$ of rank $1$ and ${\cal F}$ of rank $2$.
  Furthermore,  ${\cal F}$ admits a filtration ${\cal F}^{\bullet}:{\cal F}^1\subset {\cal F}^2={\cal F}$
    with both ${\cal F}^1$ and ${\cal F}/{\cal F}^1$
	   free ${\cal O}_{{\Bbb R}^1}^{\,\Bbb C}$-modules of rank $1$.
  By construction,  both ${\cal L}$ and ${\cal F}^{\bullet}$ are invariant under $\nabla$ as well;
  thus, in particular, one has
   $$
     ({\cal E},\nabla)\;
         =\; ({\cal L},\nabla^{\cal L})\oplus ({\cal F},\nabla^{\cal F})\,.
   $$		

 With respect to this decomposition
   $\varphi^{(2)\,\sharp}_0$ factors through
   $$
     \xymatrix{
      C^{\infty}({\Bbb C}^1)
                   \ar[rr]^-{(\varphi^{(2)\,\sharp}_{0,1}\,,\,\varphi^{(2)\,\sharp}_{0,2})}
			   \ar @/_4ex/[rrrr]_-{\varphi^{(2)\,\sharp}_0}
        &&  M_{1\times 1}(C^{\infty}({\Bbb R}^1)^{\Bbb C})
		        \times M_{2\times 2}(C^{\infty}({\Bbb R}^1)^{\Bbb C})\,
                        \ar @{^{(}->}[rr]^-{\iota}
		&&   M_{3\times 3}(C^{\infty}({\Bbb R}^1)^{\Bbb C})\,,
	  }
	$$
    where
	 $$
      \begin{array}{cccccl}
       \varphi^{(2)\, \sharp}_{0,1} & :
 	      &  C^{\infty}({\Bbb C}^1)
          &  \longrightarrow      &  M_{1\times 1}(C^{\infty}({\Bbb R}^1)^{\Bbb C})\\[1.2ex]
	    &&  y^1  & \longmapsto  &  x \cdot  \Id_{1\times 1}\\[1.2ex]
	    && y^2  & \longmapsto   &  0_{1\times 1}  & ,
      \end{array}
     $$
     $$
      \begin{array}{cccccl}
       \varphi^{(2)\, \sharp}_{0,2} & :
 	      &  C^{\infty}({\Bbb C}^1)
          &  \longrightarrow      &  M_{2\times 2}(C^{\infty}({\Bbb R}^1)^{\Bbb C})\\[1.2ex]
	    &&  y^1  & \longmapsto  &  x \cdot  \Id_{2\times 2}\\[1.2ex]
	    &&  y^2  & \longmapsto
           & \left[  \begin{array}{cc} 0  & 0 \\ 1 & 0  \end{array} \right]		 & ,
      \end{array}
     $$ 		
     $M_{1\times 1}(C^{\infty}({\Bbb R}^1)^{\Bbb C})
		        \times  M_{2\times 2}(C^{\infty}({\Bbb R}^1)^{\Bbb C})$
	    is the product ring, and 		
     $\iota$ is the inclusion of the 1-2 block-diagonal subring.
 As a consequence, one has a direct-sum decomposition
  $$
    \varphi^{(2)}_{0\,\ast}({\cal E},\nabla)\;
    =\; \varphi^{(2)}_{0,1\, \ast}({\cal L},\nabla^{\cal L})\,
        	\oplus\,   \varphi^{(2)}_{0,2\, \ast}({\cal F},\nabla^{\cal F})\,.
  $$

 By construction,  $\varphi^{(2)}_{0,1\,\ast}({\cal L},\nabla^{\cal L})$
    is a free ${\cal O}_{L_{0,0}}^{\,\Bbb C}$-module of rank $1$ with a flat connection.
  
 As for the $\varphi^{(2)}_{0,2\, \ast}({\cal F},\nabla^{\cal F})$-component,
  $Z^{(2)}_{0,2} := \Image\varphi^{(2)}_{0,2}$
    is now non-reduced, with multiplicity $2$, along the special Lagrangian line $L_{0,0}$.
 The push-forward connection $\varphi^{(2)}_{0,2\,\ast}\nabla^{\cal F}$
   on $\varphi^{(2)}_{0,2\,\ast}{\cal E}$ is defined by the application of the following:
   \begin{itemize}
    \item[$\cdot$] {\bf [push-forward connection on push-forward sheaf -- special case]}\footnote{This
              	                                                      is a notion that is required to address A-branes in string theory.
										                             So far in this project, we only discussed and constructed
                                                                      the push-forward connection case by case --
																	   cf. [L-Y6] (D(7)) and the push-forward $\varphi_{1\ast}\nabla$
																	  for the case of $\varphi_1$ in this example --
	 																   in a mathematically most natural way and
																	      partly guided by the behavior of gauge field on a D-brane
																		  as generated by open string end-points,
																	    without a full general theory of it.
																	 It should be developed more thoroughly in the future.}\\[.6ex]
	 Let
	   $\varphi:
	      (X,{\cal O}_X^{A\!z}:=\Endsheaf_{{\cal O}_X^{\,\Bbb C}}({\cal E}),{\cal E})
		    \rightarrow   Y$ be $C^k$-map, $k\ge 1$,    and
	   $$      			
		 \xymatrix{	
		   \;X_{\varphi}\;\ar[rr]^-{f_{\varphi}} \ar[d]^-{\pi_{\varphi}} && Y \\
		   \;X\; 	
		  }	
       $$
	   be the underlying maps from the surrogate $X_{\varphi}$ of $X^{\!A\!z}$ specified by $\varphi$.
     Assume that
	  \begin{itemize}
	    \item[(1)]
    	   $f_{\varphi}:X_{\varphi}\rightarrow Y$ is an embedding,
		
	    \item[(2)]
           the restriction $\pi_{\varphi}: (X_{\varphi})_{\redscriptsize}\rightarrow X$
           is a $C^k$-diffeomorphism.		
      \end{itemize}		
     Let ${\cal E}$ be a connection on ${\cal E}$.
	 Then the push-forward connection $\varphi_{\ast}\nabla$ on $\varphi_{\ast}{\cal E}$
	   is defined by
	   $$
	     (\varphi_{\ast}\nabla)_ \xi\, s\;
     		 :=\;    \nabla_{\pi_{\ast}(f_{\varphi}^{\ast}\xi) }\,  s
	   $$
	   for all $\xi\in {\cal T}_{\ast}\Image\varphi$ and $s\in \varphi_{\ast}{\cal E}$.
	 Here, we identify a local section $s$ of $\varphi_{\ast}{\cal E}$	
	   canonically as a local section, also denoted by $s$, of ${\cal E}$.			
     Furthermore,
	  assuming in addition that
	  $X_{\varphi}$ is a product space with $\pi_{\varphi}$ a projection map,
	  then:
        if $\nabla$ is a flat connection on ${\cal E}$,
 	      then $\varphi_{\ast}\nabla$ is a flat connection on $\varphi_{\ast}{\cal E}$.	
   \end{itemize}
  
 Let ${\cal I}_{L_{0,0}}\subset {\cal O}_{Z^{(2)}_{(0,2)}}$
    be the nilpotent ideal sheaf of ${\cal O}_{Z^{(2)}_{0,2}}$.
 It is generated by $y^2$ and	has the property that ${\cal I}_{L_{0,0}}^2=0$.
 Then,
   \begin{itemize}
    \item[$\cdot$]
	 The filtration
	   ${\cal I}_{L_{0,0}}\varphi^{(2)}_{0,2\,\ast}{\cal F}
	       \subset \varphi^{(2)}_{0,2\,\ast}{\cal F} $
       coincides with the filtration $\varphi_{\ast}{\cal F}^{\bullet}$.

    \item[$\cdot$]	
	 The above filtration is invariant under $\varphi^{(2)}_{0,2\,\ast}\nabla$.
   \end{itemize}

 \bigskip

 \noindent
 {\it $(3)$ Family $\varphi^{(3)}_t$}, $t\in [0,1]\;$:
 
 \medskip

 \noindent
 Over $U:={\Bbb R}^1-\{-1,0,1\}$ and for $t\in (0,1]$,
 $\varphi^{(3)}_t$ induces a splitting
  ${\cal E}_U
      ={\cal L}^{(3)}_{1,t}\oplus {\cal L}^{(3)}_{2,t}\oplus{\cal L}^{(3)}_{3,t}$
  into a direct sum of ${\cal O}_U^{\,\Bbb C}$-modules of rank $1$,
   generated respectively by the three sections
  $$
    e^{(3)}_{1,t}\;
	=\; \left(
	       \begin{array}{c}1 \\[.6ex]  -\,\frac{1}{t(x+1)}\\[1.2ex]  \frac{1}{2tx(x+1)}\end{array}
		  \right)\,,
	 \hspace{2em}	
   e^{(3)}_{2,t}\;
	=\; \left(
	       \begin{array}{c}0 \\[.6ex]    1 \\[.6ex]    -\,\frac{1}{x-1}\end{array}
		  \right)\,,
	 \hspace{2em}	
  e^{(3)}_{3,t}\;
	=\; \left(
	       \begin{array}{c}0  \\[.6ex]   0 \\[.6ex]    1  \end{array}
		  \right)
  $$
  of ${\cal E}_U$.
 The smooth locus $Z^{(3)}_{t,\smoothscriptsize}$ of
   $Z^{(3)}_t := \Image\varphi^{(3)}_t$
   is supported also in
    $L_{t, -\arctan t}\cup L_{t,0}\cup L_{t,\arctan t}$.
 The push-forward $\varphi^{(3)}_{t\,\ast}({\cal E},\nabla)$ is a rank-$1$
  ${\cal O}_{Z^{(3)}_t}^{\,\Bbb C}$-module with a flat connection with singularities.
    
 When $t=0$,
   $\varphi^{(3)}_t$ is deformed to $\varphi^{(3)}_0$ defined by the ring-homomorphism
    $$
      \begin{array}{cccccl}
       \varphi^{(3)\, \sharp}_0 & :
 	      &  C^{\infty}({\Bbb C}^1)
          &  \longrightarrow      &  M_{3\times 3}(C^{\infty}({\Bbb R}^1)^{\Bbb C})\\[1.2ex]
	    &&  y^1  & \longmapsto  &  x \cdot  \Id_{3\times 3}\\[1.2ex]
	    && y^2  & \longmapsto
           & \left[  \begin{array}{ccc} 0 & 0 & 0 \\  1 & 0  & 0 \\ 0 & 0 & 0  \end{array} \right]		 & .
      \end{array}
     $$
 Up to a relabling, this is the same situation as $\varphi^{(2)}_0$.
 Thus,  we have in particular a decomposition
  $$
    \varphi^{(3)}_{0\,\ast}({\cal E},\nabla)
     = \varphi^{(3)}_{0,2\,\ast}({\cal F},\nabla^{\cal F})
	     \oplus \varphi^{(3)}_{0,1\,\ast}({\cal L}, \nabla^{\cal L})
  $$
   with the summand $\varphi^{(3)}_{0,2\,\ast}{\cal F}\,$
           $\varphi^{(3)}_{0,2\,\ast}\nabla^{\cal F}$-invariantly  filtered.

 \bigskip

 \noindent
 {\it $(4)$ Family $\varphi^{(4)}_t$}, $t\in [0,1]\;$:

 \medskip

 \noindent
 Over $U:={\Bbb R}^1-\{-1,0,1\}$ and for $t\in (0,1]$,
 $\varphi^{(4)}_t$ induces a splitting
  ${\cal E}_U
      ={\cal L}^{(4)}_{1,t}\oplus {\cal L}^{(4)}_{2,t}\oplus{\cal L}^{(4)}_{3,t}$
  into a direct sum of ${\cal O}_U^{\,\Bbb C}$-modules of rank $1$,
   generated respectively by the three sections
  $$
    e^{(4)}_{1,t}\;
	=\; \left(
	       \begin{array}{c}1 \\[.6ex]  -\,\frac{1}{t(x+1)}\\[1.2ex]  \frac{1}{2t^2x(x+1)}\end{array}
		  \right)\,,
	 \hspace{2em}	
   e^{(4)}_{2,t}\;
	=\; \left(
	       \begin{array}{c}0 \\[.6ex]    1 \\[.6ex]    -\,\frac{1}{t(x-1)}\end{array}
		  \right)\,,
	 \hspace{2em}	
  e^{(4)}_{3,t}\;
	=\; \left(
	       \begin{array}{c}0  \\[.6ex]   0 \\[.6ex]    1  \end{array}
		  \right)
  $$
  of ${\cal E}_U$.
 The smooth locus $Z^{(4)}_{t,\smoothscriptsize}$ of
   $Z^{(4)}_t := \Image\varphi^{(4)}_t$
   is supported also in $L_{t, -\arctan t}\cup L_{t,0}\cup L_{t,\arctan t}$.
 The push-forward $\varphi^{(4)}_{t\,\ast}({\cal E},\nabla)$ is a rank-$1$
  ${\cal O}_{Z^{(4)}_t}^{\,\Bbb C}$-module with a flat connection with singularites.
    
 When $t=0$,
   $\varphi^{(4)}_t$ is deformed to $\varphi^{(4)}_0$ defined by the ring-homomorphism
    $$
      \begin{array}{cccccl}
       \varphi^{(4)\, \sharp}_0 & :
 	      &  C^{\infty}({\Bbb C}^1)
          &  \longrightarrow      &  M_{3\times 3}(C^{\infty}({\Bbb R}^1)^{\Bbb C})\\[1.2ex]
	    &&  y^1  & \longmapsto  &  x \cdot  \Id_{3\times 3}\\[1.2ex]
	    && y^2  & \longmapsto
           & \left[  \begin{array}{ccc} 0 & 0 & 0 \\  1 & 0  & 0 \\ 0 & 1 & 0  \end{array} \right]		 & .
      \end{array}
     $$
 It follows from the same construction as that for $\varphi^{(2)}_{0,2}$
  that the push-forward $\varphi^{(4)}_{0\,\ast}{\cal E}$
   is supported  on $Z^{(4)}_0 := \Image\varphi^{(4)}_0$,
    which is non-reduced with multiplicity $3$ along the special Lagrangian line
	$L_{0,0}\subset {\Bbb C}^1$.
 The push-forward connection $\varphi^{(4)}_{0\,\ast}\nabla$
   is defined  on $\varphi^{(4)}_{0\,\ast}{\cal E}$  and
   is flat.
 The nilpotent ideal sheaf ${\cal I}_{L_{0,0}}\subset {\cal O}_{Z^{(4)}_0}$
    is generated by $y^2$ with ${\cal I}_{L_{0,0}}^3=0$.
 The induced filtration
   $$
     0\; \subset\;
    {\cal I}_{L_{0,0}}^2\cdot \varphi^{(4)}_{0\,\ast}{\cal E}\;
      \subset\;  {\cal I}_{L_{0,0}}\cdot \varphi^{(4)}_{0\,\ast}{\cal E}
      \subset\;  \varphi^{(4)}_{0\,\ast}{\cal E}
   $$
   is invariant under $\varphi^{(4)}_{0\,\ast}\nabla$
   in the sense that
   $\,(\varphi^{(4)}_{0\,\ast}\nabla)
        \left(  %\rule{0ex}{1.2ex}
		 {\cal I}_{L_{0,0}}^i\cdot \varphi^{(4)}_{0\,\ast}{\cal E}
		  \right)\,
		 \subset\, {\cal I}_{L_{0,0}}^i\cdot \varphi^{(4)}_{0\,\ast}{\cal E}\,$
	 for $i=1,\,2$.
 
 \bigskip
  
 {To} conclude, we see that
   the special Lagrangian map
       $\varphi_1$ can be deformed in various ways through special Lagrangian maps
    to create $\varphi_0$'s of the same reduced image $L_{0,0}$
	but of different nature: simple, or nilpotently fuzzy, ..., etc.
 	
 Cf.~{\sc Figure}~7-2-1.
 %
 % \marginpar{\raggedright\tiny $\bullet$
 %       {\sc Figure}:  \\  sLag-flat-b-CY1-deform.pdf}

 % \begin{figure} [htbp]
 %  \bigskip
 %  \centering
 %  \includegraphics[width=0.80\textwidth]{sLag-flat-b-CY1-deform.pdf}
 %
 %  \bigskip
 %  \bigskip
 %  \centerline{\parbox{13cm}{\small\baselineskip 12pt
 %   {\sc Figure}~7-2-1.
 %    Various deformations of a special Lagrangian map $\varphi_1$
 %    	from an Azumaya real line with a fundamental module
 % 	      $({\Bbb R}^{1,A\!z},{\cal E})$  to the Calabi-Yau $1$-fold ${\Bbb C}^1$
 %	  are indicated.
 %   For one deformation $\varphi_1\Rightarrow \varphi_0$,
 %    the limit special Lagrangian map $\varphi_0$	has a reduced image,
 %       which supports $\varphi_{0\ast}{\cal E}$.
 %   For another deformation $\varphi_1\Rightarrow\varphi_0^{\prime}$,
 %    the limit special Lagrangian map $\varphi_0^{\prime}$ has a nonreduced image, carrying a nilpotent cloud.
 %   The corresponding $\varphi_{0\ast}^{\prime}{\cal E}$ has an associate filtration.
 %   {From} the target-space aspect, this suggests also a notion of scheme-theoretic-like deformations
 %     of Lagrangian cycles with a generically flat sheaf/local system with singularities.  	
 %   Note that $\varphi^{\prime}_{0\ast}{\cal E}$	is an example of `{\it T-branes}'
 %    in our setting.
 %       }}
 %  \bigskip
 % \end{figure}	
  
 \noindent\hspace{15.7cm}$\square$
}\end{example}

\bigskip

\begin{remark} $[$scheme-theoretic-like deformation of Lagrangian submanifolds$\,]$.  {\rm
 {From} the viewpoint of the target symplectic manifold $Y$,
   Example~7.2.2 gives a scheme-theoretic-like deformations of Lagrangian cycles
                % Example [deformation of special Lagrangian branes with a flat bundle
				%                  in the Calabi_Yau 1-fold ${\Bbb C}^1$]
   with a generically flat sheaf.
 It would be very interesting to see if such a notion can be formulated directly on $Y$,
  without resuming to Lagrangian morphisms from Azumaya manifolds with a fundamental module to $Y$.
}\end{remark}

\bigskip

\begin{remark} {$[$T-brane$\,]$.} {\rm
 In Example~7.2.2, the smooth maps
        % Example [deformation of special Lagrangian branes with a flat bundle
	    %                  in the Calabi_Yau 1-fold ${\Bbb C}^1$]
  $\varphi^{(2)}_{0,2}$, $\varphi^{(3)_{0,2}}$, and $\varphi^{(3)}_0$
  are examples of {\it T-branes},
  that is D-branes with a triangulation/filtration structure on its Chan-Paton module, in our setting.
 In all these cases, $y^2\in C^{\infty}({\Bbb C}^1)$ acts on the push-forward
  $\varphi_{\ast}{\cal F}$ and $\varphi_{\ast}{\cal E}$ in question
  as a nilpotent operator that gives a filtration/triangulation of  $\varphi_{\ast}{\cal F}$
  and  $\varphi_{\ast}{\cal E}$.
 See, for example, [A-H-K].
}\end{remark}

\bigskip
\bigskip
\bigskip
 
 \begin{figure} [htbp]
  \bigskip
  \centering
  \includegraphics[width=0.80\textwidth]{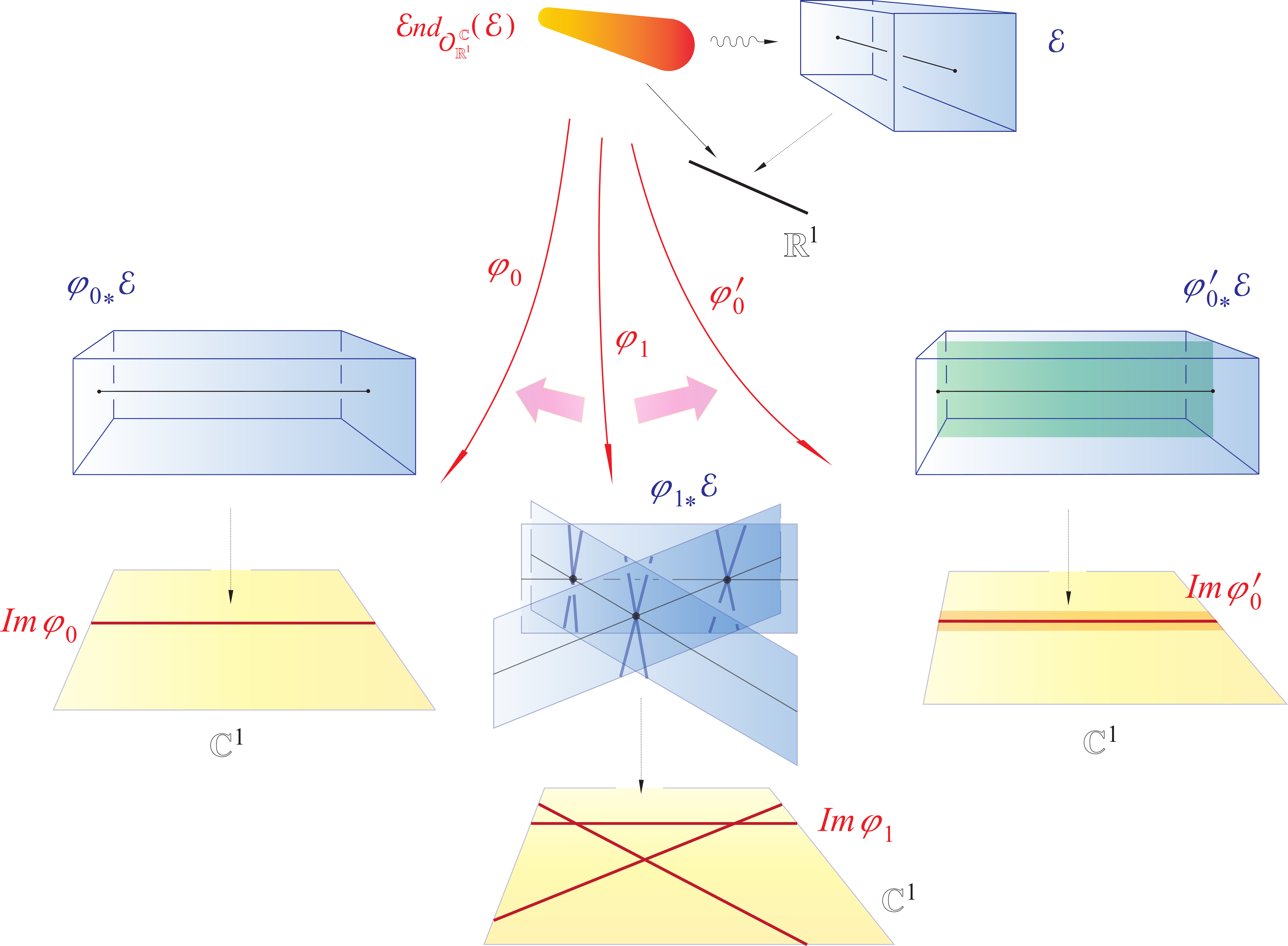}
 
  \bigskip
  \bigskip
  \centerline{\parbox{13cm}{\small\baselineskip 12pt
   {\sc Figure}~7-2-1.
    Various deformations of a special Lagrangian map $\varphi_1$
    	from an Azumaya real line with a fundamental module
 	      $({\Bbb R}^{1,A\!z},{\cal E})$  to the Calabi-Yau $1$-fold ${\Bbb C}^1$
	  are indicated.
   For one deformation $\varphi_1\Rightarrow \varphi_0$,
    the limit special Lagrangian map $\varphi_0$	has a reduced image, which supports $\varphi_{0\ast}{\cal E}$.
   For another deformation $\varphi_1\Rightarrow\varphi_0^{\prime}$,
    the limit special Lagrangian map $\varphi_0^{\prime}$ has a nonreduced image, carrying a nilpotent cloud.
   The corresponding $\varphi_{0\ast}^{\prime}{\cal E}$ has an associate filtration.
   {From} the target-space aspect, this suggests also a notion of scheme-theoretic-like deformations
     of Lagrangian cycles with a generically flat sheaf/local system with singularities.  	
   Note that $\varphi^{\prime}_{0\ast}{\cal E}$	is an example of `{\it T-branes}'
    in our setting.
       }}
  \bigskip
 \end{figure}

\newpage
\baselineskip 13pt
%references
{\footnotesize

\vspace{1em}

\noindent
chienhao.liu@gmail.com, chienliu@math.harvard.edu; \\
  % vafa@physics.harvard.edu;\\
yau@math.harvard.edu

}%endfootnotesize

\end{document}